\documentclass[10pt,a4paper]{article}

\usepackage[tt=false]{libertine}

\usepackage[utf8]{inputenc}
\usepackage{answers}
\usepackage{setspace}

\usepackage[nottoc]{tocbibind}       
\usepackage{graphicx}
\usepackage{enumitem}
\usepackage{multicol}
\usepackage{csquotes}
\usepackage[margin=2cm]{geometry}
\usepackage{tikz-cd}
\usepackage{amsmath,amsthm}
\usepackage{newtxmath}

\usepackage{pst-node}
\usepackage{mathtools}
\usepackage{diagbox}
\usepackage{quiver}
\usepackage{hyperref}
\usepackage{mathpartir}
\usepackage{xcolor}
\usepackage{tcolorbox}
\usepackage{abraces}
\usepackage{relsize}

\usepackage{xcolor}
\definecolor{newpurple}{RGB}{50, 0, 90}

\usepackage{hyperref}
\hypersetup{
	citecolor=newpurple,
	linkcolor=newpurple,
	colorlinks=true,
}

\usepackage[backend=biber,style=alphabetic]{biblatex}
\theoremstyle{definition}
\newtheorem{theorem}{Theorem}[section]
\newtheorem{lemma}[theorem]{Lemma}
\newtheorem{definition}[theorem]{Definition}

\newtheorem*{definition*}{Definition}
\newtheorem*{proposition*}{Proposition}
\newtheorem{proposition}[theorem]{Proposition}

\newtheorem{remark}[theorem]{Remark}
\newtheorem{notation}[theorem]{Notation}
\newtheorem*{notation*}{Notation}

\newtheorem*{theorem*}{Theorem}
\newtheorem*{lemma*}{Lemma}
\newtheorem*{convention*}{Convention}
\newtheorem*{example*}{Example}

\newtheorem{corollary}[theorem]{Corollary}
\newtheorem{construction}[theorem]{Construction}

\numberwithin{equation}{section}

\newcommand{\Tau}{\mathrm{T}}
\newcommand{\Ar}{\text{Ar}}

\newcommand{\PSh}{\text{PSh}}

\newcommand{\Set}{\text{Set}}

\newcommand{\Ob}{\text{Ob}}

\newcommand{\Fun}{\text{Fun}}

\newcommand{\Cat}{\text{Cat}}

\newcommand{\GAT}{\text{GAT}}
\newcommand{\yo}{\mathscr{Y}}
\newcommand{\vdashcustom}{\mkern4mu\vdash\mkern4mu}

\newcommand{\varotimes}{\mathrel{\tikz \fill (0,0) circle (0.65ex);}}

\newcommand{\btimes}{\mathrel{\tikz[baseline=-0.05ex]\draw (0,0) rectangle (1.15ex,1.15ex);}}

\newcommand{\varbtimes}{\mathrel{\tikz[baseline=-0.05ex]\fill (0,0) rectangle (1.2ex,1.2ex);}}

\newcommand{\doverline}[1]{\overline{\overline{#1}}}

\newcommand{\Cont}{\text{Cont}}
\newcommand{\iiCont}{\mathbf{Cont}}

\newcommand{\Mod}{\text{Mod}}

\renewcommand{\lim}{\underleftarrow{\text{lim }}}
\newcommand{\colim}{\underrightarrow{\text{lim }}}

\newcommand{\Type}{\mathcal S}
\newcommand{\tp}{\;\;\mathsf{sort}}
\newcommand{\tm}{\;\;\mathsf{term}}

\newcommand{\Ht}{\textsf{ht}}
\newcommand{\ctx}{\;\;\textsf{ctx}}

\newcommand{\Fam}{\text{Fam}}

\newcommand{\bbA}{\mathbb A}
\newcommand{\bbB}{\mathbb B}
\newcommand{\bbC}{\mathbb C}

\newcommand{\bbE}{\mathbb E}
\newcommand{\bbI}{\mathbb I}

\newcommand{\bbK}{\mathbb K}

\newcommand{\bbN}{\mathbb N}
\newcommand{\bbO}{\mathbb O}

\newcommand{\bbS}{\mathbb S}
\newcommand{\bbT}{\mathbb T}

\newcommand{\bbl}{\boldsymbol{\ell}}

\newcommand{\vertequiv}{\mathbin{\rotatebox[origin=c]{90}{$\equiv$}}}

\newcommand{\squa}[8]{\begin{tikzcd}[ampersand replacement=\&]
		{#1} \& {#2} \\
		{#3} \& {#4}
		\arrow["{#7}"', from=1-1, to=2-1]
		\arrow["{#6}"', from=2-1, to=2-2]
		\arrow["{#8}", from=1-2, to=2-2]
		\arrow["{#5}", from=1-1, to=1-2]
	\end{tikzcd}
}

\newcommand{\widesqua}[8]{\begin{tikzcd}[ampersand replacement=\&,column sep=2.5cm]
		#1 \arrow[swap]{d}{#7} \arrow[]{r}{#5} \& #2 \arrow[]{d}{#8}\\
		#3 \arrow[swap]{r}{#6} \& #4
	\end{tikzcd}
}

\newcommand*\uphalfsquaredraw[1][]{\tikz[#1]{ 
		\draw [line width=0.05em] (-0.25em,0.25em) rectangle (0.25em,-0.25em); 
		\draw [line width=0em, fill] (-0.25em,0em) rectangle (0.25em,0.25em);
}}

\newcommand{\uphalfsquare}{\mathrel{\uphalfsquaredraw}}

\makeindex

\begin{document}

\title{A monoidal category of dependently sorted algebraic theories\\
I: \emph{syntax}}

\author{Daniel Almeida\thanks{email: ddeal056@uottawa.ca}}

\date{}

\maketitle

\begin{abstract}
This is the first of a pair of papers where we construct and investigate a closed monoidal structure on the category of generalized algebraic theories (in the sense of Cartmell).

In the present text, as a starting point, we define the \emph{tensor product}, $\mathbb A \otimes \mathbb B$, between two generalized algebraic theories $\mathbb A$ and $\mathbb B$. This is done syntactically via an algorithm that uses the axioms of $\mathbb A$ and $\mathbb B$ in a recursive manner to produce those of $\mathbb A \otimes \mathbb B$. We provide examples of known structures that are recovered by our construction, such as tensor products of Lawvere theories, ``cellular" products of dependent type signatures, and theories of diagrams and of displayed structures. It will be verified in the second volume that, as suggested by these special cases, the category of family-valued models $\text{Mod}(\mathbb A \otimes \mathbb B,\text{Fam})$ is isomorphic to $\text{Mod}(\mathbb A,\textbf{Mod}(\mathbb B))$ and to $\text{Mod}(\mathbb B,\textbf{Mod}(\mathbb A))$ for certain contextual categories $\textbf{Mod}(\mathbb A)$ and $\textbf{Mod}(\mathbb B)$ whose underlying categories are equivalent to $\text{Mod}(\mathbb A,\text{Fam})$ and to $\text{Mod}(\mathbb B,\text{Fam})$, respectively. Moreover, the cellular structure of the tensor product is obtained by combining, via a pushout-product operation, those of the two theories.

We also construct a functor $\otimes_{\mathbb A,\mathbb B}:\mathcal C(\mathbb A) \times \mathcal C(\mathbb B) \rightarrow \mathcal C(\mathbb A \otimes \mathbb B)$ comparing the associated contextual categories, and describe isomorphisms of the forms $(\mathbb A \otimes \mathbb B) \otimes \mathbb{C} \cong \mathbb A \otimes (\mathbb B \otimes \mathbb{C})$ and $\mathbb{A} \otimes \mathbb{B} \cong \mathbb{B} \otimes \mathbb{A}$. In the sequel paper we will describe a universal property of $\otimes_{\mathbb A,\mathbb B}$, which will induce functoriality of the tensor product and thus allow us to check the monoidal category conditions.
\end{abstract}

\setcounter{tocdepth}{2}
\tableofcontents

\section{Introduction}

Many kinds of algebraic structures used in mathematics, such as groups, rings, and vector spaces over a given field, can be described as models of equational theories: they can be specified from a set of operation symbols and a set of equations between pairs of expressions built recursively from operation symbols and variables. The study of such classes of structures is the main theme of classical universal algebra, as developed by A. N. Whitehead, G. Birkhoff and others from the late 19th century to the first half of the 20th century. In 1963, in his doctoral thesis (reprinted as \cite{Law04}), W. Lawvere recast the foundations of the subject in terms of category theory by showing how equational theories and their semantics can be encoded via a certain class of categories; these were originally named \emph{algebraic theories}, and later also became known as \emph{Lawvere theories}. 

This approach is motivated by the correspondence, for a theory $\bbA$, between (1) homomorphisms between $\bbA$-models free on a finite number of generators, and (2) finite sequences of ``definable operations" (constructed from variables and the theory's basic operations) taken up to provable equality (using the theory's axioms). A consequence of this is that the category of models $\Mod(\bbA)$ can be reconstructed from its full subcategory, say $A$, whose objects are coproducts of finitely many copies of the free model on one generator. More precisely, $\Mod(\bbA)$ is equivalent to the full subcategory of $\Set^{A^{op}}$ spanned by the functors $M:A^{op} \rightarrow \Set$ that preserve finite products. At the level of objects, this equivalence is related to the fact that a $\bbA$-model is specified by two pieces of data: its underlying set, say $X$, and a coherent choice of a map $X^n \rightarrow X$ for each definable $n$-ary operation, with the requirement that variables be appropriately interpreted as cartesian projection maps. Now, it still makes sense to consider these data if, instead of a set, $X$ is an object of an arbitrary category with finite products $C$. This is an \emph{$\bbA$-model in $C$}, which can also be described by a finite-product-preserving functor $A^{op} \rightarrow C$. Thus one can, for some purposes, replace the syntactic treatment of the given class of structures by a purely semantic and presentation-independent one; and this allows for a simple characterization, in categorical language, of $\bbA$-models and homomorphisms valued in categories other than $\Set$.

\subsubsection*{Notions of finite-limit theory}

The categorical perspective on model theory was further developed, in particular, through the study of sketches (initiated by Ehresmann in the 1960s; see e.g. \cite{BasEhr72,MakPar89,AdaRos94}), which are devices used for specifying mathematical structures internal to a category $C$ as functors $I \rightarrow C$, where $I$ is a directed graph, that map certain diagrams (resp. cones, cocones) in $I$ to commutative diagrams (resp. limit cones, colimit cocones) in $C$. A particularly relevant example is that of \emph{finite-limit sketches}: those where one only specifies in $I$, in addition to a set of ``to-be-commutative" diagrams, a class of finite cones (and no cocones). In \cite{GabUlm71}, Gabriel and Ulmer characterized (in different terminology) the class of categories that occur as categories of $\Set$-valued models of finite-limit sketches: those are the \emph{locally finitely presentable} categories, which are the (cocomplete) categories that occur by completing some small finitely cocomplete category under filtered colimits. In fact, a locally finitely presentable category $\mathcal M$ occurs as the category of finite-limit-preserving functors $\mathcal M_{lp}^{op} \rightarrow \Set$ where $\mathcal M_{lp} \subset \mathcal M$ is the full subcategory spanned by the finitely presentable objects; moreover, this construction extends to an equivalence of 2-categories $\textbf{Lex}^{op} \simeq \textbf{LPr}_\omega$, known as Gabriel-Ulmer duality, where $\textbf{Lex}$ consists of finitely complete categories, finite-limit-preserving functors, and natural transformations, and $\textbf{LPr}_\omega$ of locally finitely presentable categories, filtered-colimit-preserving right adjoint functors, and natural transformations. A finite-limit sketch is then one particular way of presenting a locally finitely presentable category, and this passage only retains the semantic aspects of the sketch (more precisely, it retains enough information to define models of the sketch, as well as morphisms between them, in any category with finite limits). Around the same time, Freyd (see \cite{Fre72}) also considered a notion of ``essentially algebraic" theory, which generalize multisorted algebraic theories by allowing partial operations whose domain can be specified via equations involving previously constructed terms. Up to equivalence, the categories of models of these are also the locally finitely presentable categories (see \cite{AdaRos94}, 3.D). Further classes of logical structures that have the same categories of models as finite-limit theories, such as partial Horn theories, have been studied in \cite{PalVic07}.

\subsubsection*{Generalized algebraic theories and contextual categories}

Generalized algebraic theories (\textsc{gat}s for short), introduced by J. Cartmell in his doctoral thesis (\cite{Car78}; see also \cite{Car86}), make use of the structural rules of Martin-Löf's type theory to extend the framework of multisorted algebraic theories by allowing dependent sorts (also known as dependent types), that is, sorts (often thought of as representing sets or spaces) that depend on parameters -- given by terms (often thought of as representing elements or points) -- coming from other dependent sorts. As explained in Cartmell's work, the category of models of any essentially algebraic theory can be obtained, up to equivalence, as the category of models of a generalized algebraic theory, and vice versa. In this particular sense, finite-limit sketches, essentially algebraic theories and generalized algebraic theories have the same expressive power. However, using the latter to describe a given kind of structure is often natural due to how, in many cases, one specified the structure in terms of (total) operations and equations involving elements of sets parameterized by previously introduced elements. For example, one can construct a category by specifying a set $\Ob$ of objects, a set $\Ar$ of arrows endowed with domain and codomain maps $\Ar \rightarrow \Ob$, a composition operation $\circ:\Ar \times_\Ob \Ar \rightarrow \Ar$ such that the diagram
\[\begin{tikzcd}
	{\Ar \times_\Ob (\Ar \times_\Ob \Ar)} & {(\Ar \times_\Ob \Ar) \times_\Ob \Ar} & {\Ar \times_\Ob \Ar} \\
	{\Ar \times_\Ob \Ar} && \Ar
	\arrow["\cong", from=1-1, to=1-2]
	\arrow["{id \times \circ}"', from=1-1, to=2-1]
	\arrow["{\circ \times id}", from=1-2, to=1-3]
	\arrow["\circ", from=1-3, to=2-3]
	\arrow["\circ"', from=2-1, to=2-3]
\end{tikzcd}\]
commutes, etc. In the framework of generalized algebraic theories, one can specify a set $\Ob$ objects, a set of arrows $\Ar(x,y)$ for each $x$, $y \in \Ob$, an element $g \circ f \in \Ar(x,z)$ for each $f \in \Ar(x,y)$, $g \in \Ar(y,z)$, etc.; the associativity axiom becomes the equality $h \circ (g \circ f) = (h \circ g) \circ f$ where $f \in \Ar(x,y)$, $g \in \Ar(y,z)$, $h \in \Ar(z,w)$. Formally, this is done by considering \emph{judgments} that tell us which sorts, terms and equalities can be obtained from a list of variables of previously constructed sorts. In the above example, we could take
\begin{align*}
	& \vdashcustom \Ob \tp,\\
	x, y:\Ob & \vdashcustom \Ar(x,y) \tp,\\
	x, y, z:\Ob, f:\Ar(x,y), g:(y,z) & \vdashcustom \circ(x,y,z,f,g): \Ar(x,z),
\end{align*}
and, abbreviating $\circ(x,y,z,f,g)$ as $g \circ f$,
$$
x, y, z, w:\Ob, f:\Ar(x,y), g:\Ar(y,z), h:\Ar(z,w) \vdashcustom h \circ (g \circ f) \equiv (h \circ g) \circ f : \Ar(x,w).
$$
A list of variables which is correctly typed (in a sense to be made precise), such as the one preceding $\vdashcustom$ in each case, is called a \emph{context}. Using sequences of terms to give an appropriate definition of morphism, contexts can be assembled into a category similarly to how one constructs the Lawvere theory associated with an equational theory.

One of the main achievements of Cartmell's work was the identification of a class of algebraic structures, called \emph{contextual categories}, that encode the essential structure needed to study interpretations between different generalized algebraic theories (which, in turn, is sufficient to encode the usual set-valued models of a theory). In more detail, each theory $\bbA$ has a {syntactic category}, denoted by $\mathcal C(\bbA)$, whose objects are contexts and whose arrows are context morphisms, both taken up to provable equality. This category inherits more structure from $\bbA$, namely: (i) each object has a \emph{length}, given by the length of a corresponding context; (ii) dependent sorts define a certain class of arrows, the (length-$1$) \emph{display maps}; (iii) substitution of dependent sorts along context morphisms define certain pullback squares in $\mathcal C(\bbT)$, which we refer to as (length-$1$) \emph{distinguished squares}. A contextual category is a category with additional structure intended to encapsulate (i)-(iii) above (as well as several related properties) in an axiomatic way. By suitably assembling the classes of generalized algebraic theories and of contextual categories into categories $\GAT$ and $\Cont$, respectively, the assignment $\bbA \mapsto \mathcal C(\bbA)$ extends to an equivalence of categories $\mathcal C:\GAT \rightarrow \Cont$.

\subsubsection*{Tensor products of theories}

Shortly after Lawvere's introduction of algebraic theories, P. Freyd (\cite{Fre66}) described a way of combining two such theories $\bbA$ and $\bbB$ into a theory $\bbA \otimes \bbB$ such that for any category with finite products $C$,
$$
\Mod(\bbA, \Mod(\bbB,C)) \simeq \Mod(\bbA \otimes \bbB, C) \simeq \Mod(\bbB, \Mod(\bbA,C)).
$$
In summary, $\bbA \otimes \bbB$ contains (1) one ($n$-ary) operation for each ($n$-ary) operation in $\bbA$ or in $\bbB$, (2) all axioms from $\bbA$ or from $\bbB$, and (3) for an $m$-ary operation $f$ in $\bbA$ and an $n$-ary operation $g$ in $\bbB$, an axiom
$$
f(g(x_{11}, ..., x_{1n}), ..., g(x_{m1}, ..., x_{mn})) = g(f(x_{11}, ..., x_{m1}), ..., f(x_{1n}, ..., x_{mn}))
$$
between $mn$-operations expressed in terms of variables $x_{ij}$ for $1 \le i \le m$ and $1 \le j \le n$. This equality states a form of commutativity between $f$ and $g$; by making $f$ and $g$ range over all possible operations, we conclude that a model of $\bbA \otimes \bbB$ in a category with finite limits $C$ is an object $X$ endowed with a structure of $\bbA$-model and a structure of $\bbB$-model such that if $f$ (resp. $g$) is as above, then $f:X^m \rightarrow X$ (resp. $g:X^n \rightarrow X$) is a morphism of $\bbB$-models (resp. of $\bbA$-models).

A similar construction is given for finitely cocomplete categories (an in fact for much more general classes of enriched categories) by Kelly in \cite{Kel82}, and its counterpart for locally finitely presentable categories is discussed in Bird's doctoral thesis \cite{Bir84}.

A tensor product of sketches, possibly specifying both limits and colimits, has also been defined in the same vein; see \cite{Age92}, \cite{AdaRos94} (Exercise 1.l), \cite{Ben97}. For sketches $I$ and $J$, one has equivalences $\Mod(I,\Mod(J,C)) \simeq \Mod(I \otimes J,C) \simeq \Mod(J,\Mod(I,C))$ provided that the co/limits specified by $I$ and by $J$ commute in an appropriate sense, and that condition always holds when $I$, $J$ are limit sketches. If $I$, $J$ are finite-limit sketches, then Bird's tensor product applied to the respective categories of $\Set$-valued models satisfies $\Mod(I) \otimes \Mod(J) \simeq \Mod(I \otimes J)$.

\subsection*{Contributions of the present work}

In this article, we define the \emph{tensor product} $\bbA \otimes \bbB$ of two generalized algebraic theories $\bbA$ and $\bbB$. To motivate the construction, it is useful to think of it as an intended joint generalization of
\begin{itemize}
	\item the tensor product of Lawvere theories (\cite{Fre66});

	\item the cartesian product of locally finite direct categories, which are a categorical model of generalized algebraic theories that only have sort axioms (the \emph{dependent type signatures} from \cite{Sub21}).
\end{itemize}

Let us first explain the second point. A locally finite direct category (lfdc for short) $C$ corresponds to a generalized algebraic theory whose contextual category is (up to equivalence) opposite to the category $\PSh_{\text{fp}}(C)$ of finitely presentable presheaves on $C$ (which, in this case, are the presheaves whose category of elements is finite). Each object $c$ encodes a sort axiom whose context is the \emph{boundary} presheaf $\partial c$ obtained by removing the element $id_c$ from $\yo(c) = C(-,c)$. The corresponding display map is opposite to the inclusion $\partial c \rightarrow \yo(c)$. In concrete settings, we usually view $C$ as a category of basic combinatorial shapes to be used as building blocks for more general ones, namely, the objects of $\PSh_{\text{fp}}(C)$: a presheaf $X$ on $C$ is finitely presentable precisely if it fits into a finite sequence $\varnothing = X_0 \rightarrow X_1 \rightarrow \cdots \rightarrow X_n = X$ where each $X_i \rightarrow X_{i+1}$ is a pushout of $\partial c \rightarrow \yo(c)$ for some $c \in C$.\footnote{All small presheaves on $C$ can be obtained in this way if we allow $n$ to be any ordinal. However, the finiteness assumption is needed to establish the connection with generalized algebraic theories as they only support finitary sorts and operations (at least classically; see \cite{Sub21}, Remark 1.2.9, and the infinitary theories from \cite{BarHen25}).}

For lfdcs $C$ and $D$, the category $C \times D$ is also locally finite and direct, and the canonical functor $\PSh(C) \times \PSh(D) \rightarrow \PSh(C \times D)$\footnote{It is obtained by two applications of pointwise left Kan extension: we take the image of the Yoneda embedding $\yo:C \times D \rightarrow \PSh(C \times D)$ along the chain $\PSh(C \times D)^{C \times D} \cong (\PSh(C \times D)^C)^D \simeq \textbf{LPr}(\PSh(C),\PSh(C \times D))^D \simeq \textbf{LPr}(\PSh(D), \textbf{LPr}(\PSh(C),\PSh(C \times D))) \subset (\PSh(C \times D)^{\PSh(C)})^{\PSh(D)} \cong \PSh(C \times D)^{\PSh(C) \times \PSh(D)}$.} restricts to a functor
$$
\otimes:\PSh_{\text{fp}}(C) \times \PSh_{\text{fp}}(D) \rightarrow \PSh_{\text{fp}}(C \times D).
$$
Notably, for $c \in C$, $d \in D$, $X \in \PSh_{\text{fp}}(C)$, $Y \in \PSh_{\text{fp}}(D)$ we have
\begin{align*}
	&\yo(c) \otimes \yo(d) \cong \yo(c, d) \\
	&X \otimes 0 \cong 0 \cong 0 \otimes Y\\
	&\partial(c,d) \cong \partial c \otimes \yo(d) \sqcup_{\partial c \otimes \partial d} \yo(c) \otimes \partial d \tag{Leibniz formula}
\end{align*}
This implies in particular that, writing $|-|$ for the number of elements of a presheaf, $|X \otimes Y| = |X||Y|$ for $X \in \PSh_{\text{fp}}(C)$, $Y \in \PSh_{\text{fp}}(D)$.

When applied to theories that only have sort axioms, we expect the tensor product of \textsc{gat}s to recover $\PSh_{\text{fp}}(C \times D)$ not in terms of the cartesian product $C \times D$, but by using the Leibniz formula as a first principle. This is accomplished in \S\ref{subsec: loc fin dir cat}. In fact, these features will be present in the tensor product of two arbitrary \textsc{gat}s $\bbA$ and $\bbB$: for contexts $\textbf{X}$, $\textbf{Y}$ in $\bbA$, $\bbB$, resp., we will have a context $\textbf{X} \otimes \textbf{Y}$ in $\bbA \otimes \bbB$ satisfying the following equality in the contextual category $\mathcal C(\bbA \otimes \bbB)$:
$$
[\partial(\textbf{X} \otimes \textbf{Y})] = [\partial \textbf{X} \otimes \textbf{Y}] \times_{[\partial \textbf{X} \otimes \partial \textbf{Y}]} [\textbf{X} \otimes \partial \textbf{Y}].
$$
Together with $\textbf{X} \otimes \varnothing = \varnothing = \varnothing \otimes \textbf{Y}$, where $\varnothing$ denotes the empty context, this implies that $\ell(\textbf{X} \otimes \textbf{Y}) = \ell(\textbf{X})\ell(\textbf{Y})$, and that if $\textbf{X}'$ (resp. $\textbf{Y}'$) is an extension of $\textbf{X}$ (resp. $\textbf{Y}$), then $\textbf{X}' \otimes \textbf{Y}'$ is an extension of $\textbf{X} \otimes \textbf{Y}$ up to variable reordering.

This is achieved by introducing for sort symbols $S$, $T$ in $\bbA$, $\bbB$ a sort symbol $ST$ in $\bbA \otimes \bbB$; the context of its axiom is defined via a recursive algorithm described in \S\ref{sec: tensor product of generalized algebraic theories}.

\vspace{0.5em}

On the other hand, operations in $\bbA$ and in $\bbB$ induce operations or equalities in $\bbA \otimes \bbB$ similarly to how operations and equalities are introduced in the tensor product of two Lawvere theories. Suppose given an operation in $\bbA$, say encoded by a judgment $\textbf{X} \vdashcustom u:U$ where $\textbf{X} = (x_1:X_1, ..., x_m:X_m)$. Letting $\textbf{X}' = (\textbf{X},x:U)$ and $\textbf{u} = (x_1, ..., x_m,u)$, we have a section $[\textbf{u}]:[\textbf{X}] \rightarrow [\textbf{X}']$ of the display map $[\textbf{X}'] \twoheadrightarrow [\textbf{X}]$. For a context $\textbf{Y} = (y_1:Y_1, ..., y_n:Y_n)$ in $\bbB$, we will have in $\mathcal C(\bbA \otimes \bbB)$ a section-retraction pair
$$
[\textbf{X} \otimes \textbf{Y}] \overset{[\textbf{u} \otimes \textbf{Y}]}{\longrightarrow} [\textbf{X}' \otimes \textbf{Y}] \twoheadrightarrow [\textbf{X} \otimes \textbf{Y}].
$$
The sequence $\textbf{u} \otimes \textbf{Y}$ inducing the left arrow contains, in particular, terms $u \otimes y_1$, ..., $u \otimes y_n$ (see \S\ref{subsec: expressions from pairs of judgments}) that can be thought of as versions of $u$ whose sorts are distinct but depend coherently on the $Y_i$. This is simplified if $\bbB$ is a Lawvere theory: $Y_1 = \cdots = Y_n$ implies that $u \otimes y_1$, ..., $u \otimes y_n$ only differ from each other by a change of variables, so they can be viewed as distinct instances of the same term.

An symmetric story can be told regarding how we combine contexts in $\bbA$ and operations in $\bbB$.

To achieve this, we introduce (i) for each term symbol $s$ in $\bbA$ and each sort symbol $T$ in $\bbB$, a term symbol $sT$ in $\bbA \otimes \bbB$, and (ii) for each sort symbol $S$ in $\bbA$ and each term symbol $t$ in $\bbB$, a term symbol $St$ in $\bbA \otimes \bbB$. We explain in \S\ref{sec: tensor product of generalized algebraic theories} how, by combining these data, we obtain enough term expressions to implement the ideas outlined above.

\vspace{0.5em}

Again inspired by the tensor product of Lawvere theories, we introduce a term equality axiom for each pair $(s,t)$ consisting of term symbols $s$, $t$ in $\bbA$, $\bbB$, respectively. In fact, in \S\ref{sec: tensor product of generalized algebraic theories} we describe how to combine a pair of derivable term judgments, one in $\bbA$ and one in $\bbB$, into a term equality judgment in $\bbA \otimes \bbB$ which, by Theorem \ref{th: tensor product is a theory}, is derivable. This corresponds (see Lemma \ref{lem: tensor product of morphisms}) to the statement that if $\textbf{f}:\textbf{X} \rightarrow \textbf{X}'$, $\textbf{g}:\textbf{Y} \rightarrow \textbf{Y}'$ are context morphisms in $\bbA$, $\bbB$, then the following diagram in $\mathcal C(\bbA \otimes \bbB)$ commutes:
\[
\squa{[\textbf{X} \otimes \textbf{Y}]}{[\textbf{X} \otimes \textbf{Y}']}{[\textbf{X}' \otimes \textbf{Y}]}{[\textbf{X}' \otimes \textbf{Y}'].}{[\textbf{X} \otimes \textbf{g}]}{[\textbf{X}' \otimes \textbf{g}]}{[\textbf{f} \otimes \textbf{Y}]}{[\textbf{f} \otimes \textbf{Y}']}
\]

We also introduce in $\bbA \otimes \bbB$ equality judgments corresponding to \emph{some} pairs $(J,J')$ where $J$ (resp. $J'$) is an axiom in $\bbA$ (resp. $\bbB$) and either $J$ or $J'$ is an equality judgment.\footnote{In certain cases, for example when $J$ is a term axiom and $J'$ is a term equality axiom, we view the result of combining $J$ and $J'$ as being degenerate. This may be related to the fact that, when dealing with \textsc{gat}s with no additional structure, sorts in a given context and terms of a given sort form a set (by taking the quotient by the judgmental equality relation) rather than a higher-categorical structure. It could be worth exploring how to improve the construction by working with an intensional notation of equality.} This is summarized in Table \ref{table: 1}.

\subsubsection*{Categorical aspects}

We will verify in a sequel text (\cite{Alm26}) that this operation defines a closed monoidal structure on $\GAT$, which in turn transfers to $\Cont$; in fact, writing $\iiCont$ for the strict $2$-category obtained by endowing $\Cont$ with the natural transformations between contextual functors as $2$-cells, this operation will turn $\iiCont$ into a closed monoidal $\Cat$-enriched category.

While the purely syntactic treatment in the present article provides an algorithmic description of a tensor product of \textsc{gat}s $\bbA \otimes \bbB$, it has important limitations. It did not allow us, for example, to describe a universal property of the tensor product, to conclude that the tensor product defines a functor $\otimes:\GAT \times \GAT \rightarrow \GAT$, or to present the exponential theories $\bbB^\bbA$. These gaps will be filled in \cite{Alm26}, where we work almost entirely within the framework of contextual categories. Nonetheless, in \S\ref{sec: comparison functor} we construct a comparison functor
$$
\otimes_{\bbA,\bbB}:\mathcal C(\bbA) \times \mathcal C(\bbB) \rightarrow \mathcal C(\bbA \otimes \bbB)
$$
which will play a crucial role in the follow-up text -- it will be, in a sense to be defined, a universal \emph{double map} of contextual categories.

\vspace{0.5em}

Denoting by $\Fam$ the contextual category of families of sets described in \cite{Car86}, the aforementioned closed monoidal structure leads (up to the fact that $\Fam$ is not small, which requires us to work with contextual categories internal to a larger universe) to isomorphisms of categories
$$
\iiCont(\mathcal A,\Fam^\mathcal B) \simeq \iiCont(\mathcal A \otimes \mathcal B,\Fam) \simeq \iiCont(\mathcal B,\Fam^\mathcal A).
$$
for $\mathcal A$, $\mathcal B \in \Cont$. The contextual category $\Fam^\mathcal A$ encodes, in particular, the category of models $\iiCont(\mathcal A,\Fam)$ as its full subcategory spanned by the length-$1$ objects. A more detailed study of its features, however, is left for future work.\footnote{It is natural to ask whether we have a chain of equivalences $\Mod(\mathcal A,\Mod(\mathcal B,\Set)) \simeq \Mod(\mathcal A \otimes \mathcal B,\Set) \simeq \Mod(\mathcal B,\Mod(\mathcal A, \Set))$ where $\Mod(-,-)$ denotes the category of morphisms between two display map categories (for a contextual category, we take the underlying display map category, and for a finitely complete category, such as $\Mod(\mathcal A,\Set)$, we let all morphisms be display maps). This question, which will be studied in a future article, is subtle and cannot be answered with our current tools. The key point is that it is not true in general that $\iiCont(\mathcal A,\Fam) \simeq \Mod(\mathcal A, \Set)$; see \cite{BarHen25}, remarks 2.17 and B.54. For this equivalence to hold, we expect it to be sufficient that $\mathcal A \cong \mathcal C(\bbA)$ for some $\bbA$ that does not have sort equality axioms.}

We expect that under appropriate assumptions on the factors $\mathcal A$ and $\mathcal B$ (see footnote 4 above), our construction will correctly encode the tensor product of locally finitely presentable categories in the sense that $\Mod(\mathcal A \otimes \mathcal B) \simeq \Mod(\mathcal A) \otimes \Mod(\mathcal B)$ (where we write $\Mod(-)$ for the category of $\Set$-models of the underlying display map category). But more information is retained: $\mathcal A$ (and similarly for $\mathcal B$) endows $\Mod(\mathcal A)$ with a distinguished class of arrows, namely, the essential image under the Yoneda embedding $\mathcal A^{op} \rightarrow \Mod(\mathcal A)$ of the class of length-$1$ display maps. These special maps of models can be regarded as ``basic cofibrations" that allow us to view $\Mod(\mathcal A)$ as possessing a sort of cellular structure -- for instance, by taking the weak factorization system generated by those maps via the small object argument (in particular, if $\mathcal A = \mathcal C(\bbA)$, contexts in $\bbA$ correspond to cellular objects with respect to the basic cofibrations). This point of view is adopted, for instance, in \cite{Hen16,Fre25,BarHen25}.

To relate this extra structure to the tensor product of \textsc{gat}s, we borrow a construction from the theory of combinatorial model categories: for two pairs $(M_1,L_1)$, $(M_2,L_2)$ where $M_i$ is a locally presentable category and $L_i$ is the left class of a cofibrantly generated weak factorization system on $M_i$, we have a pair $(M_1 \otimes M_2, L_1 \otimes L_2)$ such that the canonical functor $M_1 \times M_2 \rightarrow M_1 \otimes M_2$ is (the ``cofibration part" of) a Quillen bifunctor from $((M_1,L_1), (M_2,L_2))$ to $(M_1 \otimes M_2, L_1 \otimes L_2)$, and is initial among Quillen bifunctors out of $((M_1,L_1), (M_2,L_2))$. The class $L_1 \otimes L_2$ is generated by the pushout-products $i \hat{\otimes} j$ where $i$ (resp. $j$) is a generator of $L_1$ (resp. of $L_2$); see \cite{RieVer14}, \cite{Bar20}. This is analogous to (and motivates) how the sort axioms of $\bbA \otimes \bbB$ are constructed from those of $\bbA$ and $\bbB$, as outlined in the above discussion of the Leibniz formula: semantically, i.e. by viewing derivable sort judgments as cofibrations of models, the sort axioms in $\bbA \otimes \bbB$ correspond precisely to the pushout-products of sort axioms in $\bbA$ with ones in $\bbB$. Hence as long as $\mathcal A$, $\mathcal B$ are such that $\Mod(\mathcal A \otimes \mathcal B) \simeq \Mod(\mathcal A) \otimes \Mod(\mathcal B)$, we also have $(\Mod(\mathcal A \otimes \mathcal B), L_{\mathcal A \otimes \mathcal B}) \simeq (\Mod(\mathcal A),L_\mathcal A) \otimes (\Mod(\mathcal B), L_\mathcal B)$ where $L_-$ is the weak factorization system generated by the sort axioms.

\subsubsection*{A brief overview of the construction}

For \textsc{gat}s $\bbA$ and $\bbB$, the tensor product $\bbA \otimes \bbB$ is constructed, in summary, as follows:
\begin{enumerate}[label=(\arabic*)]
	\item Firstly, we use the alphabets $\Sigma(\bbA)$ and $\Sigma(\bbB)$ to produce an alphabet $\Sigma'$. A sort symbol in $\Sigma'$ will be a pair $(S,T)$ consisting of a sort symbol of $\bbA$ and one of $\bbB$. A term symbol in $\Sigma'$ will be either $(s,T)$ where $s$ is a term symbol in $\bbA$ and $T$ is a sort symbol in $\bbB$, or $(S,t)$ where $S$ is a sort symbol in $\bbA$ and $t$ is a term symbol in $\bbB$.
	
	\item We provide an algorithm which, from derivable judgments $J$ and $J'$ in $\bbA$ and $\bbB$, respectively, constructs a judgment $J \odot J'$ in the alphabet $\Sigma'$ in accordance with Table \ref{table: 1}. The key ingredient is a procedure that produces sort/term expressions in $\Sigma'$ by combining sort/term expressions in $\Sigma(\bbA)$ with ones in $\Sigma(\bbB)$.
	
	\item We define a pretheory (in the sense of \cite{Car86}) structure on $\Sigma'$, denoted by $\bbA \otimes \bbB$, by putting as axioms all tensor products between an axiom in $\bbA$ and one in $\bbB$.
	
	\item We prove that $\bbA \otimes \bbB$ is a theory by checking that every axiom is well-formed.
\end{enumerate}

The main source of difficulty is the recursive nature of the definition of a generalized algebraic theory, which reflects both on the construction of $\bbA \otimes \bbB$ as a pretheory (essentially, on checking that our algorithm terminates for any given input expressions or judgments) and on the proof that it is a theory. While dealing with that, we are led to try to perform the required constructions and proofs by induction on the derivations of judgments, or perhaps on some numerical parameter encoding the complexity of judgments. We choose the second option for a technical reason which, I believe, is unessential: it allows us to only specify derivations (or steps thereof) when it is essential to do so.

A natural measure of the complexity of a (derivable) judgment $J$ is its \emph{height}: recursively, it is the smallest integer $n$ such that $J$ can be derived from a list of judgments of height at most $n-1$ by employing one of the inference rules. It is not, however, intrinsic to the judgment: there are multiple possible collections of inference rules for a generalized algebraic theory (i.e. that lead to the same derivable judgments as in \cite{Car86}), and different such collections induce different ``height functions". On the other hand, our proofs require the height function to have several properties that, in a way, state that the height of a judgment or expression is compatible with its complexity when measured as a mere string or tree of symbols. For example, we would like the height of a context $x_1:X_1, ..., x_n:X_n$ be to be strictly larger than that of $x_1:X_1, ..., x_{n-1}:X_{n-1}$; and the height of a derivable judgment $\textbf{X} \vdashcustom T(t_1, ..., t_n) \tp$ to be strictly larger, for $1 \le i \le n$, than that of some judgment of the form $\textbf{X} \vdashcustom t_i:U$. It turns out that not all of the desired conditions hold for the inference rules in \cite{Car86}, so we use a suitable modification of the latter. The rules that we use are presented in the appendix along with a proof that they induce the same derivable judgments as in Cartmell's approach.

\vspace{0.5em}

Deferred to \cite{Alm26}, as already mentioned, are descriptions of the functoriality of the tensor product and of a ($1$-categorical) universal property of the functors $\otimes_{\bbA,\bbB}:\mathcal C(\bbA) \times \mathcal C(\bbB) \rightarrow \mathcal C(\bbA \otimes \bbB)$. Still, in Remark \ref{rem: towards functoriality} we state a condition on the $\otimes_{\bbA,\bbB}$, which will later follow from their universal property, that yields functoriality of the tensor product.

Still, we outline a construction, for theories $\bbA$, $\bbB$, $\bbC$, of an isomorphism $(\bbA \otimes \bbB) \otimes \bbC \cong \bbA \otimes (\bbB \otimes \bbC)$. In fact,  it can be thought of as an ``equality up to change of notation": we have a canonical bijection between the alphabets $(\bbA \otimes \bbB) \otimes \bbC$ and $\bbA \otimes (\bbB \otimes \bbC)$, and, after identifying the two alphabets under this bijection, the two theories have the same derivable judgments. We also describe, in a similar way, an isomorphism $\bbA \otimes \bbB \cong \bbB \otimes \bbA$.

\subsection*{Organization of the text}

The text is structured as follows:
\begin{itemize}
	\item In Section 2, we construct the tensor product $\bbA \otimes \bbB$ of two generalized algebraic theories $\bbA$ and $\bbB$. We start by presenting an algorithm for combining sort/term expressions from $\bbA$ and $\bbB$ into expressions written in a suitable ``tensor product of alphabets" $\Sigma(\bbA) \otimes \Sigma(\bbB)$. Then we extend it to an algorithm which, from derivable judgments $J$, $J'$ in $\bbA$, $\bbB$, resp., produces a judgment $J \odot J'$ in $\Sigma(\bbA) \otimes \Sigma(\bbB)$. We let $\bbA \otimes \bbB$ be the pretheory whose signature is $\Sigma(\bbA) \otimes \Sigma(\bbB)$ and whose axioms are the judgments $J \odot J'$ where $J$ is an axiom in $\bbA$ and $J'$ is one in $\bbB$.
	
	\item In Section 3, before proving that $\bbA \otimes \bbB$ is a theory, we present several examples of our algorithm. We believe that this section can be quite useful for having a practical understanding of the construction and for recognizing some instances already present in the literature.
	
	\item Sections 4 and 5 are dedicated to proving that $\bbA \otimes \bbB$ is a theory. In order to prove that $J \odot J'$ is well-formed whenever $J$, $J'$ are axioms in $\bbA$, $\bbB$, resp., we are led to verify the more general statement that $J \odot J'$ is derivable for any derivable judgments $J$, $J'$. We prove by induction on $n \ge 0$ that this is the case whenever $\Ht(J)\Ht(J') = n$, where $\Ht$ is the height function described in the appendix.
	
	As an auxiliary definition, we say that the pair of theories $(\bbA,\bbB)$ is \emph{$h$-derivable}, where $h \ge 0$, if $J \odot J'$ is derivable whenever $\Ht(J)\Ht(J') \le h$. In Section 4, we establish several consequences of $h$-derivability, which we think of as placing us halfway between $h$-derivability and $(h+1)$-derivability. As we will see in \S\ref{sec: comparison functor}, tensoring judgments, contexts, or context morphisms (or certain combinations of these) admits a natural interpretation in the contextual category $\mathcal C(\bbA \otimes \bbB)$ by means of a canonical functor $\otimes_{\bbA,\bbB}:\mathcal C(\bbA) \times \mathcal C(\bbB) \rightarrow \mathcal C(\bbA \otimes \bbB)$. The role of Section 4 is, informally, to give us tools to reason categorically in a fragment of $\mathcal C(\bbA \otimes \bbB)$ before the latter has been constructed in its totality.
	
	In Section 5, we use the results from the previous section to prove that $(\bbA,\bbB)$ is $(h+1)$-derivable whenever it is $h$-derivable. This is done in several cases and requires that we take into account both the kind of judgment and what we call an \emph{initial inference} of $J$, $J'$, meaning that the premises have height strictly smaller than that of the conclusion.
	
	\item In Section 6, we verify a certain ``two-sided" substitution property of the tensor product $\bbA \otimes \bbB$ (propositions \ref{prop: two-sided substitution (morphism version)} and \ref{prop: two-sided substitution}), and as a consequence we obtain a comparison functor $\otimes_{\bbA,\bbB}:\mathcal C(\bbA) \times \mathcal C(\bbB) \rightarrow \mathcal C(\bbA \otimes \bbB)$.
	
	\item In Sections 7 and 8 we sketch a description of isomorphisms of the forms $(\bbA \otimes \bbB) \otimes \bbC \cong \bbA \otimes (\bbB \otimes \bbC)$ and $\bbA \otimes \bbB \cong \bbB \otimes \bbA$, respectively.
	
	\item In the appendix we give an introduction to generalized algebraic theories. Our main goal here is to discuss the particular set of inference rules used to obtain a well-behaved height function on judgments. Some technical work is required to verify the desired properties of the height function and compare our inference rules with Cartmell's. We also fix some notation and terminology regarding \textsc{gat}s that might differ from the usual ones. Readers who are familiar with the topic may consult the appendix as needed.
\end{itemize}

\subsection*{Acknowledgement}

I am extremely grateful to Simon Henry, my doctoral advisor, who introduced me to this topic, helped me navigate it through countless and insightful discussions, and provided valuable feedback on this text.

\section{The tensor product of generalized algebraic theories}

\label{sec: tensor product of generalized algebraic theories}

In this section we describe an algorithm that constructs what we will call the \emph{tensor product} of two generalized algebraic theories $\bbA$ and $\bbB$, denoted by $\bbA \otimes \bbB$. As explained in the introduction, our approach in the present text will be purely syntactic. While an abstract presentation in terms of contextual categories can be given (as will be discussed in the sequel article), the syntactic approach is suitable for explicit calculations and encodes certain data about tensor products that otherwise would not be readily available. This reflects the fact that one may regard generalized algebraic theories as \emph{presentations} of contextual categories. Thus here, as often happens with monoidal categories of algebraic structures, if $X$, $Y$ are two objects with given presentations by generators and relations, we would like to obtain a presentation for $X \otimes Y$ by explicitly combining those for $X$ and $Y$. For \textsc{gat}s $\bbA$, $\bbB$, the tensor product $\bbA \otimes \bbB$ will be a particular presentation of $\mathcal C(\bbA) \otimes \mathcal C(\bbB) \cong \mathcal C(\bbA \otimes \bbB)$ constructed directly from the axioms of $\bbA$, $\bbB$.

\vspace{0.5em}

We will define $\bbA \otimes \bbB$ in several steps. Firstly, for sort-and-term alphabets (see Definition \ref{def: alphabet}) $\Sigma$ and $\Tau$, we define an alphabet $\Sigma \otimes \Tau$. Then, specializing to the alphabets $\Sigma(\bbA)$, $\Sigma(\bbB)$ of $\bbA$, $\bbB$, we describe how to combine derivable sorts/terms in $\bbA$ with ones in $\bbB$ to obtain sort/term expressions in $\Sigma(\bbA) \otimes \Sigma(\bbB)$. This is used, in particular, to define a length-$mn$ precontext $\textbf{X} \otimes \textbf{Y}$ in $\Sigma(\bbA) \otimes \Sigma(\bbB)$ where $\textbf{X}$ (resp. $\textbf{Y}$) is a length-$m$ (resp. $n$) context in $\bbA$ (resp. $\bbB$).

Next, we define for certain (see Table \ref{table: 1}) derivable judgments $J$, $J'$ in $\bbA$, $\bbB$, resp., a judgment $J \odot J'$ in $\Sigma(\bbA) \otimes \Sigma(\bbB)$. Now, we let $\bbA \otimes \bbB$ be the pretheory (in the sense of \cite{Car86}) structure on $\Sigma(\bbA) \otimes \Sigma(\bbB)$ whose axioms are $J \odot J'$ (when it is defined) where $J$, $J'$ are axioms in $\bbA$, $\bbB$, resp. Most of the technical work in the remainder of the article will be devoted to proving:

\begin{theorem*}
Let $\bbA$, $\bbB$ be generalized algebraic theories. For any derivable judgments $J$, $J'$ in $\bbA$, $\bbB$, resp., if the judgment $J \odot J'$ is defined, then it is derivable in the pretheory $\bbA \otimes \bbB$. In particular, $\bbA \otimes \bbB$ is a theory.
\end{theorem*}

While reading this section, the reader is invited to compare each step of the construction with the examples in \S\ref{sec: examples}, especially the theory of strict double categories (\ref{subsec: strict double categories}) as it exhibits many of the features present in the general definition.

\subsection{Expressions from pairs of judgments}

\label{subsec: expressions from pairs of judgments}

We refer the reader to the appendix for the definitions of the syntactic structures that we use below -- sort-and-term alphabets, sort/term expressions, the height function, the canonical sort of a term, etc.

\begin{definition}
	\label{def: tensor product of alphabets}
	Let $\Sigma$ and $\Tau$ be sort-and-term alphabets. We define the alphabet $\Sigma \otimes \Tau$ as having the following sets of variables\footnote{We admit the possibility of different \textsc{gat}s having different (countably infinite) sets of variables; see the Appendix for the definition we will use, and Remark \ref{rem: variables of tensor product} for motivation for this convention.}, sort symbols, and term symbols:
	\begin{align*}
		(\Sigma \otimes \Tau)^{\text{var}} & = \Sigma^{\text{var}} \times \Tau^{\text{var}}\\
		(\Sigma \otimes \Tau)^{\text{sort}} & = \Sigma^{\text{sort}} \times \Tau^{\text{sort}}\\
		(\Sigma \otimes \Tau)^{\text{term}} & = (\Sigma^{\text{sort}} \times \Tau^{\text{term}}) \sqcup (\Sigma^{\text{term}} \times \Tau^{\text{sort}}).
	\end{align*}
	In each of these sets, we will write $ab$ for an ordered pair $(a,b)$ (or, by abuse of notation, for the image of a pair $(a,b)$ under the canonical maps from $\Sigma^{\text{sort}} \times \Tau^{\text{term}}$ and $\Sigma^{\text{term}} \times \Tau^{\text{sort}}$ to the above coproduct).
\end{definition}

\begin{remark}
\label{rem: variables of tensor product}
While the forms of $(\Sigma \otimes \Tau)^{\text{sort}}$ and $(\Sigma \otimes \Tau)^{\text{term}}$ are essential, that of $(\Sigma \otimes \Tau)^{\text{var}}$ can be seen as a notational convenience. Indeed, we could replace $\Sigma^{\text{var}} \times \Tau^{\text{var}}$ by any other countably infinite set, say $V$, but to perform some of the constructions that will appear in this section it is helpful to choose some (arbitrary) bijection $\Sigma^{\text{var}} \times \Tau^{\text{var}} \cong K$.
\end{remark}

We now construct operations that assign to certain pairs of judgments, one in $\bbA$ and one in $\bbB$, a ``tensor" expression in $\Sigma(\bbA) \otimes \Sigma(\bbB)$.

We let $\text{Der}_s(\bbA)$ and $\text{Der}_t(\bbA)$, respectively (and similarly for $\bbB$), be the sets of derivable sort judgments and of derivable term judgments in $\bbA$; also, we denote by $\text{Der}^+_s(\bbA)$ the set of pairs $(J,x)$ where $J \in \text{Der}_s(\bbA)$ and $x \in \Sigma(\bbA)^{\text{var}}$ does not occur in the context of $J$. Note that the elements of $\text{Der}^+_s(\bbA)$, which we think of as ``augmented" sort judgments, are in canonical bijection with the nonempty contexts of $\bbA$.\footnote{Although we could directly use the nonempty contexts of $\bbA$, this alternative form will be useful for the particular use we will make of such structures. In fact, Notation \ref{not: abuses of notation} suggests us to think of $(\textbf{X} \vdashcustom U \tp,x)$ as a structure on the variable $x$ itself, so that $x$ functions informally as a placeholder for a ``generic" element of sort $U$ in context $\textbf{X}$ (rather than, for example, the collection of all elements of that sort).}

The sets of sort expressions and of term expressions of $\Sigma(\bbA)$ (and similarly for $\Sigma(\bbB)$) will be denoted by $\text{Exp}_s(\Sigma(\bbA))$ and $\text{Exp}_t(\Sigma(\bbA))$, respectively.

Firstly, we will define maps
$$
\otimes_t: (\text{Der}_s^+(\bbA) \cup \text{Der}_t(\bbA)) \times (\text{Der}_s^+(\bbB) \cup \text{Der}_t(\bbB)) \longrightarrow \text{Exp}_t(\Sigma(\bbA) \otimes \Sigma(\bbB)),
$$
$$
\varotimes_t: (\text{Der}_s^+(\bbA) \cup \text{Der}_t(\bbA)) \times (\text{Der}_s^+(\bbB) \cup \text{Der}_t(\bbB)) \longrightarrow \text{Exp}_t(\Sigma(\bbA) \otimes \Sigma(\bbB))
$$
recursively with respect to the product of the heights of the arguments (see Definition \ref{def: derivable judgment, height, initial inference}).

\begin{remark}
This implies (using the notation introduced below) that we will combine derivable terms $u$, $v$ of sorts $U$, $V$ in $\bbA$, $\bbB$, respectively, in two different ways: as $u \otimes_t v$ and as $u \varotimes_t v$. They correspond, as reflected in the formulas $\lozenge$, $\blacklozenge$ below and Remark \ref{rem: matrix form of tensor product of terms - not variables}, that we would like to have two distinct ways of making the operations $u$, $v$ act on a matrix of variables: one where we apply $u$ to each column and then $v$ to the resulting list, and one where we apply $v$ to each row and then $u$ to the resulting list.

These expressions generalize to the dependently sorted setting the left- and right-hand sides of the equality axiom, discussed in the introduction, stating the commutativity of two operations $f$, $g$ in the tensor product of two Lawvere theories. (See also \S\ref{subsec: tensor product of lawvere theories}.) Accordingly, the equality $u \otimes_t v \equiv u \varotimes_t v$ will be derivable in $\bbA \otimes \bbB$; in fact, it will be an axiom if the term judgments attached to $u$, $v$ are axioms.
\end{remark}

This is a good point to introduce some conventions and abuses of notation that will be convenient throughout the text:

\begin{notation}
\label{not: abuses of notation}
\leavevmode
\begin{enumerate}[label=(\arabic*)]
	\item Let $J$ be a derivable term judgment $\textbf{X} \vdashcustom u:U$ in $\bbA$. Whenever it is used in one of the operations $\otimes_t$, $\varotimes_t$ below, we will denote it by $u^U$, with the context $\textbf{X}$ implicit. Similarly, if $J$ is a derivable judgment $\textbf{X} \vdashcustom U \tp$ in $\bbA$, we will denote a pair $(J,x) \in \text{Der}_s^+(\bbA)$ by $x^U$.
	
	Note that, due to the omission of the context, $x^U$ -- where $x$ is a variable -- can represent a term judgment $\textbf{X} \vdashcustom x:U$ or a sort judgment $\textbf{X} \vdashcustom U \tp$. This ambiguity only vanishes when the set of variables in the context is taken into account: if $x$ occurs in the context, we have a term judgment; if not, we have a sort judgment (or, more precisely, an element of $\text{Der}_s^+(\bbA)$).
	
	\item When dealing with a judgment of the form $\textbf{X} \vdashcustom u:\Type(u)$ (see \ref{subsec: other kinds of judgment}), we write $u$ instead of $u^{\Type(u)}$. This will occur most notably in the following two cases:
	\begin{itemize}
		\item For a derivable judgment $\textbf{X} \vdashcustom S(s_1, ..., s_k) \tp$, where $S$ is a sort symbol and $s_1$, ..., $s_k$ are sort expressions, for $1 \le i \le k$ we will -- when performing the operations $\otimes_t$ and $\varotimes_t$ -- use $s_i$ as a notation for the (derivable) judgment $\textbf{X} \vdashcustom s_i:\Type(s_i)$.
		
		\item For a derivable judgment $\textbf{X} \vdashcustom s(s_1, ..., s_k):U$, where $s$ is a term symbol and $s_1$, ..., $s_k$ are sort expressions, we denote $\textbf{X} \vdashcustom s_i:\Type(s_i)$ by $s_i$ for $1 \le i \le k$ (again, only when dealing with $\otimes_t$ and $\varotimes_t$).
	\end{itemize}
	
	\item When using the above notation for sort or term judgments, we will usually write $u^U \otimes v^V$ (resp. $u^U \varotimes v^V$) for $u^U \otimes_t v^V$ (resp. for $u^U \varotimes_t v^V$).
	
	\item Suppose given finite sequences $(u_1^{U_1}, ..., u_m^{U_m})$ in $\text{Der}_s^+(\bbA) \cup \text{Der}_t(\bbA)$ and $(v_1^{V_1}, ..., v_n^{V_n})$ in $\text{Der}_s^+(\bbB) \cup \text{Der}_t(\bbB)$. We will denote by
	$$
	(u_1^{U_1}, ..., u_m^{U_m}) \otimes (v_1^{V_1}, ..., v_n^{V_n})
	$$
	the length-$mn$ sequence $(u_1^{U_1} \otimes v_1^{V_1}, ..., u_1^{U_1} \otimes v_n^{V_n}, ..., u_m^{U_m} \otimes v_1^{V_1}, ..., u_m^{U_m} \otimes v_n^{V_n})$. Precisely, it is the family $\{1, ..., mn\} \rightarrow \text{Exp}_t(\Sigma(\bbA) \otimes \Sigma(\bbB))$ obtained by composing $(u_i^{U_i} \otimes v_j^{V_j})_{(i,j) \in \{1, ..., m\} \times \{1, ..., n\}}$ with the bijection $\{1, ..., mn\} \cong \{1, ..., m\} \times \{1,...,n\}$ corresponding to the lexicographic order on the cartesian product. We will also use the matrix notation
	$$
	\begin{pmatrix}
		u_1^{U_1} \otimes v_1^{V_1} & \cdots & u_1^{U_1} \otimes v_n^{V_n}\\
		\vdots & \ddots & \vdots\\
		u_m^{U_m} \otimes v_1^{V_1} & \cdots & u_m^{U_m} \otimes v_n^{V_n}
	\end{pmatrix}.
	$$
	(Although the matrix represents $(u_i^{U_i} \otimes v_j^{V_j})_{(i,j) \in \{1, ..., m\} \times \{1, ..., n\}}$, we use it as a notation for the reindexed family $\{1, ..., mn\} \rightarrow \text{Exp}_t(\Sigma(\bbA) \otimes \Sigma(\bbB))$.)
	
	\item For $(J,x)$ in $\text{Der}_s^+(\bbA)$ or $\text{Der}_s^+(\bbB)$, we let $\Ht(J,x) = \Ht(J)$.
\end{enumerate}
\end{notation}

\subsubsection{Construction of tensor term expressions}

Given $h \ge 0$, suppose that $J \otimes_t J'$ has been constructed whenever $J \in \text{Der}_s^+(\bbA) \cup \text{Der}_t(\bbA)$ and $J' \in \text{Der}_s^+(\bbB) \cup \text{Der}_t(\bbB)$ satisfy $\Ht(J)\Ht(J') < h$. Then, if $\Ht(J)\Ht(J') = h$, we define $J \otimes_t J'$ as follows:

\begin{enumerate}[label=(\roman*)]
	\item In the notation described above, suppose that $J$ is $x^U$ and $J'$ is $y^U$ where $x$, $y$ are variables. Then we let $J \otimes_t J'$ (or $x^U \otimes y^V$) be $xy$.
	
	More precisely, we assume that either $J$ is $(\textbf{X} \vdashcustom U \tp, x)$ where $x$ is a variable not contained in $\textbf{X}$, or it is $\textbf{X} \vdashcustom x:U$. Similarly, either $J'$ is $(\textbf{Y} \vdashcustom V \tp, y)$, where $y$ is a variable not in $\textbf{Y}$, or it is $\textbf{Y} \vdashcustom y:V$.
	
	\item Suppose that $J$ is either $(\textbf{X} \vdashcustom U \tp, x)$ or $\textbf{X} \vdashcustom x:U$, where $x$ is a variable, and $J' = (\textbf{Y} \vdashcustom v:V)$ where $v$ is not a variable. Write $U = S(s_1, ..., s_k)$\footnote{When used between two expressions built from a given sort-and-term alphabet, the symbol $=$ always indicates syntactic equality. We use $\equiv$ for the judgmental equality between sorts or terms in a (pre)theory.}, where $S$ is a sort symbol and $s_1$, ..., $s_k$ are term expressions, and $v = t(t_1, ..., t_\ell)$, where $t$ is a term symbol and $t_1$, ..., $t_k$ are term expressions. We let $J \otimes_t J'$ (or $x^U \otimes v^V$) be
	$$
	St((s_1^{\Type(s_1)}, ..., s_k^{\Type(s_k)},x^U) \otimes (t_1^{\Type(t_1)}, ..., t_\ell^{\Type(t_\ell)})).
	$$
	Following the above conventions, we write it as $St((s_1, ..., s_k,x^U) \otimes (t_1, ..., t_\ell))$ or, in matrix notation,
	$$
	St
	\begin{pmatrix}
		s_1 \otimes t_1 & \cdots & s_1 \otimes t_\ell\\
		\vdots & \ddots & \vdots\\
		s_k \otimes t_1 & \cdots & s_k \otimes t_\ell\\
		x^U \otimes t_1 & \cdots & x^U \otimes t_\ell
	\end{pmatrix}.
	$$
	Note that this only requires the use of products already constructed in the recursion process. Indeed, by Proposition \ref{prop: properties height} we have
	$$
	\Ht(\textbf{X} \vdashcustom s_i:\Type(s_i)) < \Ht(\textbf{X} \vdashcustom U \tp), \qquad\qquad
	\Ht(\textbf{Y} \vdashcustom t_j:\Type(t_j)) < \Ht(\textbf{Y} \vdashcustom v:V)
	$$
	for $1 \le i \le k$ and $1 \le j \le l$.
	
	\item Suppose that $J = (\textbf{X} \vdashcustom u:U)$, where $u$ is not a variable, and $J'$ is either $(\textbf{Y} \vdashcustom V \tp, y)$ or $\textbf{Y} \vdashcustom y:V$ where $y$ is a variable. Writing $u = s(s_1, ..., s_k)$ and $V = T(t_1, ..., t_\ell)$, we let $J \otimes_t J'$ (or $u^U \otimes y^V$) be
	$$
	sT((s_1^{\Type(s_1)}, ..., s_k^{\Type(s_k)}) \otimes (t_1^{\Type(t_1)}, ..., t_\ell^{\Type(t_\ell)}, y^V)).
	$$
	As in the previous case, we write it as $sT((s_1, ..., s_k) \otimes (t_1, ..., t_\ell, y^V))$ or, in matrix notation,
	$$
	sT
	\begin{pmatrix}
		s_1 \otimes t_1 & \cdots & s_1 \otimes t_\ell & s_1 \otimes y^V\\
		\vdots & \ddots & \vdots & \vdots\\
		s_k \otimes t_1 & \cdots & s_k \otimes t_\ell & s_k \otimes y^V
	\end{pmatrix}.
	$$
	Since, by Proposition \ref{prop: properties height},
	$$
	\Ht(\textbf{X} \vdashcustom s_i:\Type(s_i)) < \Ht(\textbf{X} \vdashcustom u:U), \qquad\qquad \Ht(\textbf{Y} \vdashcustom t_j:\Type(t_j)) < \Ht(\textbf{Y} \vdashcustom V \tp),
	$$
	this only requires products already constructed in the recursion process.
	
	\item Suppose that $J = (\textbf{X} \vdashcustom u:U)$ and $J' = (\textbf{Y} \vdashcustom v:V)$ where $u$, $v$ are not variables. If $\textbf{Y} = (y_1:Y_1, ..., y_n:Y_n)$ and $x$ is any variable not in $\textbf{X}$, we let $J \otimes_t J'$ (or $u^U \otimes v^V$) be
	\[
	\tag{$\lozenge$}
	(x^U \otimes v^V)[u^U \otimes y_1^{Y_1} \mid xy_1, ..., u^U \otimes y_n^{Y_n} \mid xy_n]
	\]
	where $y_i^{Y_i}$ denotes the pair $(\partial_{i-1} \textbf{Y} \vdashcustom Y_i \tp, y_i) \in \text{Der}_s^+(\bbB)$ for $1 \le i \le n$. Recall that ($\lozenge$) is the expression obtained from $x^U \otimes v^V$ by simultaneously replacing every occurrence of $xy_1$, ..., $xy_n$ by $u^U \otimes y_1^{Y_1}$, ..., $u^U \otimes y_n^{Y_n}$, respectively. It can be checked by induction that ($\lozenge$) does not depend on the choice of $x$.
	
	By Proposition \ref{prop: properties height} we have
	\begin{align*}
		&\Ht(\textbf{X} \vdashcustom U \tp, x) = \Ht(\textbf{X} \vdashcustom U \tp) < \Ht(\textbf{X} \vdashcustom u:U), \\
		&\Ht(\partial_{i-1}\textbf{Y} \vdashcustom Y_i \tp, y_i) = \Ht(\partial_i\textbf{Y}) \le \Ht(\textbf{Y}) < \Ht(\textbf{Y} \vdashcustom v:V),
	\end{align*}
	so only products that have already been constructed are required for ($\lozenge$).
\end{enumerate}

This concludes the definition of $\otimes_t$. On the other hand, $\varotimes_t$ is defined as follows:
\begin{enumerate}[label=(\roman*)]
	\item Suppose that $J = (\textbf{X} \vdashcustom u:U)$ and $J' = (\textbf{Y} \vdashcustom v:V)$ where $u$, $v$ are not variables. If $\textbf{X} = (x_1:X_1, ..., x_m:X_m)$ and $y$ is any variable not in $\textbf{Y}$, we let $J \varotimes_t J'$ (or $u^U \varotimes v^V$) be
	\[
	\tag{$\blacklozenge$}
	(u^U \otimes y^V)[x_1^{X_1} \otimes v^V \mid x_1y, ..., x_m^{X_m} \otimes v^V \mid x_my]
	\]
	where $x_i^{X_i}$ denotes $(\partial_{i-1}\textbf{X} \vdashcustom X_i, x_i) \in \text{Der}_s^+(\bbA)$ for $1 \le i \le m$. It can be verified by induction that ($\blacklozenge$) does not depend on the choice of $y$.
	
	\item In all other cases, i.e. if $J$ is of the form $x^U$ for a variable $x$ or $J'$ is of the form $y^V$ for a variable $y$, we let $J \varotimes_t J' = J \otimes_t J'$.
\end{enumerate}

\begin{remark}
Although ($\lozenge$) is only used to define $\otimes_t$ on pairs of terms that are not variables, it actually computes $\otimes_t$ in greater (but not full) generality. This will be convenient later when calculating tensor products of judgments. To study when ($\lozenge$) correctly computes $\otimes_t$, consider
$$
J = (\textbf{X} \vdashcustom u:U), \text{ where } \textbf{X} = (x_1:X_1, ..., x_m:X_m); \qquad J' = (\textbf{Y} \vdashcustom v:V), \text{ where } \textbf{X} = (y_1:Y_1, ..., y_n:Y_n).
$$

Firstly, suppose that $u = x_i$ for some $i \in \{1, ..., m\}$ (whereas $v$ can be a variable or not). Then ($\lozenge$) becomes
$$
(x^U \otimes v^V)[x_iy_1 \mid xy_1, ..., x_iy_n \mid xy_n] = x_i^U \otimes v^V.
$$
On the other hand, if $v = y_j$ for some $j \in \{1, ..., n\}$, then ($\lozenge$) is
$$
(xy_j)[u^U \otimes y_1^{Y_1} \mid xy_1, ..., u^U \otimes y_n \mid xy_n] = u^U \otimes y_j^{Y_j},
$$
which is not equal in general to $u^U \otimes y_j^V$; it is when $U = Y_j$.

Similarly, ($\blacklozenge$) computes $\varotimes_t$ in some cases other than those where it does by definition. Let $J$ and $J'$ be as above. If $v = y_j$ for some $j \in \{1, ..., n\}$, then ($\blacklozenge$) equals
$$
(u^U \otimes y^V)[x_1y_j \mid x_1y, ..., x_my_j \mid x_my] = u^U \otimes y_j^V = u^U \varotimes y_j^V.
$$

If $u = x_i$ for some $i \in \{1, ..., m\}$, then ($\blacklozenge$) is
$$
x_iy[x_1^{X_1} \otimes v^V \mid x_1y, ..., x_m^{X_m} \otimes v^V \mid x_my] = x_i^{X_i} \otimes v^V = x_i^{X_i} \varotimes v^V,
$$
which equals $x_i^U \varotimes v^V$ if $U = X_i$.
\end{remark}

For future reference, we state as a lemma the case we will be interested in:

\begin{lemma}
\label{lem: formulas for tensor products of terms}
Formula ($\lozenge$) applied to $u^U \in \text{Der}_t(\bbA)$ and $v^V \in \text{Der}_t(\bbB)$ calculates $u^U \otimes_t v^V$ when $V = \Type(v)$, and formula ($\blacklozenge$) applied to the same judgments calculates $u^U \varotimes_t v^V$ when $U = \Type(u)$.
\end{lemma}

\begin{remark}
\label{rem: matrix form of tensor product of terms - not variables}
Using the matrix form can be quite useful for working with formulas ($\lozenge$) and ($\blacklozenge$). Let $J = (\textbf{X} \vdashcustom u:U)$ and $J' = (\textbf{Y} \vdashcustom v:V)$ where $u$, $v$ are not variables. Write $\textbf{X} = (x_1:X_1, ..., x_m:X_m)$ and $\textbf{Y} = (y_1:Y_1, ..., y_n:Y_n)$, and let $x$, $y$ be the chosen variables not in $\textbf{X}$, $\textbf{Y}$, respectively.

If $U = S(\sigma_1, ..., \sigma_K)$ and $v = t(t_1, ..., t_\ell)$, then $u^U \otimes v^V$ can be written as
$$
St
\begin{pmatrix}
	\sigma_1 \otimes t_1 & \cdots & \sigma_1 \otimes t_\ell\\
	\vdots & \ddots & \vdots\\
	\sigma_K \otimes t_1 & \cdots & \sigma_K \otimes t_\ell\\
	x^U \otimes t_1 & \cdots & x^U \otimes t_\ell
\end{pmatrix}
[u^U \otimes y_1 \mid xy_1, ..., u^U \otimes y_n \mid xy_n].
$$
As $xy_1$, ..., $xy_n$ do not occur in the expressions $\sigma_i \otimes t_j$, by Lemma \ref{lem: formulas for tensor products of terms} the above expression equals
$$
St
\begin{pmatrix}
	\sigma_1 \otimes t_1 & \cdots & \sigma_1 \otimes t_\ell\\
	\vdots & \ddots & \vdots\\
	\sigma_K \otimes t_1 & \cdots & \sigma_K \otimes t_\ell\\
	\overline{x^U \otimes t_1} & \cdots & \overline{x^U \otimes t_\ell}
\end{pmatrix}
=
St
\begin{pmatrix}
	\sigma_1 \otimes t_1 & \cdots & \sigma_1 \otimes t_\ell\\
	\vdots & \ddots & \vdots\\
	\sigma_K \otimes t_1 & \cdots & \sigma_K \otimes t_\ell\\
	u^U \otimes t_1 & \cdots & u^U \otimes t_\ell
\end{pmatrix}
$$
where we use an upper bar to indicate an application of $[u^U \otimes y_1 \mid xy_1, ..., u^U \otimes y_n \mid xy_n]$.

On the other hand, if $u = s(s_1, ..., s_k)$ and $V = T(\tau_1, ..., \tau_\ell)$, then $u^U \varotimes v^V$ can be written as
$$
sT
\begin{pmatrix}
	s_1 \otimes \tau_1 & \cdots & s_1 \otimes \tau_\ell & s_1 \otimes y^V\\
	\vdots & \ddots & \vdots & \vdots\\
	s_k \otimes \tau_1 & \cdots & s_k \otimes \tau_\ell & s_k \otimes y^V
\end{pmatrix}
[x_1^{X_1} \otimes v^V \mid x_1y, ..., x_m^{X_m} \otimes v^V \mid x_my].
$$
As $x_1y$, ..., $x_my$ do not occur in the expressions $s_i \otimes \tau_j$, by Lemma \ref{lem: formulas for tensor products of terms} the above expression equals
$$
sT
\begin{pmatrix}
	s_1 \otimes \tau_1 & \cdots & s_1 \otimes \tau_\ell & \overline{s_1 \otimes y^V}\\
	\vdots & \ddots & \vdots & \vdots\\
	s_k \otimes \tau_1 & \cdots & s_k \otimes \tau_\ell & \overline{s_k \otimes y^V}
\end{pmatrix}
=
sT
\begin{pmatrix}
	s_1 \otimes \tau_1 & \cdots & s_1 \otimes \tau_\ell & s_1 \varotimes v^V\\
	\vdots & \ddots & \vdots & \vdots\\
	s_k \otimes \tau_1 & \cdots & s_k \otimes \tau_\ell & s_k \varotimes v^V
\end{pmatrix}
$$
where the upper bar now indicates $[x_1^{X_1} \otimes v^V \mid x_1y, ..., x_m^{X_m} \otimes v^V \mid x_my]$.
\end{remark}

\subsubsection{Construction of tensor sort expressions}

Now, we use $\otimes_t$ to define an operation
$$
\otimes_s:\text{Der}_s^+(\bbA) \otimes \text{Der}_s^+(\bbB) \longrightarrow \text{Exp}_s(\Sigma(\bbA) \otimes \Sigma(\bbB))
$$
that combines two derivable (augmented) sort judgments to produce a sort expression in the alphabet $\Sigma(\bbA) \otimes \Sigma(\bbB)$.

\begin{notation}
For a finite sequence $\alpha = (a_1, ..., a_n)$ in a given set, we let $\partial \alpha$ be its restriction $(a_1, ..., a_{n-1})$ if $n \ge 1$, and the empty sequence if $n = 0$ (i.e. if $\alpha$ is empty).
\end{notation}

Given $J = (\textbf{X} \vdashcustom U \tp, x) \in \text{Der}_s(\bbA)$ and $J' = (\textbf{Y} \vdashcustom V \tp, y) \in \text{Der}_s(\bbB)$, where
$$
U = S(s_1, ..., s_k), \qquad\qquad V = T(t_1, ..., t_\ell),
$$
we define $J \otimes_s J'$, also denoted by $U^x \otimes_s V^y$ with $\textbf{X}$, $\textbf{Y}$ implicit, as the expression
$$
ST(\partial((s_1^{\Type(s_1)}, ..., s_k^{\Type(s_k)}, x^U) \otimes_t (t_1^{\Type(t_1)}, ..., t_\ell^{\Type(t_\ell)}, y^V))).
$$
We will usually drop the subscript ``$s$" from $\otimes_s$ and, when this causes no ambiguity, also the superscript variables from $U^x$ and $V^y$.\footnote{Note that $U^x$ denotes the same structure as $x^U$ (see Notation \ref{not: abuses of notation}). However, the former is used exclusively when applying the operations $\otimes_t$ and $\varotimes_t$, and the latter only when applying $\otimes_s$.}

Following our conventions for $\otimes_t$ we write simply $ST(\partial((s_1, ..., s_k,x^U) \otimes (t_1, ..., t_\ell,y^V)))$. If we were to spell it out in the usual way, we would obtain something like
$$
ST(s_1 \otimes t_1, ..., s_1 \otimes t_\ell, s_1 \otimes y^V, ..., s_k \otimes t_1, ..., s_k \otimes t_\ell, s_k \otimes y^V, x^U \otimes t_1, ..., x^U \otimes t_\ell),
$$
which can be ambiguous. Instead, we use the ``incomplete matrix"
$$
ST
\begin{pmatrix}
	s_1 \otimes t_1 & \cdots & s_1 \otimes t_\ell & s_1 \otimes y^V\\
	\vdots & \ddots & \vdots & \vdots\\
	s_k \otimes t_1 & \cdots & s_k \otimes t_\ell & s_k \otimes y^V\\
	x^U \otimes t_1 & \cdots & x^U \otimes t_\ell & -
\end{pmatrix}
$$
(while keeping in mind that the entries are ordered lexicographically).

\subsubsection{Construction of tensor precontexts}

We use the above construction to define, for contexts $\textbf{X}$ and $\textbf{Y}$ in $\bbA$ and $\bbB$, respectively, a ``tensor" precontext $\textbf{X} \otimes \textbf{Y}$ in $\Sigma(\bbA) \otimes \Sigma(\bbB)$. If $\textbf{X} = (x_1:X_1, ..., x_m:X_m)$ and $\textbf{Y} = (y_1:Y_1, ..., y_n:Y_n)$, we let $\textbf{X} \otimes \textbf{Y}$ be the sequence
\[
\tag{\texttt{*}}
(x_1y_1:X_1^{x_1} \otimes Y_1^{y_1}, ..., x_1y_n:X_1^{x_1} \otimes Y_n^{y_n}, ..., x_my_1:X_m^{x_m} \otimes Y_1^{y_1}, ..., x_my_n:X_m^{x_m} \otimes Y_n^{y_n})
\]
obtained by lexicographically ordering the pairs $(x_iy_j:X_i^{x_i} \otimes Y_j^{y_j})$ for $1 \le i \le m$ and $1 \le j \le n$. Again, it will be convenient to use a matrix form:
\[
\begin{pmatrix}
	x_1y_1:X_1 \otimes Y_1 & \cdots & x_1y_n:X_1 \otimes Y_n \\
	\vdots & \ddots &\vdots \\
	x_my_1:X_m \otimes Y_1 & \cdots & x_my_n:X_m \otimes Y_n
\end{pmatrix}.
\]
The reason is that we will often consider precontexts obtained by dropping the last entry of (\texttt{*}), i.e.
$$
\partial(\textbf{X} \otimes \textbf{Y}) = (x_1y_1:X_1 \otimes Y_1, ..., x_1y_n:X_1 \otimes Y_n, ..., x_my_1:X_m \otimes Y_1, ..., x_my_n:X_m \otimes Y_{n-1}).
$$
We denote the latter by the ``incomplete matrix"
\[
\begin{pmatrix}
	x_1y_1:X_1 \otimes Y_1 & \cdots & x_1y_{n-1}:X_1 \otimes Y_{n-1} & x_1y_n:X_1 \otimes Y_n \\
	\vdots & \ddots &\vdots & \vdots \\
	x_{m-1}y_1:X_{m-1} \otimes Y_1 & \cdots & x_{m-1}y_{n-1}:X_{m-1} \otimes Y_{n-1} & x_{m-1}y_n:X_{m-1} \otimes Y_n\\
	x_my_1:X_m \otimes Y_1 & \cdots & x_my_{n-1}:X_m \otimes Y_{n-1} & -
\end{pmatrix}.
\]

\subsection{The tensor product of judgments}

\label{subsec: tensor product of judgments}

We will now define, for theories $\bbA$ and $\bbB$, an operation that assigns to \emph{some} pairs $(J,J')$, where $J$ (resp. $J'$) is a derivable judgment in $\bbA$ (resp. $\bbB$), a judgment $J \odot J'$ in the alphabet $\Sigma(\bbA) \otimes \Sigma(\bbB)$. This will be done according to the following table:
\begin{table}[h!]
\centering
\renewcommand{\arraystretch}{1.2}
\begin{tabular}{|c|c|c|c|c|}
	\hline
	\diagbox{Judgment in $\bbA$}{Judgment in $\bbB$} & \text{sort} & \text{term} & \text{sort eq.} & \text{term eq.} \\
	\hline
	\text{sort} & \text{sort} & \text{term} & \text{sort eq.} & \text{term eq.} \\
	\hline
	\text{term} & \text{term} & \text{term eq.} & \text{term eq.} & - \\
	\hline
	\text{sort equality} & \text{sort eq.} & \text{term eq.} & - & - \\
	\hline
	\text{term equality} & \text{term eq.} & - & - & - \\
	\hline
\end{tabular}
\caption{Rules for the tensor product of judgments}
\label{table: 1}
\end{table}

For instance, $J \odot J'$ will be a term equality judgment when $J$ (resp. $J'$) is a derivable term judgment in $\bbA$ (resp. a derivable sort equality judgment in $\bbB$), but it will not be defined when both $J$, $J'$ are term equality judgments.

\vspace{0.5em}

In what follows, we assume that every context in $\bbA$ or $\bbB$ is endowed with a variable not contained in it.\footnote{While this choice is required for the operations below to be defined, it is not essential in the sense that, once we construct the pretheory $\bbA \otimes \bbB$ and prove that it is a theory, different choices will yield sets of axioms that are equal up to variable renaming. Moreover, we can include such a choice function as part of the definition of a generalized algebraic theory without changing the definition of a theory morphism (hence without modifying the category $\GAT$ in an essential way).}

In all cases, we consider contexts $\textbf{X} = (x_1:X_1, ..., x_m:X_m)$ and $\textbf{Y} = (y_1:Y_1, ..., y_n:Y_n)$ in $\bbA$ and $\bbB$, respectively. Also, we let $x$ (resp. $y$) be the chosen variable not contained in $\textbf{X}$ (resp. $\textbf{Y}$).

\vspace{0.5em}

\newpage

\begin{center}
\textcolor{newpurple}{\textbf{(1) \hspace{0.3em}\normalsize{\underline{sort $\odot$ sort}}}}
\end{center}

\vspace{0.5em}
	
	$J$ is $\textbf{X} \vdashcustom U \tp$ and $J'$ is $\textbf{Y} \vdashcustom V \tp$. Let $\textbf{X}' = (\textbf{X}, x:U)$ and $\textbf{Y}' = (\textbf{Y}, y:U)$. We define $J \odot J'$ as
	$$
	\partial(\textbf{X}' \otimes \textbf{Y}') \vdashcustom U \otimes V \tp.
	$$
	More explicitly, writing $U = S(s_1, ..., s_k)$ and $V = T(t_1, ..., t_\ell)$, we have
	$$
	\begin{pmatrix}
		x_1y_1:X_1 \otimes Y_1 & \cdots & x_1y_n:X_1 \otimes Y_n & x_1y: X_1 \otimes V\\
		\vdots & \ddots & \vdots & \vdots\\
		x_my_1:X_m \otimes Y_1 & \cdots & x_my_n:X_m \otimes Y_n & x_my: X_m \otimes V\\
		xy_1:U \otimes Y_1 & \cdots & xy_n:U \otimes Y_n & -
	\end{pmatrix}
	\vdashcustom ST
	\begin{pmatrix}
		s_1 \otimes t_1 & \cdots & s_1 \otimes t_\ell & s_1 \otimes y^V\\
		\vdots & \ddots & \vdots & \vdots\\
		s_k \otimes t_1 & \cdots & s_k \otimes t_\ell & s_k \otimes y^V\\
		x^U \otimes t_1 & \cdots & x^U \otimes t_\ell & -
	\end{pmatrix}
	\tp.
	$$
	In particular, if $J$ and $J'$ are axioms with $U = S(x_1, ..., x_m)$ and $V = T(y_1, ..., y_n)$, then the above judgment is
	$$
	\begin{pmatrix}
		x_1y_1:X_1 \otimes Y_1 & \cdots & x_1y_n:X_1 \otimes Y_n & x_1y: X_1 \otimes V\\
		\vdots & \ddots & \vdots & \vdots\\
		x_my_1:X_m \otimes Y_1 & \cdots & x_my_n:X_m \otimes Y_n & x_my: X_m \otimes V\\
		xy_1:U \otimes Y_1 & \cdots & xy_n:U \otimes Y_n & -
	\end{pmatrix}
	\vdashcustom ST
	\begin{pmatrix}
		x_1y_1 & \cdots & x_1y_n & x_1y\\
		\vdots & \ddots & \vdots & \vdots\\
		x_my_1 & \cdots & x_my_n & x_my\\
		xy_1 & \cdots & xy_n & -
	\end{pmatrix}
	\tp.
	$$
	
	\vspace{0.5em}
	
	\begin{center}
		\textcolor{newpurple}{\textbf{(2) \hspace{0.3em}\normalsize{\underline{sort $\odot$ term}}}}
	\end{center}
	
	\vspace{0.5em}
	
	$J$ is $\textbf{X} \vdashcustom U \tp$ and $J'$ is $\textbf{Y} \vdashcustom v:V$. Let $\textbf{X}' = (\textbf{X}, x:U)$. We define $J \odot J'$ as
	$$
	\textbf{X}' \otimes \textbf{Y} \vdashcustom x^U \otimes v^V: (U \otimes V)[x_1^{X_1} \otimes v^V \mid x_1y, ..., x_m^{X_m} \otimes v^V \mid x_my].
	$$
	To describe it explicitly, write $U = S(\sigma_1, ..., \sigma_K)$ and $V = T(\tau_1, ..., \tau_L)$. Then $U \otimes V$ equals
	$$
	ST
	\begin{pmatrix}
		\sigma_1 \otimes \tau_1 & \cdots & \sigma_1 \otimes \tau_L & \sigma_1 \otimes y^V\\
		\vdots & \ddots & \vdots & \vdots\\
		\sigma_K \otimes \tau_1 & \cdots & \sigma_K \otimes \tau_L & \sigma_K \otimes y^V\\
		x^U \otimes \tau_1 & \cdots & x^U \otimes \tau_L & -
	\end{pmatrix}
	$$
	and, since $x_1y$, ..., $x_my$ do not occur in the terms $\sigma_i \otimes \tau_j$ and $x^U \otimes \tau_j$, by Lemma \ref{lem: formulas for tensor products of terms} we have
	$$
	(U \otimes V)[x_1^{X_1} \otimes v^V \mid x_1y, ..., x_m^{X_m} \otimes v^V \mid x_my] =
	ST
	\begin{pmatrix}
		\sigma_1 \otimes \tau_1 & \cdots & \sigma_1 \otimes \tau_L & \sigma_1 \varotimes v^V\\
		\vdots & \ddots & \vdots & \vdots\\
		\sigma_K \otimes \tau_1 & \cdots & \sigma_K \otimes \tau_L & \sigma_K \varotimes v^V\\
		x^U \otimes \tau_1 & \cdots & x^U \otimes \tau_L & -
	\end{pmatrix}.
	$$
	It follows that if $v = t(t_1, ..., t_\ell)$, then $J \odot J'$ equals
	$$
	\begin{pmatrix}
		x_1y_1:X_1 \otimes Y_1 & \cdots & x_1y_n:X_1 \otimes Y_n\\
		\vdots & \ddots & \vdots\\
		x_my_1:X_m \otimes Y_1 & \cdots & x_my_n:X_m \otimes Y_n\\
		xy_1:U \otimes Y_1 & \cdots & xy_n:U \otimes Y_n
	\end{pmatrix}
	\vdashcustom
	St
	\begin{pmatrix}
		\sigma_1 \otimes t_1 & \cdots & \sigma_1 \otimes t_\ell\\
		\vdots & \ddots & \vdots\\
		\sigma_K \otimes t_1 & \cdots & \sigma_K \otimes t_\ell\\
		x^U \otimes t_1 & \cdots & x^U \otimes t_\ell
	\end{pmatrix}
	:ST
	\begin{pmatrix}
		\sigma_1 \otimes \tau_1 & \cdots & \sigma_1 \otimes \tau_L & \sigma_1 \varotimes v^V\\
		\vdots & \ddots & \vdots & \vdots\\
		\sigma_K \otimes \tau_1 & \cdots & \sigma_K \otimes \tau_L & \sigma_K \varotimes v^V\\
		x^U \otimes \tau_1 & \cdots & x^U \otimes \tau_L & -
	\end{pmatrix}.
	$$
	Otherwise, if $v$ is a variable $y_j$, we replace the expression $St(\cdots)$ by $xy_j$ in the above judgment.
	
	Note that if $J$ and $J'$ are axioms, say with $U = S(x_1, ..., x_m)$ and $v = t(y_1, ..., y_n)$, then the expression $St(\cdots)$ becomes
	$$
	St
	\begin{pmatrix}
		x_1y_1 & \cdots & x_1y_n\\
		\vdots & \ddots & \vdots\\
		x_my_1 & \cdots & x_my_n\\
		xy_1 & \cdots & xy_n
	\end{pmatrix}.
	$$
	
	\vspace{0.5em}
	
	\newpage
	
	\begin{center}
		\textcolor{newpurple}{\textbf{(3) \hspace{0.3em}\normalsize{\underline{term $\odot$ sort}}}}
	\end{center}
	
	\vspace{0.5em}
	
	$J$ is $\textbf{X} \vdashcustom u:U$ and $J'$ is $\textbf{Y} \vdashcustom V \tp$. Let $\textbf{Y}' = (\textbf{Y}, y:V)$. We define $J \odot J'$ as
	$$
	\textbf{X} \otimes \textbf{Y}' \vdashcustom u^U \otimes y^V: (U \otimes V)[u^U \otimes y_1^{Y_1} \mid xy_1, ..., u^U \otimes y_n^{Y_n} \mid xy_n].
	$$
	To describe it explicitly, write $U = S(\sigma_1, ..., \sigma_K)$ and $V = T(\tau_1, ..., \tau_L)$. Then $U \otimes V$ is as in the previous item and, since $xy_1$, ..., $xy_n$ do not occur in the terms $\sigma_i \otimes \tau_j$ and $\sigma_i \otimes y^V$, by Lemma \ref{lem: formulas for tensor products of terms} we have
	$$
	(U \otimes V)[u^U \otimes y_1^{Y_1} \mid xy_1, ..., u^U \otimes y_n^{Y_n} \mid xy_n] =
	ST
	\begin{pmatrix}
		\sigma_1 \otimes \tau_1 & \cdots & \sigma_1 \otimes \tau_L & \sigma_1 \otimes y^V\\
		\vdots & \ddots & \vdots & \vdots\\
		\sigma_K \otimes \tau_1 & \cdots & \sigma_K \otimes \tau_L & \sigma_K \otimes y^V\\
		u^U \otimes \tau_1 & \cdots & u^U \otimes \tau_L & -
	\end{pmatrix}.
	$$
	It follows that if $u = s(s_1, ..., s_k)$, then $J \odot J'$ equals
	\begin{samepage}
	$$
	\begin{pmatrix}
		x_1y_1: X_1 \otimes Y_1 & \cdots & x_1y_n: X_1 \otimes Y_n & x_1y: X_1 \otimes V\\
		\vdots & \ddots & \vdots & \vdots\\
		x_my_1: X_m \otimes Y_1 & \cdots & x_my_n: X_m \otimes Y_n & x_my: X_m \otimes V
	\end{pmatrix}
	$$
	$$
	\top
	$$
	$$
	sT
	\begin{pmatrix}
		s_1 \otimes \tau_1 & \cdots & s_1 \otimes \tau_L & s_1 \otimes y^V\\
		\vdots & \ddots & \vdots & \vdots\\
		s_k \otimes \tau_1 & \cdots & s_k \otimes \tau_L & s_k \otimes y^V
	\end{pmatrix}
	:ST
	\begin{pmatrix}
		\sigma_1 \otimes \tau_1 & \cdots & \sigma_1 \otimes \tau_L & \sigma_1 \otimes y^V\\
		\vdots & \ddots & \vdots & \vdots\\
		\sigma_K \otimes \tau_1 & \cdots & \sigma_K \otimes \tau_L & \sigma_K \otimes y^V\\
		u^U \otimes \tau_1 & \cdots & u^U \otimes \tau_L & -
	\end{pmatrix}.
	$$
	\end{samepage}
	If, instead, $u$ is a variable $x_i$, we replace the expression $sT(\cdots)$ by $x_iy$ in the above judgment.
	
	Note that if $J$ and $J'$ are axioms, say with $u = s(x_1, ..., x_m)$ and $V = T(y_1, ..., y_n)$, then the expression $sT(\cdots)$ becomes
	$$
	sT
	\begin{pmatrix}
		x_1y_1 & \cdots & x_1y_n & x_1y\\
		\vdots & \ddots & \vdots & \vdots\\
		x_my_1 & \cdots & x_my_n & x_my
	\end{pmatrix}.
	$$
	
	\vspace{0.5em}
	
	\begin{center}
		\textcolor{newpurple}{\textbf{(4) \hspace{0.3em}\normalsize{\underline{term $\odot$ term}}}}
	\end{center}
	
	\vspace{0.5em}
	
	$J$ is $\textbf{X} \vdashcustom u:U$ and $J'$ is $\textbf{Y} \vdashcustom v:V$. We define $J \odot J'$ as
	$$
	\textbf{X} \otimes \textbf{Y} \vdashcustom u^U \otimes v^V \equiv u^U \varotimes v^V: \doverline{U \otimes V}
	$$
	where the double bar indicates an application of
	$$
	[x_1^{X_1} \otimes v^V \mid x_1y, ..., x_m^{X_m} \otimes v^V \mid x_my, ..., u^U \otimes y_1^{Y_1} \mid xy_1, ..., u^U \otimes y_n^{Y_n} \mid xy_n].
	$$
	To obtain a more explicit description, write $U = S(\sigma_1, ..., \sigma_K)$ and $V = T(\tau_1, ..., \tau_L)$. Then $U \otimes V$ is as in the previous items and, by using Lemma \ref{lem: formulas for tensor products of terms} to describe the passage from $U \otimes V$ to $\doverline{U \otimes V}$, we conclude that $J \odot J'$ is
	$$
	\begin{pmatrix}
		x_1y_1: X_1 \otimes Y_1 & \cdots & x_1y_n: X_1 \otimes Y_n\\
		\vdots & \ddots & \vdots\\
		x_my_1: X_m \otimes Y_1 & \cdots & x_my_n: X_m \otimes Y_n
	\end{pmatrix}
	\vdashcustom
	u^U \otimes v^V \equiv u^U \varotimes v^V: ST
	\begin{pmatrix}
		\sigma_1 \otimes \tau_1 & \cdots & \sigma_1 \otimes \tau_L & \sigma_1 \varotimes v^V\\
		\vdots & \ddots & \vdots & \vdots\\
		\sigma_K \otimes \tau_1 & \cdots & \sigma_K \otimes \tau_L & \sigma_K \varotimes v^V\\
		u^U \otimes \tau_1 & \cdots & u^U \otimes \tau_L & -
	\end{pmatrix}.
	$$
	While we cannot give a concise general description of $u^U \otimes v^V$ and $u^U \varotimes v^V$, we note that, following the previous subsection:
	\begin{itemize}
		\item[--] If $u$ or $v$ is a variable, $u^U \otimes v^V$ and $u^U \varotimes v^V$ are equal to each other by definition. If both are variables, say $x_i$ and $y_j$, their product is $x_iy_j$.
		
		\item[--] If neither $u$ nor $v$ is a variable, then $u^U \otimes v^V$ and $u^U \varotimes v^V$ have explicit matrix forms given by Remark \ref{rem: matrix form of tensor product of terms - not variables}. If $J$ and $J'$ are axioms, say with $u = s(x_1, ..., x_m)$ and $v = t(y_1, ..., y_n)$, then $u^U \otimes v^V \equiv u^U \varotimes v^V$ becomes
		$$
		St
		\begin{pmatrix}
			\sigma_1 \otimes y_1 & \cdots & \sigma_1 \otimes y_n\\
			\vdots & \ddots & \vdots\\
			\sigma_K \otimes y_1 & \cdots & \sigma_K \otimes y_n\\
			u^U \otimes y_1 & \cdots & u^U \otimes y_n
		\end{pmatrix}
		\equiv
		sT
		\begin{pmatrix}
			x_1 \otimes \tau_1 & \cdots & x_1 \otimes \tau_L & x_1 \otimes v^V\\
			\vdots & \ddots & \vdots & \vdots\\
			x_m \otimes \tau_1 & \cdots & x_m \otimes \tau_L & x_m \otimes v^V\\
		\end{pmatrix}.
		$$
	\end{itemize}
	
	\vspace{0.5em}
	
	\newpage
	
	\begin{center}
		\textcolor{newpurple}{\textbf{(5) \hspace{0.3em}\normalsize{\underline{sort $\odot$ sort equality}}}}
	\end{center}
	
	\vspace{0.5em}
	
	$J$ is $\textbf{X} \vdashcustom U \tp$ and $J'$ is $\textbf{Y} \vdashcustom V \equiv V' \tp$. Let $\textbf{X}' = (\textbf{X}, x:U)$ and $\textbf{Y}' = (\textbf{Y}, y:V)$. We define $J \odot J'$ as
	$$
	\partial(\textbf{X}' \otimes \textbf{Y}') \vdashcustom U^x \otimes V^y \equiv U^{'x} \otimes V^y \tp.
	$$
	To describe it explicitly, write $U = S(s_1, ..., s_k)$, $V = T(t_1, ..., t_\ell)$, and $V' = T'(t'_1, ..., t'_{\ell'})$. Then, proceeding as in the ``sort $\odot$ sort" case, $J \odot J$ equals
	$$
	\begin{pmatrix}
		x_1y_1:X_1 \otimes Y_1 & \cdots & x_1y_n:X_1 \otimes Y_n & x_1y: X_1 \otimes V\\
		\vdots & \ddots & \vdots & \vdots\\
		x_my_1:X_m \otimes Y_1 & \cdots & x_my_n:X_m \otimes Y_n & x_my: X_m \otimes V\\
		xy_1:U \otimes Y_1 & \cdots & xy_n:U \otimes Y_n & -
	\end{pmatrix}
	$$
	$$
	\top
	$$
	$$
	ST
	\begin{pmatrix}
		s_1 \otimes t_1 & \cdots & s_1 \otimes t_\ell & s_1 \otimes y^V\\
		\vdots & \ddots & \vdots & \vdots\\
		s_k \otimes t_1 & \cdots & s_k \otimes t_\ell & s_k \otimes y^V\\
		x^U \otimes t_1 & \cdots & x^U \otimes t_\ell & -
	\end{pmatrix}
	\equiv
	ST'
	\begin{pmatrix}
		s_1 \otimes t'_1 & \cdots & s_1 \otimes t'_{\ell'} & s_1 \otimes y^{V'}\\
		\vdots & \ddots & \vdots & \vdots\\
		s_k \otimes t'_1 & \cdots & s_k \otimes t'_{\ell'} & s_k \otimes y^{V'}\\
		x^U \otimes t'_1 & \cdots & x^U \otimes t'_{\ell'} & -
	\end{pmatrix}
	\tp.
	$$
	
	\vspace{0.5em}
	
	\begin{center}
		\textcolor{newpurple}{\textbf{(6) \hspace{0.3em}\normalsize{\underline{sort equality $\odot$ sort}}}}
	\end{center}
	
	\vspace{0.5em}
	
	$J$ is $\textbf{X} \vdashcustom U \equiv U' \tp$ and $J'$ is $\textbf{Y} \vdashcustom V \tp$. Let $\textbf{X}' = (\textbf{X}, x:U)$ and $\textbf{Y}' = (\textbf{Y}, y:V)$. We define $J \odot J'$ as
	$$
	\partial(\textbf{X}' \otimes \textbf{Y}') \vdashcustom U^x \otimes V^y \equiv U^{'x} \otimes V^y \tp.
	$$
	Writing $U = S(s_1, ..., s_k)$, $U' = S'(s'_1, ..., s'_{k'})$, and $V = T(t_1, ..., t_\ell)$, similarly to the previous case we have that the above judgment equals
	\begin{samepage}
	$$
	\begin{pmatrix}
		x_1y_1:X_1 \otimes Y_1 & \cdots & x_1y_n:X_1 \otimes Y_n & x_1y: X_1 \otimes V\\
		\vdots & \ddots & \vdots & \vdots\\
		x_my_1:X_m \otimes Y_1 & \cdots & x_my_n:X_m \otimes Y_n & x_my: X_m \otimes V\\
		xy_1:U \otimes Y_1 & \cdots & xy_n:U \otimes Y_n & -
	\end{pmatrix}
	$$
	$$
	\top
	$$
	$$
	ST
	\begin{pmatrix}
		s_1 \otimes t_1 & \cdots & s_1 \otimes t_\ell & s_1 \otimes y^V\\
		\vdots & \ddots & \vdots & \vdots\\
		s_k \otimes t_1 & \cdots & s_k \otimes t_\ell & s_k \otimes y^V\\
		x^U \otimes t_1 & \cdots & x^U \otimes t_\ell & -
	\end{pmatrix}
	\equiv
	S'T
	\begin{pmatrix}
		s'_1 \otimes t_1 & \cdots & s'_1 \otimes t_\ell & s'_1 \otimes y^V\\
		\vdots & \ddots & \vdots & \vdots\\
		s'_{k'} \otimes t_1 & \cdots & s'_{k'} \otimes t_\ell & s'_{k'} \otimes y^V\\
		x^{U'} \otimes t_1 & \cdots & x^{U'} \otimes t_\ell & -
	\end{pmatrix}
	\tp.
	$$
	\end{samepage}
		
	\vspace{0.5em}
	
	\begin{center}
		\textcolor{newpurple}{\textbf{(7) \hspace{0.3em}\normalsize{\underline{sort $\odot$ term equality}}}}
	\end{center}
	
	\vspace{0.5em}
	
	$J$ is $\textbf{X} \vdashcustom U \tp$ and $J'$ is $\textbf{Y} \vdashcustom v \equiv v': V$. Let $\textbf{X}' = (\textbf{X}, x:U)$. We define $J \odot J'$ as
	$$
	\textbf{X}' \otimes \textbf{Y} \vdashcustom x^U \otimes v^V \equiv x^U \otimes v^{'V}: (U \otimes V)[x_1^{X_1} \otimes v^V \mid x_1y, ..., x_m^{X_m} \otimes v^V \mid x_my].
	$$
	
	Let us describe it explicitly in the case where $v$ and $v'$ are not variables (otherwise, $x^U \otimes v^V$ or $x^U \otimes v^{'V}$ will be a variable). Suppose that
	$$
	U = S(\sigma_1, ..., \sigma_K), \quad V = T(\tau_1, ..., \tau_L), \quad v = t(t_1, ..., t_\ell), \quad v' = t'(t'_1, ..., t'_{\ell'}).
	$$
	Then, proceeding as in the ``sort $\odot$ term" case, $J \odot J'$ equals
	$$
	\begin{pmatrix}
		x_1y_1:X_1 \otimes Y_1 & \cdots & x_1y_n:X_1 \otimes Y_n\\
		\vdots & \ddots & \vdots\\
		x_my_1:X_m \otimes Y_1 & \cdots & x_my_n:X_m \otimes Y_n\\
		xy_1:U \otimes Y_1 & \cdots & xy_n:U \otimes Y_n
	\end{pmatrix}
	$$
	$$
	\top
	$$
	$$
	St
	\begin{pmatrix}
		\sigma_1 \otimes t_1 & \cdots & \sigma_1 \otimes t_\ell\\
		\vdots & \ddots & \vdots\\
		\sigma_K \otimes t_1 & \cdots & \sigma_K \otimes t_\ell\\
		x^U \otimes t_1 & \cdots & x^U \otimes t_\ell
	\end{pmatrix}
	\equiv
	St'
	\begin{pmatrix}
		\sigma_1 \otimes t'_1 & \cdots & \sigma_1 \otimes t'_{\ell'}\\
		\vdots & \ddots & \vdots\\
		\sigma_K \otimes t'_1 & \cdots & \sigma_K \otimes t'_{\ell'}\\
		x^U \otimes t'_1 & \cdots & x^U \otimes t'_{\ell'}
	\end{pmatrix}
	:ST
	\begin{pmatrix}
		\sigma_1 \otimes \tau_1 & \cdots & \sigma_1 \otimes \tau_L & \sigma_1 \varotimes v^V\\
		\vdots & \ddots & \vdots & \vdots\\
		\sigma_K \otimes \tau_1 & \cdots & \sigma_K \otimes \tau_L & \sigma_K \varotimes v^V\\
		x^U \otimes \tau_1 & \cdots & x^U \otimes \tau_L & -
	\end{pmatrix}.
	$$
	
	\vspace{0.5em}
	
	\begin{center}
		\textcolor{newpurple}{\textbf{(8) \hspace{0.3em}\normalsize{\underline{term equality $\odot$ sort}}}}
	\end{center}
	
	\vspace{0.5em}
	
	$J$ is $\textbf{X} \vdashcustom u \equiv u':U$ and $J'$ is $\textbf{Y} \vdashcustom V \tp$. Let $\textbf{Y}' = (\textbf{Y}, y:V)$. We define $J \odot J'$ as
	$$
	\textbf{X} \otimes \textbf{Y}' \vdashcustom u^U \otimes y^V \equiv u^{'U} \otimes y^Y: (U \otimes V)[u^U \otimes y_1^{Y_1} \mid xy_1, ..., u^U \otimes y_n^{Y_n} \mid xy_n].
	$$
	Let us describe explicitly in the case where $u$ and $u'$ are not variables. Suppose that
	$$
	U = S(\sigma_1, ..., \sigma_K), \quad V = T(\tau_1, ..., \tau_L), \quad u = s(s_1, ..., s_k), \quad u' = s'(s'_1, ..., s'_{k'}).
	$$
	Then, proceeding as in the ``term $\odot$ sort" case, $J \odot J'$ equals
	\begin{samepage}
		$$
		\begin{pmatrix}
			x_1y_1: X_1 \otimes Y_1 & \cdots & x_1y_n: X_1 \otimes Y_n & x_1y: X_1 \otimes V\\
			\vdots & \ddots & \vdots & \vdots\\
			x_my_1: X_m \otimes Y_1 & \cdots & x_my_n: X_m \otimes Y_n & x_my: X_m \otimes V
		\end{pmatrix}
		$$
		$$
		\top
		$$
		$$
		sT
		\begin{pmatrix}
			s_1 \otimes \tau_1 & \cdots & s_1 \otimes \tau_L & s_1 \otimes y^V\\
			\vdots & \ddots & \vdots & \vdots\\
			s_k \otimes \tau_1 & \cdots & s_k \otimes \tau_L & s_k \otimes y^V
		\end{pmatrix}
		\equiv
		s'T
		\begin{pmatrix}
			s'_1 \otimes \tau_1 & \cdots & s'_1 \otimes \tau_L & s'_1 \otimes y^V\\
			\vdots & \ddots & \vdots & \vdots\\
			s'_{k'} \otimes \tau_1 & \cdots & s'_{k'} \otimes \tau_L & s'_{k'} \otimes y^V
		\end{pmatrix}
		:ST
		\begin{pmatrix}
			\sigma_1 \otimes \tau_1 & \cdots & \sigma_1 \otimes \tau_L & \sigma_1 \otimes y^V\\
			\vdots & \ddots & \vdots & \vdots\\
			\sigma_K \otimes \tau_1 & \cdots & \sigma_K \otimes \tau_L & \sigma_K \otimes y^V\\
			u^U \otimes \tau_1 & \cdots & u^U \otimes \tau_L & -
		\end{pmatrix}.
		$$
	\end{samepage}

	\begin{center}
		\textcolor{newpurple}{\textbf{(9) \hspace{0.3em}\normalsize{\underline{sort equality $\odot$ term}}}}
	\end{center}
	
	\vspace{0.5em}
	
	$J$ is $\textbf{X} \vdashcustom U \equiv U' \tp$ and $J'$ is $\textbf{Y} \vdashcustom v:V$. Let $\textbf{X}' = (\textbf{X}, x:U)$. We define $J \odot J'$ as
	$$
	\textbf{X}' \otimes \textbf{Y} \vdashcustom x^U \otimes v^V \equiv x^{U'} \otimes v^V: (U \otimes V)[x_1^{X_1} \otimes v^V \mid x_1y, ..., x_m^{X_m} \otimes v^V \mid x_my].
	$$
	Let us describe it explicitly in the case where $v$ is not a variable. Suppose that
	$$
	U = S(\sigma_1, ..., \sigma_K), \quad U' = S'(\sigma'_1, ..., \sigma'_{K'}), \quad V = T(\tau_1, ..., \tau_L), \quad v = t(t_1, ..., t_\ell).
	$$
	Then, proceeding as in the ``sort $\odot$ term" case, $J \odot J'$ equals
	\begin{samepage}
	$$
	\begin{pmatrix}
		x_1y_1:X_1 \otimes Y_1 & \cdots & x_1y_n:X_1 \otimes Y_n\\
		\vdots & \ddots & \vdots\\
		x_my_1:X_m \otimes Y_1 & \cdots & x_my_n:X_m \otimes Y_n\\
		xy_1:U \otimes Y_1 & \cdots & xy_n:U \otimes Y_n
	\end{pmatrix}
	$$
	$$
	\top
	$$
	$$
	St
	\begin{pmatrix}
		\sigma_1 \otimes t_1 & \cdots & \sigma_1 \otimes t_\ell\\
		\vdots & \ddots & \vdots\\
		\sigma_K \otimes t_1 & \cdots & \sigma_K \otimes t_\ell\\
		x^U \otimes t_1 & \cdots & x^U \otimes t_\ell
	\end{pmatrix}
	\equiv
	S't
	\begin{pmatrix}
		\sigma'_1 \otimes t_1 & \cdots & \sigma'_1 \otimes t_\ell\\
		\vdots & \ddots & \vdots\\
		\sigma'_{K'} \otimes t_1 & \cdots & \sigma'_{K'} \otimes t_\ell\\
		x^{U'} \otimes t_1 & \cdots & x^{U'} \otimes t_\ell
	\end{pmatrix}
	:ST
	\begin{pmatrix}
		\sigma_1 \otimes \tau_1 & \cdots & \sigma_1 \otimes \tau_L & \sigma_1 \varotimes v^V\\
		\vdots & \ddots & \vdots & \vdots\\
		\sigma_K \otimes \tau_1 & \cdots & \sigma_K \otimes \tau_L & \sigma_K \varotimes v^V\\
		x^U \otimes \tau_1 & \cdots & x^U \otimes \tau_L & -
	\end{pmatrix}.
	$$
	\end{samepage}
	
	\vspace{0.5em}
	
	\begin{center}
		\textcolor{newpurple}{\textbf{(10) \hspace{0.3em}\normalsize{\underline{term $\odot$ sort equality}}}}
	\end{center}
	
	\vspace{0.5em}
	
	$J$ is $\textbf{X} \vdashcustom u:U$ and $J'$ is $\textbf{Y} \vdashcustom V \equiv V' \tp$. Let $\textbf{Y}' = (\textbf{Y}, y:V)$. We define $J \odot J'$ as
	$$
	\textbf{X} \otimes \textbf{Y}' \vdashcustom u^U \otimes y^V \equiv u^U \otimes y^{V'}: (U \otimes V)[u^U \otimes y_1^{Y_1} \mid xy_1, ..., u^U \otimes y_n^{Y_n} \mid xy_n].
	$$
	Let us describe it explicitly in the case where $u$ is not a variable. Suppose that
	$$
	U = S(\sigma_1, ..., \sigma_K), \quad V = T(\tau_1, ..., \tau_L), \quad V' = T'(\tau'_1, ..., \tau'_{L'}), \quad u = s(s_1, ..., s_k).
	$$
	Then, proceeding as in the ``term $\odot$ sort" case, $J \odot J'$ equals
	\begin{samepage}
		$$
		\begin{pmatrix}
			x_1y_1: X_1 \otimes Y_1 & \cdots & x_1y_n: X_1 \otimes Y_n & x_1y: X_1 \otimes V\\
			\vdots & \ddots & \vdots & \vdots\\
			x_my_1: X_m \otimes Y_1 & \cdots & x_my_n: X_m \otimes Y_n & x_my: X_m \otimes V
		\end{pmatrix}
		$$
		$$
		\top
		$$
		$$
		sT
		\begin{pmatrix}
			s_1 \otimes \tau_1 & \cdots & s_1 \otimes \tau_L & s_1 \otimes y^V\\
			\vdots & \ddots & \vdots & \vdots\\
			s_k \otimes \tau_1 & \cdots & s_k \otimes \tau_L & s_k \otimes y^V
		\end{pmatrix}
		\equiv
		sT'
		\begin{pmatrix}
			s_1 \otimes \tau'_1 & \cdots & s_1 \otimes \tau'_{L'} & s_1 \otimes y^{V'}\\
			\vdots & \ddots & \vdots & \vdots\\
			s_k \otimes \tau'_1 & \cdots & s_k \otimes \tau'_{L'} & s_k \otimes y^{V'}
		\end{pmatrix}
		:ST
		\begin{pmatrix}
			\sigma_1 \otimes \tau_1 & \cdots & \sigma_1 \otimes \tau_L & \sigma_1 \otimes y^V\\
			\vdots & \ddots & \vdots & \vdots\\
			\sigma_K \otimes \tau_1 & \cdots & \sigma_K \otimes \tau_L & \sigma_K \otimes y^V\\
			u^U \otimes \tau_1 & \cdots & u^U \otimes \tau_L & -
		\end{pmatrix}.
		$$
	\end{samepage}

\subsection{The tensor product of two theories (as a pretheory)}

\begin{definition}
\label{def: tensor product of theories}
Let $\bbA$ and $\bbB$ be generalized algebraic theories. Their \emph{tensor product}, denoted by $\bbA \otimes \bbB$, is the following pretheory (in the sense of \cite{Car86}):
\begin{itemize}
	\item its alphabet is $\Sigma(\bbA) \otimes \Sigma(\bbB)$ (see Definition \ref{def: tensor product of alphabets});
	
	\item its axioms are the judgments $J \odot J'$ where $J$ is an axiom in $\bbA$ and $J'$ is an axiom in $\bbB$.
\end{itemize}
\end{definition}

The correctness of this definition, i.e. each sort/term symbol corresponding to exactly one sort/term axiom and that axiom having the required form, follows from Table \ref{table: 1} and the explicit descriptions of $J \odot J'$ in the ``sort axiom $\odot$ sort axiom", ``sort axiom $\odot$ term axiom" and ``term axiom $\odot$ sort axiom" cases.

We will later prove (see Theorem \ref{th: tensor product is a theory}) that $\bbA \otimes \bbB$ is a theory, meaning that every axiom is \emph{well-formed} in the sense of \cite{Car86}. Due to the recursive nature of $\bbA$ and $\bbB$, this is a technically involved task: verifying that an axiom $J \odot J'$ is well-formed corresponds to checking the derivability of other judgments that are, generally, not axioms. Thus our strategy will be to prove not only the desired statement on the axioms, but also, along the way, that $J \odot J'$ is derivable for any derivable judgments $J$, $J'$. In fact, we will prove by induction on $h \ge 0$ that $J \odot J'$ is derivable whenever $\Ht(J)\Ht(J') = h$. For that reason, we use the following auxiliary definition:

\begin{definition}
\label{def: h-derivable}
Given $h \ge 0$, we say that the pair of generalized algebraic theories $(\bbA,\bbB)$ is \emph{$h$-derivable} if for any derivable judgments $J$, $J'$ in $\bbA$, $\bbB$, respectively, with $\Ht(J)\Ht(J') \le h$, we have that $J \odot J'$ is derivable.
\end{definition}

We will prove, in other words, that any pair of \textsc{gat}s $(\bbA,\bbB)$ is $h$-derivable for all $h \ge 0$.

\section{Examples}
\label{sec: examples}

In this section we describe a few instances of the tensor product of generalized algebraic theories as in Definition \ref{def: tensor product of theories}.

\subsection{An illustrative example: strict double categories}
\label{subsec: strict double categories}

One of the classical examples of a \textsc{gat} is the theory $\bbT_{\text{cat}}$ of categories, given by the following axioms\footnote{We follow the usual practice of omitting multiple consecutive occurrences of a sort in a context. For example, we write $x, y:O$ instead of $x:O, y:O$.} (see \cite{Car86}):
\begin{align*}
	&\vdashcustom O \tp \tag{sort of objects}\\[1em]
	x, y:O &\vdashcustom A(x,y) \tp \tag{sort of arrows from $x$ to $y$}\\[1em] x,y,z:O, f:A(x,y), g:A(y,z) & \vdashcustom \circ (x,y,z,f,g): A(x,z) \tag{composition}\\[1em]
	x:O &\vdashcustom id(x):A(x,x) \tag{identity arrows}\\[1em]
	x,y,z,w:O, f:A(x,y), g:A(y,z), h:A(z,w) & \vdashcustom \circ(x,y,w,f,\circ(y,z,w,g,h))\\
	& \equiv \circ(x,z,w,\circ(x,y,z,f,g),h):A(x,w) \tag{associativity}\\[1em]
	x,y:O, f:A(x,y) & \vdashcustom \circ(x,x,y,id(x),f) \equiv f: A(x,y) \tag{left unitality}\\[1em]
	x,y:O, f:A(x,y) & \vdashcustom \circ(x,y,y,f,id(y)) \equiv f: A(x,y) \tag{right unitality}
\end{align*}
Let us calculate $\bbT_{\text{cat}} \otimes \bbT_{\text{cat}}$, which will turn out to be the theory whose $\Set$-models are small strict double categories, i.e. (up to equivalence) categories internal to the $1$-category $\Cat$.

To organize the construction, in this case and in others, it can be helpful to start by identifying sub-theories of each of the factors and calculating the tensor product of those. In this case, note that the first two axioms above, which introduce $O$ and $A$, define the theory $\bbT_{\text{graph}}$ of directed graphs. Hence we will first describe $\bbT_{\text{graph}} \otimes \bbT_{\text{graph}}$, then extend it to two theories $\bbT_{\text{cat}} \otimes \bbT_{\text{graph}}$ and $\bbT_{\text{graph}} \otimes \bbT_{\text{cat}}$, and, finally, extend $(\bbT_{\text{cat}} \otimes \bbT_{\text{graph}}) \cup (\bbT_{\text{graph}} \otimes \bbT_{\text{cat}})$ to $\bbT_{\text{cat}} \otimes \bbT_{\text{cat}}$.

\subsubsection{Calculating $\bbT_{\text{graph}} \otimes \bbT_{\text{graph}}$}

Firstly, tensoring $\vdashcustom O \tp$ with itself yields
\[
\tag{Axiom 1}
\vdashcustom OO \tp.
\]
Now, letting $x$ and $k$ be the variables associated with the empty context and with $x,y:O$, respectively (recall that we choose for each context a variable not contained in it), by tensoring $\vdashcustom O \tp$ and $x,y:O \vdashcustom A(x,y) \tp$ we obtain
\[
\tag{Axiom 2}
xx:OO, xy: OO \vdashcustom OA(xx,xy) \tp.
\]
Similarly, from $x,y:O \vdashcustom A(x,y) \tp$ and $\vdashcustom O \tp$ we get
\[
\tag{Axiom 3}
xx, yx:OO \vdashcustom AO(xx,yx) \tp.
\]
The last axiom of $\bbT_{\text{graph}} \otimes \bbT_{\text{graph}}$ is obtained by tensoring $x,y:O \vdashcustom A(x,y) \tp$ with itself:
\[
\tag{Axiom 4}
\begin{pmatrix*}[l]
	xx:OO & xy:OO & xk:OA(xx,xy)\\
	yx:OO & yy:OO & yk:OA(yx,yy)\\
	kx:AO(xx,yx) & ky:AO(xy,yy) & -
\end{pmatrix*}
\vdashcustom AA
\begin{pmatrix*}[l]
	xx & xy & xk\\
	yx & yy & yk\\
	kx & ky & -
\end{pmatrix*}.
\]

Here, $OO$ can be thought of as the sort of objects; $OA(a,b)$ as the sort of ``horizontal arrows from $a$ to $b$"; $AO(a,b)$ as the sort of ``vertical arrows from $a$ to $b$"; and
$$
AA
\begin{pmatrix}
	a & b & f\\
	c & d & g\\
	h & i & -
\end{pmatrix}
$$
as the sort of ``($2$-dimensional) squares" $s$ filling the diagram that consists of the horizontal arrows $f$, $g$ and the vertical arrows $h$, $i$:
\[\begin{tikzcd}
	a & b \\
	c & d.
	\arrow["f", from=1-1, to=1-2]
	\arrow["h"', from=1-1, to=2-1]
	\arrow["s"{description}, draw=none, from=1-1, to=2-2]
	\arrow["i", from=1-2, to=2-2]
	\arrow["g"', from=2-1, to=2-2]
\end{tikzcd}\]

If we denote $OO$, $OA$, $AO$, $AA$ by $O$, $A^H$, $A^V$, $S$, respectively, rename variables, permute the sorts in the context\footnote{Note that these operations do not affect the set of derivable judgments of the resulting theory.}, and use the ``informal syntax" described in \cite{Car78} (in which one omits variables that are implicit from the context), the axioms above become
\begin{align*}
	& \vdashcustom O \tp \tag{Axiom 1'}\\
	& a,b:O \vdashcustom A^H(a,b) \tag{Axiom 2'}\\
	& a,b:O \vdashcustom A^V(a,b) \tag{Axiom 3'}\\
	& a,b,c,d:O, f:A^H(a,b), g:A^H(c,d), h:A^V(a,c), i:A^V(b,d) \vdashcustom S(f,g,h,i) \tag{Axiom 4'}
\end{align*}

Let $D$ be the category having two objects, say $0$ and $1$, and two non-identity arrows $i$, $j:0 \rightarrow 1$. Then $\Mod(\bbT_{\text{graph}} \otimes \bbT_{\text{graph}})$ is equivalent to to the category of presheaves $\PSh(D \times D) = \Fun((D \times D)^{op},\Set)$. We will call such a presheaf a \emph{double (directed) graph}.

\begin{remark}
At this point, it is possible to observe a pattern which, in general, can be verified by induction in a straightforward way: a product of derivable judgments $(\vdashcustom S \tp) \odot J$, where $\vdashcustom S \tp$ and $J$ come from theories $\bbS$ and $\bbT$, resp., can be obtained from $J$ by replacing (1) each sort symbol $T$ by $ST$, (2) each term symbol $t$ by $St$, and (3) each variable $y$ by $xy$, where $x$ is the variable associated with the empty context in $\bbT$. Similarly, the judgment $J \odot (\vdashcustom S \tp)$ is obtained from $J$ by replacing (1) each $T$ by $TS$, (2) each $t$ by $tS$, and (3) each $y$ by $yx$.
\end{remark}

\subsubsection{Calculating $\bbT_{\text{cat}} \otimes \bbT_{\text{graph}}$ and $\bbT_{\text{graph}} \otimes \bbT_{\text{cat}}$}

To describe $\bbT_{\text{cat}} \otimes \bbT_{\text{graph}}$, we must add to $\bbT_{\text{graph}} \otimes \bbT_{\text{graph}}$, as axioms, all $J \odot J'$ where $J$ is a non-sort axiom in $\bbT_{\text{cat}}$ and $J'$ is an axiom in $\bbT_{\text{graph}}$. Following the above remark on how to express products involving $\vdashcustom O \tp$, we directly obtain the following axioms, where we write $\circ^V$ for $(\circ,O)$ and $id^V$ for $(id,O)$ (the superscript $V$ stands for ``vertical"):
\begin{align*}
	&xx,yx,zx:OO,fx:AO(xx,yx),gx:AO(yx,zx) \vdashcustom \circ^V(xx,yx,zx,fx,gx):AO(xx,zx) \tag{Axiom 5}\\[1em]
	&xx:OO \vdashcustom id^V(xx): AO(xx,xx)\tag{Axiom 6}\\[1em]
	&xx,yx,zx,wx:OO, fx:AO(xx,yx), gx:AO(yx,zx), hx:AO(zx,wx) \vdashcustom\tag{Axiom 7}\\
	&\circ^V(xx,yx,wx,fx,\circ^V(yx,zx,wx,gx,hx)) \equiv \circ^V(xx,zx,wx, \circ^V(xx,yx,zx,fx,gx),hx): AO(xx,wx)\\[1em]
	&xx,yx:OO,fx:AO(xx,yx) \vdashcustom \circ^V(xx,xx,yx,id^V(xx),fx) \equiv fx:AO(xx,yx)\tag{Axiom 8}\\[1em]
	&xx,yx:OO,fx:AO(xx,yx) \vdashcustom \circ^V(xx,yx,yx,fx,id^V(yx)) \equiv fx:AO(xx,yx) \tag{Axiom 9}
\end{align*}

We think of these axioms as prescribing, respectively, the following structures/properties: composition of vertical arrows (``vertical composition"), vertical identities, associativity of vertical composition, vertical left unitality, and vertical right unitality.

In fact, by using the informal syntax, representing $\circ^V$ in infix notation, and renaming the symbols $OO$, $OA$, etc. as in the description of $\bbT_{\text{graph}} \otimes \bbT_{\text{graph}}$, we obtain the following list of axioms:
\begin{align*}
	&a,b,c:O,f:A^V(a,b),g:A^V(b,c) \vdashcustom \circ^V(f,g):A^V(a,c) \tag{Axiom 5'}\\[1em]
	&a:O \vdashcustom id^V(a): A^V(a,a)\tag{Axiom 6'}\\[1em]
	&a,b,c,d:O, f:A^V(a,b), g:A^V(b,c), h:A^V(c,d) \vdashcustom f \circ^V (g \circ^V h) \equiv (f \circ^V g) \circ^V h: A^V(a,d)\tag{Axiom 7'}\\[1em]
	&a,b:O,f:A^V(a,b) \vdashcustom id^V(a) \circ^V f \equiv f:A^V(a,b)\tag{Axiom 8'}\\[1em]
	&a,b:O,f:A^V(a,b) \vdashcustom f \circ^V id^V(b) \equiv f:A^V(a,b) \tag{Axiom 9'}
\end{align*}

Now, we add the tensor products $J \odot (x,y:O \vdashcustom A(x,y) \tp)$ where $J$ is a non-sort axiom in $\bbT_{\text{cat}}$.

When $J$ is the axiom for composition, we obtain, denoting $(\circ,A)$ by $\square^V$,
$$
\begin{pmatrix*}[l]
	xx:OO & xy:OO & xk:A(xx,xy)\\
	yx:OO & yy:OO & yk:A(yx,yy)\\
	zx:OO & zy:OO & zk:A(zx,zy)\\
	fx:AO(xx,yx) & fy:AO(xy,yy) & fk:AA(xx,xy,xk,yx,yy,yk,fx,fy)\\
	gx:AO(yx,zx) & gy:AO(yy,zy) & gk:AA(yx,yy,yk,zx,zy,zk,gx,gy)
\end{pmatrix*}
$$
$$
\top
$$
\[
\tag{Axiom 10}
\square^V
\begin{pmatrix}
xx & xy & xk\\
yx & yy & yk\\
zx & zy & zk\\
fx & fy & fk\\
gx & gy & gk
\end{pmatrix}
:AA
\begin{pmatrix*}[l]
	xx & xy & xk\\
	zx & zy & zk\\
	\circ(x,y,z,f,g)^{A(x,z)} \otimes x^O & \circ(x,y,z,f,g)^{A(x,z)} \otimes y^O & -
\end{pmatrix*}.
\]
The latter sort expression, in turn, equals
$$
AA
\begin{pmatrix*}[l]
	xx & xy & xk\\
	zx & zy & zk\\
	\circ^V(xx,yx,zx,fx,gx) & \circ^V(xy,yy,zy,fy,gy) & -
\end{pmatrix*}.
$$
We view $\square^V$ as encoding vertical composition of 2-cells: using the informal syntax, permuting the variables and writing $\square^V$ in infix notation, this axiom becomes
\[
\tag{Axiom 10'}
\begin{pmatrix*}[l]
	a:O & a':O & f:A^H(a,b)\\
	b:O & b':O & g:A^H(c,d)\\
	c:O & c':O & h:A^H(e,f)\\
	i:A^V(a,c) & j:A^V(b,d) & s:S(f,g,i,j)\\
	k:A^V(c,e) & l:A^V(d,f) & t:S(g,h,k,l)
\end{pmatrix*}
\vdashcustom
s \square^V t: S(f,h,i \circ^V k,j \circ^V l).
\]

Next, let us compute $J \odot (x,y:O \vdashcustom A(x,y) \tp)$ where $J$ is the axiom that introduces $id$: writing $id_\square^V$ (``vertical identity of a horizontal arrow") for $(id,A)$, we obtain
\[
\tag{Axiom 11}
xx:OO, xy:OO, xk:OA(xx,xy) \vdashcustom id_\square^V(xx,xy,xk):AA
\begin{pmatrix*}[l]
	xx & xy & xk\\
	xx & xy & xk\\
	id(x)^{A(x,x)} \otimes x^O & id(x)^{A(x,x)} \otimes y^O & -
\end{pmatrix*},
\]
and the latter sort expression simplifies to
$$
AA
\begin{pmatrix*}[l]
	xx & xy & xk\\
	xx & xy & xk\\
	id^V(xx) & id^V(xy) & -
\end{pmatrix*}.
$$
In the informal syntax we have
\[
\tag{Axiom 11'}
x,y:O, f:A^H(x,y) \vdashcustom id_\square^V(f): S(f,f,id(a),id(b)).
\]

Now, we calculate $J \odot (x,y:O \vdashcustom A(x,y) \tp)$ where $J$ is the associativity axiom:

$$
\begin{pmatrix*}[l]
	xx:OO & xy:OO & xk:OA(xx,xy)\\
	yx:OO & yy:OO & yk:OA(yx,yy)\\
	zx:OO & zy:OO & zk:OA(zx,zy)\\
	wx:OO & wy:OO & wk:OA(wx,wy)\\
	fx:AO(xx,yx) & fy:AO(xy,yy) & fk:AA(xx,xy,xk,yx,yy,yk,fx,fy)\\
	gx:AO(yx,zx) & gy:AO(yy,zy) & gk:AA(yx,yy,yk,zx,zy,zk,gx,gy)\\
	hx:AO(zx,wx) & hy:AO(zy,wy) & hk:AA(zx,zy,zk,wx,wy,wk,hx,hy)\\
\end{pmatrix*}
$$
$$
\top
$$
$$
\square^V
\begin{pmatrix*}[l]
	xx & xy & xk\\
	yx & yy & yk\\
	wx & wy & wk\\
	fx & xy & fk\\
	\circ(y,z,w,g,h)^{A(y,w)} \otimes x^O & \circ(y,z,w,g,h)^{A(y,w)} \otimes y^O & \circ(y,z,w,g,h)^{A(y,w)} \otimes k^{A(x,y)}
\end{pmatrix*}
$$
$$
\vertequiv
$$
$$
\square^V
\begin{pmatrix*}[l]
	xx & xy & xk\\
	zx & zy & zk\\
	wx & wy & wk\\
	\circ(x,y,z,f,g)^{A(x,z)} \otimes x^O & \circ(x,y,z,f,g)^{A(x,z)} \otimes y^O & \circ(x,y,z,f,g)^{A(x,z)} \otimes k^{A(x,y)}\\
	hx & hy & hk
\end{pmatrix*}
$$
\vspace{1mm}
$$
\cdot \cdot
$$
$$
AA
\begin{pmatrix*}[l]
	xx & xy & xk\\
	wx & wy & wk\\
	\circ(x,y,w,f,\circ(y,z,w,g,h))^{A(x,w)} \otimes x^O & \circ(x,y,w,f,\circ(y,z,w,g,h))^{A(x,w)} \otimes y^O & -
\end{pmatrix*}.
$$
\[
\tag{Axiom 12}
\]

To complete the description, we note that the terms involving ``$\otimes$" are as follows:
\begin{align*}
	\circ(y,z,w,g,h)^{A(y,w)} \otimes x^O & = \circ^V(yx,zx,wx,gx,hx)\\
	\circ(y,z,w,g,h)^{A(y,w)} \otimes y^O & = \circ^V(yy,zy,wy,gy,hy)\\
	\circ(y,z,w,g,h)^{A(y,w)} \otimes k^{A(x,y)} & = \square^V(yx,yy,yk,zx,zy,zk,wx,wy,wk,gx,gy,gk,hx,ky,hk)\\
	\circ(x,y,z,f,g)^{A(x,z)} \otimes x^O & = \circ^V(xx,yx,zx,fx,gx)\\
	\circ(x,y,z,f,g)^{A(x,z)} \otimes y^O & = \circ^V(xy,yy,zy,fy,gy)\\
	\circ(x,y,z,f,g)^{A(x,z)} \otimes k^{A(x,y)} & = \square^V(xx,xy,xk,yx,yy,yk,zx,zy,zk,fx,fy,fk,gx,gy,gk)\\
	\circ(x,y,w,f,\circ(y,z,w,g,h))^{A(x,w)} \otimes x^O & = \circ^V(xx,yx,wx,fx,\circ^V(yx,zx,wx,gx,hx))\\
	\circ(x,y,w,f,\circ(y,z,w,g,h))^{A(x,w)} \otimes y^O & = \circ^V(xy,yy,wy,fy,\circ^V(yy,zy,wy,gy,hy)).
\end{align*}

The counterpart of Axiom 12 in the informal syntax is
	$$
	\begin{pmatrix*}[l]
		a_1:O & b_1:O & f_1:A^H(a_1,b_1)\\
		a_2:O & b_2:O & f_2:A^H(a_2,b_2)\\
		a_3:O & b_3:O & f_3:A^H(a_3,b_3)\\
		a_4:O & b_4:O & f_4:A^H(a_4,b_4)\\
		g_1:A^V(a_1,a_2) & h_1:A^V(b_1,b_2) & s_1:S(f_1,f_2,g_1,h_1)\\
		g_2:A^V(a_2,a_3) & h_2:A^V(b_2,b_3) & s_2:S(f_2,f_3,g_2,h_2)\\
		g_3:A^V(a_3,a_4) & h_3:A^V(b_3,b_4) & s_3:S(f_3,f_4,g_3,h_3)\\
	\end{pmatrix*}
	$$
	$$
	\top
	$$
	$$
	s_1 \square^V (s_2 \square^V s_3) \equiv (s_1 \square^V s_2) \square^V s_3: S(f_1,f_4,g_1 \circ^V (g_2 \circ^V g_3), h_1 \circ^V (h_2 \circ^V h_3)).
	$$
	\[
	\tag{Axiom 12'}
	\]

Now, when $J$ is the left unitality axiom, $J \odot (x,y:O \vdashcustom A(x,y))$ equals
$$
\begin{pmatrix}
	xx:OO & xy:OO & xk:OA(xx,xy) \\
	yx:OO & yy:OO & yk:OA(yx,yy)\\
	fx:AO(xx,yx) & fy:AO(xy,yy) & fk:S(xx,xy,xk,yx,yy,yk,fx,fy)
\end{pmatrix}
$$
$$
\top
$$
\[
\tag{Axiom 13}
\circ(x,x,y,id(x),f) \otimes k^{A(x,y)} \equiv fk: S(xx,xy,xk,yx,yy,yk,fx,fy),
\]
and the remaining expression is given by
$$
\circ(x,x,y,id(x),f) \otimes k^{A(x,y)}
=
\square^V
\begin{pmatrix}
xx & xy & xk\\
xx & xy & xk\\
yx & yy & yk\\
id^V(xx) & id^V(xy) & id_\square^V(xx,xy,xk)\\
fx & fy & fk
\end{pmatrix}.
$$
In the informal syntax we have
\[
\tag{Axiom 13'}
\begin{pmatrix}
	a:O & b:OO & f:A^V(a,b) \\
	c:O & d:O & g:A^V(c,d)\\
	h:A^V(a,c) & i:A^V(b,d) & s:S(a,b,f,c,d,g,h,i)
\end{pmatrix}
\vdashcustom id_\square^V(f) \square^V s \equiv s: S(f,g,h,i).
\]
When $J$ is the right unitality axiom, $J \odot (x,y:O \vdashcustom A(x,y))$ is calculated similarly: we have
$$
\begin{pmatrix}
	xx:OO & xy:OO & xk:OA(xx,xy) \\
	yx:OO & yy:OO & yk:OA(yx,yy)\\
	fx:AO(xx,yx) & fy:AO(xy,yy) & fk:S(xx,xy,xk,yx,yy,yk,fx,fy)
\end{pmatrix}
$$
$$
\top
$$
\[
\tag{Axiom 14}
\circ(x,y,y,f,id(y)) \otimes k^{A(x,y)} \equiv fk: S(xx,xy,xk,yx,yy,yk,fx,fy),
\]
and the remaining expression is given by
$$
\circ(x,y,y,f,id(y)) \otimes k^{A(x,y)}
=
\square^V
\begin{pmatrix}
	xx & xy & xk\\
	yx & yy & yk\\
	yx & yy & yk\\
	fx & fy & fk\\
	id^V(yx) & id^V(yy) & id_\square^V(yx,yy,yk)
\end{pmatrix}.
$$
The corresponding judgment in the informal syntax is
\[
\tag{Axiom 14'}
\begin{pmatrix}
	a:O & b:O & f:A^V(a,b) \\
	c:O & d:O & g:A^V(c,d)\\
	h:A^V(a,c) & i:A^V(b,d) & s:S(a,b,f,c,d,g,h,i)
\end{pmatrix}
\vdashcustom
s \square^V id_\square^V(g) \equiv g:S(f,g,h,i).
\]
This concludes the description of $\bbT_{\text{cat}} \otimes \bbT_{\text{graph}}$. We can view axioms 5-14 as extending $\bbT_{\text{graph}} \otimes \bbT_{\text{graph}}$ with operations that endow a given double graph with a category structure whose objects are the horizontal arrows and whose morphisms are the squares.

\vspace{0.5em}

The description of $\bbT_{\text{graph}} \otimes \bbT_{\text{cat}}$ is symmetric to that of $\bbT_{\text{cat}} \otimes \bbT_{\text{graph}}$: the axioms will be the judgments $J \odot J'$ where $J$ is an axiom in $\bbT_{\text{graph}}$ and $\bbT$ is a non-sort axiom in $\bbT_{\text{cat}}$. We leave the full description of these axioms (15 to 24) to the interested reader. Models of $\bbT_{\text{graph}} \otimes \bbT_{\text{cat}}$ will be double graphs with a category structure having the vertical arrows as its objects and the squares as its morphisms; thus in this case we depict squares as composing horizontally. In what follows, we will denote the symbols $(O,\circ)$, $(O,id)$, $(A,\circ)$, $(A,id)$ by $\circ^H$, $id^H$, $\square^H$, $id_\square^H$, respectively.

\subsubsection{Calculating $\bbT_{\text{cat}} \otimes \bbT_{\text{cat}}$}

The theory $\bbT_{\text{cat}} \otimes \bbT_{\text{graph}} \cup \bbT_{\text{graph}} \otimes \bbT_{\text{cat}}$ has as models double graphs endowed with two category structures, each having the squares as its morphisms: one in which squares compose vertically (having horizontal arrows are their sources/targets), and one where they compose vertically (having vertical arrows as their sources/targets). However, there is no other required relation between these two categories. That role will be played by the axioms added via the extension $\bbT_{\text{cat}} \otimes \bbT_{\text{graph}} \cup \bbT_{\text{graph}} \otimes \bbT_{\text{cat}} \rightarrow \bbT_{\text{cat}} \otimes \bbT_{\text{cat}}$, which establishes a kind of ``parameterized commutativity" between the composition and identity-assigning operations on both sides.

\hspace{0.5em}

We must add as axioms the judgments $J \odot J'$ where $J$, $J'$ are non-sort axioms in $\bbT_{\text{cat}}$. In fact, it suffices to consider a pair of term axioms.

Firstly, suppose that both $J$, $J'$ are the axiom that introduces $\circ$. Then $J \odot J'$ is
\begin{small}
$$
\begin{pmatrix*}[l]
	xx:OO & xy:OO & xz:OO & xf:OA(xx,xy) & xg:OA(xy,xz)\\
	yx:OO & yy:OO & yz:OO & yf:OA(yx,yy) & yg:OA(yy,yz)\\
	zx:OO & zy:OO & zz:OO & zf:OA(zx,zy) & zg:OA(zy,zz)\\
	fx:AO(xx,yx) & fy:AO(xy,yy) & fz:AO(xz,yz) & ff:AA(xx,xy,xk,yx,yy,yk,kx,ky) & fg:AA(xy,xz,xk,yy,yz,yk,ky,kz)\\
	gx:AO(yx,zx) & gy:AO(yy,zy) & gz:AO(yz,zz) & gf:AA(yx,yy,yk,zx,zy,zk,kx,ky) & gg:AA(yy,yz,yk,zy,zz,zk,ky,kz)
\end{pmatrix*}
$$
$$
\top
$$
$$
\square^H
\begin{pmatrix}
	xx & xy & xz & xf & xg \\
	zx & zy & zz & zf & zg\\
	\circ(x,y,z,f,g) \otimes x^O & \circ(x,y,z,f,g) \otimes y^O & \circ(x,y,z,f,g) \otimes z^O & \circ(x,y,z,f,g) \otimes f^{A(x,y)} & \circ(x,y,z,f,g) \otimes g^{A(y,z)}
\end{pmatrix}
$$
$$
\vertequiv
$$
$$
\square^V
\begin{pmatrix}
	xx & xz & x^O \otimes \circ(x,y,z,f,g)\\
	yx & yz & y^O \otimes \circ(x,y,z,f,g)\\
	zx & zz & z^O \otimes \circ(x,y,z,f,g)\\
	fx & fz & f^{A(x,y)} \otimes \circ(x,y,z,f,g)\\
	gx & gz & g^{A(y,z)} \otimes \circ(x,y,z,f,g)
\end{pmatrix}
$$
$$
\cdot \cdot
$$
\[
\tag{Axiom 25}
AA
\begin{pmatrix}
	xx & xz & x^O \otimes \circ(x,y,z,f,g)\\
	zx & zz & z^O \otimes \circ(x,y,z,f,g)\\
	\circ(x,y,z,f,g) \otimes x^O & \circ(x,y,z,f,g) \otimes z^O & -
\end{pmatrix}.
\]
\end{small}

The terms involving ``$\otimes$" can be calculated in a straightforward way; for example, we have
$$
\circ(x,y,z,f,g) \otimes f^{A(x,y)} = \square^V
\begin{pmatrix}
	xx & xy & xf\\
	yx & yy & yf\\
	zx & zy & zf\\
	fx & fy & ff\\
	gx & gy & gf
\end{pmatrix},
\qquad
\circ(x,y,z,f,g) \otimes g^{A(y,z)} = \square^V
\begin{pmatrix}
	xy & xz & xg\\
	yy & yz & yg\\
	zy & zz & zg\\
	fy & fz & fg\\
	gy & gz & gg
\end{pmatrix},
$$
$$
f^{A(x,y)} \otimes \circ(x,y,z,f,g) = \square^H
\begin{pmatrix}
	xx & xy & xz & xf & xg\\
	yx & yy & yz & yf & yg\\
	fx & fy & fz & ff & fg
\end{pmatrix},
\qquad
g^{A(y,z)} \otimes \circ(x,y,z,f,g) = \square^H
\begin{pmatrix}
	yx & yy & yz & yf & yg\\
	zx & zy & zz & zf & zg\\
	gx & gy & gz & gf & gg
\end{pmatrix}.
$$
In the informal syntax, we have
\begin{small}
	$$
	\begin{pmatrix*}[l]
		a_{11}:O & a_{12}:O & a_{13}:O & h_{11}:A^H(a_{11},a_{12}) & h_{12}:A^H(a_{12},a_{13})\\
		a_{21}:O & a_{22}:O & a_{23}:O & h_{21}:A^H(a_{21},a_{22}) & h_{22}:A^H(a_{22},a_{23})\\
		a_{31}:O & a_{31}:O & a_{33}:O & h_{31}:A^H(a_{31},a_{32}) & h_{32}:A^H(a_{32},a_{33})\\
		v_{11}:A^V(a_{11},a_{21}) & v_{12}:A^V(a_{12},a_{22}) & v_{13}:A^V(a_{13},a_{23}) & s_{11}:S(h_{11},h_{21},v_{11},v_{12}) & s_{12}:S(h_{12},h_{22},v_{12},v_{13})\\
		v_{21}:A^V(a_{21},a_{31}) & v_{22}:A^V(a_{22},a_{32}) & v_{23}:A^V(a_{23},a_{33}) & s_{21}:S(h_{21},h_{31},v_{21},v_{22}) & s_{22}:S(h_{22},h_{32},v_{22},v_{23})
	\end{pmatrix*}
	$$
	$$
	\top
	$$
	\[
	(s_{11} \square^V s_{21}) \square^H (s_{12} \square^V s_{22}) \equiv (s_{11} \square^H s_{12}) \square^V (s_{21} \square^H s_{22}): S(h_{11} \circ h_{12}, h_{31} \circ h_{32},v_{11} \circ v_{21}, v_{13} \circ v_{23}).
	\]
	\[
	\tag{Axiom 25'}
	\]
\end{small}
We view this judgment as the assertion that the $2$-dimensional diagram below on the left has an unambiguous pasting as a square whose boundary is as on the right:
\[\begin{tikzcd}
	{a_{11}} & {a_{12}} & {a_{13}} &&& {a_{11}} && {a_{13}} \\
	{a_{21}} & {a_{22}} & {a_{23}} \\
	{a_{31}} & {a_{32}} & {a_{33}} &&& {a_{31}} && {a_{33}.}
	\arrow["{h_{11}}", from=1-1, to=1-2]
	\arrow["{v_{11}}"', from=1-1, to=2-1]
	\arrow["{s_{11}}"{description}, draw=none, from=1-1, to=2-2]
	\arrow["{h_{12}}", from=1-2, to=1-3]
	\arrow["{v_{12}}"{description}, from=1-2, to=2-2]
	\arrow["{s_{12}}"{description}, draw=none, from=1-2, to=2-3]
	\arrow["{v_{13}}", from=1-3, to=2-3]
	\arrow["{h_{11} \circ^H h_{12}}", from=1-6, to=1-8]
	\arrow["{v_{11} \circ^V v_{21}}"', from=1-6, to=3-6]
	\arrow["{v_{13} \circ^V v_{23}}", from=1-8, to=3-8]
	\arrow["{h_{21}}"{description}, from=2-1, to=2-2]
	\arrow["{v_{21}}"', from=2-1, to=3-1]
	\arrow["{s_{21}}"{description}, draw=none, from=2-1, to=3-2]
	\arrow["{h_{22}}"{description}, from=2-2, to=2-3]
	\arrow["{v_{22}}"{description}, from=2-2, to=3-2]
	\arrow["{s_{22}}"{description}, draw=none, from=2-2, to=3-3]
	\arrow["{v_{23}}", from=2-3, to=3-3]
	\arrow["{h_{31}}"', from=3-1, to=3-2]
	\arrow["{h_{32}}"', from=3-2, to=3-3]
	\arrow["{h_{31} \circ^H h_{32}}"', from=3-6, to=3-8]
\end{tikzcd}\]

Now, suppose that $J$ is the axiom that introduces $\circ$ and $J'$ is the one that introduces $id$. Then $J \odot J'$ equals
\[
\tag{Axiom 26}
\begin{pmatrix}
	xx:OO\\
	yx:OO\\
	zx:OO\\
	fx:AO(xx,yx)\\
	gx:AO(yx,zx)
\end{pmatrix}
\]
$$
\top
$$
$$
id_\square^H
\begin{pmatrix}
	xx\\
	\circ(x,y,z,f,g) \otimes x^O
\end{pmatrix}
\equiv
\square^V
\begin{pmatrix}
	xx & xx & x^O \otimes id(x)\\
	yx & yx & y^O \otimes id(x)\\
	zx & zx & z^O \otimes id(x)\\
	fx & fx & f^{A(x,y)} \otimes id(x)\\
	gx & gx & g^{A(y,z)} \otimes id(x)
\end{pmatrix}
:AA
\begin{pmatrix}
	xx & xx & x^O \otimes id(x)\\
	zx & zx & z^O \otimes id(x)\\
	\circ(x,y,z,f,g) \otimes x^O & \circ(x,y,z,f,g) \otimes x^O & -
\end{pmatrix}.
$$
In the informal syntax, renaming the variables and sort/term symbols, we have the judgment
\[
a_1, a_2, a_3:O, v_1:A^V(a_1, a_2), v_2:A^V(a_2,a_3)
\]
$$
\top
$$
\[
\tag{Axiom 26'}
id_\square^H(v_1 \circ^V v_2) \equiv id_\square^H(v_1) \square^V id_\square^H(v_2): S(id^H(a_1), id^H(a_3), v_1 \circ^V v_2, v_1 \circ^V v_2)
\]
Similarly, if $J$ is the axiom that introduces $id$ and $J'$ is the one that introduces $\circ$, then $J \odot J'$ equals
$$
\begin{pmatrix}
	xx:OO & xy:OO & xz:OO & xf:OA(xx,xy) & xg:OA(xy,xz)
\end{pmatrix}
$$
$$
\top
$$
$$
\square^H
\begin{pmatrix}
	xx & xy & xz & xf & xg\\
	xx & xy & xz & xf & xg\\
	id(x) \otimes x^O & id(x) \otimes y^O & id(x) \otimes x^O & id(x) \otimes f^{A(x,y)} & id(x) \otimes g^{A(y,z)}
\end{pmatrix}
\equiv
id^V
\begin{pmatrix}
	xx & xy & x^O \otimes \circ(x,y,z,f,g)^{A(x,z)}
\end{pmatrix}
$$
$$
\cdot \cdot
$$
\[
\tag{Axiom 27}
AA
\begin{pmatrix}
	xx & xz & x^O \otimes \circ(x,y,z,f,g)^{A(x,z)}\\
	xx & xz & x^O \otimes \circ(x,y,z,f,g)^{A(x,z)}\\
	id(x) \otimes x^O & id(x) \otimes z^O & -
\end{pmatrix}.
\]
In the informal syntax we obtain
$$
a_1, a_2, a_3:O, h_1:A^H(a_1, a_2), h_2:A^H(a_2, a_3)
$$
$$
\top
$$
\[
\tag{Axiom 27'}
id_\square^V(h_1) \square^H id_\square^V(h_2) \equiv id^V(h_1 \circ^H h_2):S(h_1 \circ^H h_2, h_1 \circ^H h_2,id^V(a_1), id^V(a_3)).
\]
Finally, if $J$ and $J'$ are both the axiom that introduces $id$, then $J \odot J'$ is
\[
\tag{Axiom 28}
xx:OO \vdashcustom id_\square^H(xx,xx,id(x) \otimes x^O) \equiv id_\square^V(xx,xx,x^O \otimes id(x)):AA
\begin{pmatrix}
	xx & xx & x^O \otimes id(x)\\
	xx & xx & x^O \otimes id(x)\\
	id(x) \otimes x^O & id(x) \otimes x^O & -
\end{pmatrix}.
\]
In the informal syntax we have
\[
\tag{Axiom 28'}
x:O \vdashcustom id_\square^H(id^V(x)) \equiv id_\square^V(id^H(x)):S(id^H(x),id^H(x),id^V(x),id^V(x)).
\]
This concludes the description of $\bbT_{\text{cat}} \otimes \bbT_{\text{cat}}$.

\subsection{The tensor product of (multisorted) Lawvere theories}

\label{subsec: tensor product of lawvere theories}

A Lawvere theory can be presented by a generalized algebraic theory having a single sort symbol, possibly term and term equality axioms, and no sort equality axioms. To see this, note first that, by restricting Cartmell's equivalence $\GAT \simeq \Cont$, such theories correspond up to isomorphism to the contextual categories that have a single length $1$ object. On the other hand, as verified in \cite{FioVoe20}, the full subcategory of $\Cont$ spanned by the latter is isomorphic to the category of Lawvere theories.

In particular, for a \textsc{gat} of that form, say whose sort symbol is $O$, all contexts are of the form $x_1:O, ..., x_n:O$ and all term axioms are of the form $x_1:O, ..., x_n:O \vdashcustom t(x_1, ..., x_n):O$.

\vspace{0.5em}

Suppose that $\bbA$ and $\bbB$ are (presentations of) Lawvere theories in the above sense. Writing $O$ for the sort symbol in both cases, $\bbA \otimes \bbB$ is as follows:
\begin{itemize}
	\item It has a single sort axiom $OO$, which we will also denote by $O$.
	
	\item Its term symbols are $sO := (s,O)$ where $s \in \Sigma(\bbA)^{\text{term}}$ and $Ot := (O,t)$ where $t \in \Sigma(\bbB)^{\text{term}}$. By abuse of notation, we will also write $s$ for $sO$ and $t$ for $Ot$.

	The axiom of $sO$ in $\bbA \otimes \bbB$ is, after performing the above symbol identifications, the same as the axiom of $s$ in $\bbA$. Indeed, tensoring $x_1:O, ..., x_n:O \vdashcustom s(x_1, ..., x_n):O$ with the axiom $\vdashcustom O \tp$ from $\bbB$ gives
	$$
	x_1y:OO, ..., x_ny:OO \vdashcustom sO(x_1y, ..., x_my):OO
	$$
	(where $y$ is the variable associated with the empty context of $\bbB$), which we rewrite as
	$$
	x_1y:O, ..., x_ny:O \vdashcustom s(x_1y, ..., x_my):O.
	$$
	Similarly, the axiom of $Ot$ in $\bbA \otimes \bbB$ is the same, up to symbol renaming, as the axiom of $t$ in $\bbB$.
	
	\item For each term equality axiom $J$ in $\bbA$, we have a term equality axiom $J \odot (\vdashcustom O \tp)$ in $\bbA \otimes \bbB$. Following the above conventions, the latter judgment can be viewed as $J$ itself. Similarly, every term equality axiom in $\bbB$ can be viewed as an axiom in $\bbA \otimes \bbB$.
	
	\item We also have a term equality axiom for each $s \in \Sigma(\bbA)^{\text{term}}$ and $t \in \Sigma(\bbB)^{\text{term}}$ obtained by tensoring the respective axioms, say $x_1, ..., x_m:O \vdashcustom s(x_1, ..., x_m):O$ and $y_1, ..., y_n:O \vdashcustom t(y_1, ..., y_n):O$:
	$$
	\begin{pmatrix}
		x_1y_1:OO & \cdots & x_1y_n: OO\\
		\vdots & \ddots & \vdots \\
		x_my_1:OO & \cdots & x_my_n:OO
	\end{pmatrix}
	$$
	$$
	\top
	$$
	\[
	\tag{\texttt{*}}
	Ot
	\begin{pmatrix}
		s(x_1, ..., x_m)^O \otimes y_1^O & \cdots & s(x_1, .., x_m)^O \otimes y_n^O
	\end{pmatrix}
	\equiv
	sO
	\begin{pmatrix}
		x_1^O \otimes t(y_1, ..., y_n)^O\\
		\vdots\\
		x_m^O \otimes t(y_1, ..., y_n)^O
	\end{pmatrix}
	:OO.
	\]
	But note that for $1 \le i \le m$ and $1 \le j \le n$ we have
	$$
	x_i^O \otimes t(y_1, ..., y_n)^O = Ot(x_iy_1, ..., x_iy_n),
	$$
	$$
	s(x_1, ..., x_m)^O \otimes y_j^O = sO(x_1y_j, ..., x_my_j),
	$$
	so, after writing $O$ for $OO$, etc., we conclude that (\texttt{*}) equals
	$$
	\begin{pmatrix}
		x_1y_1:O & \cdots & x_1y_n: O\\
		\vdots & \ddots & \vdots \\
		x_my_1:O & \cdots & x_my_n:O
	\end{pmatrix}
	\vdashcustom
	t
	\begin{pmatrix}
		s\begin{pmatrix}
			x_1y_1 \\
			\vdots \\
			x_my_1
		\end{pmatrix} & \cdots & s\begin{pmatrix}
		x_1y_n \\
		\vdots \\
		x_my_n
		\end{pmatrix}
	\end{pmatrix}
	\equiv
	s
	\begin{pmatrix}
		t(x_1y_1 \;\cdots\; x_1y_n)\\
		\vdots\\
		t(x_my_1 \;\cdots\; x_my_n)
	\end{pmatrix}
	:O.
	$$
	In fact, since every term is of sort $O$, it causes no ambiguity to simply write
	$$
	t(s(x_1y_1, ..., x_my_1), ..., s(x_1y_n, ..., x_my_n)) \equiv s(t(x_1y_1, ..., x_1y_n), ..., t(x_my_1, ..., x_my_n)).
	$$
\end{itemize}

Hence $\bbA \otimes \bbB$ coincides with the tensor product between $\bbA$ and $\bbB$ as in \cite{Fre66}. The above axioms can be viewed in the following way: a model of $\bbA \otimes \bbB$ is a set $X$ (given by the sort axiom) endowed with an $\bbA$-model structure (given by the terms and axioms from $\bbA$) and a $\bbB$-model structure (given by the terms and axioms of $\bbB$) such that the action on $X$ of every $n$-ary operation from $\bbB$ is an $\bbA$-model morphism $X^n \rightarrow X$. The latter condition can be replaced by an analogous one: the action on $X$ of every $n$-ary operation from $\bbA$ is a $\bbB$-model morphism $X^n \rightarrow X$.

\vspace{0.5em}

We can also consider the tensor product of two \emph{multisorted Lawvere theories}: generalized algebraic theories in which every sort axiom is of the form $\vdashcustom S \tp$, and which have no sort equality axioms. If $\bbA$ and $\bbB$ are two such theories, it can be checked $\bbA \otimes \bbB$ is given by the following data up to variable renaming:
\begin{itemize}
	\item For sort symbols $S$, $T$ in $\bbA$, $\bbB$, respectively, we have a sort symbol $ST$ introduced by
	$$
	\vdashcustom ST \tp.
	$$
	
	\item Suppose that $T$ is a sort symbol in $\bbB$.
	
	For a term symbol $s$ in $\bbA$, say introduced by $x_1:S_1, ..., x_m:S_m \vdashcustom s(x_1, ..., x_m):S$, we have a term symbol $sT$ introduced by
	$$
	x_1y:S_1T, ..., x_my:S_mT \vdashcustom sT(x_1y, ..., x_my):ST.
	$$
	Similarly, a term equality axiom $J$ in $\bbA$ yields an axiom $J \otimes (\vdashcustom T \tp)$ in $\bbA \otimes \bbB$. The latter judgment is obtained from $J$ by replacing each sort symbol $S$ by $ST$, each term symbol $s$ by $sT$, and each variable $x$ by $xy$.
	
	\item For a sort symbol $S$ in $\bbA$, we have an axiom $(\vdashcustom S \tp) \otimes J$ in $\bbA \otimes \bbB$ whenever $J$ is a term or term equality axiom in $\bbB$. An explicit description can be given as in the previous item.
	
	\item Suppose given term axioms
	\begin{align*}
		x_1:S_1, ..., x_m:S_m & \vdashcustom s(x_1, ..., x_m):S\\
		y_1:T_1, ..., y_n:T_n & \vdashcustom t(y_1, ..., y_n):T
	\end{align*}
	in $\bbA$, $\bbB$, respectively. Then we have the following axiom in $\bbA \otimes \bbB$:
	$$
	(x_iy_j:S_iT_j)_{\substack{1 \le i \le m \\ 1 \le j \le n}}
	\vdashcustom
	St
	\begin{pmatrix}
		sT_1\begin{pmatrix}
			x_1y_1 \\
			\vdots \\
			x_my_1
		\end{pmatrix} & \cdots & sT_n\begin{pmatrix}
			x_1y_n \\
			\vdots \\
			x_my_n
		\end{pmatrix}
	\end{pmatrix}
	\equiv
	sT
	\begin{pmatrix}
		S_1t(x_1y_1 \;\cdots\; x_1y_n)\\
		\vdots\\
		S_mt(x_my_1 \;\cdots\; x_my_n)
	\end{pmatrix}
	:ST.
	$$
\end{itemize}

\begin{remark}
\label{rem: categories as multisorted lawvere theories}
Note that multisorted Lawvere theories where every term symbol is unary, i.e. is introduced by an axiom of the form $x:X \vdashcustom f(x):Y$, can be used as presentations of categories. In fact, denoting the full subcategory of $\GAT$ spanned by such theories by $\GAT_{\text{cat}}$, we have an equivalence of categories $F:\GAT_{\text{cat}} \rightarrow \Cat$ that sends $\bbB$ to the full subcategory of $\mathcal C(\bbB)$ spanned by the length-$1$ objects. Moreover, $\GAT_{\text{cat}}$ is closed in $\GAT$ under $\otimes$ and, for $\bbA$, $\bbB \in \GAT_{\text{cat}}$, we have a canonical isomorphism $F(\bbA \otimes \bbB) \cong F(\bbA) \times F(\bbB)$.
\end{remark}

\subsection{Locally finite direct categories}

\label{subsec: loc fin dir cat}

Generalized algebraic theories that only have sort introduction axioms admit a convenient combinatorial model, namely, \emph{locally finite direct categories} (which we will abbreviate as lfdc). We recall (see \cite{Sub21}, Chapter 1) that a category $D$ is said to be
\begin{itemize}
	\item \emph{Locally finite} if for every object $a$ the slice category $D/a$ is finite, i.e. it has finitely many morphisms.
	
	\item \emph{Direct} if it has no infinite sequence of non-identity morphisms of the form
	$$
	\cdots \longrightarrow a_{n+1} \longrightarrow a_n \longrightarrow \cdots \longrightarrow a_1 \longrightarrow a_0.
	$$
\end{itemize}
As explained in \cite{Sub21}, 1.6.1, if $\bbT$ is a theory that only has sort axioms -- there called a \emph{type signature} --, the contextual category $\mathcal C(\bbT)$ is equivalent (as a category) to $\PSh_{\text{fp}}(D)^{op}$ where $\PSh_{\text{fp}}(D)$ is the category of finitely presentable presheaves on a lfdc $D$. Precisely, $D$ can be taken as the full subcategory of $\mathcal C(\bbT)^{op}$, say $D(\bbT)$, whose objects are the equivalence classes of contexts $\textbf{X},x:U$ where $\textbf{X} \vdashcustom U \tp$ is an axiom. On the other hand, any category of the form $\PSh_{\text{fp}}(D)^{op}$, where $D$ is a ldfc, is equivalent to $\mathcal C(\bbT)$ for some type signature $\bbT$ whose sort symbols correspond bijectively to objects of $D$.\footnote{Note, however, that we have not stated an equivalence between a category of lfdcs and one of type signatures.}

To better understand the lfdc $D(\bbT)$, we will consider the following characterization of lfdcs, which can be obtained from \cite{Sub21}, Lemma 1.1.17 by well-ordering the set of dimension $n$ objects for each $n \ge 0$ (we have omitted the details):

\begin{lemma}
\label{lem: characterization lfdc}
	A small category $D$ is locally finite and direct if and only if there exist a well-ordered set $(L,\le)$ with a maximum element $\top$ and a functor $D_*:(L,\le) \rightarrow \Cat$ such that
	\begin{enumerate}[label=(\roman*)]
		\item $D \cong D_\top$.
		
		\item For every limit element $n \in L$, the canonical functor $\colim_{m < n}D_m \rightarrow D_n$ is an isomorphism. (In particular, $D_\bot = \varnothing$ where $\bot$ is the smallest element of $L$.)
		
		\item If $n \in L$ has a successor $n^+$ (hence if $n \neq \top$), there exists a presheaf $F:D_n^{op} \rightarrow \Set$ such that $\bigsqcup_{a \in D_n}F(a)$ is finite, and the functor $\iota:D_n \rightarrow D_{n^+}$ from the diagram fits into a pushout square
		\[
		\squa{\int F}{D_n}{(\int F)^+}{D_{n^+}}{\pi}{\pi'}{}{\iota}
		\]
		where $\pi:\int F \rightarrow D_n$ is the canonical projection from the category of elements of $F$, and $(\int F)^+$ is obtained from $\int F$ by freely adjoining a terminal object.
	\end{enumerate}
\end{lemma}

The presheaf $F$ can be recovered from $\iota$ up to isomorphism: writing $1$ for the terminal object of $(\int F)^+$, so that $\pi'(1) \in D_{n^+}$ is the only object not in the image of $\iota$, we have a natural isomorphism from $F$ to $D_{n^+}(\iota -, \pi'(1)):D_n^{op} \rightarrow \Set$ that maps $x \in F(a)$ to the image of of the unique arrow $(a,x) \rightarrow 1$ under $\pi'$. Conversely, this remark allows us to characterize when $D_n \rightarrow D_{n^+}$ satisfies condition (iii): $\iota$ is full, faithful and injective on objects, $D_{n^+}$ has a single object not in the image of $\iota$, and, denoting that object by $b$, its only endomorphism is the identity and there exists no morphism $b \rightarrow a$ for $a$ in the image of $\iota$.

For an object $a$ of a lfdc $D$, the category $D/a$ being finite implies that the presheaf $Y(a) = D(-,a):D^{op} \rightarrow \Set$ has finitely many elements. The \emph{boundary} of $a$, denoted by $\partial a$, is defined as the sub-presheaf of $Y(a)$ given by $(\partial a)(a') = D(a',a)$ if $a' \neq a$, and $(\partial a)(a) = \varnothing$ (in other words, it is obtained from $Y(a)$ by removing the element $id_a \in Y(a)(a)$).

In item (iii), the object $b$ adjoined via $\iota:D_n \rightarrow D_{n^+}$ is such that $\partial b = F$ (more precisely, we refer to the restriction of $\partial b$ along $\iota$). Hence Lemma \ref{lem: characterization lfdc} states, informally, that $D$ can be obtained by successively adding objects with a specified finite presheaf as its boundary.

\vspace{0.5em}

Now, given a type signature $\bbT$, we can choose a well-ordered set $(L,\le)$ with a maximum element $\top$, and a family of theories $(\bbT_n)_{n \in L}$ such that\footnote{We assume, without loss of generality, that there exists an enumeration $(x_i)_{i \ge 0}$ of the set of variables of $\bbT$ such that the contexts of all axioms are of the form $x_1:X_1, x_2:X_2, ..., x_k:X_k$.}
\begin{itemize}
	\item $\bbT = \bbT_\top$.
	
	\item For a limit element $n \in L$ we have $\bbT_n = \bigcup_{m<n}\bbT_m$.
	
	\item If $n \in L$ has a successor $n^+$, then $\bbT_{n^+}$ is obtained from $\bbT_n$ by adjoining a sort symbol, say $S_n$, introduced by
	$$
	x_1:X^n_1, ..., x_{k(n)}:X^n_{k(n)} \vdashcustom S_n(x_1, ..., x_{k(n)}) \tp.
	$$
\end{itemize}

The induced diagram of categories $(D(\bbT_n))_{n \in L}$, where $D(\bbT_n)$ is as described previously, satisfies condition (ii). This follows from the fact that the underlying-category functor $\Cont \rightarrow \Cat$ preserves filtered colimits -- in particular, colimits indexed by linearly ordered sets.

To check that (iii) is also satisfied, note that $D(\bbT_{n^+})$ is the full subcategory of $\mathcal C(\bbT_{n^+})^{op}$ obtained from the image of the functor $D(\bbT_n) \rightarrow \mathcal C(\bbT_n) \rightarrow \mathcal C(\bbT_{n^+})$ (which is full and faithful) by adjoining a single object: the equivalence class $[\textbf{X}^n]$ of the context
$$
\textbf{X}^n = (x_1:X^n_1, ..., x_{k(n)}:X^n_{k(n)}, x_{k(n)+1}:S_n(x_1, ..., x_{k(n)})).
$$
But a morphism added via $D(\bbT_n) \rightarrow D(\bbT_{n^+})$ is either $id_{[\textbf{X}^n]}$ or opposite to one of the form $[\textbf{X}_n] \rightarrow [\textbf{X}_m]$ with $m < n$. Indeed, by working syntactically it can be verified that the only endomorphism of $\textbf{X}_n$ is the identity, and that there exists no context morphism from $\textbf{X}_m$ to $\textbf{X}_n$ for $m < n$. Moreover, the set of context morphisms
$$
\{\textbf{f}:\textbf{X}_n \rightarrow \textbf{X}_m \mid m < n\}
$$
is finite since every term symbol that occurs in $\textbf{X}_m$ must also occur in $\textbf{X}_n$.

This shows that $D(\bbT)$ is locally finite and direct. In \cite{Sub21}, it is shown that $D(\bbT)$ determines $\mathcal C(\bbT)$ in the following sense: letting $\iota:D(\bbT) \rightarrow \mathcal C(\bbT)^{op}$ be inclusion, the functor
$$
N_\bbT:\mathcal C(\bbT)^{op} \longrightarrow \PSh_{\text{fp}}(D(\bbT))
$$
given by $N_\bbT([\textbf{X}]) = \mathcal C(\bbT)([\textbf{X}],\iota -)$ (which is the restricted Yoneda embedding associated with $\iota$) is an equivalence of categories. Moreover, it is natural with respect to extensions of type signatures: if $\varepsilon:\bbT \rightarrow \bbT'$ adds a set of sort axioms to $\bbT$, the following diagram commutes up to isomorphism:
\[
\squa{\mathcal C(\bbT)^{op}}{\PSh_{\text{fp}}(D(\bbT))}{\mathcal C(\bbT')^{op}}{\PSh_{\text{fp}}(D(\bbT')),}{N_\bbT}{N_{\bbT'}}{\mathcal C(\varepsilon)^{op}}{D(\varepsilon)_!}
\]
where $D(\varepsilon)_!$ is given  by left Kan extension along $D(\varepsilon)^{op}:D(\bbT)^{op} \rightarrow D(\bbT')^{op}$ (more explicitly, $D(\varepsilon)^{op}(F)$ maps $[\textbf{X}]$ to $F([\textbf{X}])$ if $[\textbf{X}]$ is in the image of $\varepsilon$, and to $\varnothing$ otherwise).

\vspace{0.5em}

Let $\bbS$ and $\bbT$ be type signatures. Then, by definition, $\bbS \otimes \bbT$ is a type signature whose sort symbols are pairs $ST := (S,T)$ where $S$ and $T$ are sort symbols in $\bbS$ and $\bbT$, respectively. This implies that $\Ob(D(\bbS \otimes \bbT)) \cong \Ob(D(\bbS)) \times \Ob(D(\bbT))$. We will now verify that, in fact, $D(\bbS \times \bbT) \cong D(\bbS) \times D(\bbT)$.

This relies on a structure that will be introduced later: in Construction \ref{const: universal functor}, we obtain a functor
$$
\otimes_{\bbS,\bbT}:\mathcal C(\bbS) \times \mathcal C(\bbT) \rightarrow \mathcal C(\bbS \otimes \bbT)
$$
given on objects by $[\textbf{X}] \otimes_{\bbS,\bbT} [\textbf{Y}] = [\textbf{X} \otimes \textbf{Y}]$, and on arrows by $[\textbf{f}] \otimes_{\bbS,\bbT} [\textbf{g}] = [\textbf{f} \otimes \textbf{g}]$ where, if $\textbf{f} = (f_1, ..., f_m)$ and $\textbf{g} = (g_1, ..., g_n)$ are context morphisms, then $\textbf{f} \otimes \textbf{g}$ is obtained by lexicographically reindexing the family of term expressions $(f_i^{\Type(f_i)} \otimes g_j^{\Type(g_j)})_{i \le m,\;j\le n}$.

Given axioms $\textbf{X} \vdashcustom U \tp$ and $\textbf{Y} \vdashcustom V \tp$ in $\bbS$ and $\bbT$, respectively, and choosing contexts $\textbf{X}' = (\textbf{X}, x:S(x_1, ..., x_m))$ and $\textbf{Y}' = (\textbf{Y},y:T(y_1, ..., y_n))$, we have that $[\textbf{X}'] \otimes_{\bbS,\bbT} [\textbf{Y}']$ is the equivalence class of the context $(\partial(\textbf{X}' \otimes \textbf{Y}'), z:ST(\cdots))$ obtained from the axiom that introduces $ST$. This implies that $\otimes_{\bbS,\bbT}$ restricts to a bijective-on-objects functor
$$
\otimes_{\bbS,\bbT}^D:D(\bbS) \times D(\bbT) \longrightarrow D(\bbS \otimes \bbT).
$$
Our goal is to prove that it is an isomorphism. As a consequence, using the equivalence $\mathcal C(\bbA)^{op} \simeq \PSh_{\text{fp}}(D(\bbA))$, we will then have $\mathcal C(\bbS \otimes \bbT)^{op} \simeq \PSh_{\text{fp}}(D(\bbS) \times D(\bbT))$.

\vspace{0.5em}

Firstly, let us consider extensions of type signatures $\bbS \subset \bbS'$ and $\bbT \subset \bbT'$, each of which adds a single sort axiom. Denoting $D(\bbS)$, $D(\bbS')$, $D(\bbT)$, $D(\bbT')$ by $D$, $D'$, $E$, $E'$, respectively, consider the diagram
\[
\tag{\texttt{*}}
\begin{tikzcd}
	{\mathcal C(\bbS \otimes \bbT)^{op}} &&& {\PSh_{\text{fp}}(D \times E)} \\
	& {\mathcal C(\bbS' \otimes \bbT)^{op}} &&& {\PSh_{\text{fp}}(D' \times E)} \\
	\\
	{\mathcal C(\bbS \otimes \bbT')^{op}} &&& {\PSh_{\text{fp}}(D \times E')}
	\arrow["{N_{\bbS,\bbT}}"{description}, from=1-1, to=1-4]
	\arrow["{\mathcal C(\iota')^{op}}"{description}, from=1-1, to=2-2]
	\arrow["{\mathcal C(\iota)^{op}}"{description}, from=1-1, to=4-1]
	\arrow["{I'_!}"{description}, from=1-4, to=2-5]
	\arrow["{I_!}"{description, pos=0.6}, from=1-4, to=4-4]
	\arrow["{N_{\bbS',\bbT}}"{description, pos=0.3}, from=2-2, to=2-5]
	\arrow["{N_{\bbS,\bbT'}}"{description}, from=4-1, to=4-4]
\end{tikzcd}\]
where
\begin{itemize}
	\item $\iota:\bbS \otimes \bbT \rightarrow \bbS \otimes \bbT'$ and $\iota':\bbS \otimes \bbT \rightarrow \bbS' \otimes \bbT$ are the trivial interpretations.
	
	\item $N_{\bbS,\bbT}$ maps a context class $[\textbf{A}]$ to the presheaf $\mathcal C(\bbS \otimes \bbT)(- \otimes_{\bbS,\bbT}^D -,[\textbf{A}]):(D \times E)^{op} \rightarrow \Set$. In other words, $N_{\bbS,\bbT}$ is the restricted Yoneda embedding associated with the composite
	$$
	D \times E \overset{\otimes_{\bbS,\bbT}^D}{\longrightarrow} D(\bbS \otimes \bbT) \hookrightarrow \mathcal C(\bbS \otimes \bbT)^{op}.
	$$
	Similarly for $N_{\bbS,\bbT'}$ and $N_{\bbS',\bbT}$.
	
	\item $I:D \times E \rightarrow D \times E'$ and $I':D \times E \rightarrow D' \times E$ are the inclusion functors, and the subscript $!$ indicates the left Kan extension functor restricted to finitely presentable presheaves.
\end{itemize}

Now, assume that $\otimes^D_{\bbS,\bbT}$, $\otimes^D_{\bbS,\bbT'}$ and $\otimes^D_{\bbS',\bbT}$ are isomorphisms. By pseudo-naturality of $N_\bbA:\mathcal C(\bbA)^{op} \rightarrow \PSh_{\text{fp}}(D(\bbA))$ with respect to extensions of type signatures, the above diagram commutes up to isomorphism.

Write $\partial(\bbS' \otimes \bbT')$ for the theory $(\bbS' \otimes \bbT) \cup (\bbS \otimes \bbT')$. Since the extensions in
\[
\squa{\mathcal C(\bbS \otimes \bbT)}{\mathcal C(\bbS' \otimes \bbT)}{\mathcal C(\bbS \otimes \bbT')}{\mathcal C(\partial(\bbS' \otimes \bbT'))}{}{}{}{}
\]
are full and faithful, and the canonical functor
$$
(D' \times E) \times_{D \times E} (D \times E') \cong D(\bbS' \otimes \bbT) \bigsqcup_{D(\bbS \otimes \bbT)} D(\bbS \otimes \bbT') \longrightarrow D(\partial(\bbS' \otimes \bbT'))
$$
is bijective on objects, the latter is an isomorphism. This implies that the diagram
\[\begin{tikzcd}
	{\mathcal C(\bbS \otimes \bbT)^{op}} &&& {\PSh_{\text{fp}}(D \times E)} \\
	& {\mathcal C(\bbS' \otimes \bbT)^{op}} &&& {\PSh_{\text{fp}}(D' \times E)} \\
	\\
	{\mathcal C(\bbS \otimes \bbT')^{op}} &&& {\PSh_{\text{fp}}(D \times E')} \\
	& {\mathcal C(\partial(\bbS' \otimes \bbT'))^{op}} &&& {\PSh_{\text{fp}}((D' \times E) \sqcup_{D \times E} (D \times E'))}
	\arrow["{N_{\bbS,\bbT}}"{description}, from=1-1, to=1-4]
	\arrow["{\mathcal C(\iota')^{op}}"{description}, from=1-1, to=2-2]
	\arrow["{\mathcal C(\iota)^{op}}"{description}, from=1-1, to=4-1]
	\arrow["{I'_!}"{description}, from=1-4, to=2-5]
	\arrow["{I_!}"{description, pos=0.6}, from=1-4, to=4-4]
	\arrow["{N_{\bbS',\bbT}}"{description, pos=0.3}, from=2-2, to=2-5]
	\arrow[from=2-2, to=5-2]
	\arrow[from=2-5, to=5-5]
	\arrow["{N_{\bbS,\bbT'}}"{description, pos=0.7}, from=4-1, to=4-4]
	\arrow[from=4-1, to=5-2]
	\arrow[from=4-4, to=5-5]
	\arrow["{N_{\partial(\bbS' \otimes \bbT')}}"{description}, from=5-2, to=5-5]
\end{tikzcd}\]
extending ($\texttt *$) in the evident way commutes up to isomorphism, and $N_{\partial(\bbS' \otimes \bbT')}$ is an equivalence of categories.

Next, let $\textbf{X} = (\partial \textbf{X}, x:U)$ and $\textbf{Y} = (\partial \textbf{Y}, y:V)$ be contexts in $\bbS'$ and $\bbT'$, respectively, such that $\partial\textbf{X} \vdashcustom U \tp$ and $\partial\textbf{Y} \vdashcustom V \tp$ are the axioms added via $\bbS \subset \bbS'$ and $\bbT \subset \bbT'$. Then $N_{\partial(\bbS' \otimes \bbT')}$ induces isomorphisms
\begin{align*}
	N_{\partial(\bbS' \otimes \bbT')}([\textbf{X} \otimes \partial\textbf{Y}] \times_{[\partial\textbf{X} \otimes \partial\textbf{Y}]} [\partial\textbf{X} \otimes \textbf{Y}]) & \cong N_{\partial(\bbS' \otimes \bbT')}[\textbf{X} \otimes \partial\textbf{Y}] \bigsqcup_{N_{\partial(\bbS' \otimes \bbT')}[\partial\textbf{X} \otimes \partial\textbf{Y}]} N_{\partial(\bbS' \otimes \bbT')}[\partial\textbf{X} \otimes \textbf{Y}]\\
	& \cong Y([\textbf{X}], [\partial \textbf{Y}]) \bigsqcup_{Y([\partial \textbf{X}], [\partial \textbf{Y}])} Y([\partial \textbf{X}], [\textbf{Y}])\\
	& \cong \partial([\textbf{X}],[\textbf{Y}])
\end{align*}
between presheaves on $(D' \times E) \sqcup_{D \times E} (D \times E')$. Here, $\partial([\textbf{X}],[\textbf{Y}])$ denotes the boundary presheaf of $([\textbf{X}],[\textbf{Y}]) \in D' \times E'$ viewed, by restriction, as a presheaf on $(D' \times E) \sqcup_{D \times E} (D \times E')$. As the extension $\partial(\bbS' \otimes \bbT') \subset \bbS' \otimes \bbT'$ is given by adjoining a sort axiom in context $[\textbf{X} \otimes \partial\textbf{Y}] \times_{[\partial\textbf{X} \otimes \partial\textbf{Y}]} [\partial \textbf{X} \otimes \textbf{Y}]$ and $D' \times E'$ fits into a pushout square
\[
\squa{\int \partial([\textbf{X}],[\textbf{Y}])}{(D' \times E) \sqcup_{D \times E} (D \times E')}{\int (\partial([\textbf{X}],[\textbf{Y}]))^+}{D' \times E',}{}{}{}{}
\]
we obtain an isomorphism $D(\bbS' \otimes \bbT') \simeq D' \times E'$.

\vspace{0.5em}

This result can be used as the induction step in a proof that $D(\bbS) \times D(\bbT) \rightarrow D(\bbS \otimes \bbT)$ is an isomorphism for all type signatures $\bbS$, $\bbT$. In more detail, choose a family of theories $(\bbT_n)_{n \in L}$ presenting $\bbT$ in the way considered previously (with $(L,\le)$ a well-ordered set, etc.), and, analogously, choose a family $(\bbS_m)_{m \in K}$ presenting $\bbS$.

To conclude that $D(\bbS_m) \times D(\bbT_n) \rightarrow D(\bbS_m \otimes \bbT_n)$ is an isomorphism -- call this statement $P(m,n)$ -- for all $m \in K$, $n \in L$, as the poset $K \times L$ is well-founded it suffices to verify the following: for all $(m,n) \in K \times L$, if $P(m',n')$ holds for all $(m',n')<(m,n)$, then $P(m,n)$ holds. We consider two cases:

\begin{itemize}
	\item $m$ and $n$ are successor elements, say of $m'$ and $n'$, respectively. Since, by assumption, $P(m',n')$, $P(m,n')$ and $P(m',n)$ hold, it follows from the above discussion that so does $P(m,n)$.
	
	\item $m$ or $n$ is a limit element. As $\bbS_m \otimes \bbT_n = \bigcup_{(m',n') < (m,n)}\bbS_{m'} \otimes \bbT_{n'}$, the desired functor $D(\bbS_m) \times D(\bbT_n) \rightarrow D(\bbS_m \otimes \bbS_n)$ is an isomorphism, being the colimit
	$$
	\colim_{(m',n') < (m,n)}D(\bbS_{m'}) \times D(\bbT_{n'}) \longrightarrow \colim_{(m',n') < (m,n)}D(\bbS_{m'} \otimes \bbT_{n'})
	$$
	of the isomorphisms given by the induction hypothesis. (Details have been omitted.)
\end{itemize}

\subsection{The theory of morphisms vs the theory of displayed structures}

Let $\bbI$ be the \textsc{gat} given by the axioms
\begin{align*}
	& \vdashcustom S \tag{``source" sort}\\
	& \vdashcustom T \tp \tag{``target" sort}\\
	x:S & \vdashcustom \varphi(x):T \tag{``arrow" term} 
\end{align*}
A $\Fam$-model of $\bbI$ is a triple consisting of small sets $X$, $Y$, and a function $f:X \rightarrow Y$. On the other hand, let $\bbE$ be the \textsc{gat} given by a single dependent sort over a ``base" constant sort:
\begin{align*}
	& \vdashcustom B\\
	x:B & \vdashcustom E(x)
\end{align*}
A $\Fam$-model of $\bbE$ consists of a small set $X$ and, for each $x \in X$, a small set $Y(x)$.

Given a theory $\bbA$, we define its \emph{theory of morphisms} as $\bbI \otimes \bbA$, and its \emph{theory of displayed structures} as $\bbE \otimes \bbA$. They will not be studied in the present paper, in particular since the machinery from \cite{Alm26} will be more appropriate for discussing their semantics. For the moment, let us calculate them in the particular case where $\bbA$ is the theory $\bbT_{\text{cat}}$ from \S\ref{subsec: strict double categories}; we will see that $\bbI \otimes \bbT_\text{cat}$ is the theory of functors, and $\bbE \otimes \bbT_\text{cat}$ is the theory of displayed categories in the sense of \cite{AhrLum19}. More generally, the form of dependency in models of $\bbE \otimes \bbA$ is similar to the one encoded by the displayed algebras from \cite{KapKovAlt19}, but we have not established a formal connection between the two concepts.

\subsubsection{Calculating $\bbI \otimes \bbT_{\text{cat}}$}

For $\bbI \otimes \bbT_{\text{cat}}$, we start by observing that tensoring $\vdashcustom S \tp$ and $\vdashcustom T \tp$ with the axioms of $\bbT_{\text{cat}}$ has the effect of creating two disjoint copies of $\bbT_{\text{cat}}$. We will represent them by adding a subscript $S$ or $T$ to each sort/term symbol: up to variable renaming, we will have axioms

\begin{align*}
	&\vdashcustom O_S \tp \tag{sort of objects -- source}\\
	&\cdots\\
	x,y:O_S, f:A_S(x,y) & \vdashcustom \circ_S(x,y,y,f,id_S(y)) \equiv f: A_S(x,y) \tag{right unitality -- source}\\[2em]
	&\vdashcustom O_T \tp \tag{sort of objects -- target}\\
	&\cdots\\
	x,y:O_T, f:A_T(x,y) & \vdashcustom \circ_T(x,y,y,f,id_T(y)) \equiv f: A_T(x,y) \tag{right unitality -- target}
\end{align*}

We now tensor, if applicable, $x:S \vdashcustom \varphi(x):T$ with each axiom of $\bbT_{\text{cat}}$. Writing $\varphi_O$ for $\varphi O$ and $\varphi_A$ for $\varphi A$, we have
\begin{align*}
	&x:O_S \vdashcustom \varphi_O(x):O_T\\[1em]
	&x, y:O_S, f:A_S(x,y) \vdashcustom \varphi_A(x,y,f):A_T(\varphi_O(x),\varphi_O(x))\\[1em]
	&x,y,z:O_S, f:A_S(x,y), g:A_S(y,z) \vdashcustom\\
	&\circ_T(\varphi_O(x),\varphi_O(y),\varphi_O(z), \varphi_A(x,y,f), \varphi_A(y,z,g)) \equiv \varphi_A(x,z,\circ_S(x,y,z,f,g)) : A_T(\varphi_O(x),\varphi_O(z))\\[1em]
	&x:O_S \vdashcustom id_T(\varphi_O(x)) \equiv \varphi_A(x,x,id_S(x)):A_T(\varphi_O(x),\varphi_O(x))
\end{align*}
These can be expressed in the informal syntax as
\begin{align*}
	x:O_S & \vdashcustom \varphi_O(x):O_T\\
	x,y:O_S, f:A_S(x,y) & \vdashcustom \varphi_A(f):A_T(\varphi_O(x),\varphi_O(y))\\
	x,y,z:O_S, f:A_S(x,y), g:A_S(y,z) & \vdashcustom \varphi_A(f) \circ_T \varphi_A(g) \equiv \varphi_A(f \circ_S g):A_T(\varphi_O(x),\varphi_O(z))\\
	x:O_S & \vdashcustom id_T(\varphi_O(x)) \equiv \varphi_A(id_S(x)):A_T(\varphi_O(x),\varphi_O(x)).
\end{align*}
We thus recognize $\bbI \otimes \bbT_{\text{cat}}$ as a \textsc{gat} whose $\Fam$-models are triples consisting of small categories ($\Fam$-models of $\bbT_{\text{cat}}$) $X$, $Y$, and a functor $f:X \rightarrow Y$.

\begin{remark}
\label{rem: theory of diagrams}
$\bbI$ plays the role of the poset category $\{0 \rightarrow 1\}$. More generally, any category $K$ induces, as in Remark \ref{rem: categories as multisorted lawvere theories}, a \textsc{gat} $\bbK$, from which we obtain for each $\bbA$ a theory $\bbK \otimes \bbA$ encoding diagrams of shape $K$ in $\Mod(\bbA)$.
\end{remark}

\subsubsection{Calculating $\bbE \otimes \bbT_{\text{cat}}$}

Tensoring $\vdashcustom B \tp$ with the axioms of $\bbT_{\text{cat}}$ yields a copy of $\bbT_{\text{cat}}$. We represent it by adding the subscript $B$ to each sort/term symbol:
\begin{align*}
	&\vdashcustom O_B \tp \tag{sort of objects -- base}\\
	&\cdots\\
	x,y:O_B, f:A_B(x,y) & \vdashcustom \circ_B(x,y,y,f,id_B(y)) \equiv f: A_B(x,y) \tag{right unitality -- base}
\end{align*}
Now, by tensoring $x:B \vdashcustom E(x) \tp$ with the axioms of $\bbT_{\text{cat}}$ we obtain, in the informal syntax,
\begin{align*}
	 &x:O_B \vdashcustom O_E(x) \tp\\[1em]
	 &\begin{pmatrix*}[l]
		x:O_B & y:O_B & f:A_B(x,y)\\
		x':O_E(x) & y':O_E(y) & -
	\end{pmatrix*}
	\vdashcustom A_E(f,x',y') \tp\\[1em]
	&\begin{pmatrix*}[l]
		x:O_B & y:O_B & z:O_B & f:A_B(x,y) & g:A_B(y,z)\\
		x':O_E(x) & y':O_E(y) & z':O_E(z) & f':A_E(f,x',y') & g':A_E(g',y',z')
	\end{pmatrix*}
	 \vdashcustom f' \circ_E g': A_E(f \circ_B g, x', z')\\[1em]
	 &\begin{pmatrix*}[l]
	 	x:O_B\\
	 	x':O_E(x)
	 \end{pmatrix*}
	 \vdashcustom id_E(x'): A_E(id_B(x),x',x')
\end{align*}	 
\begin{align*}	 
&\begin{pmatrix*}[l]
	x:O_B & y:O_B & z:O_B & w:O_B & f:A_B(x,y) & g:A_B(y,z) & h:A_B(z,w)\\
	x':O_E(x) & y':O_E(y) & z':O_E(z) & w':O_E(w) & f':A_E(f,x',y') & g':A_E(g,y',z') & h':A_E(h,z',w')
\end{pmatrix*}\\
& \vdashcustom f' \circ_E (g' \circ_E h') \equiv (f' \circ_E g') \circ_E h': A_E(f \circ_B (g \circ_B h), x',w')\\[1em]
&\begin{pmatrix*}[l]
	x:O_B & y:O_B & f:A_B(x,y)\\
	x':O_E(x) & y':O_E(y) & f':A_E(f,x',y')
\end{pmatrix*}
\vdashcustom id_E(x') \circ_E f' \equiv f': A_E(id_B(x) \circ_B f, x',y')\\[1em]
&\begin{pmatrix*}[l]
	x:O_B & y:O_B & f:A_B(x,y)\\
	x':O_E(x) & y':O_E(y) & f':A_E(f,x',y')
\end{pmatrix*}
\vdashcustom f' \circ_E id_E(y') \equiv f': A_E(f \circ_B id_B(y), x', y')
\end{align*}

These correspond precisely to the axioms for a \emph{displayed category} in the sense of Ahrens--Lumsdaine (\cite{AhrLum19}): a model of $\bbE \otimes \bbT_{\text{cat}}$ is a pair consisting of a ``base" small category $X$ (a model of the copy of $\bbT_{\text{cat}}$ corresponding to the sort $B$) and a category $Y$ whose data is split into small sets suitably indexed by configurations of objects and arrows in $X$. Up to isomorphism, this amounts to an arbitrary functor $Y \rightarrow X$. Compare with the correspondence, due to Bénabou (see \cite{Ben00}), between functors $Y \rightarrow X$ and normal lax functors from $X$ to the bicategory of distributors.

However, $\bbI \otimes \bbT_{\text{cat}}$ and $\bbE \otimes \bbT_{\text{cat}}$ induce distinct structures on their categories of models. As mentioned in the introduction, the display maps of a contextual category $\mathcal A$ generate, via the Yoneda embedding and the small object argument, a weak factorization on $\Mod(\mathcal A)$.

In $\Mod(\bbI \otimes \bbT_{\text{cat}})$, the cofibrations, i.e. the arrows in the left class of the wfs, are generated by the commutative squares
\[
\squa{X}{X'}{Y}{Y',}{i}{j}{f}{f'}
\]
viewed as a morphism $(i,j):f \rightarrow f'$, where (1) $i = id_X$ and $j$ freely adds an object to $Y$, or (2) $i = id_X$ and $j$ freely adds an arrow to $Y$, or (3) $i$ freely adds an object to $X$ and the square is a pushout, or (4) $i$ freely adds an arrow to $X$ and the square is a pushout.

On the other hand, in $\Mod(\bbE \otimes \bbT_{\text{cat}})$, the cofibrations are generated by the squares as above where (1) $i = id_X$ and $j$ freely adds an object to $Y$, or (2) $i = id_X$ and $j$ freely adds an arrow to $Y$, or (3) $j = id_Y$ and $i$ freely adds an object that is sent by $f'$ to a specified object of $Y'$, or (4) $j = id_Y$ and $i$ freely adds an arrow that is sent by $f'$ to a specified arrow in $Y'$.

\vspace{0.5em}

We refer the reader to \cite{BarHen25}, particularly \S3.9 (which uses a variant of $\bbT_{\text{cat}}$), for a discussion of homotopy-theoretic implications of this distinction.

\section{Some consequences of $h$-derivability}
\label{sec: consequences h-derivability}

Recall from Definition \ref{def: h-derivable} that a pair of \textsc{gat}s $(\bbA,\bbB)$ is \emph{$h$-derivable}, where $h \ge 0$, if $J \odot J'$ is derivable in the pretheory $\bbA \otimes \bbB$ whenever $J$, $J'$ are derivable judgments in $\bbA$, $\bbB$, resp., such that $\Ht(J)\Ht(J') \le h$.

In this section, we will prove several consequences of $(\bbA,\bbB)$ being $h$-derivable. Roughly, some of them mean that, under certain numerical constraints on the heights of the structures involved, tensoring expressions/contexts is compatible with equality or substitution in each of the two arguments separately. Others allow us to construct contexts or morphisms in $\bbA \otimes \bbB$ by tensoring contexts or morphisms in $\bbA$ with ones in $\bbB$. They are instances of the general results from $\S\ref{sec: comparison functor}$.

We will use Proposition \ref{prop: properties height}, which lists some of the basic properties of the height function for judgments, contexts, etc., several times without explicit reference. We believe doing that will not create confusion as that proposition contains (by design) all of the properties of the height function that will be referred to throughout the text.

\vspace{0.5em}

We start by introducing some notation:

\begin{definition}
Let $\bbA$ and $\bbB$ be \textsc{gat}s. Suppose given context morphisms
\begin{align*}
	\textbf{f} = (f_1, ..., f_M):\textbf{X} = (x_1:X_1, ..., x_m:X_m) & \longrightarrow \textbf{A} = (a_1:A_1, ..., a_M:A_M),\\
	\textbf{g} = (g_1, ..., g_N):\textbf{Y} = (y_1:Y_1, ..., y_n:Y_n) & \longrightarrow \textbf{B} = (b_1:B_1, ..., b_N:B_N)
\end{align*}
in $\bbA$, $\bbB$, respectively. Recall that $\textbf{A} \otimes \textbf{B}$ is the length-$mn$ sequence obtained by lexicographically reindexing the family $(x_iy_j:X_i \otimes Y_j)_{(i,j) \in \{1, ..., M\} \times \{1, ..., N\}}$. Similarly, using the notation from \ref{not: abuses of notation}, we let
\begin{align*}
	\textbf{f} \otimes \textbf{B} & = (f_1, \ldots, f_M) \otimes (b_1, \ldots, b_N)\\
	& = (f_1^{\Type(f_1)} \otimes b_1^{B_1}, \ldots, f_1^{\Type(f_1)} \otimes b_N^{B_N}, \ldots, f_M^{\Type(f_M)} \otimes b_1^{B_1}, \ldots, f_M^{\Type(f_M)} \otimes b_N^{B_N}),\\
	& \\
	\textbf{A} \otimes \textbf{g} & = (a_1, \ldots, a_M) \otimes (g_1, \ldots, g_N) \\
	& = (a_1^{A_1} \otimes g_1^{\Type(g_1)}, \ldots, a_1^{A_1} \otimes g_N^{\Type(g_N)}, \ldots, a_M^{A_M} \otimes g_1^{\Type(g_1)}, \ldots, a_M^{A_M} \otimes g_N^{\Type(g_N)}), \\
	& \\
	\textbf{f} \otimes \textbf{g} & = (f_1, \ldots, f_M) \otimes (g_1, \ldots, g_N)\\
	& = (f_1^{\Type(f_1)} \otimes g_1^{\Type(g_1)}, \ldots, f_1^{\Type(f_1)} \otimes g_N^{\Type(g_N)}, \ldots, f_M^{\Type(f_M)} \otimes g_1^{\Type(g_1)}, \ldots, f_M^{\Type(f_M)} \otimes g_N^{\Type(g_N)}).
\end{align*}
\end{definition}

\subsection{First set of statements}

Consider a pair of \textsc{gat}s $(\bbA,\bbB)$ and a natural number $k$. Below, when we talk without further qualification about a derivable judgment, a context, or a context morphism, we mean so in $\bbA$, $\bbB$, or the pretheory $\bbA \otimes \bbB$.

We will consider the following statements:

\begin{itemize}	
	\item[\colorbox{white!90!black}{$Cont(k)$}] For all contexts $\textbf{X}$ in $\bbA$ and $\textbf{Y}$ in $\bbB$, if $\Ht(\textbf{X})\Ht(\textbf{Y}) = k$, then $\textbf{X} \otimes \textbf{Y}$ is a context.
	
	\item[\colorbox{white!90!black}{$Sub^1_{t,t}(k)$}] Suppose given in $\bbA$ a context morphism $\textbf{f}:\textbf{X} \rightarrow \textbf{A}$ and a derivable judgment $\textbf{A} \vdashcustom u \tm$ in $\bbA$, and in $\bbB$ a derivable judgment $\textbf{Y} \vdashcustom v \tm$.
	
	If ${\Ht(\textbf{f}:\textbf{X} \rightarrow \textbf{A})\Ht(\textbf{Y} \vdashcustom v \tm)}$, ${\Ht(\textbf{A} \vdashcustom u \tm)\Ht(\textbf{Y} \vdashcustom v \tm) \le k}$, then the following is derivable:
	$$
	\textbf{X} \otimes \textbf{Y} \vdashcustom (u \otimes v)[\textbf{f} \otimes \textbf{Y}] \equiv u[\textbf{f}] \otimes v \tm.
	$$
	
	\item[\colorbox{white!90!black}{$Sub^1_{t,s}(k)$}] Suppose given in $\bbA$ a context morphism $\textbf{f}:\textbf{X} \rightarrow \textbf{A}$ and a derivable judgment $\textbf{A} \vdashcustom u \tm$, and in $\bbB$ a derivable judgment $\textbf{Y} \vdashcustom V \tp$. Let $\textbf{Y}' = (\textbf{Y}, y:V)$.

	If $\Ht(\textbf{f}:\textbf{X} \rightarrow \textbf{A})\Ht(\textbf{Y} \vdashcustom V \tp)$, $\Ht(\textbf{A} \vdashcustom u \tm)\Ht(\textbf{Y} \vdashcustom V \tp) \le k$, then the following is derivable:
	$$
	\textbf{X} \otimes \textbf{Y}' \vdashcustom (u \otimes y)[\textbf{f} \otimes \textbf{Y}'] \equiv u[\textbf{f}] \otimes y \tm.
	$$
	
	\item[\colorbox{white!90!black}{$Sub^1_{s,t}(k)$}] Suppose given in $\bbA$ a context morphism $\textbf{f}:\textbf{X} \rightarrow \textbf{A}$ and a derivable judgment $\textbf{A} \vdashcustom U \tp$, and in $\bbB$ a derivable judgment $\textbf{Y} \vdashcustom v \tm$. Let $\textbf{A}' = (\textbf{A}, a:U)$, $\textbf{X}' = (\textbf{X}, x:U[\textbf{f}])$, and $\textbf{f}' = (\textbf{f}, x):\textbf{X}' \rightarrow \textbf{A}'$.
	
	If ${\Ht(\textbf{f}:\textbf{X} \rightarrow \textbf{A})\Ht(\textbf{Y} \vdashcustom v \tm)}$, ${\Ht(\textbf{A} \vdashcustom U \tp)\Ht(\textbf{Y} \vdashcustom v \tm)}$, ${\Ht(\textbf{X} \vdashcustom U[\textbf{f}])\Ht(\textbf{Y}) \le k}$, then the following is derivable:
	$$
	\textbf{X}' \otimes \textbf{Y} \vdashcustom (a \otimes v)[\textbf{f}' \otimes \textbf{Y}] \equiv x \otimes v \tm.
	$$
	
	\item[\colorbox{white!90!black}{$Sub^1_{s,s}(k)$}] Suppose given in $\bbA$ a context morphism $\textbf{f}:\textbf{X} \rightarrow \textbf{A}$ and a derivable judgment $\textbf{A} \vdashcustom U \tp$, and in $\bbB$ a derivable judgment $\textbf{Y} \vdashcustom V \tp$. Let ${\textbf{A}' = (\textbf{A}, a:U)}$, ${\textbf{X}' = (\textbf{X}, x:U[\textbf{f}])}$, ${\textbf{f}' = (\textbf{f}, x):\textbf{X}' \rightarrow \textbf{A}'}$, and ${\textbf{Y}' = (\textbf{Y}, y:V)}$.
	
	If $\Ht(\textbf{f}:\textbf{X} \rightarrow \textbf{A})\Ht(\textbf{Y} \vdashcustom V \tp)$, $\Ht(\textbf{A} \vdashcustom U \tp)\Ht(\textbf{Y} \vdashcustom V \tp)$, $\Ht(\textbf{X} \vdashcustom U[\textbf{f}] \tp)\Ht(\textbf{Y}) \le k$, then the following is derivable:
	$$
	\partial(\textbf{X}' \otimes \textbf{Y}') \vdashcustom (U \otimes V)[\partial(\textbf{f}' \otimes \textbf{Y}')] \equiv U[\textbf{f}] \otimes V \tp.
	$$
	
	\item[\colorbox{white!90!black}{$Sub^2_{t,t}(k)$}] Suppose given in $\bbA$ a derivable judgment $\textbf{X} \vdashcustom u \tm$, and in $\bbB$ a context morphism $\textbf{g}:\textbf{Y} \rightarrow \textbf{B}$ and a derivable judgment $\textbf{B} \vdashcustom v \tm$.
	
	If $\Ht(\textbf{X} \vdashcustom u \tm)\Ht(\textbf{g}:\textbf{Y} \rightarrow \textbf{B})$, $\Ht(\textbf{X} \vdashcustom u \tm)\Ht(\textbf{B} \vdashcustom v \tm) \le k$, then the following is derivable:
	$$
	\textbf{X} \otimes \textbf{Y} \vdashcustom (u \otimes v)[\textbf{X} \otimes \textbf{g}] \equiv u \otimes v[\textbf{g}] \tm.
	$$
	
	\item[\colorbox{white!90!black}{$Sub^2_{t,s}(k)$}] Suppose given in $\bbA$ a derivable judgment $\textbf{X} \vdashcustom u \tm$, and in $\bbB$ a context morphism $\textbf{g}:\textbf{Y} \rightarrow \textbf{B}$ and a derivable judgment $\textbf{B} \vdashcustom V \tp$. Let $\textbf{B}' = (\textbf{B}, b:V)$, $\textbf{Y}' = (\textbf{Y}, y:V[\textbf{g}])$, and $\textbf{g}' = (\textbf{g},y):\textbf{Y}' \rightarrow \textbf{B}'$.
	
	If $\Ht(\textbf{X} \vdashcustom u \tm)\Ht(\textbf{g}:\textbf{Y} \rightarrow \textbf{B})$, $\Ht(\textbf{X} \vdashcustom u \tm)\Ht(\textbf{B} \vdashcustom V \tp)$, $\Ht(\textbf{X})\Ht(\textbf{Y} \vdashcustom V[\textbf{g}]) \le k$, then the following is derivable:
	$$
	\textbf{X} \otimes \textbf{Y}' \vdashcustom (u \otimes b)[\textbf{X} \otimes \textbf{g}'] \equiv u \otimes y \tm.
	$$
	
	\item[\colorbox{white!90!black}{$Sub^2_{s,t}(k)$}] Suppose given in $\bbA$ a derivable judgment $\textbf{X} \vdashcustom U \tp$, and in $\bbB$ a context morphism $\textbf{g}:\textbf{Y} \rightarrow \textbf{B}$ and a derivable judgment $\textbf{B} \vdashcustom v \tm$. Let $\textbf{X}' = (\textbf{X}, x:U)$.
	
	If $\Ht(\textbf{X} \vdashcustom U \tp)\Ht(\textbf{g}:\textbf{Y} \rightarrow \textbf{B})$, $\Ht(\textbf{X} \vdashcustom U \tp)\Ht(\textbf{B} \vdashcustom v \tm) \le k$, then the following is derivable:
	$$
	\textbf{X}' \otimes \textbf{Y} \vdashcustom (x \otimes v)[\textbf{X}' \otimes \textbf{g}] \equiv x \otimes v[\textbf{g}] \tm.
	$$
	
	\item[\colorbox{white!90!black}{$Sub^2_{s,s}(k)$}] Suppose given in $\bbA$ a derivable judgment $\textbf{X} \vdashcustom U \tp$, and in $\bbB$ a context morphism $\textbf{g}:\textbf{Y} \rightarrow \textbf{B}$ and a derivable judgment $\textbf{B} \vdashcustom V \tp$. Let $\textbf{B}' = (\textbf{B}, b:V)$, $\textbf{Y}' = (\textbf{Y}, y:V[\textbf{g}])$, $\textbf{g}' = (\textbf{g},y):\textbf{Y}' \rightarrow \textbf{B}'$, and $\textbf{X}' = (\textbf{X}, x:U)$.
	
	If $\Ht(\textbf{X} \vdashcustom U \tp)\Ht(\textbf{g}:\textbf{Y} \rightarrow \textbf{B})$, $\Ht(\textbf{X} \vdashcustom U \tp)\Ht(\textbf{B} \vdashcustom V \tp)$, $\Ht(\textbf{X})\Ht(\textbf{Y} \vdashcustom V[\textbf{g}]) \le k$, then the following is derivable:
	$$
	\partial(\textbf{X}' \otimes \partial \textbf{Y}') \vdashcustom (U \otimes V)[\partial(\textbf{X}' \otimes \textbf{g}')] \equiv U \otimes V[\textbf{g}] \tp.
	$$
	
	\item[\colorbox{white!90!black}{$Mor^1(k)$}] Suppose given in $\bbA$ a context morphism $\textbf{f}:\textbf{X} \rightarrow \textbf{A}$, and in $\bbB$ a context $\textbf{Y}$.
	
	If $\Ht(\textbf{f}:\textbf{X} \rightarrow \textbf{A})\Ht(\textbf{Y}) \le k$, then the following is derivable:
	$$
	\textbf{f} \otimes \textbf{Y}:\textbf{X} \otimes \textbf{Y} \longrightarrow \textbf{A} \otimes \textbf{Y}.
	$$
	
	\item[\colorbox{white!90!black}{$Mor^1_+(k)$}] Suppose given in $\bbA$ a context morphism $\textbf{f}:\textbf{X} \rightarrow \textbf{A}$ and a derivable judgment $\textbf{A} \vdashcustom U \tp$, and in $\bbB$ a context $\textbf{Y}$. Let $\textbf{A}' = (\textbf{A},a:U)$, $\textbf{X}' = (\textbf{X},x:U[\textbf{f}])$, and $\textbf{f}' = (\textbf{f},x):\textbf{X}' \rightarrow \textbf{A}'$.
	
	If $
	\Ht(\textbf{f}:\textbf{X} \rightarrow \textbf{A})\Ht(\textbf{Y})$, $\Ht(\textbf{A} \vdashcustom U \tp)\Ht(\textbf{Y}) \le k-1$ and $\Ht(\textbf{X} \vdashcustom U[\textbf{f}] \tp)\Ht(\textbf{Y}) \le k$, then the following is derivable:
	$$
	\textbf{f}' \otimes \textbf{Y}:\textbf{X}' \otimes \textbf{Y} \longrightarrow \textbf{A}' \otimes \textbf{Y}.
	$$
	
	\item[\colorbox{white!90!black}{$Mor^2(k)$}] Suppose given in $\bbA$ a context $\textbf{X}$, and in $\bbB$ a context morphism $\textbf{g}:\textbf{Y} \rightarrow \textbf{B}$.
	
	If $\Ht(\textbf{X})\Ht(\textbf{g}:\textbf{Y} \rightarrow \textbf{B}) \le k$, then the following is derivable:
	$$
	\textbf{X} \otimes \textbf{g}:\textbf{X} \otimes \textbf{Y} \longrightarrow \textbf{X} \otimes \textbf{B}.
	$$
	
	\item[\colorbox{white!90!black}{$Mor^2_+(k)$}] Suppose given in $\bbA$ a context $\textbf{X}$, and in $\bbB$ a context morphism $\textbf{g}:\textbf{Y} \rightarrow \textbf{B}$ and a derivable judgment $\textbf{B} \vdashcustom V \tp$. Let $\textbf{B}' = (\textbf{B},b:V)$, $\textbf{Y}' = (\textbf{Y}, y:V[\textbf{g}])$, and $\textbf{g}' = (\textbf{g},y):\textbf{Y}' \rightarrow \textbf{B}'$.
	
	If $\Ht(\textbf{X})\Ht(\textbf{g}:\textbf{Y} \rightarrow \textbf{B})$, $\Ht(\textbf{X})\Ht(\textbf{B} \vdashcustom V \tp) \le k-1$ and $\Ht(\textbf{X})\Ht(\textbf{Y} \vdashcustom V[\textbf{g}]) \le k$, then the following is derivable:
	$$
	\textbf{X} \otimes \textbf{g}':\textbf{X} \otimes \textbf{Y}' \longrightarrow \textbf{X} \otimes \textbf{B}'.
	$$
\end{itemize}

\begin{tcolorbox}[colback=white!90!black, colframe=white, boxrule=0pt, arc=0mm]
	In what follows, we fix theories $\bbA$ and $\bbB$, a natural number $h$, and work under the following assumptions on $(\bbA,\bbB)$:
	\begin{itemize}
		\item it is $h$-derivable;
		
		\item for every $k < h$, it satisfies all of the above statements with parameter $k$, i.e. $Cont(k)$, ..., $Mor_+^2(k)$.
	\end{itemize}
\end{tcolorbox}

Using these, we will prove

\begin{proposition}
$(\bbA, \bbB)$ satisfies $\Cont(h)$, ..., $Mor^2_+(h)$.
\end{proposition}

\vspace{0.5em}
\colorbox{white!90!black}{Proof of $Cont(h)$}
\vspace{0.5em}

Suppose given contexts $\textbf{X}$ in $\bbA$ and $\textbf{Y}$ in $\bbB$ such that $\Ht(\textbf{X})\Ht(\textbf{Y}) = h$. Write $\textbf{X} = (x_1:X_1, \ldots, x_m:X_m)$ and $\textbf{Y} = (y_1:Y_1, \ldots, x_n:Y_n)$. We have $\Ht(\partial \textbf{X} \vdashcustom X_m \tp)\Ht(\partial \textbf{Y} \vdashcustom Y_n \tp) = h$, so by $h$-derivability we can derive $\partial(\textbf{X} \otimes \textbf{Y}) \vdashcustom X_m \otimes Y_n \tp$. It follows that $\textbf{X} \otimes \textbf{Y} = (\partial(\textbf{X} \otimes \textbf{Y}), x_my_n:X_m \otimes Y_n)$ is a context.

\vspace{1em}
\colorbox{white!90!black}{Proof of $Sub^1_{t,t}(h)$}
\vspace{0.5em}

Suppose given a context morphism
$$
\textbf{f} = (f_1, \ldots, f_M):\textbf{X} = (x_1:X_1, \ldots, x_m:X_m) \longrightarrow \textbf{A} = (a_1:A_1, \ldots, a_M:A_M)
$$
in $\bbA$, a context $\textbf{Y} = (y_1:Y_1, \ldots, y_n:Y_n)$ in $\bbB$, and derivable judgments $\textbf{A} \vdashcustom u \tm$, $\textbf{Y} \vdashcustom v \tm$. Assume that $\textbf{H}(\textbf{f}:\textbf{X} \rightarrow \textbf{A})\Ht(\textbf{Y} \vdashcustom v \tm)$ and $\Ht(\textbf{A} \vdashcustom u \tm)\Ht(\textbf{Y} \vdashcustom v \tm)$ are at most $h$.

Note that $\Ht(\textbf{f}:\textbf{X} \rightarrow \textbf{A})\Ht(\textbf{Y}) < h$, so by $Mor^1(h-1)$ we have a morphism $\textbf{f} \otimes \textbf{Y}: \textbf{X} \otimes \textbf{Y} \rightarrow \textbf{A} \otimes \textbf{Y}$.

In what follows, if $e$ is an expression in the alphabet of $\bbA \otimes \bbB$, we will write $\overline{e}$ for $e[\textbf{f} \otimes \textbf{Y}] = e[f_i \otimes y_j \mid a_iy_j]_{i \le M, j \le n}$. Similarly, if $e$ is an expression in the alphabet of $\bbA$, we write $\underline{e}$ for $e[\textbf{f}] = e[f_i \mid a_i]_{i \le M}$. Thus our goal is to derive the equality
$$
\textbf{X} \otimes \textbf{Y} \vdashcustom \overline{u \otimes v} \equiv \underline{u} \otimes v \tm.
$$

We will prove by induction on $P \in \{0, \ldots, h\}$ that this holds whenever $\Ht(\textbf{A} \vdashcustom u \tm)\Ht(\textbf{Y} \vdashcustom v \tm) = P$. Let $0 \le P \le h$ be such that the claim holds for all $P' < P$, and suppose that $\Ht(\textbf{A} \vdashcustom u \tm)\Ht(\textbf{Y} \vdashcustom v \tm) = P$.\footnote{Note that we do not verify a base case separately. In fact, for each $P$ the proof is done unconditionally for certain $u$ and $v$ $-$ namely, when both are variables.}\\

Consider contexts $\textbf{A}' = (\textbf{A},a':\Type(u))$ and $\textbf{X}' = (\textbf{X},x':\Type(\underline{u})) = (\textbf{X}, x':\underline{\Type(u)})$. We have the following cases:
\begin{itemize}
	\item $u$ is a variable $a_I$ and $v$ is a variable $y_J$. Then $\overline{u \otimes v}$ and $\underline{u} \otimes v$ are both $f_I \otimes y_J$.
	
	\item $u$ is a variable $a_I$ and $v = t(t_1, \ldots, t_l)$. Express $A_I$ as $S(\sigma_1, \ldots, \sigma_p)$. Then $\underline{A_I} = S(\underline{\sigma_1}, \ldots, \underline{\sigma_p})$ and
	$$
	\textbf{X} \vdashcustom \Type(f_I) \equiv \underline{A_I} \tp
	$$
	is derivable. Denoting this judgment by $J$, we have
	$$
	\Ht(J) \le \Ht(\textbf{X} \vdashcustom f_I:\underline{A_I}) \le \Ht(\textbf{f}:\textbf{X} \rightarrow \textbf{A}).
	$$
	Then $\Ht(J)\Ht(\textbf{Y} \vdashcustom v:\Type(v)) \le h$, so by tensoring $J$ with $\textbf{Y} \vdashcustom v:\Type(v)$ we derive an equality
	$$
	\textbf{X}' \otimes \textbf{Y} \vdashcustom x^{'\Type(f_I)} \otimes v \equiv x^{' \underline{A_I}} \otimes v \tm.
	$$

	Note that $f_I \otimes v$ is, by definition,
	$$
	(x^{' \Type(f_I)} \otimes v)[f_I \otimes y_1 \mid x'y_1, \ldots, f_I \otimes y_n \mid x'y_n],
	$$
	which is then provably equal in context $\textbf{X} \otimes \textbf{Y}$ to
	\[
	\tag{$\texttt{*}$}
	(x^{' \underline{A_I}} \otimes v)[f_I \otimes y_1 \mid x'y_1, \ldots, f_I \otimes y_n \mid x'y_n].
	\]
	Here we have used substitution along the context morphism $(x_1, \ldots, x_m,f_I) \otimes \textbf{Y}:\textbf{X} \otimes \textbf{Y} \rightarrow \textbf{X}' \otimes \textbf{Y}$.

	Expanding ($\texttt{*}$) we obtain
	\[
	\tag{1}
	St((\underline{\sigma_1}, \ldots, \underline{\sigma_p}, x^{'\underline{A_I}}) \otimes (t_1, \ldots, t_l))[f_I \otimes y_1 \mid x'y_1, \ldots, f_I \otimes y_n \mid x'y_n].
	\]
	On the other hand, $\overline{a_I \otimes v}$ is
	$$
	\overline{St((\sigma_1, \ldots, \sigma_p,a_I) \otimes (t_1, \ldots, t_l))},
	$$
	which, by the induction hypothesis, is provably equal in context $\textbf{X} \otimes \textbf{Y}$ to
	\[
	\tag{2}
	St((\underline{\sigma_1}, \ldots, \underline{\sigma_p},f_I) \otimes (t_1, \ldots, t_l)).
	\]
	To derive an equality between (1) and (2) in context $\textbf{X} \otimes \textbf{Y}$, it suffices $-$ noting that $x'$ does not occur in $\underline{\sigma_1}$, \ldots, $\underline{\sigma_p}$ $-$ to derive for $j = 1$, \ldots, $l$ an equality between $f_I \otimes t_j$ and
	$$
	(x^{'\underline{A_I}} \otimes t_j)[f_I \otimes y_1 \mid x'y_1, \ldots, f_I \otimes y_n \mid x'y_n].
	$$
	Since, by definition, $f_I \otimes t_j$ is $(x^{'\Type(f_I)} \otimes t_j)[f_I \otimes y_1 \mid x'y_1, \ldots, f_I \otimes y_n \mid x'y_n]$, the desired equality follows from
	$$
	\textbf{X}' \otimes \textbf{Y} \vdashcustom x^{'\Type(f_I)} \otimes t_j \equiv x^{'\underline{A_I}} \otimes t_j \tm
	$$
	being derivable.
	
	\item $u = s(s_1, \ldots, s_k)$ and $v$ is a variable $y_J$. Express $Y_J$ as $T(\tau_1, \ldots, \tau_q)$. Then $u \otimes v$ is $sT((s_1, \ldots, s_k) \otimes (\tau_1, \ldots, \tau_q, y_J))$, so $\overline{u \otimes v}$ is provably equal in context $\textbf{X} \otimes \textbf{Y}$, by the induction hypothesis, to
	$$
	sT((\underline{s_1}, \ldots, \underline{s_k}) \otimes (\tau_1, \ldots, \tau_q,y_J)).
	$$
	But the latter is precisely $\underline{u} \otimes v$.
	
	\item $u = s(s_1, \ldots, s_k)$ and $v = t(t_1, \ldots, t_l)$. Then $u \otimes v$ is $(a' \otimes v)[u \otimes y_1 \mid a'y_1, \ldots, u \otimes y_n \mid a'y_n]$, so
	$$
	\overline{u \otimes v} = (a' \otimes v)[u \otimes y_1 \mid a'y_1, \ldots, u \otimes y_n \mid a'y_n][f_1 \otimes y_1 \mid a_1y_1, \ldots, f_M \otimes y_n \mid a_my_n].
	$$
	Since the variables $a'y_1$, \ldots, $a'y_n$ do not occur in $f_1 \otimes y_1$, \ldots, $f_M \otimes y_n$, we conclude that
	\begin{align*}
		\overline{u \otimes v} & = (a' \otimes v)[f_i \otimes y_j \mid a_iy_j]_{i \le M, j \le n}[\overline{u \otimes y_j} \mid a'y_j]_{j \le n}\\
		 & = \overline{a' \otimes v}[\overline{u \otimes y_j} \mid a'y_j]_{j \le n}.\\
	\end{align*}

	We will verify that this expression is provably equal in context $\textbf{X} \otimes \textbf{Y}$ to $\underline{u} \otimes v$.
	
	Expressing $\Type(u)$ as $S(\sigma_1, \ldots, \sigma_p)$, we have that $a' \otimes v$ is $St((\sigma_1, \ldots, \sigma_p,a') \otimes (t_1, \ldots, t_l))$. Hence
	$$
	\overline{a' \otimes v} = St(\overline{\sigma_1 \otimes t_1}, \ldots, \overline{\sigma_p \otimes t_l}, \overline{a' \otimes t_1}, \ldots, \overline{a' \otimes t_l})
	$$
	and, as $a'$ does not occur in $\sigma_1$, \ldots, $\sigma_p$, we obtain
	\begin{align*}
		 \overline{u \otimes v} & = \overline{a' \otimes v}[\overline{u \otimes y_j} \mid a'y_j]_{j \le n}\\
		 & = St(\overline{\sigma_1 \otimes t_1}, \ldots, \overline{\sigma_p \otimes t_l}, \overline{a' \otimes t_1}[\overline{u \otimes y_j} \mid a'y_j]_{j \le n}, \ldots, \overline{a' \otimes t_l}[\overline{u \otimes y_j} \mid a'y_j]_{j \le n}) \\
		 & = St(\overline{\sigma_1 \otimes t_1}, \ldots, \overline{\sigma_p \otimes t_l}, \overline{u \otimes t_1}, \ldots, \overline{u \otimes t_l}).
	\end{align*}
	By the induction hypothesis, the latter is provably equal in context $\textbf{X} \otimes \textbf{Y}$ to
	$$
	St(\underline{\sigma_1} \otimes t_1, \ldots, \underline{\sigma_p} \otimes t_l, \underline{u} \otimes t_1, \ldots, \underline{u} \otimes t_l) = \underline{u} \otimes v.
	$$
	This concludes the proof.
\end{itemize}

\vspace{1em}
\colorbox{white!90!black}{Proof of $Sub^2_{t,t}(h)$}
\vspace{0.5em}

Suppose given a context $\textbf{X} = (x_1:X_1, \ldots, x_m:X_m)$ in $\bbA$, a context morphism
$$
\textbf{g} = (g_1, \ldots, g_N): \textbf{Y} = (y_1:Y_1, \ldots, y_n:Y_n) \rightarrow \textbf{B} = (b_1:B_1, \ldots, b_N:B_N)
$$
in $\bbB$, and derivable judgments $\textbf{X} \vdashcustom u \tm$, $\textbf{B} \vdashcustom v \tm$. Assume that $\Ht(\textbf{X} \vdashcustom u \tm)\Ht(\textbf{g}:\textbf{Y} \rightarrow \textbf{B})$ and $\Ht(\textbf{X} \vdashcustom u \tm)\Ht(\textbf{B} \vdashcustom v \tm)$ are at most $h$.

Note that $\Ht(\textbf{X})\Ht(\textbf{g}:\textbf{Y} \rightarrow \textbf{B}) < h$, so by $Mor^2(h-1)$ we have a context morphism $\textbf{X} \otimes \textbf{g}:\textbf{X} \otimes \textbf{Y} \rightarrow \textbf{X} \otimes \textbf{B}$.

Our goal, in notation analogous to the one used above, is to derive the equality
$$
\textbf{X} \otimes \textbf{Y} \vdashcustom \overline{u \otimes v} \equiv u \otimes \underline{v} \tm.
$$
We will prove by induction on $P \in \{0, \ldots, h\}$ that this holds whenever $\Ht(\textbf{X} \vdashcustom u \tm)\Ht(\textbf{B} \vdashcustom v \tm) = P$. Let $0 \le P \le h$ be such that the claim holds for all $P' < P$, and suppose that $\Ht(\textbf{X} \vdashcustom u \tm)\Ht(\textbf{B} \vdashcustom v \tm) = P$.

Consider contexts $\textbf{B}' = (\textbf{B}, b':\Type(v))$ and $\textbf{Y}' = (\textbf{Y}, y':\Type(\underline{v})) = (\textbf{Y}, y':\underline{\Type(v)})$. We have the following cases:

\begin{itemize}
	\item $u$ is a variable $x_I$ and $v$ is a variable $b_J$. Then $\overline{u \otimes v}$ and $u \otimes \underline{v}$ are both $x_I \otimes g_J$.
	
	\item $u$ is a variable $x_I$ and $v = t(t_1, \ldots, t_l)$. Express $X_I$ as $S(\sigma_1, \ldots, \sigma_p)$. Then $u \otimes v$ is $St((\sigma_1, \ldots, \sigma_p,x_I) \otimes (t_1, \ldots, t_l))$, so $\overline{u \otimes v}$ is provably equal in context $\textbf{X} \otimes \textbf{Y}$, by the induction hypothesis, to
	$$
	St((\sigma_1, \ldots, \sigma_p, x_I) \otimes (\underline{t_1}, \ldots, \underline{t_l})).
	$$
	But the latter is precisely $u \otimes \underline{v}$.
	
	\item $u = s(s_1, \ldots, s_k)$ and $v$ is a variable $b_J$. Express $B_J$ as $T(\tau_1, \ldots, \tau_q)$. Then $\underline{B_J} = T(\underline{\tau_1}, \ldots, \underline{\tau_q})$ and
	$$
	\textbf{Y} \vdashcustom \Type(g_J) \equiv \underline{B_J} \tp
	$$
	is derivable. Denoting this judgment by $J$, we have
	$$
	\Ht(J) \le \Ht(\textbf{Y} \vdashcustom g_J:\underline{B_J}) \le \Ht(\textbf{f}:\textbf{Y} \rightarrow \textbf{B}).
	$$
	Then $\Ht(\textbf{X} \vdashcustom u:\Type(u))\Ht(J) \le h$, so by tensoring $\textbf{X} \vdashcustom u:\Type(u)$ with $J$ we derive an equality
	\[
	\tag{$\texttt{*}$}
	\textbf{X} \otimes \textbf{Y'} \vdashcustom u \otimes y^{' \Type(g_J)} \equiv u \otimes y^{' \underline{B_J}} \tm.
	\]
	
	Also, note that since
	$$
	\Ht(\textbf{X} \vdashcustom u \tm)\Ht(\textbf{Y} \vdashcustom g_J \tm) \le \Ht(\textbf{X} \vdashcustom u \tm)\Ht(\textbf{f}:\textbf{Y} \rightarrow \textbf{B}) \le h,
	$$
	by $h$-derivability we can tensor $\textbf{X} \vdashcustom u \tm$ and $\textbf{Y} \vdashcustom g_J \tm$ to derive
	$$
	\textbf{X} \otimes \textbf{Y} \vdashcustom u \otimes g_J \equiv (u \otimes y^{' \Type(g_J)})[x_1 \otimes g_J \mid x_1y', \ldots, x_m \otimes g_J \mid x_my'] \tm.
	$$
	Now, by using ($\texttt{*}$) it follows that
	\[
	\tag{$\texttt{*}\texttt{*}$}
	\textbf{X} \otimes \textbf{Y} \vdashcustom u \otimes g_J \equiv (u \otimes y^{' \underline{B_J}})[x_1 \otimes g_J \mid x_1y', \ldots, x_m \otimes g_J \mid x_my'] \tm
	\]
	is derivable. In this step, we have used substitution along the context morphism $\textbf{X} \otimes (y_1, \ldots, y_n, g_J):\textbf{X} \otimes \textbf{Y} \rightarrow \textbf{X} \otimes \textbf{Y}'$.

	We will verify that the expression on the right-hand side of ($\texttt{*}\texttt{*}$) is provably equal in context $\textbf{X} \otimes \textbf{Y}$ to $\overline{u \otimes b_J}$.

	On the one hand, we have
	$$
	u \otimes y^{' \underline{B_J}} = sT((s_1, \ldots, s_k) \otimes (\underline{\tau_1}, \ldots, \underline{\tau_q}, y')),
	$$
	hence $(u \otimes y^{' \underline{B_J}})[x_1 \otimes g_J \mid x_1y', \ldots, x_m \otimes g_J \mid x_my']$ equals
	$$
	sT((s_1, \ldots, s_k) \otimes (\underline{\tau_1}, \ldots, \underline{\tau_q}, y^{' \underline{B_J}}))[x_1 \otimes g_J \mid x_1y', \ldots, x_m \otimes g_J \mid x_my']
	$$
	As $y$ does not occur in $\tau_1$, \ldots, $\tau_1$, this expression equals
	$$
	sT \bigg(s_1 \otimes \underline{\tau_1}, \ldots, s_1 \otimes \underline{\tau_q}, (s_1 \otimes y^{' \underline{B_J}})[x_1 \otimes g_J \mid x_1y', \ldots, x_m \otimes g_J \mid x_my'],
	$$
	$$
	\ldots,
	$$
	\[
	\tag{$\blacklozenge$}
	s_k \otimes \underline{\tau_1}, \ldots, s_k \otimes \underline{\tau_q}, (s_k \otimes y^{' \underline{B_J}})[x_1 \otimes g_J \mid x_1y', \ldots, x_m \otimes g_J \mid x_my'] \bigg).
	\]
	On the other hand,
	$$
	u \otimes b_j = sT((s_1, ..., s_k) \otimes (\tau_1, ..., \tau_q, b_j)),
	$$
	thus, by the induction hypothesis, the judgment
	\[
	\tag{$\blacklozenge\blacklozenge$}
	\textbf{X} \otimes \textbf{Y} \vdashcustom \overline{u \otimes b_j} \equiv sT((s_1, ..., s_k) \otimes (\underline{\tau_1}, ..., \underline{\tau_q}, g_J)) \tm
	\]
	is derivable. It remains to prove that ($\blacklozenge$) is provably equal, in context $\textbf{X} \otimes \textbf{Y}$, to the expression $sT((s_1, ..., s_k) \otimes (\underline{\tau_1}, ..., \underline{\tau_q}, g_J))$ as in ($\blacklozenge\blacklozenge$).
	
	We conclude by noting that, by the induction hypothesis, the judgment
	$$
	\textbf{X} \otimes \textbf{Y} \vdashcustom (s_i \otimes y^{' \underline{B_J}})[x_1 \otimes g_J \mid x_1y', \ldots, x_m \otimes g_J \mid x_my'] \equiv s_i \otimes g_J
	$$
	is derivable for $1 \le i \le k$.
	
	\item $u = s(s_1, ..., s_k)$ and $v = t(t_1, ..., t_l)$. Then
	$$
	u \otimes v = (x' \otimes v)[u \otimes b_1 \mid x'b_1, ..., u \otimes b_N \mid x'b_N]
	$$
	where we consider a context $\textbf{X}' = (\textbf{X}, x':\Type(u))$. It follows that
	\begin{align*}
		\overline{u \otimes v} & = \overline{(x' \otimes v)[u \otimes b_j \mid x'b_j]_{j \le N}}\\
		 & = \overline{x' \otimes v}[\overline{u \otimes b_j} \mid x'b_j]_{j \le N}.
	\end{align*}
	Expressing $\Type(u)$ as $S(\sigma_1, ..., \sigma_p)$, we have
	$$
	x' \otimes v = St((\sigma_1, ..., \sigma_p, x') \otimes (t_1, ..., t_l)),
	$$
	hence
	$$
	\overline{x' \otimes v} = St(\underline{\sigma_1 \otimes t_1}, ..., \overline{\sigma_p \otimes t_l}, \overline{x' \otimes t_1}, ..., \overline{x' \otimes t_l}).
	$$
	Noting that $x'$ does not occur in $\sigma_1$, ..., $\sigma_p$, it follows that
	\begin{align*}
		\overline{u \otimes v} & =  \overline{x' \otimes v}[\overline{u \otimes b_j} \mid x'b_j]_{j \le N}\\
		 & = St(\overline{\sigma_1 \otimes t_1}, ..., \overline{\sigma_p \otimes t_l}, \overline{x' \otimes t_1}[\overline{u \otimes b_j} \mid x'b_j]_{j \le N}, ..., \overline{x' \otimes t_l}[\overline{u \otimes b_j} \mid x'b_j]_{j \le N})\\
		 & = St(\overline{\sigma_1 \otimes t_1}, ..., \overline{\sigma_p \otimes t_l}, \overline{(x' \otimes t_1)[u \otimes b_j \mid x'b_j]_{j \le N}}, ..., \overline{(x' \otimes t_1)[u \otimes b_j \mid x'b_j]_{j \le N}})\\
		 & = St(\overline{\sigma_1 \otimes t_1}, ..., \overline{\sigma_p \otimes t_l}, \overline{u \otimes t_1}, ..., \overline{u \otimes t_l}).
	\end{align*}
	
	But by the induction hypothesis, the latter term expression is provably equal in context $\textbf{X} \otimes \textbf{Y}$ to
	\begin{align*}
		St(\sigma_1 \otimes \underline{t_1}, ..., \sigma_p \otimes \underline{t_l}, u \otimes \underline{t_1}, ..., u \otimes \underline{t_l}) & = St((\sigma_1, ..., \sigma_p, u) \otimes (\underline{t_1}, ..., \underline{t_l}))\\
		& = u \otimes \underline{v}.
	\end{align*}
\end{itemize}

\begin{minipage}{\textwidth}
\colorbox{white!90!black}{Proof of $Mor^1_+(h)$}
\vspace{0.5em}

	Suppose given a morphism
	$$
	\textbf{f} = (f_1, ..., f_M): \textbf{X} = (x_1:X_1, ..., x_m:X_m) \longrightarrow \textbf{A} = (a_1:A_1, ..., a_M:A_M)
	$$
	and a derivable judgment $\textbf{A} \vdashcustom U \tp$ in $\bbA$, and a context $\textbf{Y} = (y_1:Y_1, ..., y_n:Y_n)$ in $\bbB$. Assume that $\Ht(\textbf{f}:\textbf{X} \rightarrow \textbf{A})\Ht(\textbf{Y})$, $\Ht(\textbf{A} \vdashcustom U \tp)\Ht(\textbf{Y}) \le h-1$ and $\Ht(\textbf{X} \vdashcustom U[\textbf{f}] \tp)\Ht(\textbf{Y}) \le h$. Consider contexts $\textbf{A}' = (\textbf{A}, a':U)$ and $\textbf{X}' = (\textbf{X}, x':U[\textbf{f}])$, and write $\textbf{f}'$ for the tuple $(f_1, ..., f_M, x')$. We must verify that
	$$
	\textbf{f}' \otimes \textbf{Y}: \textbf{X}' \otimes \textbf{Y} \longrightarrow \textbf{A}' \otimes \textbf{Y}
	$$
	is a context morphism.
\end{minipage}
	
	Firstly, note that as $\Ht(\textbf{X}')\Ht(\textbf{Y})$, $\Ht(\textbf{A}')\Ht(\textbf{Y}) \le h$, we obtain from $Cont(h)$ that $\textbf{X}' \otimes \textbf{Y}$ and $\textbf{A}' \otimes \textbf{Y}$ are contexts.

	Moreover, by $Mor^1(h-1)$ we have a context morphism $\textbf{f} \otimes \textbf{Y}:\textbf{X} \otimes \textbf{Y} \rightarrow \textbf{A} \otimes \textbf{Y}$, thus also a morphism
	\[
	\tag{\texttt{*}}
	\textbf{f} \otimes \textbf{Y}:\textbf{X}' \otimes \textbf{Y} \rightarrow \textbf{A} \otimes \textbf{Y}.
	\]

	If $\Ht(\textbf{Y}) = 0$, then the tuple $\textbf{f}' \otimes \textbf{Y}$, which is empty, is the identity of the empty context $\textbf{X}' \otimes \textbf{Y}$. Otherwise, we have $\Ht(\partial \textbf{Y}) < \Ht(\textbf{Y})$, so we can use $Mor^1_+(h-1)$ to derive $\textbf{f}' \otimes \partial \textbf{Y}: \textbf{X}' \otimes \partial \textbf{Y} \rightarrow \textbf{A}' \otimes \partial \textbf{Y}$, hence
	\[
	\tag{\texttt{**}}
	\textbf{f}' \otimes \partial \textbf{Y}: \textbf{X}' \otimes \textbf{Y} \longrightarrow \textbf{A}' \otimes \partial \textbf{Y}.
	\]
	From (\texttt{*}) and (\texttt{**}) we obtain a morphism
	$$
	\partial(\textbf{f}' \otimes \textbf{Y}):\textbf{X}' \otimes \textbf{Y} \longrightarrow \partial(\textbf{A}' \otimes \textbf{Y}).
	$$
	
	To conclude, using a double bar to indicate substitution along $\partial(\textbf{f}' \otimes \textbf{Y})$, let us verify that the judgment
	$$
	\textbf{X}' \otimes \textbf{Y} \vdashcustom x'y_n: \doverline{U \otimes Y_n}
	$$
	is derivable. It suffices to prove that
	$$
	\textbf{X}' \otimes \textbf{Y} \vdashcustom \doverline{U \otimes Y_n} \equiv X' \otimes Y_n \tp
	$$
	is derivable. But that follows from $Sub^1_{s,s}(h-1)$ with $\textbf{Y} \vdashcustom V \tp$ replaced by $\partial\textbf{Y} \vdashcustom Y_n \tp$, noting that
	\begin{align*}
		&\Ht(\textbf{A} \vdashcustom U \tp)\Ht(\partial \textbf{Y} \vdashcustom Y_n \tp) = \Ht(\textbf{A} \vdashcustom U \tp)\Ht(\textbf{Y}) \le h-1, \\
		&\Ht(\textbf{X} \vdashcustom U[\textbf{f}] \tp)\Ht(\partial \textbf{Y}) < \Ht(\textbf{X} \vdashcustom U[\textbf{f}] \tp)\Ht(\textbf{Y}) \le h, \\
		&\Ht(\textbf{f}:\textbf{X} \rightarrow \textbf{A})\Ht(\partial\textbf{Y} \vdashcustom Y_n \tp) = \Ht(\textbf{f}:\textbf{X} \rightarrow \textbf{A})\Ht(\textbf{Y}) \le h-1.
	\end{align*}

\colorbox{white!90!black}{Proof of $Mor^2_+(h)$}
\vspace{0.5em}

Suppose given a context $\textbf{X} = (x_1:X_1, ..., x_m:X_m)$ in $\bbA$, and a morphism
$$
\textbf{g} = (g_1, ..., g_N): \textbf{Y} = (y_1:Y_1, ..., y_n:Y_n) \longrightarrow \textbf{B} = (b_1:B_1, ..., b_N:B_N)
$$
and a derivable judgment $\textbf{B} \vdashcustom V \tp$ in $\bbB$. Assume that $\Ht(\textbf{X})\Ht(\textbf{g}:\textbf{Y} \rightarrow \textbf{B})$, $\Ht(\textbf{X})\Ht(\textbf{B} \vdashcustom V \tp) \le h-1$ and $\Ht(\textbf{X})\Ht(\textbf{Y} \vdashcustom V[\textbf{g}] \tp) \le h$. Consider contexts $\textbf{B}' = (\textbf{B}, b':V)$ and $\textbf{Y}' = (\textbf{Y}, y':V[\textbf{g}])$, and write $\textbf{g}'$ for the tuple $(g_1, ..., g_N,y')$. We must verify that
$$
\textbf{X} \otimes \textbf{g}': \textbf{X} \otimes \textbf{Y}' \longrightarrow \textbf{X} \otimes \textbf{B}'
$$
is a context morphism.

Firstly, note that as $\Ht(\textbf{X})\Ht(\textbf{Y}')$, $\Ht(\textbf{X})\Ht(\textbf{B}') \le h$, we obtain from $Cont(h)$ that $\textbf{X} \otimes \textbf{Y}'$ and $\textbf{X} \otimes \textbf{B}'$ are contexts.

Moreover, by $Mor^2(h-1)$ we have a context morphism $\textbf{X} \otimes \textbf{g}:\textbf{X} \otimes \textbf{Y} \rightarrow \textbf{X} \otimes \textbf{B}$, thus also a morphism
\[
\tag{\texttt{*}}
\textbf{X} \otimes \textbf{g}: \textbf{X} \otimes \textbf{Y}' \longrightarrow \textbf{X} \otimes \textbf{B}.
\]
If $\Ht(\textbf{X}) = 0$, then the tuple $\textbf{X} \otimes \textbf{g}'$, which is empty, is the identity of the empty context $\textbf{X} \otimes \textbf{Y}'$. Otherwise, we have $\Ht(\partial\textbf{X}) < \Ht(\textbf{X})$, so we can use $Mor^2_+(h-1)$ to derive $\partial \textbf{X} \otimes \textbf{g}':\partial \textbf{X} \otimes \textbf{Y}' \rightarrow \partial \textbf{X} \otimes \textbf{B}'$, hence
\[
\tag{\texttt{**}}
\partial \textbf{X} \otimes \textbf{g}':\textbf{X} \otimes \textbf{Y}' \longrightarrow \partial \textbf{X} \otimes \textbf{B}'.
\]
From (\texttt{*}) and (\texttt{**}) we obtain a morphism
$$
\partial(\textbf{X} \otimes \textbf{g}'):\textbf{X} \otimes \textbf{Y}' \longrightarrow \partial(\textbf{X} \otimes \textbf{B}').
$$
To conclude, using a double bar to indicate substitution along $\partial(\textbf{X} \otimes \textbf{g}')$, let us verify that the judgment
$$
\textbf{X} \otimes \textbf{Y}' \vdashcustom x_m y':\doverline{X_m \otimes V}
$$
is derivable. It suffices to prove that
$$
\textbf{X} \otimes \textbf{Y}' \vdashcustom \doverline{X_m \otimes V} \equiv X_m \otimes Y' \tp
$$
is derivable. But that follows from $Sub^2_{s,s}(h-1)$ with $\textbf{X} \vdashcustom U \tp$ replaced by $\partial \textbf{X} \vdashcustom Y_m \tp$, noting that
\begin{align*}
	&\Ht(\partial \textbf{X} \vdashcustom X_m \tp)\Ht(\textbf{B} \vdashcustom V \tp) = \Ht(\textbf{X})\Ht(\textbf{B} \vdashcustom V \tp) \le h-1, \\
	&\Ht(\partial \textbf{X})\Ht(\textbf{Y} \vdashcustom V[\textbf{g}]) < \Ht(\textbf{X})\Ht(\textbf{X} \vdashcustom V[\textbf{g}]) \le h, \\
	&\Ht(\partial \textbf{X} \vdashcustom X_m \tp)\Ht(\textbf{g}:\textbf{Y} \rightarrow \textbf{B}) = \Ht(\textbf{X})\Ht(\textbf{g}:\textbf{Y} \rightarrow \textbf{B}) \le h-1.
\end{align*}

\colorbox{white!90!black}{Proof of $Sub^1_{t,s}(h)$}
\vspace{0.5em}

Suppose given a context morphism
$$
\textbf{f} = (f_1, \ldots, f_M):\textbf{X} = (x_1:X_1, \ldots, x_m:X_m) \longrightarrow \textbf{A} = (a_1:A_1, \ldots, a_M:A_M)
$$
in $\bbA$, a context $\textbf{Y} = (y_1:Y_1, \ldots, y_n:Y_n)$ in $\bbB$, and derivable judgments $\textbf{A} \vdashcustom u \tm$, $\textbf{Y} \vdashcustom V \tp$ such that $\Ht(\textbf{f}:\textbf{X} \rightarrow \textbf{A})\Ht(\textbf{Y} \vdashcustom V \tp)$ and $\Ht(\textbf{A} \vdashcustom u \tm)\Ht(\textbf{Y} \vdashcustom V \tp)$ are at most $h$.

Let $\textbf{Y}' = (\textbf{Y}, y':V)$. Since $\Ht(\textbf{X})\Ht(\textbf{Y}') = \Ht(\textbf{X})\Ht(\textbf{Y} \vdashcustom V \tp) \le h$, by $Cont(h)$ we have that $\textbf{X} \otimes \textbf{Y}'$ is a context.

Using notation analogous to that of the previous lemma, we will verify by induction on $P \in \{0, \ldots, h\}$ that
$$
\textbf{X} \otimes \textbf{Y}' \vdashcustom \overline{u \otimes y'} \equiv \underline{u} \otimes y' \tm
$$
is derivable whenever $\Ht(\textbf{A} \vdashcustom u \tm)\Ht(\textbf{Y} \vdashcustom V \tp) = P$.

Assume that the claim holds for all $P' < P = \Ht(\textbf{A} \vdashcustom u \tm)\Ht(\textbf{Y} \vdashcustom V \tp)$.

If $u$ is a variable $a_I$, then
$$
\overline{u \otimes y'} = \overline{a_Iy'} = f_I \otimes y' = \underline{u} \otimes y'.
$$

If it is not a variable, let $u = s(s_1, \ldots, s_p)$ and $V = T(t_1, \ldots, t_q)$. Then
$$
\underline{u} \otimes y' = s(\underline{s_1}, \ldots, \underline{s_p}) \otimes T(t_1, \ldots, t_q) = sT((\underline{s_1}, \ldots, \underline{s_p}) \otimes (t_1, \ldots, t_q, y')).
$$

Now, for $1 \le i \le p$ and $1 \le j \le q$, note that $\Ht(\textbf{A} \vdashcustom s_i \tm)\Ht(\textbf{Y} \vdashcustom t_j \tm) < \Ht(\textbf{A} \vdashcustom u \tm)\Ht(\textbf{Y} \vdashcustom V \tp) \le h$, so by $Sub^1_{t,t}(h-1)$ we can derive $\overline{s_i \otimes t_j} \equiv \underline{s_i} \otimes t_j$ in context $\textbf{X} \otimes \textbf{Y}$, hence in $\textbf{X} \otimes \textbf{Y}'$.

Moreover, for $1 \le i \le p$ we have $\Ht(\textbf{A} \vdashcustom s_i \tm)\Ht(\textbf{Y} \vdashcustom V \tp) < \Ht(\textbf{A} \vdashcustom u \tm)\Ht(\textbf{Y} \vdashcustom V \tp)$, so by the induction hypothesis we can derive $\overline{s_i \otimes y'} \equiv \underline{s_i} \otimes y'$ in context $\textbf{X} \otimes \textbf{Y}'$.

It follows that
$$
\textbf{X} \otimes \textbf{Y}' \vdashcustom \overline{sT((s_1, \ldots, s_p) \otimes (t_1, \ldots, t_q,y'))} \equiv sT((\underline{s_1}, \ldots, \underline{s_p}) \otimes (t_1,\ldots,t_q,y'))  \tm
$$
is derivable. But this judgment is precisely $\textbf{X} \otimes \textbf{Y}' \vdashcustom \overline{u \otimes y'} \equiv \underline{u} \otimes y' \tm$.

\vspace{1em}
\colorbox{white!90!black}{Proof of $Sub^2_{s,t}(h)$}
\vspace{0.5em}

Suppose given a context $\textbf{X} = (x_1:X_1, ..., x_m:X_m)$ in $\bbA$, a context morphism
$$
\textbf{g}:(g_1, ..., g_N): \textbf{Y} = (y_1:Y_1, ..., y_n:Y_n) \longrightarrow \textbf{B} = (b_1:B_1, ..., b_N:B_N)
$$
in $\bbB$, and derivable judgments $\textbf{X} \vdashcustom U \tp$, $\textbf{B} \vdashcustom v \tm$ such that $\Ht(\textbf{X} \vdashcustom U \tp)\Ht(\textbf{g}:\textbf{Y} \rightarrow \textbf{B})$ and $\Ht(\textbf{X} \vdashcustom U \tp)\Ht(\textbf{B} \vdashcustom v \tm)$ are at most $h$.

Let $\textbf{X}' = (\textbf{X}, x':U)$. Since $\Ht(\textbf{X}')\Ht(\textbf{Y}) = \Ht(\textbf{X} \vdashcustom U \tp)\Ht(\textbf{Y}) \le h$, by $Cont(h)$ we have that $\textbf{X}' \otimes \textbf{Y}$ is a context.

In notation analogous to that of the previous lemma, we will verify by induction on $P \in \{0, ..., h\}$ that
$$
\textbf{X}' \otimes \textbf{Y} \vdashcustom \doverline{x' \otimes v} \equiv x' \otimes \underline{v} \tm
$$
is derivable if $\Ht(\textbf{X} \vdashcustom U \tp)\Ht(\textbf{B} \vdashcustom v \tm) = P$.

Assume that the claim holds for all $P' < P = \Ht(\textbf{X} \vdashcustom U \tp)\Ht(\textbf{B} \vdashcustom v \tm)$.

If $v$ is a variable $b_J$, then
$$
\doverline{x' \otimes v} = \doverline{x'b_J} = x' \otimes g_J = x' \otimes \underline{v}.
$$
If it is not a variable, let $U = S(s_1, ..., s_p)$ and $v = t(t_1, ..., t_q)$. Then
$$
x' \otimes \underline{v} = S(s_1, ..., s_p) \otimes t(\underline{t_1}, ..., \underline{t_q}) = St((s_1, ..., s_p,x') \otimes (\underline{t_1}, ..., \underline{t_q})).
$$
Now, for $1 \le i \le p$ and $1 \le j \le q$, note that $\Ht(\textbf{X} \vdashcustom s_i \tm)\Ht(\textbf{B} \vdashcustom t_j \tm) < \Ht(\textbf{X} \vdashcustom U \tp)\Ht(\textbf{B} \vdashcustom v \tm) \le h$, so by $Sub^2_{t,t}(h-1)$ we can derive $\doverline{s_i \otimes t_j} \equiv s_i \otimes \underline{t_j}$ in context $\textbf{X} \otimes \textbf{Y}$, hence in $\textbf{X}' \otimes \textbf{Y}$.

Moreover, for $1 \le j \le q$ we have $\Ht(\textbf{X} \vdashcustom U \tp)\Ht(\textbf{B} \vdashcustom t_j \tm) < \Ht(\textbf{X} \vdashcustom U \tp)\Ht(\textbf{B} \vdashcustom v \tm) \le h$, so by the induction hypothesis we can derive $\doverline{x' \otimes t_j} \equiv x' \otimes \underline{t_j}$ in context $\textbf{X}' \otimes \textbf{Y}$.

It follows that
$$
\textbf{X}' \otimes \textbf{Y} \vdashcustom \doverline{St((s_1, ..., s_p,x') \otimes (t_1, ..., t_q))} \equiv St((s_1, ..., s_p,x') \otimes (\underline{t_1}, ..., \underline{t_q})) \tm
$$
is derivable. But this judgment is precisely $\textbf{X}' \otimes \textbf{Y} \vdashcustom \doverline{x' \otimes v} \equiv x' \otimes \underline{v} \tm$.

\vspace{1em}
\colorbox{white!90!black}{Proof of $Sub^1_{s,t}(h)$}
\vspace{0.5em}

Suppose given a context morphism
$$
\textbf{f} = (f_1, \ldots, f_M):\textbf{X} = (x_1:X_1, \ldots, x_m:X_m) \longrightarrow \textbf{A} = (a_1:A_1, \ldots, a_M:A_M)
$$
in $\bbA$, a context $\textbf{Y} = (y_1:Y_1, \ldots, y_n:Y_n)$ in $\bbB$, and derivable judgments $\textbf{A} \vdashcustom U \tp$, $\textbf{Y} \vdashcustom v \tm$ such that $\Ht(\textbf{f}:\textbf{X} \rightarrow \textbf{A})\Ht(\textbf{Y} \vdashcustom v \tm)$, $\Ht(\textbf{X} \vdashcustom U[\textbf{f}] \tp)\Ht(\textbf{Y})$ and $\Ht(\textbf{A} \vdashcustom U \tp)\Ht(\textbf{Y} \vdashcustom v \tm)$ are at most $h$.

Let $\textbf{A}' = (\textbf{A}, a':U)$ and $\textbf{X}' = (\textbf{X}, x':\underline{U})$ where $\underline{U} = U[\textbf{f}]$. Since $\Ht(\textbf{X}')\Ht(\textbf{Y}) \le h$, by $Cont(h)$ we have that $\textbf{X}' \otimes \textbf{Y}$ is a context.

We will verify by induction on $P \in \{0, \ldots, h\}$ that
$$
\textbf{X}' \otimes \textbf{Y} \vdashcustom \doverline{a' \otimes v} \equiv x' \otimes v \tm
$$
is derivable whenever $\Ht(\textbf{A} \vdashcustom U \tp)\Ht(\textbf{Y} \vdashcustom v \tm) = P$.

In fact, we will start by verifying, without using the induction hypothesis, that $\textbf{X}' \otimes \textbf{Y} \vdashcustom \doverline{a' \otimes v} \tm$ is derivable; the induction hypothesis will then be used to derive the equality with $x' \otimes v$ in context $\textbf{X}' \otimes \textbf{Y}$, which in particular will imply that $\textbf{X}' \otimes \textbf{Y} \vdashcustom x' \otimes v \tm$ is derivable.

For the first part, note that, as
\begin{align*}
	&\Ht(\textbf{f}:\textbf{X} \rightarrow \textbf{A})\Ht(\textbf{Y}) < \Ht(\textbf{f}:\textbf{X} \rightarrow \textbf{A})\Ht(\textbf{Y} \vdashcustom v \tm) \le h, \\
	&\Ht(\textbf{A} \vdashcustom U \tp)\Ht(\textbf{Y}) < \Ht(\textbf{A} \vdashcustom U \tp)\Ht(\textbf{Y} \vdashcustom v \tm) \le h, \\
	&\Ht(\textbf{X} \vdashcustom \underline{U} \tp)\Ht(\textbf{Y}) \le h,
\end{align*}
by $Mor^1_+(h)$ we can derive a morphism
\[
\tag{\texttt{*}}
\textbf{f}' \otimes \textbf{Y}:\textbf{X}' \otimes \textbf{Y} \longrightarrow \textbf{A}' \otimes \textbf{Y}.
\]

Moreover, as $\Ht(\textbf{A} \vdashcustom U \tp)\Ht(\textbf{Y} \vdashcustom v \tm) \le h$, by $h$-derivability we can tensor $\textbf{A} \vdashcustom U \tp$ and $\textbf{Y} \vdashcustom v \tm$ to derive
\[
\tag{\texttt{**}}
\textbf{A}' \otimes \textbf{Y} \vdashcustom a' \otimes v \tm.
\]
By substitution of (\texttt{**}) along (\texttt{*}) we derive $\textbf{X}' \otimes \textbf{Y} \vdashcustom \doverline{a' \otimes v} \tm$.

Now, let us turn to the inductive argument. Assume that the claim holds for all $P' < P = \Ht(\textbf{A} \vdashcustom U \tp)\Ht(\textbf{Y} \vdashcustom v \tm)$.

If $v$ is a variable $y_J$, then
$$
\doverline{a' \otimes v} = \doverline{a'y_J} = x'y_J = \underline{a'} \otimes v.
$$
If it is not a variable, let $U = S(s_1, ..., s_p)$ and $v = t(t_1, ..., t_q)$. Then $\underline{U} = S(\underline{s_1}, ..., \underline{s_p})$, so
$$
x' \otimes v = S(\underline{s_1}, ..., \underline{s_p}) \otimes t(t_1, ..., t_q) = St((\underline{s_1}, ..., \underline{s_p}, x') \otimes (t_1, ..., t_q)).
$$
For $1 \le i \le p$ and $1 \le j \le q$, note that $\Ht(\textbf{A} \vdashcustom s_i \tm)\Ht(\textbf{Y} \vdashcustom t_j \tm) < \Ht(\textbf{A} \vdashcustom U \tp)\Ht(\textbf{Y} \vdashcustom v \tm) \le h$, so by $Sub^1_{t,t}(h-1)$ we can derive $\doverline{s_i \otimes t_j} \equiv \underline{s_i} \otimes t_j$ in context $\textbf{X} \otimes \textbf{Y}$, hence in $\textbf{X}' \otimes \textbf{Y}$.

Also, for $1 \le j \le q$ we have $\Ht(\textbf{A} \vdashcustom U \tp)\Ht(\textbf{Y} \vdashcustom t_j \tm) < \Ht(\textbf{A} \vdashcustom U \tp)\Ht(\textbf{Y} \vdashcustom v \tm)$, so by the induction hypothesis we can derive $\doverline{a' \otimes t_j} \equiv x' \otimes t_j$ in context $\textbf{X}' \otimes \textbf{Y}$.

It follows that
$$
\textbf{X}' \otimes \textbf{Y} \vdashcustom \doverline{St((s_1, ..., s_p,a') \otimes (t_1, ..., t_q))} \equiv St((\underline{s_1}, ..., \underline{s_p}, x') \otimes (t_1, ..., t_q)) \tm
$$
is derivable. But this judgment is precisely $\textbf{X}' \otimes \textbf{Y} \vdashcustom \doverline{a' \otimes v} \equiv x' \otimes v \tm$.

\vspace{1em}
\colorbox{white!90!black}{Proof of $Sub^2_{t,s}(h)$}
\vspace{0.5em}

Suppose given a context $\textbf{X} = (x_1:X_1, ..., x_m:X_m)$ in $\bbA$, a context morphism
$$
\textbf{g}:(g_1, ..., g_N): \textbf{Y} = (y_1:Y_1, ..., y_n:Y_n) \longrightarrow \textbf{B} = (b_1:B_1, ..., b_N:B_N)
$$
in $\bbB$, and derivable judgments $\textbf{X} \vdashcustom u \tm$, $\textbf{B} \vdashcustom V \tp$ such that $\Ht(\textbf{X} \vdashcustom u \tm)\Ht(\textbf{g}:\textbf{Y} \rightarrow \textbf{B})$, $\Ht(\textbf{X})\Ht(\textbf{Y} \vdashcustom V[\textbf{g}] \tp)$ and $\Ht(\textbf{X} \vdashcustom u \tm)\Ht(\textbf{B} \vdashcustom V \tp)$ are at most $h$.

Let $\textbf{B}' = (\textbf{B}, b':V)$ and $\textbf{Y}' = (\textbf{Y}, y':\underline{V})$ where $\underline{V} = V[\textbf{g}]$. Since $\Ht(\textbf{X})\Ht(\textbf{Y}') \le h$, by $Cont(h)$ we have that $\textbf{X} \otimes \textbf{Y}'$ is a context.

We will verify by induction on $P \in \{0, ..., h\}$ that
$$
\textbf{X} \otimes \textbf{Y}' \vdashcustom \doverline{u \otimes b'} \equiv u \otimes y' \tm
$$
is derivable whenever $\Ht(\textbf{X} \vdashcustom u \tm)\Ht(\textbf{B} \vdashcustom V \tp) = P$.

As in the previous lemma, we will start by verifying, without using the induction hypothesis, that $\textbf{X} \otimes \textbf{Y}' \vdashcustom \doverline{u \otimes b'} \tm$ is derivable; the we will use the induction hypothesis to derive the equality with $u \otimes y'$ in context $\textbf{X} \otimes \textbf{Y}'$, which in particular implies that $\textbf{X} \otimes \textbf{Y}' \vdashcustom u \otimes y' \tm$ is derivable.

For the first part, note that, as
\begin{align*}
	&\Ht(\textbf{X})\Ht(\textbf{g}:\textbf{Y} \rightarrow \textbf{B}) < \Ht(\textbf{X} \vdashcustom u \tm)\Ht(\textbf{g}:\textbf{Y} \rightarrow \textbf{B}) \le h, \\
	&\Ht(\textbf{X})\Ht(\textbf{B} \vdashcustom V \tp) < \Ht(\textbf{X} \vdashcustom u \tm)\Ht(\textbf{B} \vdashcustom V \tp) \le h, \\
	&\Ht(\textbf{X})\Ht(\textbf{Y} \vdashcustom \underline{V}) \le h,
\end{align*}
by $Mor^2_+(h)$ we can derive
\[
\tag{\texttt{*}}
\textbf{X} \otimes \textbf{g}': \textbf{X} \otimes \textbf{Y}' \longrightarrow \textbf{X} \otimes \textbf{B}'.
\]
Moreover, as $\Ht(\textbf{X} \vdashcustom u \tm)\Ht(\textbf{Y} \vdashcustom V \tp) \le h$, by $h$-derivability we can tensor $\textbf{X} \vdashcustom u \tm$ and $\textbf{Y} \vdashcustom V \tp$ to derive
\[
\tag{\texttt{**}}
\textbf{X} \otimes \textbf{B}' \vdashcustom u \otimes b' \tm.
\]
By substitution of (\texttt{**}) along (\texttt{*}) we derive $\textbf{X} \otimes \textbf{Y}' \vdashcustom \doverline{u \otimes b'} \tm$.

Now, we turn to the inductive argument. Assume that the claim holds for all $P' < P = \Ht(\textbf{X} \vdashcustom u \tm)\Ht(\textbf{B} \vdashcustom V \tp)$.

If $u$ is a variable $a_I$, then
$$
\doverline{u \otimes b'} = \doverline{a_Ib'} = a_Iy' = u \otimes \underline{b'}.
$$
If it is not a variable, let $u = s(s_1, ..., s_p)$ and $V = T(t_1, ..., t_q)$. Then $\underline{V} = T(\underline{t_1}, ..., \underline{t_q})$, so
$$
u \otimes y' = s(s_1, ..., s_p) \otimes T(\underline{t_1}, ..., \underline{t_q}) = sT((s_1, ..., s_p) \otimes (\underline{t_1}, ..., \underline{t_1}, y')).
$$
For $1 \le i \le p$ and $1 \le j \le q$, note that $\Ht(\textbf{X} \vdashcustom s_i \tm)\Ht(\textbf{B} \vdashcustom t_j \tm) < \Ht(\textbf{X} \vdashcustom u \tm)\Ht(\textbf{B} \vdashcustom V \tp) \le h$, so by $Sub^2_{t,t}(h)$ we can derive $\doverline{s_i \otimes t_j} \equiv s_i \otimes \underline{t_j}$ in context $\textbf{X} \otimes \textbf{Y}$, hence in $\textbf{X} \otimes \textbf{Y}'$.

Also, for $1 \le i \le p$ we have $\Ht(\textbf{X} \vdashcustom s_i \tm)\Ht(\textbf{B} \vdashcustom V \tp) < \Ht(\textbf{X} \vdashcustom u \tm)\Ht(\textbf{B} \vdashcustom V \tp)$, so by the induction hypothesis we can derive $\doverline{s_i \otimes b'} \equiv s_i \otimes y'$ in context $\textbf{X} \otimes \textbf{Y}'$.

It follows that
$$
\textbf{X} \otimes \textbf{Y}' \vdashcustom \doverline{sT((s_1, ..., s_p) \otimes (t_1, ..., t_q,b'))} \equiv sT((s_1, ..., s_p) \otimes (\underline{t_1}, ..., \underline{t_q}, y')) \tm
$$
is derivable. But this judgment is precisely $\textbf{X} \otimes \textbf{Y}' \vdashcustom \doverline{u \otimes b'} \equiv u \otimes y' \tm$.

\vspace{1em}
\colorbox{white!90!black}{Proof of $Sub^1_{s,s}(h)$}
\vspace{0.5em}

Suppose given a context morphism
$$
\textbf{f} = (f_1, \ldots, f_M):\textbf{X} = (x_1:X_1, \ldots, x_m:X_m) \longrightarrow \textbf{A} = (a_1:A_1, \ldots, a_M:A_M)
$$
in $\bbA$, a context $\textbf{Y} = (y_1:Y_1, \ldots, y_n:Y_n)$ in $\bbB$, and derivable judgments $\textbf{A} \vdashcustom U \tp$ and $\textbf{Y} \vdashcustom V \tp$ such that $\Ht(\textbf{A} \vdashcustom U \tp)\Ht(\textbf{Y} \vdashcustom V \tp)$, $\Ht(\textbf{X} \vdashcustom U[\textbf{f}] \tp)\Ht(\textbf{Y})$ and $\Ht(\textbf{f}:\textbf{X} \rightarrow \textbf{A})\Ht(\textbf{Y} \vdashcustom V \tp)$ are at most $h$.

Define contexts $\textbf{A}' = (\textbf{A}, a':U)$, $\textbf{X}' = (\textbf{X},x':\underline{U})$, and $\textbf{Y}' = (\textbf{Y}, y':V)$ where $\underline{U} = U[\textbf{f}]$. Using the same notation as in the previous lemma, our goal is to derive
$$
\partial(\textbf{X}' \otimes \textbf{Y}') \vdashcustom \doverline{U \otimes V} \equiv \underline{U} \otimes V \tp.
$$

We will prove by induction on $P \in \{0, \ldots, h\}$ that this judgment is derivable whenever $\Ht(\textbf{A} \vdashcustom U \tp)\Ht(\textbf{Y} \vdashcustom V \tp) = P$. Suppose that the claim holds for all $P' < P = \Ht(\textbf{A} \vdashcustom U \tp)\Ht(\textbf{Y} \vdashcustom V \tp)$.

Let $U = S(s_1, ..., s_k)$ and $V = T(t_1, ..., t_l)$. Then
$$
\doverline{U \otimes V} = \doverline{ST(\partial((s_1, ..., s_k,a') \otimes (t_1, ..., t_l,y')))},
$$
which can be written in matrix form as
$$
ST
\begin{pmatrix}
	\doverline{s_1 \otimes t_1} & \cdots & \doverline{s_1 \otimes t_l} & \doverline{s_1 \otimes y'}\\
	\vdots & \ddots & \vdots & \vdots\\
	\doverline{s_k \otimes t_1} & \cdots & \doverline{s_k \otimes t_l} & \doverline{s_k \otimes y'}\\
	\doverline{a' \otimes t_1} & \cdots & \doverline{a' \otimes t_l} & -\\
\end{pmatrix}.
$$
We can describe each entry as follows:
\begin{itemize}
	\item For $i \le k$ and $j \le l$, the expression $\doverline{s_i \otimes t_j}$ is obtained from $s_i \otimes t_j$ by substitution along $\textbf{f} \otimes \textbf{Y}:\textbf{X} \otimes \textbf{Y} \rightarrow \textbf{A} \otimes \textbf{Y}$. But since
	\begin{align*}
		&\Ht(\textbf{A} \vdashcustom s_i \tm)\Ht(\textbf{Y} \vdashcustom t_j \tm) < \Ht(\textbf{A} \vdashcustom U \tp)\Ht(\textbf{Y} \vdashcustom V \tp) \le h, \\
		&\Ht(\textbf{f}:\textbf{X} \rightarrow \textbf{A})\Ht(\textbf{Y} \vdashcustom t_j \tm) < \Ht(\textbf{f}:\textbf{X} \rightarrow \textbf{A})\Ht(\textbf{Y} \vdashcustom V \tp) \le h,
	\end{align*}
	by $Sub^1_{t,t}(h-1)$ we can derive $\textbf{X} \otimes \textbf{Y} \vdashcustom \doverline{s_i \otimes t_j} \equiv \underline{s_i} \otimes t_j \tm$.
	
	\item To describe $\doverline{a' \otimes t_j}$, note that since
	\begin{align*}
		&\Ht(\textbf{f}:\textbf{X} \rightarrow \textbf{A})\Ht(\textbf{Y} \vdashcustom t_j \tm) < \Ht(\textbf{f}:\textbf{X} \rightarrow \textbf{A})\Ht(\textbf{Y} \vdashcustom V \tp) \le h, \\
		&\Ht(\textbf{X} \vdashcustom \underline{U} \tp)\Ht(\textbf{Y}) \le h, \\
		&\Ht(\textbf{A} \vdashcustom U \tp)\Ht(\textbf{Y} \vdashcustom t_j \tm) < \Ht(\textbf{A} \vdashcustom U \tp)\Ht(\textbf{Y} \vdashcustom V \tp) \le h,
	\end{align*}
	we can use $Sub^1_{s,t}(h)$ to derive
	$$
	\textbf{X}' \otimes \textbf{Y} \vdashcustom \doverline{a' \otimes t_j} \equiv x' \otimes t_j \tm.
	$$
	
	\item To describe $\doverline{s_i \otimes y'}$, since
	\begin{align*}
		&\Ht(\textbf{f}:\textbf{X} \rightarrow \textbf{A})\Ht(\textbf{Y} \vdashcustom V \tp) \le h, \\
		&\Ht(\textbf{A} \vdashcustom s_i \tm)\Ht(\textbf{Y} \vdashcustom V \tp) \le \Ht(\textbf{A} \vdashcustom U \tp)\Ht(\textbf{Y} \vdashcustom V \tp) \le h,
	\end{align*}
	we can use $Sub^1_{t,s}(h)$ to derive
	$$
	\textbf{X}' \otimes \textbf{Y} \vdashcustom \doverline{s_i \otimes y'} \equiv \underline{s_i} \otimes y' \tm.
	$$
\end{itemize}

We conclude that $\doverline{U \otimes V}$ is provably equal in context $\partial(\textbf{X}' \otimes \textbf{Y}')$ to
$$
ST
\begin{pmatrix}
	\underline{s_1} \otimes t_1 & \cdots & \underline{s_1} \otimes t_l & \underline{s_1} \otimes y'\\
	\vdots & \ddots & \vdots & \vdots\\
	\underline{s_k} \otimes t_1 & \cdots & \underline{s_k} \otimes t_l & \underline{s_k} \otimes y'\\
	x' \otimes t_1 & \cdots & x' \otimes t_l & -\\
\end{pmatrix}.
$$
This, in turn, is the matrix form of $\underline{U} \otimes V = ST(\partial((\underline{s_1}, ..., \underline{s_k},x') \otimes (t_1, ..., t_l,y')))$.

\vspace{1em}
\colorbox{white!90!black}{Proof of $Sub^2_{s,s}(h)$}
\vspace{0.5em}

Suppose given a context $\textbf{X} = (x_1:X_1, ..., x_m:X_m)$ in $\bbA$, a context morphism
$$
\textbf{g}:(g_1, ..., g_N): \textbf{Y} = (y_1:Y_1, ..., y_n:Y_n) \longrightarrow \textbf{B} = (b_1:B_1, ..., b_N:B_N)
$$
in $\bbB$, and derivable judgments $\textbf{X} \vdashcustom U \tp$ and $\textbf{B} \vdashcustom V \tp$ such that $\Ht(\textbf{X} \vdashcustom U \tp)\Ht(\textbf{B} \vdashcustom V \tp)$, $\Ht(\textbf{X})\Ht(\textbf{Y} \vdashcustom V[\textbf{g}] \tp)$ and $\Ht(\textbf{X} \vdashcustom U \tp)\Ht(\textbf{g}:\textbf{Y} \rightarrow \textbf{B})$ are at most $h$.

Define contexts $\textbf{X}' = (\textbf{X}, x':U)$, $\textbf{B}' = (\textbf{B}, b':V)$, and $\textbf{Y}' = (\textbf{Y}, y':\underline{V})$ where $\underline{V} = V[\textbf{g}]$. Using notation analogous to that in the previous lemma, our goal is to derive
$$
\partial(\textbf{X}' \otimes \textbf{Y}') \vdashcustom \doverline{U \otimes V} \equiv U \otimes \underline{V} \tp.
$$
We will prove by induction on $P \in \{0, ..., h\}$ that this judgment is derivable whenever $\Ht(\textbf{X} \vdashcustom U \tp)\Ht(\textbf{B} \vdashcustom V \tp) = P$. Suppose that the claim holds for all $P' < P = \Ht(\textbf{X} \vdashcustom U \tp)\Ht(\textbf{B} \vdashcustom V \tp)$.

Let $U = S(s_1, ..., s_k)$ and $V = T(t_1, ..., t_l)$. Then, in matrix form, $\doverline{U \otimes V}$ is
$$
ST
\begin{pmatrix}
	\doverline{s_1 \otimes t_1} & \cdots & \doverline{s_1 \otimes t_l} & \doverline{s_1 \otimes b'}\\
	\vdots & \ddots & \vdots & \vdots\\
	\doverline{s_k \otimes t_1} & \cdots & \doverline{s_k \otimes t_l} & \doverline{s_k \otimes b'}\\
	\doverline{x' \otimes t_1} & \cdots & \doverline{x' \otimes t_l} & -\\
\end{pmatrix}.
$$
We describe each entry as follows:
\begin{itemize}
	\item For $i \le k$ and $j \le l$, the expression $\doverline{s_i \otimes t_j}$ is obtained from $s_i \otimes t_j$ by substitution along $\textbf{X} \otimes \textbf{g}: \textbf{X} \otimes \textbf{B} \rightarrow \textbf{X} \otimes \textbf{Y}$. But since
	\begin{align*}
		&\Ht(\textbf{X} \vdashcustom s_i \tm)\Ht(\textbf{B} \vdashcustom t_j \tm) < \Ht(\textbf{X} \vdashcustom U \tp)\Ht(\textbf{B} \vdashcustom V \tp) \le h, \\
		&\Ht(\textbf{X} \vdashcustom s_i \tm)\Ht(\textbf{g}:\textbf{Y} \rightarrow \textbf{B}) < \Ht(\textbf{X} \vdashcustom U \tp)\Ht(\textbf{g}:\textbf{Y} \rightarrow \textbf{B}) \le h,
	\end{align*}
	by $Sub^2_{t,t}(h-1)$ we can derive $\textbf{X} \otimes \textbf{Y} \vdashcustom \doverline{s_i \otimes t_j} \equiv s_i \otimes \underline{t_j} \tm$.
	
	\item To describe $\doverline{s_i \otimes b'}$, note that since
	\begin{align*}
		&\Ht(\textbf{X} \vdashcustom s_i \tm)\Ht(\textbf{g}:\textbf{Y} \rightarrow \textbf{B}) < \Ht(\textbf{X} \vdashcustom U \tp)\Ht(\textbf{g}:\textbf{Y} \rightarrow \textbf{B}) \le h, \\
		&\Ht(\textbf{X})\Ht(\textbf{Y} \vdashcustom \underline{V} \tp) \le h, \\
		&\Ht(\textbf{X} \vdashcustom s_i \tm)\Ht(\textbf{B} \vdashcustom V \tp) < \Ht(\textbf{X} \vdashcustom U \tp)\Ht(\textbf{B} \vdashcustom V \tp) \le h,
	\end{align*}
	we can use $Sub^2_{t,s}(h)$ to derive
	$$
	\textbf{X} \otimes \textbf{Y}' \vdashcustom \doverline{s_i \otimes b'} \equiv s_i \otimes y' \tm.
	$$
	
	\item To describe $\doverline{x' \otimes t_j}$, since
	\begin{align*}
		&\Ht(\textbf{X} \vdashcustom U \tp)\Ht(\textbf{g}:\textbf{Y} \rightarrow \textbf{B}) \le h, \\
		&\Ht(\textbf{X} \vdashcustom U \tp)\Ht(\textbf{B} \vdashcustom t_j \tm) \le \Ht(\textbf{X} \vdashcustom U \tp)\Ht(\textbf{B} \vdashcustom V \tp) \le h,
	\end{align*}
	we can use $Sub^2_{s,t}(h)$ to derive
	$$
	\textbf{X} \otimes \textbf{Y}' \vdashcustom \doverline{x' \otimes t_j} \equiv x' \otimes \underline{t_j} \tm.
	$$
\end{itemize}
We conclude that $\doverline{U \otimes V}$ is provably equal in context $\partial(\textbf{X}' \otimes \textbf{Y}')$ to
$$
ST
\begin{pmatrix}
	s_1 \otimes \underline{t_1} & \cdots & s_1 \otimes \underline{t_l} & s_1 \otimes y'\\
	\vdots & \ddots & \vdots & \vdots\\
	s_k \otimes \underline{t_1} & \cdots & s_k \otimes \underline{t_l} & s_k \otimes y'\\
	x' \otimes \underline{t_1} & \cdots & x' \otimes \underline{t_l} & -\\
\end{pmatrix}.
$$
This, in turn, is the matrix form of $U \otimes \underline{V} = ST(\partial((s_1, ..., s_k,x') \otimes (\underline{t_1}, ..., \underline{t_l}, y')))$.

\vspace{1em}
\colorbox{white!90!black}{Proof of $Mor^1(h)$}
\vspace{0.5em}

Write $\mathscr S$ for the set of all pairs consisting of a context morphism $\textbf{f}:\textbf{X} \rightarrow \textbf{A}$ in $\bbA$ and a context $\textbf{Y}$ in $\bbB$ such that $\Ht(\textbf{f}:\textbf{X} \rightarrow \textbf{A})\Ht(\textbf{Y}) = h$. Note that in this case, since $\Ht(\textbf{X})$, $\Ht(\textbf{A}) \le \Ht(\textbf{f}:\textbf{X} \rightarrow \textbf{A})$, by $Cont(h)$ we have that $\textbf{X} \otimes \textbf{Y}$ and $\textbf{A} \otimes \textbf{Y}$ are contexts.

We will prove by induction on $L \ge 0$ that $\textbf{f} \otimes \textbf{Y}$ is a morphism from $\textbf{X} \otimes \textbf{Y}$ to $\textbf{A} \otimes \textbf{Y}$ whenever $l(\textbf{A})l(\textbf{Y}) = L$.

The claim holds for $L = 0$ as in this case $\textbf{A} \otimes \textbf{Y}$ is the empty context and $\textbf{f} \otimes \textbf{Y}$ is the empty sequence. Now, let $L \ge 1$ and assume that the claim holds for $0$, \ldots, $L-1$. Suppose that $(\textbf{f}:\textbf{X} \rightarrow \textbf{A},\textbf{Y}) \in \mathscr S$ satisfies $l(\textbf{A})l(\textbf{Y}) = L$.

Since $\Ht(\partial \textbf{f}: \textbf{X} \rightarrow \partial\textbf{A}) \le \Ht(\textbf{f}:\textbf{X} \rightarrow \textbf{A})$ and $l(\partial \textbf{A}) < l(\textbf{A})$, by the induction hypothesis we have a morphism
\[
\tag{$\texttt{*}$}
\partial \textbf{f} \otimes \textbf{Y}: \textbf{X} \otimes \textbf{Y} \longrightarrow \partial\textbf{A} \otimes \textbf{Y}.
\]
Similarly, as $\Ht(\partial \textbf{Y}) < \Ht(\textbf{Y})$ we have a morphism $\textbf{f} \otimes \partial\textbf{Y}:\textbf{X} \otimes \partial \textbf{Y} \longrightarrow \textbf{A} \otimes \partial \textbf{Y}$, hence a morphism
\[
\tag{$\texttt{*}\texttt{*}$}
\textbf{f} \otimes \partial\textbf{Y}:\textbf{X} \otimes \textbf{Y} \longrightarrow \textbf{A} \otimes \partial \textbf{Y}.
\]
From ($\texttt{*}$) and ($\texttt{*}\texttt{*}$) we obtain a morphism
$$
\partial(\textbf{f} \otimes \textbf{Y}): \textbf{X} \otimes \textbf{Y} \longrightarrow \partial(\textbf{A} \otimes \textbf{Y}).
$$
To derive $\textbf{f} \otimes \textbf{Y}:\textbf{X} \otimes \textbf{Y} \rightarrow \textbf{A} \otimes \textbf{Y}$, it remains to derive the term judgment corresponding to its last entry, i.e.
$$
\textbf{X} \otimes \textbf{Y} \vdashcustom f_M \otimes y_n : \doverline{A_M \otimes Y_n}
$$
where the double bar indicates the substitution operation $[f_i \otimes y_j \mid a_iy_j]_{(i,j) < (M,n)}$.

Using a lower bar to indicate the operation $[f_i \mid a_i]_{i < M}$, we have
$$
\Ht(\textbf{X} \vdashcustom f_M: \underline{A_M})\Ht(\partial \textbf{Y} \vdashcustom Y_n \tp) \le \Ht(\textbf{f}:\textbf{X} \rightarrow \textbf{A})\Ht(\textbf{Y}) = h,
$$
by $h$-derivability we can derive $(\textbf{X} \vdashcustom f_M:\underline{A_M}) \otimes (\partial \textbf{Y} \vdashcustom Y_n \tp)$, which is
$$
\textbf{X} \otimes \textbf{Y} \vdashcustom f_M \otimes y_n: (\underline{A_M} \otimes Y_n)[f_M \otimes y_j]_{j < n}.
$$
Thus it suffices to derive
$$
\partial(\textbf{X}' \otimes \textbf{Y}) \vdashcustom \doverline{A_M \otimes Y_n} \equiv \underline{A_M} \otimes Y_n \tp
$$
where $\textbf{X}' = (\textbf{X}, x':\underline{A_M})$. But that is obtained from $Sub^1_{s,s}(h)$ applied to the morphism $\partial \textbf{f}:\textbf{X} \rightarrow \partial \textbf{A}$ and the judgments $\partial\textbf{A} \vdashcustom A_M \tp$ and $\partial \textbf{Y} \vdashcustom Y_n \tp$, noting that
\begin{align*}
	&\Ht(\partial \textbf{A} \vdashcustom A_M \tp)\Ht(\partial \textbf{Y} \vdashcustom Y_n \tp) = \Ht(\textbf{A})\Ht(\textbf{Y}) \le h,\\
	&\Ht(\textbf{X} \vdashcustom \underline{A_M} \tp)\Ht(\partial \textbf{Y}) \le \Ht(\textbf{X} \vdashcustom f_M:\underline{A_M})\Ht(\partial \textbf{Y}) < \Ht(\textbf{f}:\textbf{X} \rightarrow \textbf{A})\Ht(\textbf{Y}) \le h,\\
	&\Ht(\partial \textbf{f}:\textbf{X} \rightarrow \partial \textbf{A})\Ht(\partial \textbf{Y} \vdashcustom Y_n \tp) \le \Ht(\textbf{f}:\textbf{X} \rightarrow \textbf{A})\Ht(\textbf{Y}) \le h.
\end{align*}
This concludes the proof of the claim for $L$.

\vspace{1em}
\colorbox{white!90!black}{Proof of $Mor^2(h)$}
\vspace{0.5em}

Write $\mathscr S$ for the set of all pairs consisting of a context $\textbf{X}$ in $\bbA$ and a context morphism $\textbf{g}:\textbf{Y} \rightarrow \textbf{B}$ in $\bbB$ such that $\Ht(\textbf{X})\Ht(\textbf{g}:\textbf{Y} \rightarrow \textbf{B}) = h$. Note that in this case, since $\Ht(\textbf{Y})$, $\Ht(\textbf{B}) \le \Ht(\textbf{g}:\textbf{Y} \rightarrow \textbf{B})$, by $Cont(h)$ we have that $\textbf{X} \otimes \textbf{Y}$ and $\textbf{X} \otimes \textbf{B}$ are contexts.

We will prove by induction on $L \ge 0$ that $\textbf{X} \otimes \textbf{g}$ is a morphism from $\textbf{X} \otimes \textbf{Y}$ to $\textbf{X} \otimes \textbf{B}$ whenever $l(\textbf{X})l(\textbf{B}) = L$.

The claim holds for $L = 0$ as in this case $\textbf{X} \otimes \textbf{B}$ is the empty context and $\textbf{X} \otimes \textbf{g}$ is the empty sequence. Now, let $L \ge 1$ and assume that the claim holds for $0$, ..., $L-1$. Suppose that $\Ht(\textbf{X})\Ht(\textbf{g}:\textbf{Y} \rightarrow \textbf{B}) \in \mathscr S$ satisfies $l(\textbf{X})l(\textbf{B}) = L$.

Since $\Ht(\partial \textbf{g}:\textbf{Y} \rightarrow \partial \textbf{B}) \le \Ht(\textbf{g}:\textbf{Y} \rightarrow \textbf{B})$ and $l(\partial \textbf{B}) < l(\textbf{B})$, by the induction hypothesis we have a morphism
\[
\tag{$\texttt{*}$}
\textbf{X} \otimes \partial \textbf{g}:\textbf{X} \otimes \textbf{Y} \longrightarrow \textbf{X} \otimes \partial \textbf{B}.
\]
Similarly, as $\Ht(\partial \textbf{X}) < \Ht(\textbf{X})$ we have a morphism $\partial \textbf{X} \otimes \textbf{g}: \partial \textbf{X} \otimes \textbf{Y} \longrightarrow \partial \textbf{X} \otimes \textbf{B}$, hence a morphism
\[
\tag{$\texttt{*}\texttt{*}$}
\partial \textbf{X} \otimes \textbf{g}: \textbf{X} \otimes \textbf{Y} \longrightarrow \partial \textbf{X} \otimes \textbf{B}.
\]
From ($\texttt{*}$) and ($\texttt{*}\texttt{*}$) we obtain a morphism
$$
\partial(\textbf{X} \otimes \textbf{g}):\textbf{X} \otimes \textbf{Y} \longrightarrow \partial(\textbf{X} \otimes \textbf{B}).
$$
To derive $\textbf{X} \otimes \textbf{g}:\textbf{X} \otimes \textbf{Y} \rightarrow \textbf{X} \otimes \textbf{B}$, it remains to derive the term judgment corresponding to its last entry, i.e.
$$
\textbf{X} \otimes \textbf{Y} \vdashcustom x_m \otimes g_N: \doverline{X_m \otimes B_N}
$$
where the upper double bar indicates the substitution operation $[x_i \otimes g_j \mid x_ib_j]_{(i,j) < (m,N)}$.

Using a lower bar to indicate the operation $[g_i \mid b_j]_{j < N}$, we have
$$
\Ht(\partial \textbf{X} \vdashcustom X_m \tp)\Ht(\textbf{Y} \vdashcustom g_N:\underline{B_N}) \le \Ht(\textbf{X})\Ht(\textbf{g}:\textbf{Y} \rightarrow \textbf{B}) = h,
$$
so by $h$-derivability we can derive $(\partial \textbf{X} \vdashcustom X_m \tp) \otimes (\textbf{Y} \vdashcustom g_N:\underline{B_N})$, which is
$$
\textbf{X} \otimes \textbf{Y} \vdashcustom x_m \otimes g_N: (X_m \otimes \underline{B_N})[x_i \otimes g_N]_{i < m}.
$$
Thus it suffices to derive
$$
\partial (\textbf{X} \otimes \textbf{Y}') \vdashcustom \doverline{X_m \otimes B_N} \equiv X_m \otimes \underline{B_N} \tp
$$
where $\textbf{Y}' = (\textbf{Y}, y':\underline{B_N})$. But that is obtained from $Sub^2_{s,s}(h)$ applied to the morphism $\partial \textbf{g}:\textbf{Y} \rightarrow \partial \textbf{B}$ and the judgments $\partial \textbf{X} \vdashcustom X_m \tp$ and $\partial \textbf{B} \vdashcustom B_N \tp$, noting that
\begin{align*}
	&\Ht(\partial \textbf{X} \vdashcustom X_m \tp)\Ht(\partial \textbf{B} \vdashcustom B_N \tp) = \Ht(\textbf{X})\Ht(\textbf{B}) \le h, \\
	&\Ht(\partial \textbf{X})\Ht(\textbf{Y} \vdashcustom \underline{B_N} \tp) \le \Ht(\partial \textbf{X})\Ht(\textbf{Y} \vdashcustom g_N:\underline{B_N}) < \Ht(\textbf{X})\Ht(\textbf{g}:\textbf{Y} \rightarrow \textbf{B}) \le h, \\
	&\Ht(\partial \textbf{X} \vdashcustom X_m \tp)\Ht(\partial \textbf{g}:\textbf{Y} \rightarrow \partial \textbf{B}) \le \Ht(\textbf{X})\Ht(\textbf{g}:\textbf{Y} \rightarrow \textbf{B}) \le h.
\end{align*}
This concludes the proof of the claim for $L$.

\subsection{Context equalities and context morphism equalities}

We will now consider the following statements:

\begin{itemize}	
	\item[\colorbox{white!90!black}{$Conteq^1(h)$}] Suppose given in $\bbA$ a derivable judgment $\textbf{X} \vdashcustom U \equiv U' \tp$, and in $\bbB$ a context $\textbf{Y}$. Let $\textbf{X}' = (\textbf{X}, x:U)$, $\textbf{X}'' = (\textbf{X}, x:U')$.
	
	If $\Ht(\textbf{X} \vdashcustom U \equiv U' \tp)\Ht(\textbf{Y}) \le h$, then $\textbf{X}' \otimes \textbf{Y} \equiv \textbf{X}'' \otimes \textbf{Y} \ctx$ is derivable.
	
	\item[\colorbox{white!90!black}{$Conteq^2(h)$}] Suppose given in $\bbA$ a context $\textbf{X}$, and in $\bbB$ a derivable judgment $\textbf{Y} \vdashcustom V \equiv V' \tp$. Let $\textbf{Y}' = (\textbf{Y}, y:V)$, $\textbf{Y}'' = (\textbf{Y}, y:V')$.
	
	If $\Ht(\textbf{X})\Ht(\textbf{Y} \vdashcustom V \equiv V' \tp) \le h$, then $\textbf{X} \otimes \textbf{Y}' \equiv \textbf{X} \otimes \textbf{Y}'' \ctx$ is derivable.
	
	\item[\colorbox{white!90!black}{$Moreq^1(h)$}] Suppose given in $\bbA$ a derivable morphism equality $\textbf{f} \equiv \textbf{g}:\textbf{X} \rightarrow \textbf{A}$, and in $\bbB$ a context $\textbf{Y}$.
	
	If $\Ht(\textbf{f} \equiv \textbf{g}:\textbf{X} \rightarrow \textbf{A})\Ht(\textbf{Y}) \le h$, then $\textbf{f} \otimes \textbf{Y} \equiv \textbf{g} \otimes \textbf{Y}:\textbf{X} \otimes \textbf{Y} \rightarrow \textbf{A} \otimes \textbf{Y}$ is derivable.
	
	\item[\colorbox{white!90!black}{$Moreq^1(h)$}] Suppose given in $\bbA$ a context $\textbf{X}$, and in $\bbB$ a derivable morphism equality $\textbf{f} \equiv \textbf{g}:\textbf{Y} \rightarrow \textbf{B}$.
	
	If $\Ht(\textbf{X})\Ht(\textbf{f} \equiv \textbf{g}:\textbf{Y} \rightarrow \textbf{B}) \le h$, then $\textbf{X} \otimes \textbf{f} \equiv \textbf{X} \otimes \textbf{g}:\textbf{X} \otimes \textbf{Y} \rightarrow \textbf{X} \otimes \textbf{B}$ is derivable.
\end{itemize}

\begin{lemma}
If $(\bbA,\bbB)$ is $h$-derivable, then it satisfies $Conteq^1(h)$ and $Conteq^2(h)$.
\end{lemma}

\begin{proof}
We will only verify $Conteq^1(h)$, the proof of $Conteq^2(h)$ being analogous.

Suppose given a derivable judgment $\textbf{X} \vdashcustom U \equiv U' \tp$ in $\bbA$ and a context $\textbf{Y}$ in $\bbB$ such that $\Ht(\textbf{X} \vdashcustom U \equiv U' \tp)\Ht(\textbf{Y}) \le h$. Let us prove that, considering contexts $\textbf{X}' = (\textbf{X}, x:U)$ and $\textbf{X}'' = (\textbf{X}, x:U')$, the judgment $\textbf{X}' \otimes \textbf{Y} \equiv \textbf{X}'' \otimes \textbf{Y} \ctx$ is derivable.

If $\textbf{Y}$ is the empty context, the claim is trivial. Otherwise, expressing $\textbf{Y}$ as $(\partial \textbf{Y}, y:V)$, by $h$-derivability we can tensor $\textbf{X} \vdashcustom U \equiv U' \tp$ and $\partial\textbf{Y} \vdashcustom V \tp$ to derive
$$
\partial(\textbf{X}' \otimes \textbf{Y}) \vdashcustom U \otimes V \equiv U' \otimes V \tp.
$$
Thus to obtain the desired context equality it remains to derive $\partial(\textbf{X}' \otimes \textbf{Y}) \equiv \partial(\textbf{X}'' \otimes \textbf{Y})$. But as $\Ht(\partial \textbf{Y}) < \Ht(\textbf{Y})$, this follows by using $Conteq^1(h-1)$ to derive $\textbf{X}' \otimes \partial \textbf{Y} \equiv \textbf{X}'' \otimes \partial \textbf{Y} \ctx$.
\end{proof}

\begin{lemma}
If $(\bbA,\bbB)$ is $h$-derivable, then it satisfies $Moreq^1(h)$ and $Moreq^2(h)$.
\end{lemma}

\begin{proof}
We will only check $Moreq^1(h)$, the proof of $Moreq^2(h)$ being analogous.

Suppose given a derivable equality $\textbf{f} \equiv \textbf{g}:\textbf{X} \rightarrow \textbf{A}$ in $\bbA$ and a context $\textbf{Y}$ in $\bbB$ such that $\Ht(\textbf{f} \equiv \textbf{g}:\textbf{X} \rightarrow \textbf{A})\Ht(\textbf{Y}) \le h$. Write $\textbf{Y} = (y_1:Y_1, ..., y_n:Y_n)$, $\textbf{f} = (f_1, ..., f_M)$, and $\textbf{g} = (g_1, ..., g_M)$. To derive $\textbf{f} \otimes \textbf{Y} \equiv \textbf{g} \otimes \textbf{Y}:\textbf{X} \otimes \textbf{Y} \rightarrow \textbf{A} \otimes \textbf{Y}$, it suffices to derive
$$
\textbf{X} \otimes \textbf{Y} \vdashcustom f_i \otimes y_j \equiv g_i \otimes y_j \tm
$$
for $1 \le i \le M$ and $1 \le j \le n$. But this follows from $h$-derivability by tensoring $\textbf{X} \vdashcustom f_i \equiv g_i:U$ and $\partial_{j-1} \textbf{Y} \vdashcustom Y_j \tp$ where the former is chosen so that $\Ht(\textbf{X} \vdashcustom f_i \equiv g_i:U) \le \Ht(\textbf{f} \equiv \textbf{g}:\textbf{X} \rightarrow \textbf{A})$.
\end{proof}

\subsection{Tensoring by a section of a display map}

We again work under the assumption that $(\bbA,\bbB)$ is $h$-derivable. Suppose given contexts $\textbf{X} = (x_1:X_1, ..., x_m:X_m)$ in $\bbA$ and $\textbf{Y} = (y_1:Y_1, ..., y_n:Y_n)$ in $\bbB$.

\begin{lemma}
\label{lem: tensoring a term and a context}
Consider a derivable judgment $\textbf{X} \vdashcustom u:U$ and let $\textbf{X}' = (\textbf{X}, x:U)$. If $\Ht(\textbf{X} \vdashcustom u:U)\Ht(\textbf{Y}) \le h$, then we have a context morphism
$$
(x_1, ..., x_m, u^U) \otimes \textbf{Y}:\textbf{X} \otimes \textbf{Y} \longrightarrow \textbf{X}' \otimes \textbf{Y}.
$$
Similarly, consider a derivable judgment $\textbf{Y} \vdashcustom v:V$ and let $\textbf{Y}' = (\textbf{Y}, y:V)$. If $\Ht(\textbf{X})\Ht(\textbf{Y} \vdashcustom v:V) \le h$, then we have a context morphism
$$
\textbf{X} \otimes (y_1, ..., y_n,v^V):\textbf{X} \otimes \textbf{Y} \longrightarrow \textbf{X} \otimes \textbf{Y}'.
$$
\end{lemma}

\begin{proposition}
\label{prop: term from pair of terms}
Suppose given derivable judgments $\textbf{X} \vdashcustom u:U$ and $\textbf{Y} \vdashcustom v:V$ such that $\Ht(\textbf{X} \vdashcustom u:U)\Ht(\textbf{Y} \vdashcustom v:V) \le h+1$. Then
$$
\textbf{X} \otimes \textbf{Y} \vdashcustom u^U \otimes v^V : \doverline{U \otimes V}, \qquad \textbf{X} \otimes \textbf{Y} \vdashcustom u^U \varotimes v^V : \doverline{U \otimes V}
$$
are derivable, where the double bar indicates the substitution operation
$$
[u^U \otimes y_1 \mid xy_1, ..., u^U \otimes y_n \mid xy_n, x_1 \otimes v^V \mid x_1y, ..., x_m \otimes v^V \mid x_my]
$$
for $x$ (resp. $y$) a variable of sort $U$ (resp. $V$).
\end{proposition}

\begin{proof}
Firstly, note that $\Ht(\textbf{X} \vdashcustom U \tp)\Ht(\textbf{Y} \vdashcustom v:V) < h+1$, so by $h$-derivability we can tensor $\textbf{X} \vdashcustom U \tp$ and $\textbf{Y} \vdashcustom v:V$ to derive
\[
\tag{\texttt{*}}
\textbf{X}' \otimes \textbf{Y} \vdashcustom x \otimes v^V:(U \otimes V)[x_1 \otimes v^V \mid x_1y, ..., x_m \otimes v^V \mid x_my]
\]
where $\textbf{X}' = (\textbf{X}, x:U)$.

Also, by Lemma \ref{lem: tensoring a term and a context} we have a context morphism $(x_1, ..., x_m,u^U) \otimes \textbf{Y}:\textbf{X} \otimes \textbf{Y} \rightarrow \textbf{X}' \otimes \textbf{Y}$. By substitution along this map, we obtain from (\texttt{*}) a derivable judgment
\begin{align*}
	\textbf{X} \otimes \textbf{Y} & \vdashcustom (x \otimes v^V)[u^U \otimes y_1 \mid xy_1, ..., u^U \otimes y_n \mid xy_n] \\
	&: (U \otimes V)[x_1 \otimes v^V \mid x_1y, ..., x_m \otimes v^V \mid x_my][u^U \otimes y_1 \mid xy_1, ..., u^U \otimes y_n \mid xy_n],
\end{align*}
which is precisely $\textbf{X} \otimes \textbf{Y} \vdashcustom u^U \otimes v^V : \doverline{U \otimes V}$.

Similarly, we can tensor $\textbf{X} \vdashcustom u:U$ and $\textbf{Y} \vdashcustom V \tp$ to derive
\[
\tag{\texttt{**}}
\textbf{X} \otimes \textbf{Y}' \vdashcustom u^U \otimes y : (U \otimes V)[u^U \otimes y_1 \mid xy_1, ..., u^U \otimes y_n \mid xy_n].
\]
where $\textbf{Y}' = (\textbf{Y}, y:V)$. From Lemma \ref{lem: tensoring a term and a context} we obtain a context morphism $\textbf{X} \otimes (y_1, ..., y_n, v^V): \textbf{X} \otimes \textbf{Y} \rightarrow \textbf{X} \otimes \textbf{Y}'$, hence by substitution in (\texttt{**}) we can derive
\begin{align*}
	\textbf{X} \otimes \textbf{Y} & \vdashcustom (u^U \otimes y)[x_1 \otimes v^V \mid x_1y, ..., x_m \otimes v^V \mid x_my] \\
	&: (U \otimes V)[u^U \otimes y_1 \mid xy_1, ..., u^U \otimes y_n \mid xy_n][x_1 \otimes v^V \mid x_1y, ..., x_m \otimes v^V \mid x_my],
\end{align*}
which equals $\textbf{X} \otimes \textbf{Y} \vdashcustom u^U \varotimes v^V: \doverline{U \otimes V}$.
\end{proof}

\begin{lemma}
\label{lem: morphism equality from (term eq,ctx)}
Suppose given a derivable judgment $\textbf{X} \vdashcustom u \equiv v:U$ in $\bbA$ such that $\Ht(\textbf{X} \vdashcustom u \equiv v:U)\Ht(\textbf{Y}) \le h$. Then, writing $\textbf{X}' = (\textbf{X}, x':U)$, the context morphism equality
$$
(x_1, ..., x_m,u^U) \otimes \textbf{Y} \equiv (x_1, ..., x_m,v^U) \otimes \textbf{Y}: \textbf{X} \otimes \textbf{Y} \longrightarrow \textbf{X}' \otimes \textbf{Y}
$$
is derivable.
\end{lemma}

\begin{proof}
As $\Ht(\textbf{X} \vdashcustom u:U)$, $\Ht(\textbf{X} \vdashcustom v:U) \le \Ht(\textbf{X} \vdashcustom u \equiv v:U)$, by Lemma \ref{lem: tensoring a term and a context} the two premorphisms in the statement are derivable. To obtain the desired morphism equality, it suffices to derive
$$
\textbf{X} \otimes \textbf{Y} \vdashcustom u^U \otimes y_j \equiv v^U \otimes y_j \tm
$$
for $1 \le j \le n$. For that, it suffices to derive $\textbf{X} \otimes \partial_j \textbf{Y} \vdashcustom u^U \otimes y_j \equiv v^V \otimes y_j \tm$, which is accomplished by tensoring $\textbf{X} \vdashcustom u \equiv v:U$ and $\partial_{j-1}\textbf{Y} \vdashcustom Y_j \tp$ using $h$-derivability and the fact that $\Ht(\partial_{j-1}\textbf{Y} \vdashcustom Y_j \tp) \le \Ht(\textbf{Y})$.
\end{proof}

\begin{lemma}
\label{lem: morphism equality from (ctx, term eq)}
Suppose given a derivable judgment $\textbf{Y} \vdashcustom u \equiv v:V$ in $\bbB$ such that $\Ht(\textbf{X})\Ht(\textbf{Y} \vdashcustom u \equiv v:V) \le h$. Then, writing $\textbf{Y}' = (\textbf{Y}, y':V)$, the context morphism equality
$$
\textbf{X} \otimes (y_1, ..., y_n,u) \equiv \textbf{X} \otimes (y_1, ..., y_n,v): \textbf{X} \otimes \textbf{Y} \longrightarrow \textbf{X} \otimes \textbf{Y}'
$$
is derivable.
\end{lemma}

\begin{proof}
Analogous to Lemma \ref{lem: morphism equality from (term eq,ctx)}.
\end{proof}

\section{$h$-derivability implies $(h+1)$-derivability}

This section is devoted to proving the following statement:

\begin{proposition}
\label{prop: induction step}
Given $h \ge 0$, suppose that a pair of generalized algebraic theories $(\bbA, \bbB)$ is $h$-derivable. Then it is $(h+1)$-derivable.
\end{proposition}

It will then be used as the induction step in a proof that the tensor product of any two \textsc{gat}s is a theory; see Theorem \ref{th: tensor product is a theory}.

\vspace{0.5em}

In what follows we let $J$ and $J'$ be derivable standard judgments\footnote{By a \emph{standard judgment} we mean a sort, term, sort equality or term equality judgment.} in $\bbA$ and $\bbB$, respectively, that satisfy $\Ht(J)\Ht(J') = h+1$. Our goal is to prove that $J \odot J'$ is derivable. This will be done in several steps, according to the kinds of judgment being considered and the rules used to derive them. Note that we can apply to $(\bbA,\bbB)$ the results of the previous section, that is, $Cont(k)$, $Sub^1_{t,t}(k)$, etc., whenever $k \le h$.

Below, we always write $\textbf{X} = (x_1:X_1,...,x_m:X_m)$ for the context of $J$ and $\textbf{Y} = (y_1:Y_1, ..., y_n:Y_n)$ for that of $J'$. We will repeatedly use the fact that $\textbf{X} \otimes \textbf{Y}$ is a context, which follows from $Cont(h)$ and $\Ht(\textbf{X})\Ht(\textbf{Y}) < \Ht(J)\Ht(J') = h+1$.

\vspace{0.5em}

As in the previous section, we use Proposition \ref{prop: properties height}, without explicit reference, as a source of all properties of the height function that will be needed.

\begin{center}
	\textcolor{newpurple}{\textbf{\normalsize{\underline{sort $\odot$ sort}}}}
\end{center}

\vspace{0.2em}
	
Suppose that $J$ is $\textbf{X} \vdashcustom X' \tp$ and $J'$ is $\textbf{Y} \vdashcustom Y' \tp$. Let $\textbf{X}' = (\textbf{X},x': X')$ and $\textbf{Y}' = (\textbf{Y},y':Y')$. Then $J \odot J'$ is
$$
\partial(\textbf{X}' \otimes \textbf{Y}') \vdashcustom X' \otimes Y' \tp.
$$
As $\Ht(\textbf{X}) < \Ht(\textbf{X}')$ and $\Ht(\textbf{Y}) < \Ht(\textbf{Y}')$, we have that $\Ht(\textbf{X}')\Ht(\textbf{Y})$ and $\Ht(\textbf{X})\Ht(\textbf{Y}')$ are at most $h$. It follows that $\textbf{X}' \otimes \textbf{Y}$ and $\textbf{X} \otimes \textbf{Y}'$, thus also $\partial(\textbf{X}' \otimes \textbf{Y}')$, are contexts.
	
We consider the following (not mutually exclusive) cases:
	
	\begin{enumerate}[label=(\arabic*)]
		\item Both $J$ and $J'$ are  axioms, say $\textbf{X} \vdashcustom S(x_1, ..., x_m) \tp$ and $\textbf{Y} \vdashcustom T(y_1, ..., y_n) \tp$, respectively. Then $J \odot J'$ is the axiom
		$$
		\partial(\textbf{X}' \otimes \textbf{Y}') \vdashcustom ST
		\begin{pmatrix}
			x_1y_1 & \cdots & x_1y_n & x_1y'\\
			\vdots & \ddots & \vdots & \vdots \\
			x_my_1 & \cdots & x_my_n & x_my'\\
			x'y_1 & \cdots & x'y_m & -
		\end{pmatrix}
		\tp.
		$$
		which is derivable as it is an axiom in $\bbA \otimes \bbB$ and $\partial(\textbf{X}' \otimes \textbf{Y}')$ is a context.
		
		\item $J$ has an initial inference of the form (T-sub). Then there exist an axiom $\textbf{A} \vdashcustom U \tp$ and a morphism $\textbf{f}:\textbf{X} \rightarrow \textbf{A}$ such that $X' = U[\textbf{f}]$ and $\Ht(\textbf{A} \vdashcustom U \tp)$, $\Ht(\textbf{f}:\textbf{X} \rightarrow \textbf{A}) < \Ht(\textbf{X} \vdashcustom X' \tp)$. Let $\textbf{A}'$ be a context $(\textbf{A}, a:U)$.
		
		Note that $\Ht(\textbf{Y}) < \Ht(\textbf{Y} \vdashcustom Y' \tp)$, so
		\begin{align*}
			&\Ht(\textbf{A} \vdashcustom U \tp)\Ht(\textbf{Y} \vdashcustom Y' \tp), \\
			&\Ht(\textbf{X} \vdashcustom X' \tp)\Ht(\textbf{Y}), \\
			&\Ht(\textbf{f}:\textbf{X} \rightarrow \textbf{A})\Ht(\textbf{Y} \vdashcustom Y' \tp)
		\end{align*}
		are strictly smaller than $\Ht(\textbf{X} \vdashcustom X' \tp)\Ht(\textbf{Y} \vdashcustom Y' \tp) = h+1$. This allows us to apply $Sub^1_{s,s}(h)$ to derive $\partial(\textbf{X}' \otimes \textbf{Y}') \vdashcustom X' \otimes Y' \tp$.
		
		\item $J'$ has an initial inference of the form (T-sub). Then there exist an axiom $\textbf{B} \vdashcustom V \tp$ and a morphism $\textbf{g}:\textbf{Y} \rightarrow \textbf{B}$ such that $Y' = V[\textbf{g}]$ and $\Ht(\textbf{B} \vdashcustom V \tp)$, $\Ht(\textbf{g}:\textbf{Y} \rightarrow \textbf{B}) < \Ht(\textbf{Y} \vdashcustom Y' \tp)$. Let $\textbf{B}'$ be a context $(\textbf{B}, b:V)$.
		
		Note that $\Ht(\textbf{X}) < \Ht(\textbf{X} \vdashcustom X' \tp)$, so
		\begin{align*}
			&\Ht(\textbf{X} \vdashcustom X' \tp)\Ht(\textbf{B} \vdashcustom V \tp), \\
			&\Ht(\textbf{X})\Ht(\textbf{Y} \vdashcustom Y' \tp), \\
			&\Ht(\textbf{X} \vdashcustom X' \tp)\Ht(\textbf{g}:\textbf{Y} \rightarrow \textbf{B})
		\end{align*}
		are strictly smaller than $\Ht(\textbf{X} \vdashcustom X' \tp)\Ht(\textbf{Y} \vdashcustom Y' \tp) = h+1$. Thus we can apply $Sub^2_{s,s}(h)$ to derive $\partial(\textbf{X}' \otimes \textbf{Y}') \vdashcustom X' \otimes Y' \tp$.
	\end{enumerate}
	
	\vspace{0.5em}
	
	\begin{center}
		\textcolor{newpurple}{\textbf{\normalsize{\underline{sort $\odot$ term}}}}
	\end{center}
	
	\vspace{0.2em}
	
	Suppose that $J$ is $\textbf{X} \vdashcustom X' \tp$ and $J'$ is $\textbf{Y} \vdashcustom v:Y'$. Taking contexts $\textbf{X}' = (\textbf{X},x':X')$ and $\textbf{Y}' = (\textbf{Y},y':Y')$, recall that $J \odot J'$ is
	$$
	\textbf{X}' \otimes \textbf{Y} \vdashcustom x' \otimes v: (X' \otimes Y')[x_1 \otimes v \mid x_1y', ..., x_m \otimes v \mid x_my'].
	$$
	
	We have the following cases:
	
	\begin{enumerate}[label=(\arabic*)]
	\item $J$ and $J'$ are axioms, say $\textbf{X} \vdashcustom S(x_1, ..., x_m) \tp$ and $\textbf{Y} \vdashcustom t(y_1, ..., y_n):Y'$, respectively. Then $J \odot J'$ is the axiom
	$$
	\textbf{X}' \otimes \textbf{Y} \vdashcustom St
	\begin{pmatrix}
		x_1y_1 & \cdots & x_1y_n \\
		\vdots & \ddots & \vdots \\
		x_my_1 & \cdots & x_my_n \\
		x'y_1 & \cdots & x'y_n
	\end{pmatrix}
	: (X' \otimes Y')[x_1 \otimes v \mid x_1y', ..., x_m \otimes v \mid x_my'].
	$$
	For it to be derivable, it must be well-formed, i.e.
	\[
	\tag{\texttt{*}}
	\textbf{X}' \otimes \textbf{Y} \vdashcustom (X' \otimes Y')[x_1 \otimes v \mid x_1y', ..., x_m \otimes v \mid x_my'] \tp
	\]
	must be derivable. To check that this is the case, we start by tensoring $\textbf{X} \vdashcustom X' \tp$ and $\textbf{Y} \vdashcustom Y' \tp$ to derive
	\[
	\tag{\texttt{**}}
	\partial(\textbf{X}' \otimes \textbf{Y}') \vdashcustom X' \otimes Y' \tp.
	\]
	On the other hand, as $\Ht(\textbf{X})\Ht(\textbf{Y} \vdashcustom v:Y') \le h$, by Lemma \ref{lem: tensoring a term and a context} we have a morphism $\textbf{X} \otimes (y_1, ..., y_n, v):\textbf{X} \otimes \textbf{Y} \rightarrow \textbf{X} \otimes \textbf{Y}'$. By amalgamating it with the identity morphism of $\textbf{X}'$ we derive
	$$
	\partial(\textbf{X}' \otimes (y_1, ..., y_n,v)):\textbf{X}' \otimes \textbf{Y} \longrightarrow \partial(\textbf{X}' \otimes \textbf{Y}').
	$$
	Now, substitution of (\texttt{**}) along the latter morphism yields the desired judgment (\texttt{*}).
	
	\item $J$ has an initial inference of the form (T-sub). Then there exist an axiom $\textbf{A} \vdashcustom U \tp$ and a morphism $\textbf{f} = (f_1, ..., f_M):\textbf{X} \rightarrow \textbf{A}$ such that $X' = U[\textbf{f}]$ and $\Ht(\textbf{A} \vdashcustom U \tp)$, $\Ht(\textbf{f}:\textbf{X} \rightarrow \textbf{A}) < \Ht(\textbf{X} \vdashcustom X' \tp)$. Let $\textbf{A}'$ be a context $(\textbf{A},a:U)$. Note that
	$$
	\Ht(\textbf{f}:\textbf{X} \rightarrow \textbf{A})\Ht(\textbf{Y} \vdashcustom v \tm), \qquad \Ht(\textbf{X} \vdashcustom X' \tp)\Ht(\textbf{Y}), \qquad \Ht(\textbf{A} \vdashcustom U \tp)\Ht(\textbf{Y} \vdashcustom v \tm)
	$$
	are strictly smaller than $\Ht(\textbf{X} \vdashcustom X' \tp)\Ht(\textbf{Y} \vdashcustom v:Y') = h+1$. Thus we can use $Sub^1_{s,t}(h)$ to derive
	$$
	\textbf{X}' \otimes \textbf{Y} \vdashcustom \doverline{a \otimes v} \equiv x' \otimes v \tm,
	$$
	where the double bar indicates substitution along $(f_1, ..., f_M,x') \otimes \textbf{Y}:\textbf{X}' \otimes \textbf{Y} \rightarrow \textbf{A}' \otimes \textbf{Y}$; the latter being a morphism follows from $Mor^1_+(h)$.
	
	It remains to show that $x' \otimes v$ is of the desired sort, that is,
	\[
	\tag{$\varheartsuit$}
	(X' \otimes Y')[x_1 \otimes v \mid x_1y', ..., x_m \otimes v \mid x_my'].
	\]
	We will do this by checking that $\doverline{a \otimes v}$ is provably of sort ($\varheartsuit$) in context $\textbf{X}' \otimes \textbf{Y}$.
	
	Writing $U = S(a_1, ..., a_M)$ and $Y' = T(\tau_1, ..., \tau_L)$, the matrix form of ($\varheartsuit$) is
	$$
	ST
	\begin{pmatrix}
		f_1 \otimes \tau_1 & \cdots & f_1 \otimes \tau_L & f_1 \varotimes v^{Y'}\\
		\vdots & \ddots & \vdots & \vdots\\
		f_M \otimes \tau_1 & \cdots & f_M \otimes \tau_L & f_M \varotimes v^{Y'}\\
		x' \otimes \tau_1 & \cdots & x' \otimes \tau_L & -
	\end{pmatrix}.
	$$
	
	On the other hand, let us study the sort of $\doverline{a \otimes v}$. By $h$-derivability, $\textbf{A}' \otimes \textbf{Y} \vdashcustom a \otimes v: (U \otimes Y')[a_1 \otimes v \mid a_1 y', ..., a_M \otimes v \mid a_M y']$ is derivable, so we can also derive
	$$
	\textbf{X}' \otimes \textbf{Y} \vdashcustom \doverline{a \otimes v}: \doverline{(U \otimes Y')[a_1 \otimes v \mid a_1 y', ..., a_M \otimes v \mid a_M y']}.
	$$
	We can express $(U \otimes Y')[a_1 \otimes v \mid a_1 y', ..., a_M \otimes v \mid a_M y']$ as
	$$
	ST
	\begin{pmatrix}
		a_1 \otimes \tau_1 & \cdots & a_1 \otimes \tau_L & a_1 \otimes v\\
		\vdots & \ddots & \vdots & \vdots\\
		a_M \otimes \tau_1 & \cdots & a_M \otimes \tau_L & a_M \otimes v\\
		a \otimes \tau_1 & \cdots & a \otimes \tau_L & -
	\end{pmatrix}.
	$$
	Now, observe that
	\begin{itemize}
		\item For $1 \le i \le M$ and $w$ among $\tau_1$, ..., $\tau_L$, we have $\doverline{a_i \otimes w} = (a_i \otimes w)[\textbf{f} \otimes \textbf{Y}]$, which is provably equal in context $\textbf{X} \otimes \textbf{Y}$ to $f_i \otimes w$. This follows from $Sub^1_{t,t}(h)$ using that
		\begin{align*}
			&\Ht(\textbf{f}:\textbf{X} \rightarrow \textbf{A})\Ht(\textbf{Y} \vdashcustom w \tm) \le \Ht(\textbf{f}:\textbf{X} \rightarrow \textbf{A})\Ht(\textbf{Y} \vdashcustom v:Y') \le h, \\
			&\Ht(\textbf{A} \vdashcustom a_i \tm)\Ht(\textbf{Y} \vdashcustom w \tm) < \Ht(\textbf{A} \vdashcustom U \tp)\Ht(\textbf{Y} \vdashcustom v:Y') \le h.
		\end{align*}
		
		Also, by tensoring $\textbf{X} \vdashcustom f_i:\Type(f_i)$ (which has height smaller than that of $\textbf{f}:\textbf{X} \rightarrow \textbf{A}$) and $\textbf{Y} \vdashcustom v:Y'$ we obtain that $f_i \otimes v$ is provably equal to $f_i \varotimes v^V$.
		
		\item For $1 \le j \le L$, we can use $Sub^1_{s,s}(h)$ to derive $\textbf{X}' \otimes \textbf{Y} \vdashcustom \doverline{a \otimes \tau_j} \equiv x' \otimes \tau_j \tm$.
	\end{itemize}
	
	By applying these modifications to the matrix form of $(U \otimes Y')[a_1 \otimes v \mid a_1 y', ..., a_M \otimes v \mid a_M y']$, we conclude that $\doverline{(U \otimes Y')[a_1 \otimes v \mid a_1 y', ..., a_M \otimes v \mid a_M y']}$ is provably equal to ($\varheartsuit$) in $\textbf{X}' \otimes \textbf{Y}$, as required.\\
	
	\item $J'$ has an initial inference
	$$
	\inferrule{\textbf{Y} \vdashcustom V \equiv Y' \tp \\ \textbf{Y} \vdashcustom v:V}{\textbf{Y} \vdashcustom v:Y'}\text{(Teq/t)}.
	$$
	In this case, by $h$-derivability we can tensor $\textbf{X} \vdashcustom X' \tp$ and $\textbf{Y} \vdashcustom v:V \tp$ to obtain
	$$
	\textbf{X}' \otimes \textbf{Y} \vdashcustom x' \otimes v: (X' \otimes V)[x_1 \otimes v \mid x_1y', ..., x_m \otimes v \mid x_my'].
	$$
	Hence it suffices to derive
	$$
	\textbf{X}' \otimes \textbf{Y} \vdashcustom (X' \otimes Y')[x_1 \otimes v \mid x_1y', ..., x_m \otimes v \mid x_my'] \equiv (X' \otimes V)[x_1 \otimes v \mid x_1y', ..., x_m \otimes v \mid x_my'] \tp.
	$$
	For that, by tensoring $\textbf{X} \vdashcustom X' \tp$ and $\textbf{Y} \vdashcustom V \equiv Y' \tp$ we derive $\partial(\textbf{X}' \otimes \textbf{Y}'') \vdashcustom X' \otimes V \equiv X' \otimes Y' \tp$ where $\textbf{Y}'' = (\textbf{Y}, y':V)$. But by $Conteq^2(h)$ we obtain $\textbf{X}' \otimes \textbf{Y}' \equiv \textbf{X}' \otimes \textbf{Y}'' \ctx$, thus also
	\[
	\tag{\texttt{*}}
	\partial(\textbf{X}' \otimes \textbf{Y}') \vdashcustom X' \otimes V \equiv X' \otimes Y' \tp.
	\]
	Finally, by Lemma \ref{lem: tensoring a term and a context} (using that $\Ht(\textbf{X}')\Ht(\textbf{Y} \vdashcustom v:V) \le h$) we can derive $\textbf{X}' \otimes (y_1, ..., y_n,v):\textbf{X}' \otimes \textbf{Y} \rightarrow \textbf{X}' \otimes \textbf{Y}'$, and substitution of (\texttt{*}) along this morphism yields the desired sort equality.
	
	\item There exists $j \in \{1, ..., n\}$ such that $Y' = Y_j$, $v = y_j$, and $J'$ has an initial inference
	$$
	\inferrule{\textbf{Y} \vdashcustom Y_j \tp}{\textbf{Y} \vdashcustom y_j:Y_j}\text{(var)}.
	$$	
	By $Cont(h)$ we have that $\textbf{X}' \otimes \textbf{Y}$ is a context. But $(X' \otimes Y_j)[x_1y_j \mid x_1y', ..., x_my_j \mid x_my']$ is, by definition, the sort of the variable $x'y_j$ in $\textbf{X}' \otimes \textbf{Y}$. An application of (var) then yields $J \odot J'$.
	
	\item $J'$ has an initial inference (t-sub): there exist an axiom $\textbf{B} \vdashcustom w:W$, where $\textbf{B} = (b_1:B_1, ..., b_N:B_N)$, and a morphism $\textbf{g}:\textbf{Y} \rightarrow \textbf{B}$ such that $Y' = W[\textbf{g}]$, $v = w[\textbf{g}]$, and $\Ht(\textbf{B} \vdashcustom w:W)$, $\Ht(\textbf{g}:\textbf{Y} \rightarrow \textbf{B}) < \Ht(\textbf{Y} \vdashcustom v:Y')$.
	
	Taking $\textbf{B}' = (\textbf{B}, b':W)$, we can tensor $\textbf{X} \vdashcustom X' \tp$ and $\textbf{B} \vdashcustom w:W$ to derive
	$$
	\textbf{X}' \otimes \textbf{B} \vdashcustom x' \otimes w: (X' \otimes W)[x_1 \otimes w \mid x_1b', ..., x_m \otimes w \mid x_mb'].
	$$
	As $\Ht(\textbf{X}')\Ht(\textbf{g}:\textbf{Y} \rightarrow \textbf{B}) \le h$, by $Mor^2(h)$ we have a morphism $\textbf{X}' \otimes \textbf{g}: \textbf{X}' \otimes \textbf{Y} \rightarrow \textbf{X}' \otimes \textbf{B}$. Also, by $Sub^2_{s,t}(h)$ we obtain
	$$
	\textbf{X}' \otimes \textbf{Y} \vdashcustom (x' \otimes w)[\textbf{X}' \otimes \textbf{g}] \equiv x' \otimes v \tm,
	$$
	so the claim will follow if we derive
	\[
	\tag{$\varheartsuit$}
	(X' \otimes W)[x_1 \otimes w \mid x_1b', ..., x_m \otimes w \mid x_mb'][\textbf{X}' \otimes \textbf{g}] \equiv (X' \otimes Y')[x_1 \otimes v \mid x_1y', ..., x_m \otimes v \mid x_my']
	\]
	in context $\textbf{X}' \otimes \textbf{Y}$.
	
	Writing $X' = S(\sigma_1, ..., \sigma_K)$ and $W = T(\tau_1, ..., \tau_L)$, we have
	$$
	(X' \otimes W)[x_1 \otimes w \mid x_1b', ..., x_m \otimes w \mid x_mb'] = ST
	\begin{pmatrix}
		\sigma_1 \otimes \tau_1 & \cdots & \sigma_1 \otimes \tau_L & \sigma_1 \varotimes w\\
		\vdots & \ddots & \vdots & \vdots \\
		\sigma_K \otimes \tau_1 & \cdots & \sigma_K \otimes \tau_L & \sigma_K \varotimes w\\
		x' \otimes \tau_1 & \cdots & x' \otimes \tau_L & -
	\end{pmatrix}.
	$$
	On the other hand,
	$$
	(X' \otimes Y')[x_1 \otimes v \mid x_1y', ..., x_m \otimes v \mid x_my'] = ST
	\begin{pmatrix}
		\sigma_1 \otimes \tau_1[\textbf{g}] & \cdots & \sigma_1 \otimes \tau_L[\textbf{g}] & \sigma_1 \varotimes v\\
		\vdots & \ddots & \vdots & \vdots \\
		\sigma_K \otimes \tau_1[\textbf{g}] & \cdots & \sigma_K \otimes \tau_L[\textbf{g}] & \sigma_K \varotimes v\\
		x' \otimes \tau_1[\textbf{g}] & \cdots & x' \otimes \tau_L[\textbf{g}] & -
	\end{pmatrix}.
	$$
	Now, we obtain ($\varheartsuit$) by deriving an equality, for each entry of the first matrix above, between its pullback along $\textbf{X}' \otimes \textbf{g}$ and the corresponding entry in the second matrix:
	\begin{itemize}
		\item For $1 \le i \le K$ and $1 \le j \le L$, we have $(\sigma_i \otimes \tau_j)[\textbf{X}' \otimes \textbf{g}] = (\sigma_i \otimes \tau_j)[\textbf{X} \otimes \textbf{g}]$. By $Sub^2_{t,s}(h)$, the latter is provably equal to $\sigma_i\otimes \tau_j[\textbf{g}]$.
		
		\item For $1 \le i \le K$, we have
		\begin{align*}
			(\sigma_i \varotimes w)[\textbf{X}' \otimes \textbf{g}] & = (\sigma_i \varotimes w)[\textbf{X} \otimes \textbf{g}] \\
			& \equiv (\sigma_i \otimes w)[\textbf{X} \otimes \textbf{g}] \tag{$h$-derivability}\\
			& \equiv \sigma_i \otimes v \tag{$Sub^2_{t,t}(h)$}\\
			& \equiv \sigma_i \varotimes v. \tag{$h$-derivability}
		\end{align*}
		
		\item For $1 \le j \le L$, by $Sub^2_{s,t}(h)$ we derive $(x' \otimes \tau_j)[\textbf{X}' \otimes \textbf{g}] \equiv x' \otimes \tau_j[\textbf{g}]$ in context $\textbf{X}' \otimes \textbf{Y}$.
	\end{itemize}
This concludes the proof in the case where $J$ is a sort judgment and $J'$ is a term judgment.
\end{enumerate}

\begin{center}
	\textcolor{newpurple}{\textbf{\normalsize{\underline{term $\odot$ sort}}}}
\end{center}

\vspace{0.2em}

Suppose that $J$ is $\textbf{X} \vdashcustom u:X'$ and $J'$ is $\textbf{Y} \vdashcustom Y' \tp$. Letting $\textbf{X}' = (\textbf{X}, x':X)$ and $\textbf{Y}' = (\textbf{Y}, y':Y')$, recall that $J \odot J'$ is
$$
\textbf{X} \otimes \textbf{Y}' \vdashcustom u \otimes y': (X' \otimes Y')[u \otimes y_1 \mid x'y_1, ..., u \otimes y_n \mid x'y_n].
$$

We have the following cases:

\begin{enumerate}[label=(\arabic*)]
	\item $J$ and $J'$ are axioms, say $\textbf{X} \vdashcustom s(x_1, ..., x_m):X'$ and $\textbf{Y} \vdashcustom T(y_1, ..., y_n) \tp$, respectively. Then $J \odot J'$ is the axiom
	$$
	\textbf{X} \otimes \textbf{Y}' \vdashcustom
	\begin{pmatrix}
		x_1y_1 & \cdots & x_1y_n & x_1y'\\
		\vdots & \ddots & \vdots & \vdots\\
		x_my_1 & \cdots & x_my_n & x_my'
	\end{pmatrix}
	(X' \otimes Y')[u \otimes y_1 \mid x'y_1, ..., u \otimes y_n \mid x'y_n].
	$$
	For it to be derivable, it must be well-formed, i.e.
	\[
	\tag{\texttt{*}}
	\textbf{X} \otimes \textbf{Y}' \vdashcustom (X' \otimes Y')[u \otimes y_1 \mid x'y_1, ..., u \otimes y_n \mid x'y_n] \tp
	\]
	must be derivable. To check that this is the case, we start by tensoring $\textbf{X} \vdashcustom X' \tp$ and $\textbf{Y} \vdashcustom Y' \tp$ to derive
	\[
	\tag{\texttt{**}}
	\partial(\textbf{X}' \otimes \textbf{Y}') \vdashcustom X' \otimes Y' \tp.
	\]
	On the other hand, as $\Ht(\textbf{X} \vdashcustom u:X')\Ht(\textbf{Y}) \le h$, by Lemma \ref{lem: tensoring a term and a context} we have a morphism $(x_1, ..., x_m,u):\textbf{X} \otimes \textbf{Y} \rightarrow \textbf{X}' \otimes \textbf{Y}$. By amalgamating it with the identity morphism of $\textbf{Y}'$ we derive
	$$
	\partial((x_1, ..., x_m,u) \otimes \textbf{Y}'): \textbf{X} \otimes \textbf{Y}' \longrightarrow \partial (\textbf{X}' \otimes \textbf{Y}').
	$$
	Now, substitution of (\texttt{**}) along the latter morphism yields (\texttt{*}).
	
	\item $J'$ has an initial inference of the form (T-sub). Then there exist a sort axiom $\textbf{B} \vdashcustom V \tp$ and a morphism $\textbf{g} = (g_1, ..., g_N):\textbf{Y} \rightarrow \textbf{B}$ such that $Y' = V[\textbf{g}]$ and $\Ht(\textbf{B} \vdashcustom V \tp)$, $\Ht(\textbf{g}:\textbf{Y} \rightarrow \textbf{B}) < \Ht(\textbf{Y} \vdashcustom Y' \tp)$. Let $\textbf{B}' = (\textbf{B}, b:V)$. Note that
	$$
	\Ht(\textbf{X} \vdashcustom u \tm)\Ht(\textbf{g}:\textbf{Y} \rightarrow \textbf{B}), \qquad \Ht(\textbf{X})\Ht(\textbf{Y} \vdashcustom Y' \tp),\qquad \Ht(\textbf{X} \vdashcustom u \tm)\Ht(\textbf{B} \vdashcustom V \tp)
	$$
	are strictly smaller than $\Ht(\textbf{X} \vdashcustom u:X')\Ht(\textbf{Y} \vdashcustom Y' \tp) = h+1$. Thus we can use $Sub^2_{t,s}(h)$ to derive
	$$
	\textbf{X} \otimes \textbf{Y}' \vdashcustom \doverline{u \otimes v} \equiv u \otimes y' \tm,
	$$
	where the double bar indicates substitution along $\textbf{X} \otimes (g_1, ..., g_N,y'):\textbf{X} \otimes \textbf{Y}' \rightarrow \textbf{X} \otimes \textbf{B}'$; the latter being a morphism follows from $Mor^2_+(h)$.
	
	It remains to show that $u \otimes y'$ is of the desired sort, that is,
	\[
	\tag{$\varheartsuit$}
	(X' \otimes Y')[u \otimes y_1 \mid x'y_1, ..., u \otimes y_n \mid x'y_n].
	\]
	We will do this by checking that $\doverline{u \otimes b}$ is provably of sort ($\varheartsuit$) in context $\textbf{X} \otimes \textbf{Y}'$.
	
	Writing $X' = S(\sigma_1, ..., \sigma_K)$ and $V = T(b_1, ..., b_N)$, the matrix form of ($\varheartsuit$) is
	$$
	ST
	\begin{pmatrix}
		\sigma_1 \otimes g_1 & \cdots & \sigma_1 \otimes g_N & \sigma_1 \otimes y'\\
		\vdots & \ddots & \vdots & \vdots\\
		\sigma_K \otimes g_1 & \cdots & \sigma_K \otimes g_N & \sigma_K \otimes y'\\
		u^{X'} \otimes g_1 & \cdots & u^{X'} \otimes g_N & -
	\end{pmatrix}.
	$$
	On the other hand, let us study the sort of $\doverline{u \otimes b}$. By $h$-derivability, $\textbf{X} \otimes \textbf{B}' \vdashcustom u \otimes b: (X' \otimes V)[u \otimes b_1 \mid x'b_1, ..., u \otimes b_N \mid x'b_N]$ is derivable, so we can also derive
	$$
	\textbf{X} \otimes \textbf{Y}' \vdashcustom \doverline{u \otimes b}: \doverline{(X' \otimes V)[u \otimes b_1 \mid x'b_1, ..., u \otimes b_N \mid x'b_N]}.
	$$
	We can express $(X' \otimes V)[u \otimes b_1 \mid x'b_1, ..., u \otimes b_N \mid x'b_N]$ as
	$$
	ST
	\begin{pmatrix}
		\sigma_1 \otimes b_1 & \cdots & \sigma_1 \otimes b_N & \sigma_1 \otimes b\\
		\vdots & \ddots & \vdots & \vdots\\
		\sigma_K \otimes b_1 & \cdots & \sigma_K \otimes b_N & \sigma_K \otimes b\\
		u \otimes b_1 & \cdots & u \otimes b_N & -
	\end{pmatrix}.
	$$
	Now, observe that
	\begin{itemize}
		\item For $1 \le j \le N$ and $w$ among $\sigma_1$, ..., $\sigma_K$, we have $\doverline{w \otimes b_j} = (w \otimes b_j)[\textbf{X} \otimes \textbf{g}]$, which is provably equal in $\textbf{X} \otimes \textbf{Y}$ to $w \otimes g_j$. This follows from $Sub^2_{t,t}(h)$ using that
		\begin{align*}
			&\Ht(\textbf{X} \vdashcustom w \tm)\Ht(\textbf{g}:\textbf{Y} \rightarrow \textbf{B}) \le \Ht(\textbf{X} \vdashcustom u:X')\Ht(\textbf{g}:\textbf{Y} \rightarrow \textbf{B}) \le h, \\
			&\Ht(\textbf{X} \vdashcustom w \tm)\Ht(\textbf{B} \vdashcustom b_j \tm) < \Ht(\textbf{X} \vdashcustom u:X')\Ht(\textbf{B} \vdashcustom V \tp) \le h.
		\end{align*}
		\item For $1 \le i \le K$, we can use $Sub^2_{s,s}(h)$ to derive $\textbf{X} \otimes \textbf{Y}' \vdashcustom \doverline{\sigma_i \otimes b} \equiv \sigma_i \otimes y' \tm$.
	\end{itemize}
	By applying these modifications to the matrix form of $(X' \otimes V)[u \otimes b_1 \mid x'b_1, ..., u \otimes b_N \mid x'b_N]$, we conclude that $\doverline{(X' \otimes V)[u \otimes b_1 \mid x'b_1, ..., u \otimes b_N \mid x'b_N]}$ is provably equal to ($\varheartsuit$) in $\textbf{X} \otimes \textbf{Y}'$, as required.
	
	\item $J$ has an initial inference
	$$
	\inferrule{\textbf{X} \vdashcustom U \equiv X' \tp \\ \textbf{X} \vdashcustom u:U}{\textbf{X} \vdashcustom u:X'}\text{(Teq/t)}.
	$$
	In this case, by $h$-derivability we can tensor $\textbf{X} \vdashcustom u:U$ and $\textbf{Y} \vdashcustom Y' \tp$ to obtain
	$$
	\textbf{X} \otimes \textbf{Y}' \vdashcustom u \otimes y':(U \otimes Y')[u \otimes y_1 \mid x'y_1, ..., u \otimes y_n \mid x'y_n].
	$$
	Hence it suffices to derive
	$$
	\textbf{X} \otimes \textbf{Y}' \vdashcustom (X' \otimes Y')[u \otimes y_1 \mid x'y_1, ..., u \otimes y_n \mid x'y_n] \equiv (U \otimes Y')[u \otimes y_1 \mid x'y_1, ..., u \otimes y_n \mid x'y_n] \tp.
	$$
	For that, by tensoring $\textbf{X} \vdashcustom U \equiv X' \tp$ and $\textbf{Y} \vdashcustom Y' \tp$ we derive $\partial(\textbf{X}'' \otimes \textbf{Y}') \vdashcustom U \otimes Y' \equiv X' \otimes Y' \tp$ where $\textbf{X}'' = (\textbf{X}, x',U)$. But by $Conteq^1(h)$ we obtain $\textbf{X}' \otimes \textbf{Y}' \equiv \textbf{X}'' \otimes \textbf{Y}' \ctx$, thus also
	\[
	\tag{\texttt{*}}
	\partial(\textbf{X}' \otimes \textbf{Y}') \vdashcustom U \otimes Y' \equiv X' \otimes Y' \tp.
	\]
	Finally, by Lemma \ref{lem: tensoring a term and a context} (using that $\Ht(\textbf{X} \vdashcustom u:U)\Ht(\textbf{Y}') \le h$) we can derive $(x_1, ..., x_m,u) \otimes \textbf{Y}':\textbf{X} \otimes \textbf{Y}' \rightarrow \textbf{X}' \otimes \textbf{Y}'$, and substitution of (\texttt{*}) along this morphism yields the desired sort equality.
	
	\item There exists $i \in \{1, ..., m\}$ such that $X' = X_i$, $u = x_i$, and $J$ has an initial inference
	$$
	\inferrule{\textbf{X} \vdashcustom X_i \tp}{\textbf{X} \vdashcustom x_i:X_i}\text{(var)}.
	$$
	By $Cont(h)$ we have that $\textbf{X}' \otimes \textbf{Y}$ is a context. But $(X_i \otimes Y')[x_iy_1 \mid x'y_1, ..., x_iy_n \mid x'y_n]$ is, by definition, the sort of the variable $x_iy'$ in $\textbf{X} \otimes \textbf{Y}'$. An application of (var) then yields $J \odot J'$.
	
	\item $J$ has an initial inference (t-sub): there exist an axiom $\textbf{A} \vdashcustom w:W$, where $\textbf{A} = (a_1:A_1, ..., a_M:A_M)$, and a morphism $\textbf{f}:\textbf{X} \rightarrow \textbf{A}$ such that $X' = W[\textbf{f}]$, $u = w[\textbf{f}]$, and $\Ht(\textbf{A} \vdashcustom w:W)$, $\Ht(\textbf{f}:\textbf{X} \rightarrow \textbf{A}) < \Ht(\textbf{X} \vdashcustom u:X')$.
	
	Taking $\textbf{A}' = (\textbf{A}, a':W)$, we can tensor $\textbf{A} \vdashcustom w:W$ and $\textbf{Y} \vdashcustom Y' \tp$ to derive
	$$
	\textbf{A} \otimes \textbf{Y}' \vdashcustom w \otimes y': (W \otimes Y')[w \otimes y_1 \mid a'y_1, ..., w \otimes y_n \mid a'y_n].
	$$
	As $\Ht(\textbf{f}:\textbf{X} \rightarrow \textbf{A})\Ht(\textbf{Y}') \le h$, by $Mor^1(h)$ we have a morphism $\textbf{f} \otimes \textbf{Y}':\textbf{X} \otimes \textbf{Y}' \rightarrow \textbf{A} \otimes \textbf{Y}'$. Also, by $Sub^1_{t,s}(h)$ we obtain
	$$
	\textbf{X} \otimes \textbf{Y}' \vdashcustom (w \otimes y')[\textbf{f} \otimes \textbf{Y}'] \equiv u \otimes y' \tm,
	$$
	so the claim will follow if we derive
	\[
	\tag{$\varheartsuit$}
	(W \otimes Y')[w \otimes y_1 \mid a'y_1, ..., w \otimes y_1 \mid a'y_1][\textbf{f} \otimes \textbf{Y}'] \equiv (X' \otimes Y')[u \otimes y_1 \mid x'y_1, ..., u \otimes y_n \mid x'y_n]
	\]
	in context $\textbf{X} \otimes \textbf{Y}'$.
	
	Writing $W = S(\sigma_1, ..., \sigma_K)$ and $Y' = T(\tau_1, ..., \tau_L)$, we have
	$$
	(W \otimes Y')[w \otimes y_1 \mid a'y_1, ..., w \otimes y_n \mid a'y_n] = ST
	\begin{pmatrix}
		\sigma_1 \otimes \tau_1 & \cdots & \sigma_1 \otimes \tau_L & \sigma_1 \otimes y'\\
		\vdots & \ddots & \vdots & \vdots\\
		\sigma_K \otimes \tau_1 & \dots & \sigma_K \otimes \tau_L & \sigma_K \otimes y'\\
		w \otimes \tau_1 & \cdots & w \otimes \tau_L & -
	\end{pmatrix}.
	$$
	On the other hand,
	$$
	(X' \otimes Y')[u \otimes y_1 \mid x'y_1, ..., u \otimes y_n \mid x'y_n] = ST
	\begin{pmatrix}
		\sigma_1[\textbf{f}] \otimes \tau_1 & \cdots & \sigma_1[\textbf{f}] \otimes \tau_L & \sigma_1[\textbf{f}] \otimes y'\\
		\vdots & \ddots & \vdots & \vdots\\
		\sigma_K[\textbf{f}] \otimes \tau_1 & \cdots & \sigma_K[\textbf{f}] \otimes \tau_L & \sigma_K[\textbf{f}] \otimes y'\\
		u \otimes \tau_1 & \cdots & u \otimes \tau_L & -
	\end{pmatrix}.
	$$
	Now, we obtain ($\varheartsuit$) by deriving an equality, for each entry of the first matrix above, between its pullback along $\textbf{f} \otimes \textbf{Y}'$ and the corresponding entry in the second matrix:
	\begin{itemize}
		\item For $1 \le i \le K$ and $1 \le j \le L$, we have $(\sigma_i \otimes \tau_j)[\textbf{f} \otimes \textbf{Y}'] = (\sigma_i \otimes \tau_j)[\textbf{f} \otimes \textbf{Y}]$. By $Sub^1_{s,t}(h)$, the latter is provably equal to $\sigma_i[\textbf{f}] \otimes \tau_j$.
		
		\item For $1 \le j \le L$, by $Sub^1_{t,t}(h)$ we derive $(w \otimes \tau_j)[\textbf{f} \otimes \textbf{Y}] \equiv u \otimes \tau_j$.
		
		\item For $1 \le i \le K$, by $Sub^1_{t,s}(h)$ we derive $(\sigma_i \otimes y')[\textbf{f} \otimes \textbf{Y}'] \equiv \sigma_i[\textbf{f}] \otimes y'$.
	\end{itemize}
\end{enumerate}

This concludes the proof in the case where $J$ is a term judgment and $J'$ is a sort judgment.

\vspace{0.5em}

\begin{center}
	\textcolor{newpurple}{\textbf{\normalsize{\underline{term $\odot$ term}}}}
\end{center}

\vspace{0.2em}
	
	Suppose that $J$ is $\textbf{X} \vdashcustom u:U$ and $J'$ is $\textbf{Y} \vdashcustom v:V$. Consider contexts $\textbf{X}' = (\textbf{X}, x':U)$ and $\textbf{Y}' = (\textbf{Y}, y':V)$.
	
	Then $J \odot J'$ is
	$$
	\textbf{X} \otimes \textbf{Y} \vdashcustom u^U \otimes v^V \equiv u^U \varotimes v^V: \doverline{U \otimes V}
	$$
	where the double bar denotes an application of
	$$
	[u^U \otimes y_1 \mid x'y_1, ..., u^U \otimes y_n \mid x'y_n , x_1 \otimes v^V \mid x_1y', ..., x_m \otimes v^V \mid x_my'].
	$$
	
	The problem can be simplified as follows: since, by Proposition \ref{prop: term from pair of terms}, the judgment $\textbf{X} \otimes \textbf{Y} \vdashcustom u^U \otimes v^V: \doverline{U \otimes V}$ (as well as $\textbf{X} \otimes \textbf{Y} \vdashcustom u^U \varotimes v^V: \doverline{U \otimes V}$) is derivable, it suffices to derive
	$$
	\textbf{X} \otimes \textbf{Y} \vdashcustom u^U \otimes v^V \equiv u^U \varotimes v^V \tm.
	$$
	
	Moreover, if $u$ or $v$ is a variable the claim is trivial since $u^U \varotimes v^V$ is defined as $u^U \otimes v^V$. Thus in what follows we assume that neither $u$ nor $v$ is a variable. (In particular, we do not consider the cases where $J$ or $J'$ has an initial inference of the form (var).)
	
	We have the following cases:
	
	\begin{enumerate}[label=(\arabic*)]
		\item $J$ and $J'$ are axioms. Then $J \odot J'$ is an axiom, so its derivability follows from that of the corresponding premises of the rule (t-a), i.e.
		$$
		\textbf{X} \otimes \textbf{Y} \vdashcustom u \otimes v: \doverline{U \otimes V}, \;\;\;\;\;\;\;\;\;\;\; \textbf{X} \otimes \textbf{Y} \vdashcustom u \varotimes v: \doverline{U \otimes V}.
		$$
		
		\item $J$ has an initial inference of the form (t-sub). Then there exist an axiom $\textbf{A} \vdashcustom w:W$, where $\textbf{A} = (a_1:A_1, ..., a_M:A_M)$, and a morphism $\textbf{f}:\textbf{X} \rightarrow \textbf{A}$ such that $U = W[\textbf{f}]$, $u = w[\textbf{f}]$, and $\Ht(\textbf{A} \vdashcustom w:W)$, $\Ht(\textbf{f}:\textbf{X} \rightarrow \textbf{A}) < \Ht(\textbf{X} \vdashcustom u:U)$.
		
		Writing $\textbf{A}'$ for a context $(\textbf{A}, a':W)$, by tensoring $\textbf{A} \vdashcustom w:W$ and $\textbf{Y} \vdashcustom v:V$ we derive
		\[
		\tag{\texttt{*}}
		\textbf{A} \otimes \textbf{Y} \vdashcustom w \otimes v^V \equiv w \varotimes v^V \tm.
		\]
		
		Moreover, as $\Ht(\textbf{f}:\textbf{X} \rightarrow \textbf{A})\Ht(\textbf{Y}) \le h$, by $Mor^1(h)$ we have a context morphism $\textbf{f} \otimes \textbf{Y}:\textbf{X} \otimes \textbf{Y} \rightarrow \textbf{A} \otimes \textbf{Y}$. We will verify that $J \odot J'$ is derivable by describing the pullback along $\textbf{f} \otimes \textbf{Y}$ of the two term expressions in (\texttt{*}). 
		
		Write $w = s(a_1, ..., a_M)$, $v = t(t_1, ..., t_l)$, $W = S(\sigma_1, ..., \sigma_K)$, and $V = T(\tau_1, ..., \tau_L)$; also, we use an upper bar (resp. lower bar) to indicate substitution along $\textbf{f} \otimes \textbf{Y}$ (resp. along $\textbf{f}$). Then:
		
		\begin{itemize}
			\item $w \otimes v^V$ has matrix form
			$$
			St
			\begin{pmatrix}
				\sigma_1 \otimes t_1 & \cdots & \sigma_1 \otimes t_l\\
				\vdots & \ddots & \vdots \\
				\sigma_K \otimes t_1 & \cdots & \sigma_K \otimes t_l\\
				w \otimes t_1 & \cdots & w \otimes t_l
			\end{pmatrix},
			$$
			so, by $Sub^1_{t,t}(h)$, we have that $\overline{w \otimes v^V}$ is provably equal in context $\textbf{X} \otimes \textbf{Y}$ to
			$$
			St
			\begin{pmatrix}
				\underline{\sigma_1} \otimes t_1 & \cdots & \underline{\sigma_1} \otimes t_l\\
				\vdots & \ddots & \vdots \\
				\underline{\sigma_K} \otimes t_1 & \cdots & \underline{\sigma_K} \otimes t_l\\
				\underline{w} \otimes t_1 & \cdots & \underline{w} \otimes t_l
			\end{pmatrix}
			= 
			St
			\begin{pmatrix}
				\underline{\sigma_1} \otimes t_1 & \cdots & \underline{\sigma_1} \otimes t_l\\
				\vdots & \ddots & \vdots \\
				\underline{\sigma_K} \otimes t_1 & \cdots & \underline{\sigma_K} \otimes t_l\\
				u \otimes t_1 & \cdots & u \otimes t_l
			\end{pmatrix}
			=
			u \otimes v^V.
			$$
			
			\item $w \varotimes v^V$ has matrix form
			$$
			sT
			\begin{pmatrix}
				a_1 \otimes \tau_1 & \cdots & a_1 \otimes \tau_L & a_1 \otimes v^V\\
				\vdots & \ddots & \vdots & \vdots \\
				a_M \otimes \tau_1 & \cdots & a_M \otimes \tau_L & a_M \otimes v^V
			\end{pmatrix}.
			$$
			By $Sub^1_{t,t}(h)$, it follows that $\overline{w \varotimes v^V}$ is provably equal in context $\textbf{X} \otimes \textbf{Y}$ to
			$$
			sT
			\begin{pmatrix}
				f_1 \otimes \tau_1 & \cdots & f_1 \otimes \tau_L & \overline{a_1 \otimes v^V}\\
				\vdots & \ddots & \vdots & \vdots \\
				f_M \otimes \tau_1 & \cdots & f_M \otimes \tau_L & \overline{a_M \otimes v^V}
			\end{pmatrix}.
			$$
			Now, note that, for $1 \le i \le M$,
			$$
			\overline{a_i \otimes v^V} = \overline{a_i \otimes v} \overset{\textbf{X} \otimes \textbf{Y}}{\equiv} f_i \otimes v = f_i \otimes v^V \overset{\textbf{X} \otimes \textbf{Y}}{\equiv} f_i \varotimes v^V,
			$$
			where the latter equality follows from $h$-derivability. It follows that the latter matrix is provably equal in context $\textbf{X} \otimes \textbf{Y}$ to
			$$
			sT
			\begin{pmatrix}
				f_1 \otimes \tau_1 & \cdots & f_1 \otimes \tau_L & f_1 \varotimes v^V\\
				\vdots & \ddots & \vdots & \vdots \\
				f_M \otimes \tau_1 & \cdots & f_M \otimes \tau_L & f_M \varotimes v^V
			\end{pmatrix}
			= u \varotimes v^V.
			$$
		\end{itemize}
		
		\item The case where $J'$ has an initial inference of the form (t-sub) can be studied analogously to (2); the verification has been omitted.
		
		\item $J$ has an initial inference of the form (Teq/t), say
		$$
		\inferrule{\textbf{X} \vdashcustom U'\equiv U \tp \\ \textbf{X} \vdashcustom u:U'}{\textbf{X} \vdashcustom u:U}.
		$$
		By $h$-derivability, we can tensor $\textbf{X} \vdashcustom u:U'$ and $\textbf{Y} \vdashcustom v:V$ to derive
		$$
		\textbf{X} \otimes \textbf{Y} \vdashcustom u^{U'} \otimes v^V \equiv u^{U'} \varotimes v^V \tm,
		$$
		hence it suffices to derive
		\begin{align*}
			&\textbf{X} \otimes \textbf{Y} \vdashcustom u^{U'} \otimes v^V \equiv u^U \otimes v^V \tm, \tag{\texttt{*}}\\
			&\textbf{X} \otimes \textbf{Y} \vdashcustom u^{U'} \varotimes v^V \equiv u^U \varotimes v^V \tm. \tag{\texttt{**}}
		\end{align*}		
		Recalling that $u$, $v$ are assumed to not be variables, we have that $u^{U'} \varotimes v^V$ and $u^U \varotimes v^V$ are the same expression. Hence it suffices to derive (\texttt{*}). For that, we tensor $\textbf{X} \vdashcustom U' \equiv U \tp$ and $\textbf{Y} \vdashcustom v:V$ to derive
		\[
		\tag{\texttt{***}}
		\textbf{X}' \otimes \textbf{Y} \vdashcustom z^{U'} \otimes v^V \equiv z^U \otimes v^V \tm.
		\]
		Now, it can be checked that we have an equality of context morphisms
		$$
		(x_1, ..., x_m,u^{U'}) \otimes \textbf{Y} \equiv (x_1, ..., x_m,u^U) \otimes \textbf{Y}: \textbf{X} \otimes \textbf{Y} \longrightarrow \textbf{X}' \otimes \textbf{Y}.
		$$
		By taking the pullback of the left term expression in (\texttt{***}) along $(x_1, ..., x_m,u^{U'})$ and of the right one along $(x_1, ..., x_m,u^U)$, we obtain
		$$
		\textbf{X} \otimes \textbf{Y} \vdashcustom (z^{U'} \otimes v^V)[u^{U'} \otimes y_1 \mid zy_1, ..., u^{U'} \otimes y_n \mid zy_n] \equiv (z^U \otimes v^V)[u^U \otimes y_1 \mid zy_1, ..., u^U \otimes y_n \mid zy_n] \tm,
		$$
		which is precisely (\texttt{*}).
		
		\item If $J'$ has an initial inference of the form (Teq/t), the proof is analogous to the one in (4).
	\end{enumerate}
	
	\vspace{0.1em}
	
	\begin{center}
		\textcolor{newpurple}{\textbf{\normalsize{\underline{sort equality $\odot$ sort}}}}
	\end{center}
	
	\vspace{0.2em}

	Suppose that $J$ is $\textbf{X} \vdashcustom X' \equiv X'' \tp$ and $J'$ is $\textbf{Y} \vdashcustom Y' \tp$. Let $\textbf{X}' = (\textbf{X},x':X')$, $\textbf{X}'' = (\textbf{X},x':X'')$, and $\textbf{Y}' = (\textbf{Y}, y':Y')$. Then $J \odot J'$ is
	$$
	\partial(\textbf{X}' \otimes \textbf{Y}') \vdashcustom X' \otimes Y' \equiv X'' \otimes Y' \tp.
	$$
	
	Firstly, note that this judgment is well-formed:
	\begin{itemize}
		\item As $\Ht(\textbf{X} \vdashcustom X' \tp) < \Ht(\textbf{X} \vdashcustom X' \equiv X'' \tp)$, by $h$-derivability we can tensor $\textbf{X} \vdashcustom X' \tp$ and $\textbf{Y}' \vdashcustom Y' \tp$ to derive $\partial(\textbf{X}' \otimes \textbf{Y}') \vdashcustom X' \otimes Y' \tp$.
		
		\item Similarly, we can tensor $\textbf{X} \vdashcustom X'' \tp$ and $\textbf{Y} \vdashcustom Y' \tp$ to derive $\partial(\textbf{X}'' \otimes \textbf{Y}'') \vdashcustom X' \otimes Y'' \tp$. But by $Conteq^1(h)$ we can derive $\textbf{X}' \otimes \textbf{Y} \equiv \textbf{X}'' \otimes \textbf{Y} \ctx$, from which we conclude that $\partial(\textbf{X}' \otimes \textbf{Y}')$ and $\partial(\textbf{X}'' \otimes \textbf{Y}')$ are provably equal.
	\end{itemize}

	Now, to derive $J \odot J'$, we have the following cases:
	
	\begin{enumerate}[label=(\arabic*)]
		\item $J$ and $J'$ are axioms. Then $J \odot J'$ is derivable as it is well-formed.

		\item $X'$ equals $X''$, and $J$ has an initial inference
		$$
		\inferrule{\textbf{X} \vdashcustom X' \tp}{\textbf{X} \vdashcustom X' \equiv X' \tp}(\text{T1}).
		$$
		Then we can derive $J \odot J'$ from $\partial(\textbf{X}' \otimes \textbf{Y}') \vdashcustom X' \otimes Y' \tp$ by using (T1).
		
		\item $J$ has an initial inference
		$$
		\inferrule{\textbf{X} \vdashcustom X'' \equiv X' \tp}{\textbf{X} \vdashcustom X' \equiv X'' \tp}(\text{T2}).
		$$
		By $h$-derivability, we can tensor the premise and $\textbf{Y} \vdashcustom Y' \tp$ to derive
		$$
		\partial(\textbf{X}'' \otimes \textbf{Y}') \vdashcustom X'' \otimes Y' \equiv X' \otimes Y' \tp.
		$$
		Then we obtain $J \odot J'$ from $\partial(\textbf{X}' \otimes \textbf{Y}') \equiv \partial(\textbf{X}'' \otimes \textbf{Y}') \ctx$ and an application of (T2).
		
		\item $J$ has an initial inference
		$$
		\inferrule{\textbf{X} \vdashcustom X' \equiv X''' \tp \\ \textbf{X} \vdashcustom X''' \equiv X''}{\textbf{X} \vdashcustom X' \equiv X'' \tp}(\text{T3}).
		$$
		By $h$-derivability, the judgments $(\textbf{X} \vdashcustom X' \equiv X''' \tp) \otimes J'$ and $(\textbf{X} \vdashcustom X''' \equiv X'' \tp) \otimes J'$ are derivable. These are, respectively,
		$$
		\partial(\textbf{X}' \otimes \textbf{Y}') \vdashcustom X' \otimes Y' \equiv X''' \otimes Y' \tp,
		$$
		$$
		\partial(\textbf{X}''' \otimes \textbf{Y}') \vdashcustom X''' \otimes Y' \equiv X'' \otimes Y' \tp,
		$$
		where $\textbf{X}''' = (\textbf{X}, x':X''')$. Then we obtain $J \odot J'$ by using $\partial(\textbf{X}' \otimes \textbf{Y}') \equiv \partial(\textbf{X}''' \otimes \textbf{Y}') \ctx$ and an instance of (T3).
		
		\item $J$ has an initial inference
		$$
		\inferrule{\textbf{A} \vdashcustom U \equiv V \tp \\ \textbf{X} \vdashcustom X' \tp \\ \textbf{X} \vdashcustom X'' \tp \\ \textbf{f}:\textbf{X} \rightarrow \textbf{A}}{\textbf{X} \vdashcustom X' \equiv X'' \tp}\text{(Teq-sub-1)}
		$$
		where $\textbf{A} = (a_1:A_1, ..., a_M:A_M)$, $X' = U[\textbf{f}]$ and $X'' = V[\textbf{f}]$. By tensoring $\textbf{A} \vdashcustom U \equiv V \tp$ and $\textbf{Y} \vdashcustom Y' \tp$, we derive
		\[
		\tag{\texttt{*}}
		\partial(\textbf{A}' \otimes \textbf{Y}') \vdashcustom U \otimes Y' \equiv V \otimes Y' \tp
		\]
		where $\textbf{A}' = (\textbf{A}, a:U)$. By $Mor^1_+(h)$, we have a morphism $(f_1, ..., f_M,x') \otimes \textbf{Y}:\textbf{X}' \otimes \textbf{Y} \rightarrow \textbf{A}' \otimes \textbf{Y}$, from which we obtain $\partial((f_1, ..., f_M, x') \otimes \textbf{Y}'):\partial(\textbf{X}' \otimes \textbf{Y}') \rightarrow \partial(\textbf{A}' \otimes \textbf{Y}')$.
		
		By taking the pullback along the latter of (\texttt{*}) and applying $Sub^1_{s,s}$ we derive $\partial(\textbf{X}' \otimes \textbf{Y}') \vdashcustom X' \otimes Y' \equiv X'' \otimes Y' \tp$, as required.
		
		\item $J$ has an initial inference
		$$
		\inferrule{\textbf{A} \vdashcustom U \tp \\ \textbf{X} \vdashcustom X' \tp \\ \textbf{X} \vdashcustom X'' \tp \\ \textbf{f} \equiv \textbf{g}:\textbf{X} \rightarrow \textbf{A}}{\textbf{X} \vdashcustom X' \equiv X'' \tp}\text{(Teq-sub-2)}
		$$
		where $X' = U[\textbf{f}]$ and $X'' = U[\textbf{g}]$. By $h$-normality, we can tensor $\textbf{A} \vdashcustom U \tp$ and $\textbf{Y} \vdashcustom Y' \tp$ to derive
		\[
		\tag{\texttt{*}}
		\partial(\textbf{A}' \otimes \textbf{Y}') \vdashcustom U \otimes Y' \tp
		\]
		where $\textbf{A}' = (\textbf{A}, a:U)$. Now, observe that
		\begin{itemize}
			\item By $Moreq^1(h)$ we have a morphism equality $\textbf{f} \otimes \textbf{Y}' \equiv \textbf{g} \otimes \textbf{Y}' : \textbf{X} \otimes \textbf{Y}' \rightarrow \textbf{A} \otimes \textbf{Y}'$.
			
			\item Let $\textbf{f} = (f_1, ..., f_M,x')$ and $\textbf{g} = (g_1, ..., g_M,x')$. Since $\Ht(\textbf{Y}) < \Ht(\textbf{Y}')$, we can use $Mor^1_+(h)$ to obtain morphisms
			$$
			\textbf{f}' \otimes \textbf{Y}:\textbf{X}' \otimes \textbf{Y} \longrightarrow \textbf{A}' \otimes \textbf{Y}, \quad\quad\quad\quad \textbf{g}' \otimes \textbf{Y}:\textbf{X}'' \otimes \textbf{Y} \longrightarrow \textbf{A}' \otimes \textbf{Y}.
			$$
			But $\textbf{X}' \otimes \textbf{Y}$ and $\textbf{X}'' \otimes \textbf{Y}$ are provably equal, so the judgment
			$$
			\textbf{f}' \otimes \textbf{Y} \equiv \textbf{g}' \otimes \textbf{Y}: \textbf{X}' \otimes \textbf{Y} \longrightarrow \textbf{A}' \otimes \textbf{Y}
			$$
			is well-formed. In fact, it is derivable: the two maps extend $\textbf{f} \otimes \textbf{Y}$ and $\textbf{g} \otimes \textbf{Y}$, respectively, by the same list of expressions (namely, $x'y_1$, ..., $x'y_n$).
		\end{itemize}
		From the two equalities above we obtain
		$$
		\partial(\textbf{f}' \otimes \textbf{Y}') \equiv \partial(\textbf{g}' \otimes \textbf{Y}'): \partial(\textbf{X}' \otimes \textbf{Y}') \longrightarrow \partial(\textbf{A}' \otimes \textbf{Y}').
		$$
		Finally, by taking the pullback of (\texttt{*}) along $\partial(\textbf{f}' \otimes \textbf{Y}')$ and $\partial(\textbf{g}' \otimes \textbf{Y}')$, and using $Sub^1_{s,s}$ in each case, we derive
		$$
		\partial(\textbf{X}' \otimes \textbf{Y}') \vdashcustom U[\textbf{f}] \otimes Y' \equiv U[\textbf{g}] \otimes Y' \tp.
		$$
		
		\item $J'$ has an initial inference of the form (T-sub). Then there exist an axiom $\textbf{B} \vdashcustom U \tp$, where $\textbf{B} = (b_1:B_1, ..., b_N:B_N)$, and a morphism $\textbf{g}:\textbf{Y} \rightarrow \textbf{B}$ such that $Y' = U[\textbf{g}]$ and $\Ht(\textbf{B} \vdashcustom U \tp)$, $\Ht(\textbf{g}:\textbf{Y} \rightarrow \textbf{B}) < \Ht(\textbf{Y} \vdashcustom Y' \tp)$.
		
		By $h$-derivability, we can tensor $J$ and $\textbf{B} \vdashcustom U \tp$ to obtain
		\[
		\tag{\texttt{*}}
		\partial(\textbf{X}' \otimes \textbf{B}') \vdashcustom X' \otimes U \equiv X'' \otimes U \tp
		\]
		where $\textbf{B}' = (\textbf{B},b:U)$. Also,
		\begin{itemize}
			\item By $Mor^2(h)$, we have a morphism $\textbf{X}' \otimes \textbf{g}: \textbf{X}' \otimes \textbf{Y} \rightarrow \textbf{X}' \otimes \textbf{B}$.
			
			\item By $Mor^2_+(h)$, we have a morphism $\textbf{X} \otimes \textbf{g}': \textbf{X} \otimes \textbf{Y}' \rightarrow \textbf{X} \otimes \textbf{B}'$ where $\textbf{g}' = (g_1, ..., g_N,y')$.
		\end{itemize}
		Thus we obtain a map $\partial(\textbf{X}' \otimes \textbf{g}'):\partial(\textbf{X}' \otimes \textbf{Y}') \rightarrow \partial(\textbf{X}' \otimes \textbf{B}')$. With an analogous argument we derive $\partial(\textbf{X}'' \otimes \textbf{g}'):\partial(\textbf{X}'' \otimes \textbf{Y}') \rightarrow \partial(\textbf{X}'' \otimes \textbf{B}')$, and, as the domains (resp. codomains) of these morphisms are provably equal, the judgment
		\[
		\tag{\texttt{**}}
		\partial(\textbf{X}' \otimes \textbf{g}') \equiv \partial(\textbf{X}'' \otimes \textbf{g}'):\partial(\textbf{X}' \otimes \textbf{Y}') \longrightarrow \partial(\textbf{X}' \otimes \textbf{B}')
		\]
		is well-formed. Now, note that the entries in which these two maps differ are of the form $x^{'X'} \otimes g_i$ and $x^{'X''} \otimes g_i$, respectively, for $1 \le i \le N$. By tensoring $\textbf{X} \vdashcustom X' \equiv X'' \tp$ and $\textbf{Y} \vdashcustom g_j:\Type(g_j)$, it follows (using that $\Ht(\textbf{Y} \vdashcustom g_j:\Type(g_j)) \le \Ht(\textbf{g}:\textbf{Y} \rightarrow \textbf{B})$) that $x^{'X'} \otimes g_i \equiv x^{'X''} \otimes g_i$ in context $\textbf{X}' \otimes \textbf{Y}$.
		
		This implies that (\texttt{**}) is derivable, and we can take the pullback of the two sort expressions in (\texttt{*}) along the two morphisms in (\texttt{**}) to derive
		\[
		\tag{\texttt{***}}
		\partial(\textbf{X}' \otimes \textbf{Y}') \vdashcustom (X' \otimes U)[\textbf{X}' \otimes \textbf{g}'] \equiv (X'' \otimes U)[\textbf{X}'' \otimes \textbf{g}'] \tp.
		\]
		Now, by $Sub^2_{s,s}$ we obtain
		\begin{align*}
			&\partial(\textbf{X}' \otimes \textbf{Y}') \vdashcustom (X' \otimes U)[\textbf{X}' \otimes \textbf{g}'] \equiv X' \otimes Y' \tp, \\
			&\partial(\textbf{X}'' \otimes \textbf{Y}') \vdashcustom (X'' \otimes U)[\textbf{X}'' \otimes \textbf{g}'] \equiv X'' \otimes Y' \tp,
		\end{align*}
		and, finally, by replacing the two sort expressions in (\texttt{***}) according to the latter equalities we derive the desired judgment $\partial(\textbf{X}' \otimes \textbf{Y}') \vdashcustom X' \otimes Y' \equiv X'' \otimes Y' \tp$.
\end{enumerate}
	
	\vspace{0.1em}
	
	\begin{center}
		\textcolor{newpurple}{\textbf{\normalsize{\underline{sort equality $\odot$ term}}}}
	\end{center}
	
	\vspace{0.2em}
	
	Suppose that $J$ is $\textbf{X} \vdashcustom X' \equiv X'' \tp$ and $J'$ is $\textbf{Y} \vdashcustom u:Y'$. Let $\textbf{X}' = (\textbf{X},x':X')$, $\textbf{X}'' = (\textbf{X},x':X'')$, and $\textbf{Y}' = (\textbf{Y}, y':Y')$. Then $J \odot J'$ is
	$$
	\textbf{X}' \otimes \textbf{Y} \vdashcustom x^{'X'} \otimes u \equiv x^{'X''} \otimes u : (X' \otimes Y')[x_1 \otimes u \mid x_1y', ..., x_m \otimes u \mid x_my'].
	$$
	As in the previous cases, we start by performing some simplifications. Since $\Ht(\textbf{X} \vdashcustom X' \tp)$, $\Ht(\textbf{X} \vdashcustom X'' \tp) < \Ht(J)$, we can tensor $\textbf{X} \vdashcustom X' \tp$ and $\textbf{X} \vdashcustom X'' \tp$ with $J'$ to derive
	\begin{align*}
		\textbf{X}' \otimes \textbf{Y} & \vdashcustom x^{'X'} \otimes u: (X' \otimes Y')[x_1 \otimes u \mid x_1y', ..., x_m \otimes u \mid x_my'], \tag{\texttt{*}} \\
		\textbf{X}'' \otimes \textbf{Y} & \vdashcustom x^{''X'} \otimes u: (X'' \otimes Y')[x_1 \otimes u \mid x_1y', ..., x_m \otimes u \mid x_my']. \tag{\texttt{**}}
	\end{align*}
	To combine these judgments, observe that tensoring $\textbf{X} \vdashcustom X' \equiv X'' \tp$ and $\textbf{Y} \vdashcustom Y' \tp$ yields
	\[
	\tag{\texttt{***}}
	\partial(\textbf{X}' \otimes \textbf{Y}') \vdashcustom X' \otimes Y' \equiv X'' \otimes Y' \tp.
	\]
	On the other hand, by Lemma \ref{lem: tensoring a term and a context} we have a morphism $\textbf{X} \otimes (y_1, ..., y_n,u):\textbf{X} \otimes \textbf{Y} \rightarrow \textbf{X} \otimes \textbf{Y}'$, hence also $\partial(\textbf{X}' \otimes (y_1, ..., y_n,u)): \textbf{X}' \otimes \textbf{Y} \longrightarrow \partial(\textbf{X}' \otimes \textbf{Y}')$. By substitution of (\texttt{***}) along the latter map we obtain
	$$
	\textbf{X}' \otimes \textbf{Y} \vdashcustom (X' \otimes Y')[x_1 \otimes u \mid x_1y', ..., x_m \otimes u \mid x_my'] \equiv (X'' \otimes Y')[x_1 \otimes u \mid x_1y', ..., x_m \otimes u \mid x_my'] \tp.
	$$
	We conclude that $J \odot J'$ is well-formed and, in particular, its derivability will follow from that of 
	\[
	\tag{$\varheartsuit$}
	\textbf{X}' \otimes \textbf{Y} \vdashcustom x^{'X'} \otimes u \equiv x^{'X''} \otimes u \tm.
	\]
	
	We consider the following cases:
	\begin{enumerate}[label=(\arabic*)]
		\item $J$ and $J'$ are axioms. Then $J \odot J'$ is a well-formed axiom, so it is derivable.
		
		\item $X'$ equals $X''$, and $J$ has an initial inference
		$$
		\inferrule{\textbf{X} \vdashcustom X' \tp}{\textbf{X} \vdashcustom X' \equiv X' \tp}(\text{T1}).
		$$
		Then an instance of (T1) allows us to derive ($\varheartsuit$) from $\textbf{X}' \otimes \textbf{Y} \vdashcustom x^{'X'} \otimes u \tm$.
		
		\item $J$ has an initial inference
		$$
		\inferrule{\textbf{X} \vdashcustom X'' \equiv X' \tp}{\textbf{X} \vdashcustom X' \equiv X'' \tp}(\text{T2}).
		$$
		Then $h$-derivability yields $\textbf{X}'' \otimes \textbf{Y} \vdashcustom x^{'X''} \otimes u \equiv x^{'X''} \otimes u \tm$, and we obtain ($\varheartsuit$) from $\textbf{X}' \otimes \textbf{Y} \equiv \textbf{X}'' \otimes \textbf{Y} \ctx$ and an instance of (T2).
		
		\item $J$ has an initial inference
		$$
		\inferrule{\textbf{X} \vdashcustom X' \equiv X''' \tp \\ \textbf{X} \vdashcustom X''' \equiv X''}{\textbf{X} \vdashcustom X' \equiv X'' \tp}(\text{T3}).
		$$
		Let $\textbf{X}''' = (\textbf{X}, x':X''')$. By $h$-derivability we obtain
		$$
		\textbf{X}' \otimes \textbf{Y} \vdashcustom x^{'X'} \otimes u \equiv x^{'X'''} \otimes u \tm, \quad\quad \textbf{X}''' \otimes \textbf{Y} \vdashcustom x^{'X'''} \otimes u \equiv x^{'X''} \otimes u \tm.
		$$
		Now, $\textbf{X}' \otimes \textbf{Y} \equiv \textbf{X}''' \otimes \textbf{Y} \ctx$ and an instance of (T3) yield ($\varheartsuit$).
		
		\item $J$ has an initial inference
		$$
		\inferrule{\textbf{A} \vdashcustom U \equiv V \tp \\ \textbf{X} \vdashcustom X' \tp \\ \textbf{X} \vdashcustom X'' \tp \\ \textbf{f}:\textbf{X} \rightarrow \textbf{A}}{\textbf{X} \vdashcustom X' \equiv X'' \tp}\text{(Teq-sub-1)}
		$$
		where $\textbf{A} = (a_1:A_1, ..., a_M:A_M)$, $X' = U[\textbf{f}]$ and $X'' = V[\textbf{f}]$. By tensoring $\textbf{A} \vdashcustom U \equiv V \tp$ and $\textbf{Y} \vdashcustom u:Y'$ we derive
		\[
		\tag{\texttt{*}}
		\textbf{A}' \otimes \textbf{Y} \vdashcustom a^U \otimes u \equiv a^V \otimes u \tm
		\]
		where $\textbf{A}' = (\textbf{A}, a:U)$. Now, by using $Sub^1_{s,t}$ twice we obtain
		\begin{align*}
			\textbf{X}' \otimes \textbf{Y} &\vdashcustom x^{'X'} \otimes u \equiv (a^U \otimes u)[\textbf{f}' \otimes \textbf{Y}] \tm, \tag{\texttt{**}}\\
			\textbf{X}'' \otimes \textbf{Y} &\vdashcustom x^{'X''} \otimes u \equiv (a^V \otimes u)[\textbf{f}' \otimes \textbf{Y}] \tm \tag{\texttt{***}}
		\end{align*}
		where $\textbf{f}' = (f_1, ..., f_M, x')$. Since $\textbf{X}' \otimes \textbf{Y} \equiv \textbf{X}'' \otimes \textbf{Y} \ctx$ is derivable, we can combine (\texttt{*}), (\texttt{**}) and (\texttt{***}) to obtain ($\varheartsuit$).
		
		\item $J$ has an initial inference
		$$
		\inferrule{\textbf{A} \vdashcustom U \tp \\ \textbf{X} \vdashcustom X' \tp \\ \textbf{X} \vdashcustom X'' \tp \\ \textbf{f} \equiv \textbf{g}:\textbf{X} \rightarrow \textbf{A}}{\textbf{X} \vdashcustom X' \equiv X'' \tp}\text{(Teq-sub-2)}
		$$
		where $X' = U[\textbf{f}]$ and $X'' = U[\textbf{g}]$. By tensoring $\textbf{A} \vdashcustom U \tp$ and $\textbf{Y} \vdashcustom u:Y'$ we derive
		\[
		\tag{\texttt{*}}
		\textbf{A}' \otimes \textbf{Y} \vdashcustom a^U \otimes u \tm
		\]
		where $\textbf{A}' = (\textbf{A}, a:U)$. On the other hand, as $\Ht(\textbf{f} \equiv \textbf{g}:\textbf{X} \rightarrow \textbf{A})\Ht(\textbf{Y}) \le h$, by $Moreq^1(h)$ we obtain $\textbf{f} \otimes \textbf{Y} \equiv \textbf{g} \otimes \textbf{Y}:\textbf{X} \otimes \textbf{Y} \longrightarrow \textbf{A} \otimes \textbf{Y}$, hence also
		\[
		\tag{\texttt{**}}
		\textbf{f}' \otimes \textbf{Y} \equiv \textbf{g}' \otimes \textbf{Y}: \textbf{X}' \otimes \textbf{Y} \longrightarrow \textbf{A}' \otimes \textbf{Y}.
		\]
		By substitution of (\texttt{*}) along the two morphisms in (\texttt{**}) we derive
		$$
		\textbf{X}' \otimes \textbf{Y} \vdashcustom (a^U \otimes u)[\textbf{f}' \otimes \textbf{Y}] \equiv (a^U \otimes u)[\textbf{g}' \otimes \textbf{Y}] \tm,
		$$
		and applying $Sub^1_{s,t}(h)$ to these two term expressions yields ($\varheartsuit$).
		
		\item $J'$ has an initial inference
		$$
		\inferrule{\textbf{Y} \vdashcustom Y'' \equiv Y' \tp \\ \textbf{Y} \vdashcustom u:Y''}{\textbf{Y} \vdashcustom u:Y'}\text{(Teq/t)}.
		$$
		By $h$-derivability, we can tensor $\textbf{X} \vdashcustom X' \equiv X'' \tp$ and $\textbf{Y} \vdashcustom u:Y''$ to derive, in particular,
		$$
		\textbf{X}' \otimes \textbf{Y} \vdashcustom x^{'X'} \otimes u \equiv x^{'X''} \otimes u \tm,
		$$
		which is the required judgment.
		
		\item There exists $j \in \{1, ..., n\}$ such that $Y' = Y_j$, $u = y_j$, and $J'$ has an initial inference
		$$
		\inferrule{\textbf{Y} \vdashcustom Y_j \tp}{\textbf{Y} \vdashcustom y_j:Y_j}\text{(var)}.
		$$
		Then ($\varheartsuit$) is $\textbf{X}' \otimes \textbf{Y} \vdashcustom x'y_j \equiv x'y_j \tm$. Since it is well-formed, by (T1) it is derivable.
		
		\item $J'$ has an initial inference of the form (t-sub). Then there exist an axiom $\textbf{B} \vdashcustom w:W$, where $\textbf{B} = (b_1:B_1, ..., b_N:B_N)$, and a morphism $\textbf{g}:\textbf{Y} \rightarrow \textbf{B}$ such that $Y' = W[\textbf{g}]$, $u = w[\textbf{g}]$, and $\Ht(\textbf{B} \vdashcustom w:W)$, $\Ht(\textbf{g}:\textbf{Y} \rightarrow \textbf{B}) < \Ht(\textbf{Y} \vdashcustom u:Y')$.
		
		Taking $\textbf{B}' = (\textbf{B}, b':W)$, by $h$-derivability we can tensor $\textbf{X} \vdashcustom X' \equiv X'' \tp$ and $\textbf{B} \vdashcustom w:W$ to derive, in particular,
		\[
		\tag{\texttt{*}}
		\textbf{X}' \otimes \textbf{B} \vdashcustom x^{'X'} \otimes w \equiv x^{'X''} \otimes w \tm.
		\]
		Also, by $Sub^2_{s,t}(h)$ we obtain
		\begin{align*}
			\textbf{X}' \otimes \textbf{Y} & \vdashcustom (x^{'X'} \otimes w)[\textbf{X}' \otimes \textbf{g}] \equiv x^{'X'} \otimes u \tm, \\
			\textbf{X}'' \otimes \textbf{Y} & \vdashcustom (x^{'X''} \otimes w)[\textbf{X}' \otimes \textbf{g}] \equiv x^{'X''} \otimes u \tm.
		\end{align*}
		As $\textbf{X}' \otimes \textbf{Y}$ and $\textbf{X}'' \otimes \textbf{Y}$ are provably equal, we can combine the above judgments with (\texttt{*}) to derive $\textbf{X}' \otimes \textbf{Y} \vdashcustom x^{'X'} \otimes u \equiv x^{'X''} \otimes u \tm$, as required.
	\end{enumerate}
	
	\vspace{0.1em}
	
	\begin{center}
		\textcolor{newpurple}{\textbf{\normalsize{\underline{sort $\odot$ term equality}}}}
	\end{center}
	
	\vspace{0.2em}
		
	Suppose that $J$ is $\textbf{X} \vdashcustom U \tp$ and $J'$ is $\textbf{Y} \vdashcustom u \equiv v:V$. Writing $\textbf{X}' = (\textbf{X}, x':U)$, we want to derive
	$$
	\textbf{X}' \otimes \textbf{Y} \vdashcustom x' \otimes u^V \equiv x' \otimes v^V: (U \otimes V)[x_1 \otimes u^V \mid x_1y', ..., x_m \otimes u^V \mid x_my'].
	$$
	To start, note that we can tensor $\textbf{X} \vdashcustom U \tp$ and $\textbf{Y} \vdashcustom u:V$ to derive
	$$
	\textbf{X}' \otimes \textbf{Y} \vdashcustom x' \otimes u^V: (U \otimes V)[x_1 \otimes u^V \mid x_1y', ..., x_m \otimes u^V \mid x_my']
	$$
	where $y'$ is a variable of sort $V$.
	
	As a consequence, to obtain $J \odot J'$ it suffices to derive $\textbf{X}' \otimes \textbf{Y} \vdashcustom x' \otimes u^V \equiv x' \otimes v^V \tm$ or, equivalently, $\textbf{X}' \otimes \textbf{Y} \vdashcustom x' \otimes u \equiv x' \otimes v \tm$.
	
	We have the following cases:
	
	\begin{enumerate}[label=(\arabic*)]
		\item $J$ and $J'$ are axioms. As $J \odot J'$ is an axiom, to conclude that it is derivable we must verify that it is well-formed; we will do this by checking that the following judgment is derivable:
		\[
		\tag{\texttt{*}}
		\textbf{X}' \otimes \textbf{Y} \vdashcustom (U \otimes V)[x_1 \otimes u^V \mid x_1y', ..., x_m \otimes u^V \mid x_my'] \equiv (U \otimes V)[x_1 \otimes v^V \mid x_1y', ..., x_m \otimes v^V \mid x_my'] \tp.
		\]
		
		By Lemma \ref{lem: morphism equality from (ctx, term eq)}, letting $\textbf{Y}' = (\textbf{Y}, y':V)$, we can derive a context morphism equality
		$$
		\textbf{X} \otimes (y_1, ..., y_n,u^V) \equiv \textbf{X} \otimes (y_1, ..., y_n,v^V): \textbf{X} \otimes \textbf{Y} \longrightarrow \textbf{X} \otimes \textbf{Y}',
		$$
		which in turn yields an equality
		\[
		\partial(\textbf{X}' \otimes (y_1, ..., y_n,u^V)) \equiv \partial(\textbf{X}' \otimes (y_1, ..., y_n,v^V)): \textbf{X}' \otimes \textbf{Y} \longrightarrow \partial(\textbf{X}' \otimes \textbf{Y}').
		\]
		Now, by taking the pullback of $\partial(\textbf{X}' \otimes \textbf{Y}') \vdashcustom U \otimes V \tp$ along the two morphisms above we obtain (\texttt{*}), as required.
		
		\item $J$ has an initial inference of the form (T-sub): there exist an axiom $\textbf{A} \vdashcustom W \tp$, say where $\textbf{A} = (a_1:A_1, ..., a_M:A_M)$, and a morphism $\textbf{f} = (f_1, ..., f_M):\textbf{X} \rightarrow \textbf{A}$ such that $U = W[\textbf{f}]$ and $\Ht(\textbf{f}:\textbf{X} \rightarrow \textbf{A})$, $\Ht(\textbf{A} \vdashcustom W \tp) < \Ht(\textbf{X} \vdashcustom U \tp)$. Let $\textbf{A}'$ be a context $(\textbf{A}, a':W)$.
		
		By $h$-derivability, we can tensor $\textbf{A} \vdashcustom W \tp$ and $\textbf{Y} \vdashcustom u \equiv v:V$ to derive
		$$
		\textbf{A}' \otimes \textbf{Y} \vdashcustom a' \otimes u \equiv a' \otimes v \tm.
		$$
		Now, by using $Sub^1_{s,t}(h)$ twice we obtain
		$$
		x' \otimes u \equiv \overline{a \otimes u} \equiv \overline{a \otimes u} \equiv x' \otimes v
		$$
		in context $\textbf{X}' \otimes \textbf{Y}$, where the upper bar denotes substitution along $(f_1, ..., f_M,x') \otimes \textbf{Y}:\textbf{X}' \otimes \textbf{Y} \rightarrow \textbf{A}' \otimes \textbf{Y}$.
		
		\item $J'$ has an initial inference
		$$
		\inferrule{\textbf{Y} \vdashcustom u:V}{\textbf{Y} \vdashcustom u \equiv u:V}\text{(t1)}.
		$$
		Then the desired judgment is $\textbf{X}' \otimes \textbf{Y} \vdashcustom x' \otimes u \equiv x' \otimes u \tm$, which is derivable by (t1).
		
		\item $J'$ has an initial inference
		$$
		\inferrule{\textbf{Y} \vdashcustom v \equiv u:V}{\textbf{Y} \vdashcustom u \equiv v:V}\text{(t2)}.
		$$
		As $\Ht(\textbf{X} \vdashcustom U \tp)\Ht(\textbf{Y} \vdashcustom v \equiv u:V) \le h$, we can derive $\textbf{X}' \otimes \textbf{Y} \vdashcustom x' \otimes v \equiv x' \otimes u \tm$; we conclude by applying (t2) to the latter judgment.
		
		\item $J'$ has an initial inference
		$$
		\inferrule{\textbf{Y} \vdashcustom u \equiv u':V \\ \textbf{Y} \vdashcustom u' \equiv v:V}{\textbf{Y} \vdashcustom u \equiv v:V}\text{(t3)}.
		$$
		By $h$-derivability, we can tensor $\textbf{X} \vdashcustom U \tp$ with each of the premises to obtain
		$$
		\textbf{X}' \otimes \textbf{Y} \vdashcustom x' \otimes u \equiv x' \otimes u' \tm \;\;\;\;\;\;\;\;\;\; \textbf{X}' \otimes \textbf{Y} \vdashcustom x' \otimes u' \equiv x' \otimes v \tm.
		$$
		By applying (t3) we derive $\textbf{X}' \otimes \textbf{Y} \vdashcustom x' \otimes u \equiv x' \otimes v \tm$, as required.
		
		\item $J'$ has an initial inference
		$$
		\inferrule{\textbf{Y} \vdashcustom V' \equiv V \tp \\ \textbf{Y} \vdashcustom u \equiv v:V' \\ \textbf{Y} \vdashcustom u:V \\ \textbf{Y} \vdashcustom v:V}{\textbf{Y} \vdashcustom u \equiv v:V}\text{(Teq/teq)}.
		$$
		By $h$-derivability, we can tensor $\textbf{X} \vdashcustom U \tp$ and $\textbf{Y} \vdashcustom u \equiv v:V'$ to obtain
		$$
		\textbf{X}' \otimes \textbf{Y} \vdashcustom x' \otimes u^{V'} \equiv x' \otimes v^{V'} \tm.
		$$
		This is the desired judgment as $x' \otimes u^{V'} = x' \otimes u$ and $x' \otimes v^{V'} = x' \otimes v$.
		
		\item $J'$ has an initial inference
		$$
		\inferrule{\textbf{B} \vdashcustom \alpha \equiv \beta:Z \\ \textbf{Y} \vdashcustom u:V \\ \textbf{Y} \vdashcustom v:V \\ \textbf{g}:\textbf{Y} \rightarrow \textbf{B}}{\textbf{Y} \vdashcustom u \equiv v:V}\text{(teq-sub-1)}
		$$
		where $u = \alpha[\textbf{g}]$, $v = \beta[\textbf{g}]$, and $V = Z[\textbf{g}]$. By $h$-derivability, we can tensor $\textbf{X} \vdashcustom U \tp$ and $\textbf{B} \vdashcustom \alpha \equiv \beta:Z$ to derive
		$$
		\textbf{X}' \otimes \textbf{B} \vdashcustom x' \otimes \alpha \equiv x' \otimes \beta \tm.
		$$
		
		Noting that $\Ht(\textbf{X} \vdashcustom U \tp)\Ht(\textbf{g}:\textbf{Y} \rightarrow \textbf{B})$ and $\Ht(\textbf{X} \vdashcustom U \tp)\Ht(\textbf{B} \vdashcustom \alpha \tm)$ are at most $h$, by $Sub^2_{s,t}(h)$ we can derive $\textbf{X}' \otimes \textbf{Y} \vdashcustom (x' \otimes \alpha)[\textbf{X}' \otimes \textbf{g}] \equiv x' \otimes u \tm$. Similarly we derive $\textbf{X}' \otimes \textbf{Y} \vdashcustom (x' \otimes \beta)[\textbf{X}' \otimes \textbf{g}] \equiv x' \otimes v \tm$, from which it follows that
		$$
		x' \otimes u \equiv (x' \otimes \alpha)[\textbf{X}' \otimes \textbf{g}] \equiv (x' \otimes \beta)[\textbf{X}' \otimes \textbf{g}] \equiv x' \otimes v
		$$
		in context $\textbf{X}' \otimes \textbf{Y}$.
		
		\item $J'$ has an initial inference
		$$
		\inferrule{\textbf{B} \vdashcustom \alpha:Z \\ \textbf{Y} \vdashcustom u:V \\ \textbf{Y} \vdashcustom v:V \\ \textbf{f} \equiv \textbf{g}:\textbf{Y} \rightarrow \textbf{B}}{\textbf{Y} \vdashcustom u \equiv v:V}
		$$
		where $u = \alpha[\textbf{f}]$, $v = \alpha[\textbf{g}]$, and $V = Z[\textbf{f}]$.
		
		By tensoring $\textbf{X} \vdashcustom U \tp$ and $\textbf{B} \vdashcustom \alpha:Z$ we derive $\textbf{X}' \otimes \textbf{B} \vdashcustom x' \otimes \alpha \tm$. On the other hand, as $\Ht(\textbf{X}')\Ht(\textbf{f} \equiv \textbf{g}: \textbf{Y} \rightarrow \textbf{B}) \le h$, by $Moreq^2(h)$ we derive $\textbf{X}' \otimes \textbf{f} \equiv \textbf{X}' \otimes \textbf{g}: \textbf{X}' \otimes \textbf{Y} \longrightarrow \textbf{X}' \otimes \textbf{B}$.
		
		It follows that
		$$
		x' \otimes u \equiv (x' \otimes \alpha)[\textbf{X}' \otimes \textbf{f}] \equiv (x' \otimes \alpha)[\textbf{X}' \otimes \textbf{g}] \equiv x' \otimes v
		$$
		in context $\textbf{X}' \otimes \textbf{Y}$, where the middle equality is obtained by substitution, and the first and third ones by $Mor^2_{s,t}(h)$.
	\end{enumerate}
	This concludes the proof in the case where $J$ is a sort judgment and $J'$ is a term equality judgment.
	
	\vspace{0.5em}
	
	\begin{center}
		\textcolor{newpurple}{\textbf{\normalsize{\underline{term equality $\odot$ sort}}}}
	\end{center}
	
	\vspace{0.2em}
	
	The proof is similar to the previous one. Suppose that $J$ is $\textbf{X} \vdashcustom u \equiv v:U$ and $J'$ is $\textbf{Y} \vdashcustom V \tp$. Writing $\textbf{Y}' = (\textbf{Y}, y':V)$, we want to derive
	$$
	\textbf{X} \otimes \textbf{Y}' \vdashcustom u \otimes y' \equiv v \otimes y': (U \otimes V)[u \otimes y_1 \mid x'y_1, ..., u \otimes y_n \mid x'y_n]
	$$
	where $x'$ is a variable of sort $U$.
	
	By tensoring $\textbf{X} \vdashcustom u:U$ and $\textbf{Y} \vdashcustom V \tp$ we derive
	$$
	\textbf{X} \otimes \textbf{Y}' \vdashcustom u \otimes y':(U \otimes V)[u \otimes y_1 \mid x'y_1, ..., u \otimes y_n \mid x'y_n],
	$$
	so to obtain $J \odot J'$ it suffices to derive $\textbf{X} \otimes \textbf{Y}' \vdashcustom u \otimes y' \equiv v \otimes y' \tm$. We have the following cases:
	\begin{enumerate}[label=(\arabic*)]
		\item $J$ and $J'$ are axioms. To derive $J \odot J'$, which is an axiom, we must verify that it is well-formed; for that, we will derive
		\[
		\textbf{X} \otimes \textbf{Y}' \vdashcustom (U \otimes V)[u \otimes y_1 \mid x'y_1, ..., u \otimes y_n \mid x'y_n] \equiv (U \otimes V)[v \otimes y_1 \mid x'y_1, ..., v \otimes y_n \mid x'y_n] \tp.
		\]
		By Lemma \ref{lem: morphism equality from (term eq,ctx)}, letting $\textbf{X}' = (\textbf{X}, x:U)$, we can derive a morphism equality
		$$
		(x_1, ..., x_m, u) \otimes \textbf{Y} \equiv (x_1, ..., x_m,v): \textbf{X} \otimes \textbf{Y} \longrightarrow \textbf{X}' \otimes \textbf{Y},
		$$
		hence also
		$$
		\partial((x_1, ..., x_m,u) \otimes \textbf{Y}') \equiv \partial((x_1, ..., x_m,v) \otimes \textbf{Y}'): \textbf{X} \otimes \textbf{Y}' \longrightarrow \partial(\textbf{X}' \otimes \textbf{Y}').
		$$
		The desired judgment (\texttt{*}) is now obtained by substitution of $\partial(\textbf{X}' \otimes \textbf{Y}') \vdashcustom U \otimes V \tp$ along the two morphisms above.
		
		\item $J'$ has an initial inference of the form (T-sub). Then there exist an axiom $\textbf{B} \vdashcustom W \tp$, say where $\textbf{B} = (b_1:B_1, ..., b_N:B_N)$, and a morphism $\textbf{g} = (g_1, ..., g_N):\textbf{Y} \rightarrow \textbf{B}$ such that $V = W[\textbf{g}]$ and $\Ht(\textbf{g}:\textbf{Y} \rightarrow \textbf{B})$, $\Ht(\textbf{B} \vdashcustom W \tp) < \Ht(\textbf{Y} \vdashcustom V \tp)$. Let $\textbf{B}'$ be a context $(\textbf{B}, b':W)$.
		
		By $h$-derivability, we can tensor $\textbf{X} \vdashcustom u \equiv v:U$ and $\textbf{B} \vdashcustom W \tp$ to derive
		$$
		\textbf{X} \otimes \textbf{B}' \vdashcustom u \otimes b' \equiv v \otimes b' \tm.
		$$
		We conclude by using $Sub^2_{t,s}$ twice to obtain
		$$
		u \otimes y' \equiv \overline{u \otimes b'} \equiv \overline{v \otimes b'} \equiv v \otimes y'
		$$
		in context $\textbf{X} \otimes \textbf{Y}'$, where the upper bar denotes substitution along $\textbf{X} \otimes (g_1, ..., g_N,y'):\textbf{X} \otimes \textbf{Y}' \rightarrow \textbf{X} \otimes \textbf{B}'$.
		
		\item $J$ has an initial inference
		$$
		\inferrule{\textbf{X} \vdashcustom u:U}{\textbf{X} \vdashcustom u \equiv u:U}\text{(t1)}.
		$$
		Then the desired judgment is $\textbf{X} \otimes \textbf{Y}' \vdashcustom u \otimes y' \equiv u \otimes y' \tm$, which is derivable by (t1).
		
		\item $J$ has an initial inference
		$$
		\inferrule{\textbf{X} \vdashcustom v \equiv u:V}{\textbf{X} \vdashcustom u \equiv v:V}\text{(t2)}.
		$$
		As $\Ht(\textbf{X} \vdashcustom v \equiv u:V)\Ht(\textbf{Y} \vdashcustom V \tp) \le h$, we can derive $\textbf{X} \otimes \textbf{Y}' \vdashcustom u \otimes y' \equiv v \otimes y' \tm$; we conclude by applying (t2) to the latter judgment.
		
		\item $J$ has an initial inference
		$$
		\inferrule{\textbf{X} \vdashcustom u \equiv u':U \\ \textbf{X} \vdashcustom u' \equiv v:V}{\textbf{X} \vdashcustom u \equiv v:V}\text{(t3)}.
		$$
		By $h$-derivability, we can tensor each of the premises with $\textbf{Y} \vdashcustom V \tp$ to obtain
		$$
		\textbf{X} \otimes \textbf{Y}' \vdashcustom u \otimes y' \equiv u' \otimes y' \tm \;\;\;\;\;\;\;\;\;\; \textbf{X} \otimes \textbf{Y}' \vdashcustom u' \otimes y' \equiv v \otimes y' \tm.
		$$
		By applying (t3) we derive $\textbf{X} \otimes \textbf{Y}' \vdashcustom u \otimes y' \equiv v \otimes y' \tm$, as required.
		
		\item $J$ has an initial inference
		$$
		\inferrule{\textbf{X} \vdashcustom U' \equiv U \tp \\ \textbf{X} \vdashcustom u \equiv v:U' \\ \textbf{X} \vdashcustom u:U \\ \textbf{X} \vdashcustom v:U}{\textbf{X} \vdashcustom u \equiv v:U}\text{(Teq/teq)}.
		$$
		By $h$-derivability, we can tensor $\textbf{X} \vdashcustom u \equiv v:U'$ and $\textbf{Y} \vdashcustom V \tp$ to obtain
		$$
		\textbf{X} \otimes \textbf{Y}' \vdashcustom u^{U'} \otimes y' \equiv v^{U'} \otimes y' \tm.
		$$
		This is the desired judgment as $u^{U'} \otimes y' = u \otimes y'$ and $v^{U'} \otimes y' = v \otimes y'$.
		
		\item $J$ has an initial inference
		$$
		\inferrule{\textbf{A} \vdashcustom \alpha \equiv \beta:Z \\ \textbf{X} \vdashcustom u:U \\ \textbf{X} \vdashcustom v:U \\ \textbf{f}:\textbf{X} \rightarrow \textbf{A}}{\textbf{X} \vdashcustom u \equiv v:U}\text{(teq-sub-1)}
		$$
		where $u = \alpha[\textbf{f}]$, $v = \beta[\textbf{g}]$, and $U = Z[\textbf{f}]$. By $h$-derivability, we can tensor $\textbf{A} \vdashcustom \alpha \equiv \beta:Z$ and $\textbf{Y} \vdashcustom V \tp$ to derive
		$$
		\textbf{A} \otimes \textbf{Y}' \vdashcustom \alpha \otimes y' \equiv \beta \otimes y' \tm.
		$$
		Noting that $\Ht(\textbf{f}:\textbf{X} \rightarrow \textbf{A})\Ht(\textbf{Y} \vdashcustom V \tp)$ and $\Ht(\textbf{A} \vdashcustom \alpha \tm)\Ht(\textbf{Y} \vdashcustom V \tp)$ are at most $h$, by $Sub^1_{t,s}(h)$ we can derive $\textbf{X} \otimes \textbf{Y}' \vdashcustom (\alpha \otimes y')[\textbf{f} \otimes \textbf{Y}'] \equiv u \otimes y' \tm$. Similarly we derive $\textbf{X} \otimes \textbf{Y}' \vdashcustom (\beta \otimes y')[\textbf{f} \otimes \textbf{Y}'] \equiv v \otimes y' \tm$, from which it follows that
		$$
		u \otimes y' \equiv (\alpha \otimes y')[\textbf{f} \otimes \textbf{Y}'] \equiv (\beta \otimes y')[\textbf{f} \otimes \textbf{Y}'] \equiv v \otimes y'
		$$
		in context $\textbf{X} \otimes \textbf{Y}'$.
		
		\item $J$ has an initial inference
		$$
		\inferrule{\textbf{A} \vdashcustom \alpha:Z \\ \textbf{X} \vdashcustom u:U \\ \textbf{X} \vdashcustom v:U \\ \textbf{f} \equiv \textbf{g}:\textbf{X} \rightarrow \textbf{A}}{\textbf{X} \vdashcustom u \equiv v:U}
		$$
		where $u = \alpha[\textbf{f}]$, $v = \alpha[\textbf{g}]$, and $U = Z[\textbf{f}]$.
		
		By tensoring $\textbf{A} \vdashcustom \alpha:Z$ and $\textbf{Y} \vdashcustom V \tp$ we derive $\textbf{A} \otimes \textbf{Y}' \vdashcustom \alpha \otimes y' \tm$. On the other hand, as $\Ht(\textbf{f} \equiv \textbf{g}:\textbf{X} \rightarrow \textbf{A})\Ht(\textbf{Y}') \le h$, by $Moreq^1(h)$ we derive $\textbf{f} \otimes \textbf{Y}' \equiv \textbf{g} \otimes \textbf{Y}': \textbf{X} \otimes \textbf{Y}' \rightarrow \textbf{A} \otimes \textbf{Y}'$.
		
		It follows that
		$$
		u \otimes y' \equiv (\alpha \otimes y')[\textbf{f} \otimes \textbf{Y}'] \equiv (\alpha \otimes y')[\textbf{g} \otimes \textbf{Y}'] \equiv v \otimes y'
		$$
		in context $\textbf{X} \otimes \textbf{Y}'$, where the middle equality is obtained by substitution, and the first and third ones by $Sub^1_{t,s}(h)$.
	\end{enumerate}

\vspace{0.5em}

This concludes the proof of Proposition \ref{prop: induction step}. By induction, we obtain the main result of the text:

\begin{theorem}
\label{th: tensor product is a theory}
Let $\bbA$ and $\bbB$ be generalized algebraic theories. Then $(\bbA,\bbB)$ is $h$-derivable for all $h \ge 0$. In particular, the pretheory $\bbA \otimes \bbB$ is a theory.
\end{theorem}

\section{Two-sided substitution and the functor $\mathcal C(\bbA) \times \mathcal C(\bbB) \rightarrow \mathcal C(\bbA \otimes \bbB)$}

\label{sec: comparison functor}

Let $\bbA$ and $\bbB$ be generalized algebraic theories. We have proved (Theorem \ref{th: tensor product is a theory}) that $\bbA \otimes \bbB$ is a theory; in fact, that $J \odot J'$ is derivable whenever $J$, $J'$ are derivable judgments in $\bbA$, $\bbB$, respectively. Also, all the statements from \S\ref{sec: consequences h-derivability} hold without reference to the parameter $h$; we will thus say, for example, that $(\bbA,\bbB)$ satisfies $Sub^1_{t,t}$, meaning that it satisfies $Sub^1_{t,t}(h)$ for all $h \ge 0$, and similarly for the other conditions in that section.

We will now consider tensor products of context morphisms and describe a form of commutativity between such tensor products and the substitution operation (which we refer to as ``two-sided substitution"). Using that, we will define a comparison functor $\mathcal C(\bbA) \times \mathcal C(\bbB) \rightarrow \mathcal C(\bbA \otimes \bbB)$ between the corresponding contextual categories.

We refer the reader to \cite{Car86} (and, for more detail, to \cite{Car78}) for some aspects of generalized algebraic theories not covered in the appendix (and which do not depend on the differences between Cartmell's presentation and ours). Notably, we have a category $\GAT$ of generalized algebraic theories and interpretations, and a category $\Cont$ of contextual categories and contextual functors.

Each theory $\bbA$ has an associated contextual category $\mathcal C(\bbA)$ where, in particular, objects are equivalence classes (with respect to provable equality) of contexts, and whose arrows are equivalence classes of context morphisms. These assemble into a functor
$$
\mathcal C:\GAT \longrightarrow \Cont,
$$
which is proved in \cite{Car78} to be an equivalence of categories.

Given a context $\textbf{X}$ in $\bbA$, the corresponding object in $\mathcal C(\bbA)$ will be denoted by $[\textbf{X}]$. Similarly, for a context morphism $\textbf{f}:\textbf{X} \rightarrow \textbf{Y}$ we write $[\textbf{f}]$ for the induced arrow from $[\textbf{X}]$ to $[\textbf{Y}]$.

\begin{lemma}
\label{lem: tensor product of morphisms}
Suppose given context morphisms
$$
\textbf{f} = (f_1, ..., f_M): \textbf{X} = (x_1:X_1, ..., x_m:X_m) \longrightarrow \textbf{A} = (a_1:A_1, ..., a_M:A_M)
$$
in $\bbA$ and
$$
\textbf{g} = (g_1, ..., g_N): \textbf{Y} = (y_1:Y_1, ..., y_n:Y_n) \longrightarrow \textbf{B} = (b_1:B_1, ..., b_N:B_N)
$$
in $\bbB$. Then $\textbf{f} \otimes \textbf{g}:\textbf{X} \otimes \textbf{Y} \rightarrow \textbf{A} \otimes \textbf{B}$ is derivable. Moreover, it is equal to both composites in the square
\[
\squa{\textbf{X} \otimes \textbf{Y}}{\textbf{A} \otimes \textbf{Y}}{\textbf{X} \otimes \textbf{B}}{\textbf{A} \otimes \textbf{B}.}{\textbf{f} \otimes \textbf{Y}}{\textbf{f} \otimes \textbf{B}}{\textbf{X} \otimes \textbf{g}}{\textbf{A} \otimes \textbf{g}}
\]
Here, we have used $Mor^1$ and $Mor^2$ to conclude that these four sequences are morphisms with the indicated co/domains.
\end{lemma}

\begin{proof}
We check that the entries of $\textbf{f} \otimes \textbf{g}$ are provably equal in context $\textbf{X} \otimes \textbf{Y}$ to the corresponding entries in $(\textbf{A} \otimes \textbf{g}) \circ (\textbf{f} \otimes \textbf{Y})$ and, similarly, in $(\textbf{f} \otimes \textbf{B}) \circ (\textbf{X} \otimes \textbf{g})$. Either statement implies, in particular, that $\textbf{f} \otimes \textbf{g}$ is a morphism from $\textbf{X} \otimes \textbf{Y}$ to $\textbf{A} \otimes \textbf{B}$.

Let $1 \le i \le M$ and $1 \le j \le N$. The $(i,j)$-entry of the matrix form of $(\textbf{A} \otimes \textbf{g}) \circ (\textbf{f} \otimes \textbf{Y})$ is $(a_i \otimes g_j)[\textbf{f} \otimes \textbf{Y}]$. By $Sub^1_{t,t}$, the latter is provably equal in $\textbf{X} \otimes \textbf{Y}$ to $f_i \otimes g_j$, which is the $(i,j)$-entry of $\textbf{f} \otimes \textbf{g}$.

The other case follows similarly from $Sub^2_{t,t}$.
\end{proof}

\begin{lemma}
\label{lem: one-sided composition}
Consider context morphisms
$$
\textbf{X} \overset{\textbf{f}}{\longrightarrow} \textbf{X}' \overset{\textbf{f}'}{\longrightarrow} \textbf{X}'', \qquad\qquad\qquad \textbf{Y} \overset{\textbf{g}}{\longrightarrow} \textbf{Y}' \overset{\textbf{g}'}{\longrightarrow} \textbf{Y}''
$$
in $\bbA$ and $\bbB$, respectively. Then the composites
$$
\textbf{X} \otimes \textbf{Y} \overset{\textbf{f} \otimes \textbf{Y}}{\longrightarrow} \textbf{X}' \otimes \textbf{Y} \overset{\textbf{f}' \otimes \textbf{Y}}{\longrightarrow} \textbf{X}'' \otimes \textbf{Y}, \qquad\qquad\qquad \textbf{X} \otimes \textbf{Y}' \overset{\textbf{X} \otimes \textbf{g}}{\longrightarrow} \textbf{X} \otimes \textbf{Y}' \overset{\textbf{X} \otimes \textbf{g}'}{\longrightarrow} \textbf{X} \otimes \textbf{Y}''
$$
are provably equal in $\bbA \otimes \bbB$ to $(\textbf{f'} \circ \textbf{f}) \otimes \textbf{Y}:\textbf{X} \otimes \textbf{Y} \rightarrow \textbf{X}'' \otimes \textbf{Y}$ and to $\textbf{X} \otimes (\textbf{g}' \circ \textbf{g}):\textbf{X} \otimes \textbf{Y} \rightarrow \textbf{X} \otimes \textbf{Y}''$, respectively.
\end{lemma}

\begin{proof}
The proof is similar to that of the previous lemma: the $(i,j)$-entry of the left-hand side composite is $(f'_i \otimes y_j)[\textbf{f} \otimes \textbf{Y}]$, which by $Sub^1_{t,t}$ is provably equal in context $\textbf{X} \otimes \textbf{Y}$ to $f'_i[\textbf{f}] \otimes y_j = (\textbf{f}' \circ \textbf{f})_i \otimes y_j$. The second case is verified similarly from $Sub^2_{t,t}$.
\end{proof}

\begin{proposition}
\label{prop: two-sided substitution (morphism version)}
In the setting of Lemma \ref{lem: one-sided composition}, the morphisms
$$
(\textbf{f}' \otimes \textbf{g}') \circ (\textbf{f} \otimes \textbf{g}), \qquad\qquad (\textbf{f}' \circ \textbf{f}) \otimes (\textbf{g}' \circ \textbf{g})
$$
from $\textbf{X} \otimes \textbf{Y}$ to $\textbf{X}'' \otimes \textbf{Y}''$ are provably equal.
\end{proposition}

\begin{proof}
By the two lemmas above, we have a diagram of context morphisms
\[\begin{tikzcd}
	{\textbf{X} \otimes \textbf{Y}} && {\textbf{X}' \otimes \textbf{Y}} && {\textbf{X}'' \otimes \textbf{Y}} \\
	\\
	&& {\textbf{X}' \otimes \textbf{Y}'} && {\textbf{X}'' \otimes \textbf{Y}'} \\
	\\
	&&&& {\textbf{X}'' \otimes \textbf{Y}''}
	\arrow["{\textbf{f} \otimes \textbf{Y}}"{description}, from=1-1, to=1-3]
	\arrow["{(\textbf{f}' \circ \textbf{f}) \otimes \textbf{Y}}", curve={height=-24pt}, from=1-1, to=1-5]
	\arrow["{f \otimes g}"', from=1-1, to=3-3]
	\arrow["{\textbf{f}' \otimes \textbf{Y}}"{description}, from=1-3, to=1-5]
	\arrow["{\textbf{X}' \otimes \textbf{g}}"{description}, from=1-3, to=3-3]
	\arrow["{\textbf{X}'' \otimes \textbf{g}}"{description}, from=1-5, to=3-5]
	\arrow["{\textbf{X}'' \otimes (\textbf{g}' \circ \textbf{g})}", curve={height=-24pt}, from=1-5, to=5-5]
	\arrow["{\textbf{f}' \otimes \textbf{Y}'}"{description}, from=3-3, to=3-5]
	\arrow["{f' \otimes g'}"', from=3-3, to=5-5]
	\arrow["{\textbf{X}'' \otimes \textbf{g}''}"{description}, from=3-5, to=5-5]
\end{tikzcd}\]
that commutes up to provable equality. Also, by Lemma \ref{lem: tensor product of morphisms} we can derive a morphism equality
$$
\big(\textbf{X}'' \otimes (\textbf{g}' \circ \textbf{g})\big) \circ \big( (\textbf{f}' \otimes \textbf{f}) \otimes \textbf{Y} \big) \equiv (\textbf{g}' \circ \textbf{g}) \otimes (\textbf{f}' \otimes \textbf{f}).
$$
\end{proof}

\begin{proposition}[Two-sided substitution]
\label{prop: two-sided substitution}
Suppose given a derivable judgment $\textbf{A} \vdashcustom u \tm$ and a context morphism $\textbf{f}:\textbf{X} \rightarrow \textbf{A}$ in $\bbA$ and, similarly, $\textbf{B} \vdashcustom v \tm$ and $\textbf{g}:\textbf{Y} \rightarrow \textbf{B}$ in $\bbB$. Then the judgment
$$
\textbf{X} \otimes \textbf{Y} \vdashcustom (u \otimes v)[\textbf{f} \otimes \textbf{g}] \equiv u[\textbf{f}] \otimes v[\textbf{g}] \tm
$$
is derivable.
\end{proposition}

\begin{proof}
By Lemma \ref{lem: tensor product of morphisms}, $Sub^1$ and $Sub^2$, we have the following derivable equalities in context $\textbf{X} \otimes \textbf{Y}$:
\begin{align*}
	(u \otimes v)[\textbf{f} \otimes \textbf{g}] & \equiv (u \otimes v)[(\textbf{f} \otimes \textbf{B}) \circ (\textbf{X} \otimes \textbf{g})]\\
	& \equiv (u \otimes v)[\textbf{f} \otimes \textbf{B}][\textbf{X} \otimes \textbf{g}]\\
	& \equiv (u[\textbf{f}] \otimes v)[\textbf{X} \otimes \textbf{g}]\\
	& \equiv u[\textbf{f}] \otimes v[\textbf{g}].
\end{align*}
\end{proof}

\begin{construction}
\label{const: universal functor}
We define a functor $\otimes_{\bbA,\bbB}:\mathcal C(\bbA) \times \mathcal C(\bbB) \rightarrow \mathcal C(\bbA \otimes \bbB)$, also denoted by $\otimes$ by abuse of notation, as follows:
\begin{itemize}
	\item Given derivable equalities $\textbf{X} \equiv \textbf{X}' \ctx$ in $\bbA$ and $\textbf{Y} \equiv \textbf{Y}' \ctx$ in $\bbB$, from $Conteq^1$ and $Conteq^2$ we obtain
	$$
	\textbf{X} \otimes \textbf{Y} \equiv \textbf{X}' \otimes \textbf{Y} \equiv \textbf{X}' \otimes \textbf{Y}'.
	$$
	We then let $[\textbf{X}] \otimes [\textbf{Y}] = [\textbf{X} \otimes \textbf{Y}]$.
	
	\item Given derivable equalities $\textbf{f} \otimes \textbf{f}':\textbf{X} \rightarrow \textbf{A}$ in $\bbA$ and $\textbf{g} \equiv \textbf{g}':\textbf{Y} \rightarrow \textbf{B}$ in $\bbB$, we obtain an equality
	$$
	\textbf{f} \otimes \textbf{g} \equiv (\textbf{A} \otimes \textbf{g}) \circ (\textbf{f} \otimes \textbf{Y}) \equiv (\textbf{A} \otimes \textbf{g}') \circ (\textbf{f}' \otimes \textbf{Y}) \equiv \textbf{f}' \otimes \textbf{g}
	$$
	between morphisms from $\textbf{X} \otimes \textbf{Y}$ to $\textbf{A} \otimes \textbf{B}$ by using $Moreq^1$, $Moreq^2$ and Proposition \ref{lem: tensor product of morphisms}.
	
	We let $[\textbf{f}] \otimes [\textbf{g}]:[\textbf{X}] \otimes [\textbf{Y}] \rightarrow [\textbf{A}] \otimes [\textbf{B}]$ be $[\textbf{f} \otimes \textbf{g}]$.
	
	\item For functionality, note firstly that the equality $(x_1, ..., x_m) \otimes (y_1, ..., y_n) = (x_1y_1, ..., x_my_n)$ yields $id_{[\textbf{X}]} \otimes id_{[\textbf{Y}]} = id_{[\textbf{X} \otimes \textbf{Y}]} = id_{[\textbf{X}] \otimes [\textbf{Y}]}$. On the other hand, preservation of composition follows from Proposition \ref{prop: two-sided substitution (morphism version)}.
\end{itemize}
\end{construction}

\begin{remark}
\label{rem: functoriality particular case}
Suppose that we extend $\bbA$ and $\bbB$ to theories $\bbA'$ and $\bbB'$, respectively, by adding a set of axioms (including, possibly, new symbols). It can be checked by induction on the height that all the operations of the form $- \otimes -$ between structures in $\bbA$ and $\bbB$ (such as tensoring two judgments, contexts, or context morphisms) are identical to the restrictions to structures in $\bbA$ and $\bbB$ of the corresponding operations for $\bbA'$ and $\bbB'$. This implies a weak form of naturality of the functors given by Construction \ref{const: universal functor}: the diagram
\[
\widesqua{\mathcal C(\bbA) \times \mathcal C(\bbB)}{\mathcal C(\bbA') \times \mathcal C(\bbB')}{\mathcal C(\bbA \otimes \bbB)}{\mathcal C(\bbA' \otimes \bbB')}{\mathcal C(I) \times \mathcal C(I)}{\mathcal C(I)}{\otimes_{\bbA,\bbB}}{\otimes_{\bbA',\bbB'}}
\]
commutes where $I$ is, in each case, the identity interpretation.
\end{remark}

\begin{remark}
\label{rem: towards functoriality}
Consider the following statement: given morphisms of theories $F:\bbA \rightarrow \bbA'$ and $G:\bbB \rightarrow \bbB'$, there exists a unique contextual functor $H:\mathcal C(\bbA \otimes \bbB) \rightarrow \mathcal C(\bbA' \otimes \bbB')$ such that
\[
\widesqua{\mathcal C(\bbA) \times \mathcal C(\bbB)}{\mathcal C(\bbA') \times \mathcal C(\bbB')}{\mathcal C(\bbA \otimes \bbB)}{\mathcal C(\bbA' \otimes \bbB')}{\mathcal C(F) \times \mathcal C(G)}{H}{\otimes_{\bbA,\bbB}}{\otimes_{\bbA',\bbB'}}
\]
commutes. This would immediately turn the tensor product of theories into a functor $\otimes:\GAT \times \GAT \rightarrow \GAT$. However, giving a purely syntactic proof of the above statement seems to be a laborious task due to the recursive nature of morphisms of \textsc{gat}s (as equivalence classes of interpretations) is; compare, for example, with the definition of a contextual functor. We will come back to this discussion in \cite{Alm26}.
\end{remark}

\section{Towards associativity}

In this section we outline a description of an isomorphism $(\bbA \otimes \bbB) \otimes \bbC \cong \bbA \otimes (\bbB \otimes \bbC)$ for arbitrary \textsc{gat}s $\bbA$, $\bbB$, $\bbC$. (But we only see this as associativity in a restricted sense; see Remark \ref{rem: no monoidal structure now}.)

To start, observe that the tensor product of sort-and-term alphabets from Definition \ref{def: tensor product of alphabets} is associative. Given $\Sigma_1$, $\Sigma_2$ and $\Sigma_3$, the alphabet $(\Sigma_1 \otimes \Sigma_2) \otimes \Sigma_3$ has a set of variables $(\Sigma_1^{\text{var}} \times \Sigma_2^{\text{var}}) \times \Sigma_3^{\text{var}}$, a set of sort symbols $(\Sigma_1^{\text{sort}} \times \Sigma_2^{\text{sort}}) \times \Sigma_3^{\text{sort}}$, and a set of term symbols $(\Sigma_1 \otimes \Sigma_2)^{\text{sort}} \times \Sigma_3^{\text{term}}\;\; \sqcup \;\; (\Sigma_1 \otimes \Sigma_2)^{\text{term}} \times \Sigma_3^{\text{sort}}$, which equals
$$
(\Sigma_1^{\text{sort}} \times \Sigma_2^{\text{sort}}) \times \Sigma_3^{\text{term}} \;\; \sqcup \;\; (\Sigma_1^{\text{sort}} \times \Sigma_2^{\text{term}} \sqcup \Sigma_1^{\text{term}} \times \Sigma_2^{\text{sort}}) \times \Sigma_3^{\text{term}}.
$$
Similarly, $\Sigma_1 \otimes (\Sigma_2 \otimes \Sigma_3)$ has a set of variables $\Sigma_1^{\text{var}} \times (\Sigma_2^{\text{var}} \times \Sigma_3^{\text{var}})$, a set of sort symbols $\Sigma_1^{\text{sort}} \times (\Sigma_2^{\text{sort}} \times \Sigma_3^{\text{sort}})$, and a sort of term symbols $\Sigma_1^{\text{sort}} \times (\Sigma_2^{\text{sort}} \times \Sigma_3^{\text{term}} \sqcup\Sigma_2^{\text{term}} \times \Sigma_3^{\text{sort}}) \;\;\sqcup\;\; \Sigma_1^{\text{term}} \times (\Sigma_2^{\text{sort}} \times \Sigma_3^{\text{sort}})$. Each of the three sets in the first case is canonically isomorphic to the corresponding set in the second one; hence we will identify both alphabets with the alphabet $\Sigma_1 \otimes \Sigma_2 \otimes \Sigma_3$ defined as having a set of variables $\Sigma_1^{\text{var}} \times \Sigma_2^{\text{var}} \times \Sigma_3^{\text{var}}$, a set of sort symbols $\Sigma_1^{\text{sort}} \times \Sigma_2^{\text{sort}} \times \Sigma_3^{\text{sort}}$, and a set of term symbols
$$
(\Sigma_1^{\text{sort}} \times \Sigma_2^{\text{sort}} \times \Sigma_3^{\text{term}}) \sqcup (\Sigma_1^{\text{sort}} \times \Sigma_2^{\text{term}} \times \Sigma_3^{\text{sort}}) \sqcup (\Sigma_1^{\text{term}} \times \Sigma_2^{\text{sort}} \times \Sigma_3^{\text{sort}}).
$$
We will denote a triple in one of these sets by concatenating the entries; for example, given $R \in \Sigma_1^{\text{sort}}$, $S \in \Sigma_2^{\text{sort}}$ and $t \in \Sigma_3^{\text{term}}$, we write $RSt$ for $(R,S,t)$.

For theories $\bbA$, $\bbB$ and $\bbC$, we will abuse notation and view $(\bbA \otimes \bbB) \otimes \bbC$ and $\bbA \otimes (\bbB \otimes \bbC)$ as having the same alphabet, namely, $\Sigma(\bbA) \otimes \Sigma(\bbB) \otimes \Sigma(\bbC)$. More precisely, we will also write $(\bbA \otimes \bbB) \otimes \bbC$ and $\bbA \otimes (\bbB \otimes \bbC)$ for the theories in the alphabet $\Sigma(\bbA) \otimes \Sigma(\bbB) \otimes \Sigma(\bbC)$ obtained by translating the respective sets of axioms (hence of derivable judgments) along the canonical bijections between sets of symbols. Under this convention, our goal will be to prove that these two theories have the same derivable judgments.

\subsection{Preliminary notations}

Before discussing the associativity of the tensor product of theories, we introduce a syntactic structure that will be useful for describing the tensor product of a sequence of expressions coming from two or more theories.

\begin{notation}
Let $\bbA$ be a theory. For a derivable judgment $\textbf{X} \vdashcustom U \tp$, where $U = T(\alpha_1, ..., \alpha_k)$, and a variable $x$ not occurring in $\textbf{X}$, we let $U\{x\}$ be the pair $(T(\alpha_1, ..., \alpha_k,x), \; \textbf{X})$. We view $\textbf{X}$ as a label on the expression $T(\alpha_1, ..., \alpha_k,x)$ that, in particular, endows each variable occurring in it with a sort: the variables in $\alpha_1$, ..., $\alpha_k$ have the sorts specified by $\textbf{X}$, and $x$ has sort $U$.

Similarly, for derivable $\textbf{X} \vdashcustom u:U$ we write $U\{u\}$ for $(T(\alpha_1, ..., \alpha_k,u), \; \textbf{X})$.

In either case, we omit the context $\textbf{X}$ from the notation when it is implicit.

Note that $U$ can be recovered from $U\{x\} = T(\alpha_1, ..., \alpha_k,x)$ or $U\{u\} = T(\alpha_1, ..., \alpha_k,u)$ by removing the last entry between parentheses. However, it will be useful for performing calculations to denote $U\{x\}$ (resp. $U\{u\}$) by $A(\alpha_1, ..., \alpha_k,x^U)$ (resp. $A(\alpha_1, ..., \alpha_k,u^U)$).
\end{notation}

The usefulness of this notation comes from the fact that, in many cases, such expressions can be naturally used to describe the operations introduced in \S\ref{sec: tensor product of generalized algebraic theories}. To understand that, we introduce a further set of notations.

\vspace{0.5em}

Let $\bbA$ and $\bbB$ be theories, and consider expressions $U\{u\}$, $V\{v\}$ of one of the two forms described above with respect to $\bbA$, $\bbB$, respectively. That is, we are given
\begin{itemize}
	\item in $\bbA$, either a derivable judgment $\textbf{X} \vdashcustom u:U$, or a derivable judgment $\textbf{X} \vdashcustom U \tp$ such that $u$ is the chosen variable not occurring in $\textbf{X}$; and
	
	\item in $\bbB$, either a derivable judgment $\textbf{Y} \vdashcustom v:V$, or a derivable judgment $\textbf{Y} \vdashcustom V \tp$ such that $v$ is the chosen variable not in $\textbf{Y}$.
\end{itemize}
Write $U = R(\alpha_1, ..., \alpha_M)$ and $V = S(\beta_1, ..., \beta_N)$. Then $U\{u\} = R(\alpha_1, ..., \alpha_m,u^U)$ and $V\{v\} = S(\beta_1, ..., \beta_n,v^V)$, and we define their ``box product" as
$$
U\{u\} \btimes V\{v\} := RS((\alpha_1, ..., \alpha_m,u^U) \otimes (\beta_1, ..., \beta_N,v^V)) =
RS
\begin{pmatrix}
\alpha_1 \otimes \beta_1 & \cdots & \alpha_1 \otimes \beta_N & \alpha_1 \otimes v^V\\
\vdots & \ddots & \vdots & \vdots\\
\alpha_M \otimes \beta_1 & \cdots & \alpha_M \otimes \beta_N & \alpha_M \otimes v^V\\
u^U \otimes \beta_1 & \cdots & u^U \otimes \beta_N & u^U \otimes v^V
\end{pmatrix}.
$$
We also consider the following variants, where we replace some instances of $\otimes$ by $\varotimes$:
\begin{align*}
	U\{u\} \varbtimes V\{v\} :=
	RS
	\begin{pmatrix}
		\alpha_1 \otimes \beta_1 & \cdots & \alpha_1 \otimes \beta_N & \alpha_1 \varotimes v^V\\
		\vdots & \ddots & \vdots & \vdots\\
		\alpha_M \otimes \beta_1 & \cdots & \alpha_M \otimes \beta_N & \alpha_M \varotimes v^V\\
		u^U \otimes \beta_1 & \cdots & u^U \otimes \beta_N & u^U \varotimes v^V
	\end{pmatrix},\\
	&\\
	U\{u\} \uphalfsquare V\{v\} :=
	RS
	\begin{pmatrix}
		\alpha_1 \otimes \beta_1 & \cdots & \alpha_1 \otimes \beta_N & \alpha_1 \varotimes v^V\\
		\vdots & \ddots & \vdots & \vdots\\
		\alpha_M \otimes \beta_1 & \cdots & \alpha_M \otimes \beta_N & \alpha_M \varotimes v^V\\
		u^U \otimes \beta_1 & \cdots & u^U \otimes \beta_N & u^U \otimes v^V
	\end{pmatrix}.
\end{align*}

Similarly, if $u$ is not a variable, say $u = s(a_1, ..., a_m)$, we let
\begin{align*}
	u \btimes V\{v\} & := rS((a_1, ..., a_m) \otimes (\beta_1, ..., \beta_N,v^V)) = rS
	\begin{pmatrix}
		a_1 \otimes \beta_1 & \cdots & a_1 \otimes \beta_N & a_1 \otimes v^V\\
		\vdots & \ddots & \vdots & \vdots\\
		a_m \otimes \beta_1 & \cdots & a_m \otimes \beta_N & a_m \otimes v^V
	\end{pmatrix},\\
	&\\
	u \varbtimes V\{v\} & := rS((a_1, ..., a_m) \varotimes (\beta_1, ..., \beta_N,v^V)) := rS
	\begin{pmatrix}
		a_1 \otimes \beta_1 & \cdots & a_1 \otimes \beta_N & a_1 \varotimes v^V\\
		\vdots & \ddots & \vdots & \vdots\\
		a_m \otimes \beta_1 & \cdots & a_m \otimes \beta_N & a_m \varotimes v^V
	\end{pmatrix},
\end{align*}

and when $v$ is not a variable, say $v = s(b_1, ..., b_n)$, we let

\begin{align*}
	U\{u\} \btimes v & := Rs((\alpha_1, ..., \alpha_m,u^U) \otimes (b_1, ..., b_n)) = Rs
	\begin{pmatrix}
		\alpha_1 \otimes b_1 & \cdots & \alpha_1 \otimes b_n\\
		\vdots & \ddots & \vdots\\
		\alpha_M \otimes b_1 & \cdots & \alpha_M \otimes b_n\\
		u^U \otimes b_1 & \cdots & u^U \otimes b_n
	\end{pmatrix},\\
	&\\
	U\{u\} \varbtimes v & := Rs((\alpha_1, ..., \alpha_m,u^U) \varotimes (b_1, ..., b_n)) := Rs
	\begin{pmatrix}
		\alpha_1 \otimes b_1 & \cdots & \alpha_1 \varotimes b_n\\
		\vdots & \ddots & \vdots\\
		\alpha_M \otimes b_1 & \cdots & \alpha_M \varotimes b_n\\
		u^U \otimes b_1 & \cdots & u^U \varotimes b_n
	\end{pmatrix}.
\end{align*}

These expressions allow us to present tensor products between sorts and terms in a way that can be easily iterated. For example, if $u$ and $v$ are not variables, then
$$
u^U \otimes v^V = U\{u\} \btimes v, \qquad\qquad u^U \varotimes v^V = u \varbtimes V\{v\},
$$
and the tensor product of $\textbf{X} \vdashcustom u:U$ and $\textbf{Y} \vdashcustom v:V$ equals
$$
\textbf{X} \otimes \textbf{Y} \vdashcustom U\{u\} \btimes v \equiv u \varbtimes V\{v\}: \partial(U\{u\} \uphalfsquare V\{v\}).
$$

\begin{remark}
Denoting by $K$ the sort $\partial(U\{u\} \uphalfsquare V\{v\}) = \partial(U\{u\} \varbtimes V\{v\})$ in the above judgment, we have
$$
K\{u^U \otimes v^V\} = U\{u\} \uphalfsquare V\{v\}, \qquad\qquad K\{u^U \varotimes v^V\} = U\{u\} \varbtimes V\{v\}.
$$
This exemplifies a general compatibility, which will be the key for verifying associativity of the tensor product of theories, between the box product and the tensor product of judgments.
\end{remark}

Let us use the above formalism to describe other combinations of judgments.

The tensor product of $\textbf{X} \vdashcustom U \tp$ and $\textbf{Y} \vdashcustom V \tp$ is
$$
\partial(\textbf{X}' \otimes \textbf{Y}') \vdashcustom U \otimes V \tp
$$
where $\textbf{X}' = (\textbf{X}, x:U)$, $\textbf{Y}' = (\textbf{Y}, y:V)$, and
$$
U \otimes V = RS
\begin{pmatrix}
	\alpha_1 \otimes \beta_1 & \cdots & \alpha_1 \otimes \beta_N & \alpha_1 \otimes y^V\\
	\vdots & \ddots & \vdots & \vdots\\
	\alpha_M \otimes \beta_1 & \cdots & \alpha_M \otimes \beta_N & \alpha_M \otimes y^V\\
	x^U \otimes \beta_1 & \cdots & x^U \otimes \beta_N & -
\end{pmatrix}
= \partial(U\{x\} \btimes V\{y\}).
$$

This implies that $(U \otimes V)\{xy\} = U\{x\} \btimes V\{y\}$.

Similarly, it can be checked that the tensor product of $\textbf{X} \vdashcustom u:U$ and $\textbf{Y} \vdashcustom V \tp$ is
$$
\textbf{X} \otimes \textbf{Y}' \vdashcustom u \btimes V\{y\}: \partial(U\{u\} \btimes V\{y\}),
$$
and that of $\textbf{X} \vdashcustom U \tp$ and $\textbf{Y} \vdashcustom v:V$ is
$$
\textbf{X}' \otimes \textbf{Y} \vdashcustom U\{x\} \btimes v: \partial(U\{x\} \varbtimes V\{v\}).
$$

These formulas can be extended in a straightforward way to ones for describing $J \odot J'$ when $J$ or $J'$ is an equality judgment.

\vspace{0.5em}

We now state the main result of this section:

\begin{theorem}
Let $\bbA$, $\bbB$ and $\bbC$ be generalized algebraic theories. Then $(\bbA \otimes \bbB) \otimes \bbC$ and $\bbA \otimes (\bbB \otimes \bbC)$, both viewed over the alphabet $\Sigma(\bbA) \otimes \Sigma(\bbB) \otimes \Sigma(\bbC)$, have the same derivable judgments. As a consequence, we obtain an isomorphism $(\bbA \otimes \bbB) \otimes \bbC \cong \bbA \otimes (\bbB \otimes \bbC)$.
\end{theorem}

We will sketch a proof that if $J$, $J'$ and $J''$ are derivable judgments in $\bbA$, $\bbB$ and $\bbC$, respectively, then $(J \odot J') \odot J''$ is derivable in $\bbA \otimes (\bbB \otimes \bbC)$. By considering axioms $J$, $J'$ and $J''$, we conclude that every axiom in $(\bbA \otimes \bbB) \otimes \bbC$ is derivable in $\bbA \otimes (\bbB \otimes \bbC)$, hence that every judgment derivable in $(\bbA \otimes \bbB) \otimes \bbC$ is also derivable in $\bbA \otimes (\bbB \otimes \bbC)$. It can be verified analogously that, conversely, if a judgment is derivable in $\bbA \otimes (\bbB \otimes \bbC)$, then it is derivable in $(\bbA \otimes \bbB) \otimes \bbC$.

\vspace{0.5em}

Let us fix some terminology: in what follows, by \emph{derivable} or \emph{provably equal} we mean so in $\bbA \otimes (\bbB \otimes \bbC)$ whenever we are dealing with judgments or contexts/expressions in $\Sigma(\bbA) \otimes \Sigma(\bbB) \otimes \Sigma(\bbC)$. We also use the following convention: given derivable judgments $J$, $J'$ in the alphabet of $\bbA \otimes (\bbB \otimes \bbC)$, we write $J \sim J'$ when $J$, $J'$ have the same interpretation in the contextual category $\mathcal C = \mathcal C(\bbA \otimes (\bbB \otimes \bbC))$. Precisely, one of the following holds:
\begin{itemize}
	\item[--] $J$, $J'$ are sort judgments that correspond to the same display map in $\mathcal C$.
	
	\item[--] $J$, $J'$ are term judgments that correspond to the same section of a display map in $\mathcal C$.
	
	\item[--] $J$, $J'$ are sort equality judgments, say $\textbf{X} \vdashcustom U \equiv U' \tp$ and $\textbf{Y} \vdashcustom V \equiv V' \tp$, respectively, such that $(\textbf{X} \vdashcustom U \tp) \sim (\textbf{Y} \vdashcustom V \tp)$ and $(\textbf{X} \vdashcustom U' \tp) \sim (\textbf{Y} \vdashcustom V' \tp)$.
	
	\item[--] $J$, $J'$ are term equality judgments, say $\textbf{X} \vdashcustom u \equiv u':U$ and $\textbf{Y} \vdashcustom v \equiv v':V$, respectively, such that $(\textbf{X} \vdashcustom u:U) \sim (\textbf{Y} \vdashcustom v:V)$ and $(\textbf{X} \vdashcustom u':U) \sim (\textbf{Y} \vdashcustom v':V)$.
\end{itemize}

In other words, $J = (\textbf{X} \vdashcustom \blacksquare$) is derivable and $J'$ can be obtained from $J$ by a combination of the following kinds of operations: replacing the context $\textbf{X}$ by a provably equal one, and replacing a sort/term expression in $\blacksquare$ by a provably equal one.

\vspace{0.5em}

We check by induction on $h \ge 0$ that the following three statements hold:

\begin{enumerate}[label=(A\arabic*)]
	\item If $\Ht(J)\Ht(J')\Ht(J'') = h$ and $(J \odot J') \odot J''$ is derivable, then so is $J \odot (J' \odot J'')$, and $(J \odot J') \odot J'' \Rightarrow J \odot (J' \odot J'')$.
	
	\item Given contexts $\textbf{X}$, $\textbf{Y}$, $\textbf{Z}$ in $\bbA$, $\bbB$, $\bbC$, resp., such that $\Ht(\textbf{X})\Ht(\textbf{Y})\Ht(\textbf{Z}) = h$, the judgment
	$$
	(\textbf{X} \otimes \textbf{Y}) \otimes \textbf{Z} \equiv \textbf{X} \otimes (\textbf{Y} \otimes \textbf{Z}) \ctx
	$$
	is derivable.
	
	\item Consider derivable judgments $J$, $J'$, $J''$ in $\bbA$, $\bbB$, $\bbC$, respectively, each of which is a sort judgment or a term judgment, and such that $\Ht(J)\Ht(J')\Ht(J'') = h$. Let $\textbf{K}$ be the context of the judgment $(J \odot J') \odot J''$. Then, in the notation of \ref{subsec: expressions from pairs of judgments}, the following terms are derivable and provably equal to each other in $\textbf{K}$:
	\begin{align*}
		(J \otimes_t J') \otimes_t J'', &\qquad J \otimes_t (J' \otimes_t J''),\\
		(J \otimes_t J') \varotimes_t J'', &\qquad J \otimes_t (J' \varotimes_t J''),\\
		(J \varotimes_t J') \otimes_t J'', &\qquad J \varotimes_t (J' \otimes_t J''),\\
		(J \varotimes_t J') \varotimes_t J'', &\qquad J \varotimes_t (J' \varotimes_t J'').
	\end{align*}
\end{enumerate}

Observe that, as we are working in $(\bbA \otimes \bbB) \otimes \bbC$, by using term equalities $\varphi \otimes \psi \equiv \varphi \varotimes \psi$ obtained by tensoring two term judgments we can derive equalities between the four expressions in the left column:
$$
(J \otimes_t J') \otimes_t J'' \equiv (J \otimes_t J') \varotimes_t J'' \equiv (J \varotimes_t J') \varotimes_t J'' \equiv (J \varotimes_t J') \otimes_t J''.
$$

Before proceeding, let us make the statement more explicit. If $J = (\textbf{X} \vdashcustom u:U)$, $J' = (\textbf{Y} \vdashcustom v:V)$ and $J'' = (\textbf{Z} \vdashcustom w:W)$, then the above term expressions are
\begin{align*}
(u^U \otimes v^V) \otimes w^W, &\qquad u^U \otimes (v^V \otimes w^W),\\
(u^U \otimes v^V) \varotimes w^W, &\qquad u^U \otimes (v^V \varotimes w^W),\\
(u^U \varotimes v^V) \otimes w^W, &\qquad u^U \varotimes (v^V \otimes w^W),\\
(u^U \varotimes v^V) \varotimes w^W, &\qquad u^U \varotimes (v^V \varotimes w^W),
\end{align*}
and we are claiming their derivability and equality in context $(\textbf{X} \otimes \textbf{Y}) \otimes \textbf{Z}$. When some among $J$, $J'$, $J''$ are, instead, sort judgments, we suitably expand the context and modify the above expressions by using fresh variables in place of $u$, $v$ or $w$. For example, if we replace $J$ by $\textbf{X} \vdashcustom U \tp$, then we consider
\begin{align*}
	(x^U \otimes v^V) \otimes w^W, &\qquad x^U \otimes (v^V \otimes w^W),\\
	(x^U \otimes v^V) \varotimes w^W, &\qquad x^U \otimes (v^V \varotimes w^W),\\
	(x^U \varotimes v^V) \otimes w^W, &\qquad x^U \varotimes (v^V \otimes w^W),\\
	(x^U \varotimes v^V) \varotimes w^W, &\qquad x^U \varotimes (v^V \varotimes w^W)
\end{align*}
in context $(\textbf{X}' \otimes \textbf{Y}) \otimes \textbf{Z}$, where $x$ is the chosen variable not occurring in $\textbf{X}$, and $\textbf{X}' = (\textbf{X},x:U)$.

\vspace{0.5em}

Now, let us sketch how to prove statements (A1)-(A3) for a given $h \ge 0$ assuming that they hold for $h' < h$.

\subsection{Sketch of proofs of (A1) and (A2)}

\label{subsec: proof of (A1)}

Before studying (A1), let us prove a preliminary lemma: given contexts $\textbf{X}$, $\textbf{Y}$, $\textbf{Z}$ in $\bbA$, $\bbB$, $\bbC$, resp. such that $\Ht(\textbf{X})\Ht(\textbf{Y})\Ht(\textbf{Z}) \le h$, the judgment $\partial ((\textbf{X} \otimes \textbf{Y}) \otimes \textbf{Z}) \equiv \partial (\textbf{X} \otimes (\textbf{Y} \otimes \textbf{Z})) \ctx$ is derivable.

We start by observing that (the object of $\mathcal C((\bbA \otimes \bbB) \otimes \bbC)$ corresponding to) $\partial((\textbf{X} \otimes \textbf{Y}) \otimes \textbf{Z})$ is, in a canonical way, a limit of the diagram of projections
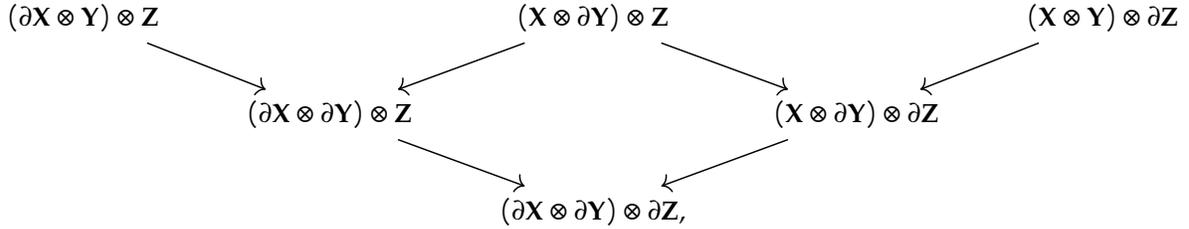
\[\begin{tikzcd}
	{(\partial\textbf{X} \otimes \textbf{Y}) \otimes\textbf{Z}} && {(\textbf{X} \otimes \partial\textbf{Y}) \otimes\textbf{Z}} && {(\textbf{X} \otimes \textbf{Y}) \otimes \partial\textbf{Z}} \\
	& {(\partial\textbf{X} \otimes \partial\textbf{Y}) \otimes\textbf{Z}} && {(\textbf{X} \otimes \partial\textbf{Y}) \otimes \partial\textbf{Z}} \\
	&& {(\partial\textbf{X} \otimes \partial\textbf{Y}) \otimes\partial\textbf{Z},}
	\arrow[from=1-1, to=2-2]
	\arrow[from=1-3, to=2-2]
	\arrow[from=1-3, to=2-4]
	\arrow[from=1-5, to=2-4]
	\arrow[from=2-2, to=3-3]
	\arrow[from=2-4, to=3-3]
\end{tikzcd}\]
i.e. where each arrow is dual to the corresponding inclusion between sets of variables -- for example, $(\partial\textbf{X} \otimes \textbf{Y}) \otimes \textbf{Z} \rightarrow (\partial\textbf{X} \otimes \partial \textbf{Y}) \otimes \textbf{Z}$ is given by lexicographically reindexing the family $(x_iy_jz_k)_{(i,j,k) \in \{1,...,m-1\} \times \{1, ..., n-1\} \times \{1, ..., p\}}$ whose entries are taken from the set of variables of $(\partial \textbf{X} \otimes \textbf{Y}) \otimes \textbf{Z}$.

Since $\Ht(\partial \textbf{X}) < \Ht(\textbf{X})$, $\Ht(\partial \textbf{Y}) < \Ht(\textbf{Y})$ and $\Ht(\partial\textbf{Z}) < \Ht(\textbf{Z})$, it follows from the induction hypothesis that each context in the above diagram is provably equal to the one having the same three factors but associated differently -- i.e. $(\partial \textbf{X} \otimes \textbf{Y}) \otimes \textbf{Z} \equiv \partial \textbf{X} \otimes (\textbf{Y} \otimes \textbf{Z}) \ctx$, etc. are derivable. Hence $\partial(\textbf{X} \otimes (\textbf{Y} \otimes \textbf{Z}))$ is also, canonically, a limit of the diagram. This implies that the sequence of variables $\partial(((x_1, ..., x_m) \otimes (y_1, ..., y_n)) \otimes (z_1, ..., z_p))$ is a morphism from $\partial((\textbf{X} \otimes \textbf{Y}) \otimes \textbf{Z})$ to $\partial(\textbf{X} \otimes (\textbf{Y} \otimes \textbf{Z}))$, thus that $\partial((\textbf{X} \otimes \textbf{Y}) \otimes \textbf{Z}) \equiv \partial(\textbf{X} \otimes (\textbf{Y} \otimes \textbf{Z})) \ctx$ is derivable.

\subsubsection{Statement (A1)}

Now, let us sketch a proof of (A1). We have different cases depending on the kinds of judgments being considered. Below, whenever applicable we let $\textbf{X}' = (\textbf{X},x:U)$, $\textbf{Y}' = (\textbf{Y}, y:V)$ and $\textbf{Z}' = (\textbf{Z}, z:W)$ where $x$, $y$, $z$ are the chosen variables not occurring in $\textbf{X}$, $\textbf{Y}$, $\textbf{Z}$, respectively.

\begin{itemize}
	\item $J$ is $\textbf{X} \vdashcustom U \tp$, $J'$ is $\textbf{Y} \vdashcustom V \tp$, and $J''$ is $\textbf{Z} \vdashcustom W \tp$. Then
	\begin{align*}
		 & (J \odot J') \odot J''\\
		 = \quad & (\partial(\textbf{X}' \otimes \textbf{Y}') \vdashcustom U \otimes V \tp) \odot (\textbf{Z} \vdashcustom W \tp)\\
		 = \quad& \partial((\textbf{X}' \otimes \textbf{Y}') \otimes \textbf{Z}') \vdashcustom (U \otimes V) \otimes W \tp\\
		 = \quad&  \partial((\textbf{X}' \otimes \textbf{Y}') \otimes \textbf{Z}') \vdashcustom \partial((U \otimes V)\{xy\} \btimes W\{z\}) \tp\\
		 = \quad&  \partial((\textbf{X}' \otimes \textbf{Y}') \otimes \textbf{Z}') \vdashcustom \partial((U\{x\} \btimes V\{y\}) \btimes W\{z\}) \tp\\
		 \Rightarrow \quad & \partial((\textbf{X}' \otimes \textbf{Y}') \otimes \textbf{Z}') \vdashcustom \partial(U\{x\} \btimes (V\{y\} \btimes W\{z\})) \tp \tag{calculation below}\\
		 \Rightarrow \quad & \partial(\textbf{X}' \otimes (\textbf{Y}' \otimes \textbf{Z}')) \vdashcustom \partial(U\{x\} \btimes (V\{y\} \btimes W\{z\})) \tp \tag{lemma}\\
		 = \quad & \partial(\textbf{X}' \otimes (\textbf{Y}' \otimes \textbf{Z}')) \vdashcustom U \otimes (V \otimes W) \tp\\
		 = \quad & J \odot (J' \odot J'').
	\end{align*}
	
	It remains to verify that $(U\{x\} \btimes V\{y\}) \btimes W\{z\} \equiv U\{x\} \btimes (V\{x\} \btimes W\{z\})$ is entrywise derivable in context $\textbf{K} = \partial((\textbf{X}' \otimes \textbf{Y}') \otimes \textbf{Z}')$. Writing
	$$
	U = R(\alpha_1, ..., \alpha_M), \quad V = S(\beta_1, ..., \beta_N), \quad W = T(\gamma_1, ..., \gamma_P),
	$$
	we must derive
	$$
	RST(((\alpha_1, ..., \alpha_M,x^U) \otimes (\beta_1, ..., \beta_N,y^V)) \otimes (\gamma_1, ..., \gamma_P,z^W))
	$$
	$$
	\overset{\textbf{K}}{\equiv} RST((\alpha_1, ..., \alpha_M,x^U) \otimes ((\beta_1, ..., \beta_N,y^V) \otimes (\gamma_1, ..., \gamma_P,z^W))).
	$$
	This corresponds to deriving in $\textbf{K}$ the following equalities:
	\begin{enumerate}[label=(\arabic*)]
		\item $(\alpha_i \otimes \beta_j) \otimes \gamma_k \equiv \alpha_i \otimes (\beta_j \otimes \gamma_k)$.
		
		Its derivability follows from (A3), via the induction hypothesis, since
		\begin{align*}
			\Ht(\textbf{X} \vdashcustom \alpha_i:\Type(\alpha_i)) & < \Ht(\textbf{X} \vdashcustom U \tp),\\
			\Ht(\textbf{Y} \vdashcustom \beta_j:\Type(\beta_j)) & < \Ht(\textbf{Y} \vdashcustom V \tp),\\
			\Ht(\textbf{Z} \vdashcustom \gamma_k:\Type(\gamma_k)) & < \Ht(\textbf{Z} \vdashcustom W \tp).
		\end{align*}
		(Note that this is stronger than required to use the induction hypothesis: it is sufficient that one of these three inequalities be strict. This remark will be used repeatedly below.)

		\item $(\alpha_i \otimes \beta_j) \otimes z^W \equiv \alpha_i \otimes (\beta_j \otimes z^W)$.

		\item $(\alpha_i \otimes y^V) \otimes \gamma_k \equiv \alpha_i \otimes (y^V \otimes \gamma_k)$.

		\item $(\alpha_i \otimes y^V) \otimes z^W \equiv \alpha_i \otimes (y^V \otimes z^W)$.

		\item $(x^U \otimes \beta_j) \otimes \gamma_k \equiv x^U \otimes (\beta_j \otimes \gamma_k)$.
		
		\item $(x^U \otimes \beta_j) \otimes z^W \equiv x^U \otimes (\beta_j \otimes z^W)$.

		\item $(x^U \otimes y^V) \otimes \gamma_k \equiv x^U \otimes (y^V \otimes \gamma_k)$.
		
		These cases are verified similarly to (1). Here, when applicable, we use $\textbf{X} \vdashcustom U \tp$, $\textbf{Y} \vdashcustom V \tp$, $\textbf{Z} \vdashcustom W \tp$, respectively, instead of $\textbf{X} \vdashcustom \alpha_i:\Type(\alpha_i)$, $\textbf{Y} \vdashcustom \beta_j:\Type(\beta_j)$, $\textbf{Z} \vdashcustom \gamma_k:\Type(\gamma_k)$.

		\item $(x^U \otimes y^V) \otimes z^W \equiv x^U \otimes (y^V \otimes z^W)$.
		
		The left- and right-hand sides equal $xyz$.
	\end{enumerate}

	\item $J$ is $\textbf{X} \vdashcustom U \tp$, $J'$ is $\textbf{Y} \vdashcustom V \tp$, and $J''$ is $\textbf{Z} \vdashcustom w:W$.
	\begin{align*}
		 & (J \odot J') \odot J''\\
		= \quad& (\partial(\textbf{X}' \otimes \textbf{Y}') \vdashcustom U \otimes V \tp) \odot (\textbf{Z} \vdashcustom w:W)\\
		= \quad & (\textbf{X}' \otimes \textbf{Y}') \otimes \textbf{Z} \vdashcustom (U \otimes V)\{xy\} \btimes w: \partial((U \otimes V)\{xy\} \varbtimes W\{w\})\\
		= \quad& (\textbf{X}' \otimes \textbf{Y}') \otimes \textbf{Z} \vdashcustom (U\{x\} \btimes V\{y\}) \btimes w: \partial((U\{x\} \btimes V\{y\}) \varbtimes W\{w\})\\
		\Rightarrow \quad & (\textbf{X}' \otimes \textbf{Y}') \otimes \textbf{Z} \vdashcustom U\{x\} \btimes (V\{y\} \btimes w): \partial(U\{x\} \varbtimes (V\{y\} \varbtimes W\{w\})) \tag{calculation below}\\
		\Rightarrow \quad & \textbf{X}' \otimes (\textbf{Y}' \otimes \textbf{Z}) \vdashcustom U\{x\} \btimes (V\{y\} \btimes w): \partial(U\{x\} \varbtimes (V\{y\} \varbtimes W\{w\})) \tag{A2}\\
		= \quad & (\textbf{X} \vdashcustom U \tp) \odot (\textbf{Y}' \otimes \textbf{Z} \vdashcustom V\{y\} \btimes w: \partial(V\{y\} \varbtimes W\{w\}))\\
		= \quad & J \odot (J' \odot J'').
	\end{align*}
	
	Let us now check that
	\[
	\tag{\texttt{*}}
	(U\{x\} \btimes V\{y\}) \btimes w \equiv U\{x\} \btimes (V\{y\} \btimes w),
	\]
	\[
	\tag{\texttt{**}}
	\partial( (U\{x\} \btimes V\{y\}) \varbtimes W\{w\} ) \equiv \partial( U\{x\} \varbtimes (V\{y\} \varbtimes W\{w\}))
	\]
	are entrywise derivable in $\textbf{K} = \partial((\textbf{X}' \otimes \textbf{Y}') \otimes \textbf{Z})$.
	
	In (\texttt{*}), if $w$ is a variable, then the left- and right-hand sides are both $xyw$. Otherwise, write
	$$
	U = R(\alpha_1, ..., \alpha_M), \quad V = S(\beta_1, ..., \beta_N), \quad w = t(c_1, ..., c_p).
	$$
	Then (\texttt{*}) is
	$$
	RSt( ((\alpha_1, ..., \alpha_M,x^U) \otimes (\beta_1, ..., \beta_N,y^V)) \otimes (c_1, ..., c_p) )
	$$
	$$
	\overset{\textbf{K}}{\equiv} RSt( (\alpha_1, ..., \alpha_M,x^U) \otimes ((\beta_1, ..., \beta_N,y^V) \otimes (c_1, ..., c_p)) ),
	$$
	so it suffices to derive, in context $\textbf{K}$, all equalities of one of the forms
	
	\begin{enumerate}[label=(\arabic*)]
		\item $(\alpha_i \otimes \beta_j) \otimes c_k \equiv \alpha_i \otimes (\beta_j \otimes c_k)$.
		
		\item $(\alpha_i \otimes y^V) \otimes c_k \equiv \alpha_i \otimes (y^V \otimes c_k)$
		
		\item $(x^U \otimes \beta_j) \otimes c_k \equiv x^U \otimes (\beta_j \otimes c_k)$.
		
		\item $(x^U \otimes y^V) \otimes c_k \equiv x^U \otimes (y^V \otimes c_k)$.
	\end{enumerate}
	
	These follow from (A3) since
	\begin{align*}
		 \Ht(\textbf{X} \vdashcustom \alpha_i:\Type(\alpha_i)) & < \Ht(\textbf{X} \vdashcustom U \tp), \\
		 \Ht(\textbf{Y} \vdashcustom \beta_j:\Type(\beta_j)) & < \Ht(\textbf{Y} \vdashcustom V \tp), \\
		 \Ht(\textbf{Z} \vdashcustom c_k:\Type(c_k)) & < \Ht(\textbf{Z} \vdashcustom w:W).
	\end{align*}
	
	For (\texttt{**}), write
	$$
	U = R(\alpha_1, ..., \alpha_M), \quad V = S(\beta_1, ..., \beta_N), \quad W = T(\gamma_1, ..., \gamma_P).
	$$
	Then the desired equality is
	$$
	RST( \partial ( ((\alpha_1, ..., \alpha_M,x^U) \otimes (\beta_1, ..., \beta_N,y^V)) \varotimes (\gamma_1, ..., \gamma_P,w^W) ))
	$$
	$$
	\overset{\textbf{K}}{\equiv} RST( \partial ( (\alpha_1, ..., \alpha_M,x^U) \varotimes ((\beta_1, ..., \beta_N,y^V) \varotimes (\gamma_1, ..., \gamma_P,w^W)) )),
	$$
	for which we must derive:
	\begin{enumerate}[label=(\arabic*)]
		\item $(\alpha_i \otimes \beta_j) \otimes \gamma_k \equiv \alpha_i \otimes (\beta_j \otimes \gamma_k)$.
		
		\item $(\alpha_i \otimes \beta_j) \varotimes w^W \equiv \alpha_i \otimes (\beta_j \varotimes w^W)$.
		
		\item $(\alpha_i \otimes y^V) \otimes \gamma_k \equiv \alpha_i \otimes (y^V \otimes \gamma_k)$.
		
		\item $(\alpha_i \otimes y^V) \varotimes w^W \equiv \alpha_i \varotimes (y^V \varotimes w^W)$.
		
		\item $(x^U \otimes \beta_j) \otimes \gamma_k \equiv x^U \otimes (\beta_j \otimes \gamma_k)$.
		
		\item $(x^U \otimes \beta_j) \varotimes w^W \equiv x^U \otimes (\beta_j \varotimes w^W)$.
		
		\item $(x^U \otimes y^V) \otimes \gamma_k \equiv x^U \otimes (y^V \otimes \gamma_k)$.		
	\end{enumerate}
	All of these follow from (A3).
	
	\item $J$ is $\textbf{X} \vdashcustom U \tp$, $J'$ is $\textbf{Y} \vdashcustom V \tp$, and $J''$ is $\textbf{Z} \vdashcustom W \equiv W' \tp$. Then
	\begin{align*}
		& (J \odot J') \odot J''\\
		= \quad & (\partial(\textbf{X}' \otimes \textbf{Y}') \vdashcustom U \otimes V \tp) \odot (\textbf{Z} \vdashcustom W \equiv W' \tp)\\
		= \quad& \partial((\textbf{X}' \otimes \textbf{Y}') \otimes \textbf{Z}') \vdashcustom (U \otimes V) \otimes W \equiv (U \otimes V) \otimes W' \tp\\
		= \quad&  \partial((\textbf{X}' \otimes \textbf{Y}') \otimes \textbf{Z}') \vdashcustom \partial((U \otimes V)\{xy\} \btimes W\{z\}) \equiv \partial((U \otimes V)\{xy\} \btimes W'\{z\}) \tp\\
		= \quad&  \partial((\textbf{X}' \otimes \textbf{Y}') \otimes \textbf{Z}') \vdashcustom \partial((U\{x\} \btimes V\{y\}) \btimes W\{z\}) \equiv \partial((U\{x\} \btimes V\{y\}) \btimes W'\{z\}) \tp\\
		\Rightarrow \quad & \partial((\textbf{X}' \otimes \textbf{Y}') \otimes \textbf{Z}') \vdashcustom \partial(U\{x\} \btimes (V\{y\} \btimes W\{z\})) \equiv \partial(U\{x\} \btimes (V\{y\} \btimes W'\{z\})) \tp \tag{see below}\\
		\Rightarrow \quad & \partial(\textbf{X}' \otimes (\textbf{Y}' \otimes \textbf{Z}')) \vdashcustom \partial(U\{x\} \btimes (V\{y\} \btimes W\{z\})) \equiv \partial(U\{x\} \btimes (V\{y\} \btimes W'\{z\})) \tp \tag{lemma}\\
		= \quad & \partial(\textbf{X}' \otimes (\textbf{Y}' \otimes \textbf{Z}')) \vdashcustom U \otimes (V \otimes W) \equiv U \otimes (V \otimes W')  \tp\\
		= \quad & J \odot (J' \odot J'').
	\end{align*}
	
	For the remaining step, calculations analogous to those in the previous case show that each of the type equalities
	\begin{align*}
		\partial((U\{x\} \btimes V\{y\}) \btimes W\{z\}) & \equiv \partial(U\{x\} \btimes (V\{y\} \btimes W\{z\}))\\
		\partial((U\{x\} \btimes V\{y\}) \btimes W'\{z\}) & \equiv \partial(U\{x\} \btimes (V\{y\} \btimes W'\{z\}))
	\end{align*}
	is entrywise derivable in context $\partial((\textbf{X}' \otimes \textbf{Y}') \otimes \textbf{Z}')$.
	
	\item $J$ is $\textbf{X} \vdashcustom U \tp$, $J'$ is $\textbf{Y} \vdashcustom V \tp$, and $J''$ is $\textbf{Z} \vdashcustom w \equiv w':W$. Then
	\begin{align*}
		& (J \odot J') \odot J''\\
		= \quad& (\partial(\textbf{X}' \otimes \textbf{Y}') \vdashcustom U \otimes V \tp) \odot (\textbf{Z} \vdashcustom w \equiv w':W)\\
		= \quad & (\textbf{X}' \otimes \textbf{Y}') \otimes \textbf{Z} \vdashcustom (U \otimes V)\{xy\} \btimes w \equiv (U \otimes V)\{xy\} \btimes w': \partial((U \otimes V)\{xy\} \varbtimes W\{w\})\\
		= \quad& (\textbf{X}' \otimes \textbf{Y}') \otimes \textbf{Z} \vdashcustom (U\{x\} \btimes V\{y\}) \btimes w \equiv (U\{x\} \btimes V\{y\}) \btimes w': \partial((U\{x\} \btimes V\{y\}) \varbtimes W\{w\})\\
		\Rightarrow \quad & (\textbf{X}' \otimes \textbf{Y}') \otimes \textbf{Z} \vdashcustom U\{x\} \btimes (V\{y\} \btimes w) \equiv U\{x\} \btimes (V\{y\} \btimes w'): \partial(U\{x\} \varbtimes (V\{y\} \varbtimes W\{w\})) \tag{see below}\\
		\Rightarrow \quad & \textbf{X}' \otimes (\textbf{Y}' \otimes \textbf{Z}) \vdashcustom U\{x\} \btimes (V\{y\} \btimes w) \equiv U\{x\} \btimes (V\{y\} \btimes w'): \partial(U\{x\} \varbtimes (V\{y\} \varbtimes W\{w\})) \tag{A2}\\
		= \quad & (\textbf{X} \vdashcustom U \tp) \odot (\textbf{Y}' \otimes \textbf{Z} \vdashcustom V\{y\} \btimes w \equiv V\{y\} \btimes w': \partial(V\{y\} \varbtimes W\{w\}))\\
		= \quad & J \odot (J' \odot J'').
	\end{align*}
	
	\item $J$ is $\textbf{X} \vdashcustom U \tp$, $J'$ is $\textbf{Y} \vdashcustom v:V$, and $J''$ is $\textbf{Z} \vdashcustom w:W$. Then
	\begin{align*}
		 & (J \odot J') \odot J'' \\
		= \quad & (\textbf{X}' \otimes \textbf{Y} \vdashcustom U\{x\} \btimes v: \partial(U\{x\} \varbtimes V\{v\})) \odot (\textbf{Z} \vdashcustom w:W)\\
		= \quad & (\textbf{X}' \otimes \textbf{Y}) \otimes \textbf{Z} \vdashcustom (U\{x\} \varbtimes V\{v\}) \btimes w \equiv (U\{x\} \btimes v) \varbtimes W\{w\}: \partial((U\{x\} \varbtimes V\{v\}) \uphalfsquare W\{w\})\\
		\Rightarrow \quad & (\textbf{X}' \otimes \textbf{Y}) \otimes \textbf{Z} \vdashcustom U\{x\} \btimes (V\{v\} \btimes w) \equiv U\{x\} \btimes (v \varbtimes W\{w\}) : \partial (U\{x\} \varbtimes (V\{v\} \uphalfsquare W\{w\})) \tag{see below}\\
		\Rightarrow \quad & \textbf{X}' \otimes (\textbf{Y} \otimes \textbf{Z}) \vdashcustom U\{x\} \btimes (V\{v\} \btimes w) \equiv U\{x\} \btimes (v \varbtimes W\{w\}) : \partial (U\{x\} \varbtimes (V\{v\} \uphalfsquare W\{w\})) \tag{A2}\\
		= \quad & (\textbf{X} \vdashcustom U \tp) \odot (\textbf{Y} \otimes \textbf{Z} \vdashcustom V\{v\} \btimes w \equiv v \varbtimes W\{w\} : \partial(V\{v\} \varbtimes W\{w\}))\\
		= \quad & J \odot (J' \odot J'').
	\end{align*}
	
	It remains to verify that the following equalities are entrywise derivable in context $(\textbf{X}' \otimes \textbf{Y}) \otimes \textbf{Z}$:
	\begin{align*}
		(U\{x\} \varbtimes V\{v\}) \btimes w & \equiv U\{x\} \btimes (V\{v\} \btimes w),\\
		(U\{x\} \btimes v) \varbtimes W\{w\} & \equiv U\{x\} \btimes (v \varbtimes W\{w\}),\\
		\partial((U\{x\} \varbtimes V\{v\}) \uphalfsquare W\{w\}) & \equiv \partial(U\{x\} \varbtimes (V\{v\} \uphalfsquare W\{w\})).
	\end{align*}
	These can be expanded and derived using (A3), in a way analogous to our approach in the previous cases.

	\item $J$ is $\textbf{X} \vdashcustom U \tp$, $J'$ is $\textbf{Y} \vdashcustom v:V$, and $J''$ is $\textbf{Z} \vdashcustom W \equiv W' \tp$. Then
	\begin{align*}
		 & (J \odot J') \odot J'' \\
		= \quad & (\textbf{X}' \otimes \textbf{Y} \vdashcustom U\{x\} \btimes v: \partial(U\{x\} \varbtimes  V\{v\})) \otimes (\textbf{Z} \vdashcustom W \equiv W' \tp)\\
		= \quad & (\textbf{X}' \otimes \textbf{Y}) \otimes \textbf{Z}' \vdashcustom (U\{x\} \btimes v) \btimes W\{z\} \equiv (U\{x\} \btimes v) \btimes W'\{z\}: \partial((U\{x\} \varbtimes V\{v\}) \varbtimes W\{z\})\\
		\Rightarrow \quad & (\textbf{X}' \otimes \textbf{Y}) \otimes \textbf{Z}' \vdashcustom U\{x\} \btimes (v \btimes W\{z\}) \equiv U\{x\} \btimes (v \btimes W'\{z\}): \partial(U\{x\} \varbtimes (V\{v\} \btimes W\{z\})) \tag{see below}\\
		\Rightarrow \quad & \textbf{X}' \otimes (\textbf{Y} \otimes \textbf{Z}') \vdashcustom U\{x\} \btimes (v \btimes W\{z\}) \equiv U\{x\} \btimes (v \btimes W'\{z\}): \partial(U\{x\} \varbtimes (V\{v\} \btimes W\{z\})) \tag{A2}\\
		= \quad & (\textbf{X} \vdashcustom U \tp) \otimes (\textbf{Y} \otimes \textbf{Z}' \vdashcustom v \btimes W\{z\} \equiv v \btimes W'\{z\}: \partial(V\{v\} \btimes W\{z\}))\\
		= \quad & J \odot (J' \odot J'').
	\end{align*}

As in the previous cases, we can conclude by expanding the equalities
\begin{align*}
	(U\{x\} \btimes v) \btimes W\{z\} & \equiv U\{x\} \btimes (v \btimes W\{z\})\\
	(U\{x\} \btimes v) \btimes W'\{z\} & \equiv U\{x\} \btimes (v \btimes W'\{z\})\\
	\partial((U\{x\} \varbtimes V\{v\}) \varbtimes W\{z\}) & \equiv \partial(U\{x\} \varbtimes (V\{v\} \btimes W\{z\}))
\end{align*}
and deriving them entrywise from (A3).

\item $(J \odot J') \odot J'' \Rightarrow J \odot (J' \odot J'')$ can be proved analogously in the remaining cases, namely, when $J$, $J'$, $J''$ are, respectively, judgments of the following forms:

\begin{multicols}{2}
	\begin{itemize}
		\item[\texttt{*}] sort, term, sort
		\item[\texttt{*}] term, sort, sort
		\item[\texttt{*}] sort, sort equality, sort
		\item[\texttt{*}] sort equality, sort, sort
		\item[\texttt{*}] sort, term equality, sort
		\item[\texttt{*}] term equality, sort, sort
		\item[\texttt{*}] term, sort, term
		\item[\texttt{*}] term, term, sort
		\item[\texttt{*}] sort, sort equality, term
		\item[\texttt{*}] sort equality, sort, term
		\item[\texttt{*}] sort equality, term, sort
		\item[\texttt{*}] term, sort, sort equality
		\item[\texttt{*}] term, sort equality, sort
	\end{itemize}
\end{multicols}
\end{itemize}

\subsubsection{Statement (A2)}

Consider contexts
$$\textbf{X} = (x_1:X_1, ..., x_m:X_m), \qquad \textbf{Y} = (y_1:Y_1, ..., y_n:Y_n), \qquad \textbf{Z} = (z_1:Z_1, ..., z_p:Z_p)
$$
in $\bbA$, $\bbB$, $\bbC$, respectively, such that $\Ht(\textbf{X})\Ht(\textbf{Y})\Ht(\textbf{Z}) = h \ge 1$. We will verify that $(\textbf{X} \otimes \textbf{Y}) \otimes \textbf{Z} \equiv \textbf{X} \otimes (\textbf{Y} \otimes \textbf{Z}) \ctx$ is derivable under the assumption that the analogous statement always holds with $0$, ..., $h-1$ in place of $h$.

\vspace{0.5em}

By the lemma proved at the beginning of the section, the judgment
\[
\tag{\texttt{*}}
\partial((\textbf{X} \otimes \textbf{Y}) \otimes \textbf{Z}) \equiv \partial(\textbf{X} \otimes (\textbf{Y} \otimes \textbf{Z})) \ctx
\]
is derivable. Moreover, applying (A1) to $J = (\partial \textbf{X} \vdashcustom X_m \tp)$, $J' = (\partial \textbf{Y} \vdashcustom Y_n \tp)$, $J'' = (\partial \textbf{Z} \vdashcustom Z_p \tp)$ yields derivability of
\[
\tag{\texttt{**}}
\partial((\textbf{X} \otimes \textbf{Y}) \otimes \textbf{Z}) \vdashcustom (X_m \otimes Y_n) \otimes Z_p \equiv X_m \otimes (Y_n \otimes Z_p) \tp.
\]

By combining (\texttt{*}) and (\texttt{**}) we conclude that $\textbf{X} \otimes (\textbf{Y} \otimes \textbf{Z})$ is a context and is provably equal to $(\textbf{X} \otimes \textbf{Y}) \otimes \textbf{Z}$ in $(\bbA \otimes \bbB) \otimes \bbC$.

\subsection{Sketch of proof of (A3)}

Let $J$, $J'$, $J''$ and $\textbf{K}$ be as in the statement of (A3), and consider the following list of equalities -- to be verified -- in context $\textbf{K}$:
\begin{enumerate}[label=(E\arabic*)]
	\item $(J \otimes_t J') \otimes_t J'' \equiv J \otimes_t (J' \otimes_t J'')$.
	
	\item $(J \otimes_t J') \varotimes_t J'' \equiv J \otimes_t (J' \varotimes_t J'')$.
	
	\item $(J \varotimes_t J') \otimes_t J'' \equiv J \varotimes_t (J' \otimes_t J'')$.
	
	\item $(J \varotimes_t J') \varotimes_t J'' \equiv J \varotimes_t (J' \varotimes_t J'')$.
\end{enumerate}

\begin{remark}
\label{rem: reducing cases - associativity}
Observe that it is possible to reduce the set of cases to be checked by repeatedly using term equalities $\varphi \otimes \psi \equiv \varphi \varotimes \psi$ obtained by tensoring two term judgments. Indeed, this shows derivability of
$$
(J \otimes_t J') \otimes_t J'' \equiv (J \otimes_t J') \varotimes_t J'' \equiv (J \varotimes_t J') \varotimes_t J'' \equiv (J \varotimes_t J') \otimes_t J''
$$
in $\textbf{K}$. This will be combined with the fact that if $J$ is either a sort judgment or $\textbf{X} \vdashcustom x:U$ for a variable $x$, then, by definition, $\otimes_t$ and $\varotimes_t$ give identical expressions when one of the arguments is $J$; similarly for $J'$, $J''$. For example, if $J$ is of that form, then
$$
J \otimes_t (J' \otimes_t J'') = J \varotimes_t (J' \otimes_t J''), \qquad J \otimes_t (J' \varotimes_t J'') = J \varotimes_t (J' \varotimes_t J''),
$$
so it suffices to derive $\big($(E1) or (E3)$\big)$ and $\big($(E2) or (E4)$\big)$. Similarly, if $J'$ or $J''$ is of that form, it suffices to derive $\big($(E1) or (E2)$\big)$ and $\big($(E3) or (E4)$\big)$; if two among $J$, $J'$, $J''$ are of that form, then (E1)-(E4) are the same; and if the three judgments are of that form, then (E1)-(E4) are trivially derivable as they equal $xyz \equiv xyz$ for variables $x$, $y$ and $z$.
\end{remark}

We will now outline the verification of the claim in some of these cases.

\vspace{0.5em}

\textbf{\underline{1st case}}: suppose that $J$, $J'$, $J''$ are term judgments whose term expressions are not variables, say
\begin{align*}
	J &= (\textbf{X} \vdashcustom u:U) = (\textbf{X} \vdashcustom r(a_1, ..., a_m):R(\alpha_1, ..., \alpha_M)),\\
	J' &= (\textbf{Y} \vdashcustom v:V) = (\textbf{Y} \vdashcustom s(b_1, ..., b_n): S(\beta_1, ..., \beta_N)),\\
	J'' &= (\textbf{Z} \vdashcustom w:W) = (\textbf{Z} \vdashcustom t(c_1, ..., c_p): T(\gamma_1, ..., \gamma_P)).
\end{align*}
Then we consider the following equalities in context $(\textbf{X} \otimes \textbf{Y}) \otimes \textbf{Z}$:

\begin{itemize}
	\item[\textbf{(E1)}] $(u^U \otimes v^V) \otimes w^W \equiv u^U \otimes (v^V \otimes w^W)$.
	
	Note that we can express it as
	$$
	RS
	\begin{pmatrix}
		\alpha_1 \otimes \beta_1 & \cdots & \alpha_1 \otimes \beta_N & \alpha_1 \varotimes v^V\\
		\vdots & \ddots & \vdots & \vdots\\
		\alpha_M \otimes \beta_1 & \cdots & \alpha_M \otimes \beta_N & \alpha_M \varotimes v^V\\
		u^U \otimes \beta_1 & \cdots & u^U \otimes \beta_N & u^U \otimes v^V
	\end{pmatrix}
	\btimes 
	t(c_1, ..., c_p)
	\equiv R(\alpha_1, ..., \alpha_M, u^U)
	\btimes St
	\begin{pmatrix}
		\beta_1 \otimes c_1 & \cdots & \beta_1 \otimes c_p\\
		\vdots & \ddots & \vdots\\
		\beta_N \otimes c_1 & \cdots & \beta_N \otimes c_p\\
		v^V \otimes c_1 & \cdots & v^V \otimes c_p
	\end{pmatrix}.
	$$
	By the induction hypothesis, the left- and right-hand sides are entrywise provably equal, i.e. the following equalities are derivable in $(\textbf{X} \otimes \textbf{Y}) \otimes \textbf{Z}$:
	\begin{align*}
		(\alpha_i \otimes \beta_j) \otimes c_k &\equiv \alpha_i \otimes (\beta_j \otimes c_k) \tag{for all $i$, $j$, $k$}\\
		(\alpha_i \varotimes v^V) \otimes c_k &\equiv \alpha_i \otimes (v^V \otimes c_k) \tag{for all $i$, $k$}\\
		(u^U \otimes \beta_j) \otimes c_k &\equiv u^U \otimes (\beta_j \otimes c_k) \tag{for all $j$, $k$}\\
		(u^U \otimes v^V) \otimes c_k &\equiv u^U \otimes (v^V \otimes c_k) \tag{for all $k$}
	\end{align*}
	
	\item[\textbf{(E2)}] $(u^U \otimes v^V) \varotimes w^W \equiv u^U \otimes (v^V \varotimes w^W)$.
	
	This can be expressed as
	$$
	Rs
	\begin{pmatrix}
		\alpha_1 \otimes b_1 & \cdots & \alpha_1 \otimes b_n\\
		\vdots & \ddots & \vdots\\
		\alpha_M \otimes b_1 & \cdots & \alpha_M \otimes b_n\\
		u^U \otimes b_1 & \cdots & u^U \otimes b_n
	\end{pmatrix}
	\varbtimes
	T(\gamma_1, ..., \gamma_P, w^W)
	\equiv
	R(\alpha_1, ..., \alpha_M,u^U)
	\btimes
	sT
	\begin{pmatrix}
		b_1 \otimes \gamma_1 & \cdots & b_1 \otimes \gamma_P & b_1 \varotimes w^W\\
		\vdots & \ddots & \vdots & \vdots\\
		b_n \otimes \gamma_1 & \cdots & b_n \otimes \gamma_P & b_n \varotimes w^W
	\end{pmatrix}.
	$$
	Similarly to the previous case, we conclude by using the induction hypothesis to derive the term equality corresponding to each of the entries:
	\begin{align*}
		(\alpha_i \otimes b_j) \otimes \gamma_k &\equiv \alpha_i \otimes (b_j \otimes \gamma_k) \tag{for all $i$, $j$, $k$}\\
		(\alpha_i \otimes b_j) \varotimes w^W &\equiv \alpha_i \otimes (b_j \varotimes w^W) \tag{for all $i$, $j$}\\
		(u^U \otimes b_j) \otimes \gamma_k &\equiv u^U \otimes (b_j \otimes \gamma_k) \tag{for all $j$, $k$}\\
		(u^U \otimes b_j) \varotimes w^W &\equiv u^U \otimes (b_j \varotimes w^W) \tag{for all $j$}
	\end{align*}
	
	\item[\textbf{(E4)}] $(u^U \varotimes v^V) \varotimes w^W \equiv u^U \varotimes (v^V \varotimes w^W)$.
	
	We can express it as
	$$
	rS
	\begin{pmatrix}
		a_1 \otimes \beta_1 & \cdots & a_1 \otimes \beta_N & a_1 \varotimes v^V\\
		\vdots & \ddots & \vdots & \vdots\\
		a_m \otimes \beta_1 & \cdots & a_m \otimes \beta_N & a_m \varotimes v^V
	\end{pmatrix}
	\varbtimes
	T(\gamma_1, ..., \gamma_P, w^W) \equiv
	r(a_1, ..., a_m) \varbtimes ST
	\begin{pmatrix}
		\beta_1 \otimes \gamma_1 & \cdots & \beta_1 \otimes \gamma_P & \beta_1 \varotimes w^W\\
		\vdots & \ddots & \vdots & \vdots\\
		\beta_N \otimes \gamma_1 & \cdots & \beta_N \otimes \gamma_P & \beta_N \varotimes w^W\\
		v^V \otimes \gamma_1 & \cdots & v^V \otimes \gamma_P & v^V \varotimes w^W
	\end{pmatrix}.
	$$
	Its derivability follows from that of:
	\begin{align*}
		(a_i \otimes \beta_j) \otimes \gamma_k & \equiv a_i \otimes (\beta_j \otimes \gamma_k) \tag{for all $i$, $j$, $k$}\\
		(a_i \otimes \beta_j) \varotimes w^W & \equiv a_i \otimes (\beta_j \varotimes w^W) \tag{for all $i$, $j$}\\
		(a_i \varotimes v^V) \otimes \gamma_k & \equiv a_i \otimes (v^V \otimes \gamma_k) \tag{for all $i$, $k$}\\
		(a_i \varotimes v^V) \varotimes w^W & \equiv a_i \varotimes (v^V \varotimes w^W)\tag{for all $i$}
	\end{align*}

	\item[\textbf{(E3)}] $(u^U \varotimes v^V) \otimes w^W \equiv u^U \varotimes (v^V \otimes w^W)$.
	
	This case differs from the others: while the outer term symbol in the left expression is $RSt$, the one in the right expression is $rST$, so it is not sufficient to derive an entrywise equality. Instead, we start by observing that $u^U \varotimes (v^V \otimes w^W)$ equals
	$$
	r(a_1, ..., a_m) \varbtimes ST
	\begin{pmatrix}
		\beta_1 \otimes \gamma_1 & \cdots & \beta_1 \otimes \gamma_P & \beta_1 \varotimes w^W\\
		\vdots & \ddots & \vdots & \vdots\\
		\beta_N \otimes \gamma_1 & \cdots & \beta_N \otimes \gamma_P & \beta_N \varotimes w^W\\
		v^V \otimes \gamma_1 & \cdots & v^V \otimes \gamma_P & v^V \otimes w^W
	\end{pmatrix},
	$$
	which is -- by the induction hypothesis -- entrywise provably equal to $u^U \varotimes (v^V \varotimes w^W)$ in $(\textbf{X} \otimes \textbf{Y}) \otimes \textbf{Z}$. The desired equality then follows from
	$$
	u^U \varotimes (v^V \varotimes w^W) \equiv (u^U \varotimes v^V) \varotimes w^W \equiv (u^U \varotimes v^V) \otimes w^W.
	$$
\end{itemize}

\textbf{\underline{2nd case}}: suppose that
\begin{align*}
	J &= (\textbf{X} \vdashcustom x:U) = (\textbf{X} \vdashcustom x:R(\alpha_1, ..., \alpha_M)),\\
	J' &= (\textbf{Y} \vdashcustom v:V) = (\textbf{Y} \vdashcustom s(b_1, ..., b_n): S(\beta_1, ..., \beta_N)),\\
	J'' &= (\textbf{Z} \vdashcustom w:W) = (\textbf{Z} \vdashcustom t(c_1, ..., c_p): T(\gamma_1, ..., \gamma_P))
\end{align*}
where $x$ is a variable occurring in $\textbf{X}$. By Remark \ref{rem: reducing cases - associativity}, it suffices to derive the following equalities in context $(\textbf{X} \otimes \textbf{Y}) \otimes \textbf{Z}$:

\begin{itemize}
	\item[\textbf{(E1)}] $(x^U \otimes v^V) \otimes w^W \equiv x^U \otimes (v^V \otimes w^W)$. The proof of its derivability is analogous to that of (E1) in the 1st case.

	\item[\textbf{(E2)}] $(x^U \otimes v^V) \varotimes w^W \equiv x^U \otimes (v^V \varotimes w^W)$. Here, we argue as for (E2) in the 1st case.
\end{itemize}

A similar argument allows us to derive (E1) and (E2) when $J$ is replaced by $\textbf{X} \vdashcustom U \tp$.

\vspace{1em}

\textbf{\underline{3rd case}}: suppose that
\begin{align*}
	J &= (\textbf{X} \vdashcustom u:U) = (\textbf{X} \vdashcustom r(a_1, ..., a_m):R(\alpha_1, ..., \alpha_M)),\\
	J' &= (\textbf{Y} \vdashcustom y:V) = (\textbf{Y} \vdashcustom y: S(\beta_1, ..., \beta_N)),\\
	J'' &= (\textbf{Z} \vdashcustom w:W) = (\textbf{Z} \vdashcustom t(c_1, ..., c_p): T(\gamma_1, ..., \gamma_P))
\end{align*}
where $y$ is a variable occurring in $\textbf{Y}$. By Remark \ref{rem: reducing cases - associativity}, it suffices to derive the following equalities in $(\textbf{X} \otimes \textbf{Y}) \otimes \textbf{Z}$:

\begin{itemize}
	\item[\textbf{(E1)}] $(u^U \otimes y^V) \otimes w^W \equiv u^U \otimes (y^V \otimes w^W)$. Derived as (E1) in the 1st case.
	
	\item[\textbf{(E4)}] $(u^U \varotimes y^V) \varotimes w^W \equiv u^U \varotimes (y^V \varotimes w^W)$. Derived as (E4) in the 1st case.
\end{itemize}

With a similar argument we can derive (E1) and (E4) when $J'$ is replaced by $\textbf{Y} \vdashcustom V \tp$.

\vspace{1em}

\textbf{\underline{4th case}}: suppose that
\begin{align*}
	J &= (\textbf{X} \vdashcustom u:U) = (\textbf{X} \vdashcustom r(a_1, ..., a_m):R(\alpha_1, ..., \alpha_M)),\\
	J' &= (\textbf{Y} \vdashcustom v:V) = (\textbf{Y} \vdashcustom s(b_1, ..., b_n): S(\beta_1, ..., \beta_N)),\\
	J'' &= (\textbf{Z} \vdashcustom z:W) = (\textbf{Z} \vdashcustom z: T(\gamma_1, ..., \gamma_P))
\end{align*}
where $z$ is a variable occurring in $\textbf{Z}$. By Remark \ref{rem: reducing cases - associativity}, it suffices to derive the following equalities in $(\textbf{X} \otimes \textbf{Y}) \otimes \textbf{Z}$:

\begin{itemize}
	\item[\textbf{(E2)}] $(u^U \otimes v^V) \varotimes z^W \equiv u^U \otimes (v^V \varotimes z^W)$. Derived as (E2) in the 1st case.

	\item[\textbf{(E4)}] $(u^U \varotimes v^V) \varotimes z^W \equiv u^U \varotimes (v^V \varotimes z^W)$. Derived as (E4) in the 1st case.
\end{itemize}

With a similar argument we can derive (E2) and (E4) when $J''$ is replaced by $\textbf{Z} \vdashcustom W \tp$.

\vspace{1em}

\textbf{\underline{5th case}}: suppose that
\begin{align*}
	J &= (\textbf{X} \vdashcustom x:U) = (\textbf{X} \vdashcustom x:R(\alpha_1, ..., \alpha_M)),\\
	J' &= (\textbf{Y} \vdashcustom y:V) = (\textbf{Y} \vdashcustom y: S(\beta_1, ..., \beta_N)),\\
	J'' &= (\textbf{Z} \vdashcustom w:W) = (\textbf{Z} \vdashcustom t(c_1, ..., c_p): T(\gamma_1, ..., \gamma_P))
\end{align*}
where $x$, $y$ are variables in $\textbf{X}$, $\textbf{Y}$, respectively. By Remark \ref{rem: reducing cases - associativity}, it suffices to derive (E1) -- $(x^U \otimes y^V) \otimes w^W \equiv x^U \otimes (y^V \otimes w^W)$ -- in context $(\textbf{X} \otimes \textbf{Y}) \otimes \textbf{Z}$. Here, we can argue as for (E1) in the 1st case.

\vspace{0.5em}

We can proceed similarly when $J$ or $J'$ is replaced by $\textbf{X} \vdashcustom U \tp$ or $\textbf{Y} \vdashcustom V \tp$, respectively.

\vspace{1em}

\textbf{\underline{6th case}}: suppose that
\begin{align*}
	J &= (\textbf{X} \vdashcustom x:U) = (\textbf{X} \vdashcustom x:R(\alpha_1, ..., \alpha_M)),\\
	J' &= (\textbf{Y} \vdashcustom v:V) = (\textbf{Y} \vdashcustom s(b_1, ..., b_n): S(\beta_1, ..., \beta_N)),\\
	J'' &= (\textbf{Z} \vdashcustom z:W) = (\textbf{Z} \vdashcustom z: T(\gamma_1, ..., \gamma_P))
\end{align*}
where $x$, $z$ are variables in $\textbf{X}$, $\textbf{Z}$, respectively. By Remark \ref{rem: reducing cases - associativity}, it suffices to derive (E2) -- $(x^U \otimes y^V) \varotimes w^W \equiv x^U \otimes (y^V \varotimes w^W)$ -- in context $(\textbf{X} \otimes \textbf{Y}) \otimes \textbf{Z}$. For that, we can proceed as for (E2) in the 1st case.

\vspace{0.5em}

We can use a similar argument when $J$ or $J''$ is replaced by $\textbf{X} \vdashcustom U \tp$ or $\textbf{Z} \vdashcustom W \tp$, respectively.

\vspace{1em}

\textbf{\underline{7th case}}: suppose that
\begin{align*}
	J &= (\textbf{X} \vdashcustom u:U) = (\textbf{X} \vdashcustom r(a_1, ..., a_m):R(\alpha_1, ..., \alpha_M)),\\
	J' &= (\textbf{Y} \vdashcustom y:V) = (\textbf{Y} \vdashcustom y: S(\beta_1, ..., \beta_N)),\\
	J'' &= (\textbf{Z} \vdashcustom z:W) = (\textbf{Z} \vdashcustom z: T(\gamma_1, ..., \gamma_P))
\end{align*}
where $y$, $z$ are variables in $\textbf{Y}$, $\textbf{Z}$, respectively. By Remark \ref{rem: reducing cases - associativity}, it suffices to derive (E4) -- $(x^U \varotimes y^V) \varotimes w^W \equiv x^U \varotimes (y^V \varotimes w^W)$ -- in context $(\textbf{X} \otimes \textbf{Y}) \otimes \textbf{Z}$. For that, we can proceed as for (E4) in the 1st case.

\vspace{0.5em}

A similar argument can be used when $J'$ or $J''$ is replaced by $\textbf{Y} \vdashcustom V \tp$ or $\textbf{Z} \vdashcustom W \tp$, respectively.

\vspace{0.5em}

This concludes the sketch of proof that for any $J$, $J'$, $J''$, the eight term expressions under consideration are derivable and provably equal to each other in the appropriate context.

\begin{remark}
\label{rem: no monoidal structure now}
It is natural to try use the isomorphisms $(\bbA \otimes \bbB) \otimes \bbC \equiv \bbA \otimes (\bbB \otimes \bbC)$ described in this section as components of an associator for a monoidal structure on $\GAT$, the unit of which would be the theory $\bbO_1$ having a single axiom $\vdashcustom O \tp$. We defer this discussion to \cite{Alm26} since we have, at this point, little to no access to functoriality of the tensor product.
\end{remark}

\section{Towards symmetry}

We will now describe for generalized algebraic theories $\bbA$ and $\bbB$ an isomorphism $\bbA \otimes \bbB \cong \bbB \otimes \bbA$.

Firstly, observe that we have a canonical isomorphism between the alphabets $\Sigma(\bbA) \otimes \Sigma(\bbB)$ and $\Sigma(\bbB) \otimes \Sigma(\bbA)$ given by swapping the entries of each element: a variable $xy$, a sort symbol $ST$, and term symbols $sT$ and $St$ in $\Sigma(\bbA) \otimes \Sigma(\bbB)$ correspond, respectively, to $yx$, $ST$, $St$ and $sT$ in $\Sigma(\bbB) \otimes \Sigma(\bbA)$. Note also that, because of the ad-hoc (and asymmetric) choice of lexicographic order in the definition of tensor expressions (\S\ref{sec: tensor product of generalized algebraic theories}), the ``arity matrix" of (for example) $sT$ in $\bbA \otimes \bbB$ is transpose to the one of $Ts$ in $\bbB \otimes \bbA$.

\vspace{0.5em}

Precisely, we define a preinterpretation, in the sense of \cite{Car86}, \S12\footnote{However, since we allow different theories to have different sets of variables, we must modify the definition of a preinterpretation by requiring a choice of function between the respective sets of variables.}, of $\bbA \otimes \bbB$ in $\bbB \otimes \bbA$ as follows:

\begin{itemize}
	\item A variable $xy$ is sent to $yx$.
	
	\item A sort symbol $ST$, say introduced as
	$$
	ST
	\begin{pmatrix}
		x_1y_1 & \cdots & x_1y_n & x_1y'\\
		\vdots & \ddots & \vdots & \vdots \\
		x_my_1 & \cdots & x_my_n & x_my'\\
		x'y_1 & \cdots & x'y_n & -
	\end{pmatrix},
	$$
	is sent to the expression
	$$
	TS
	\begin{pmatrix}
		y_1x_1 & \cdots & y_1x_m & y_1x'\\
		\vdots & \ddots & \vdots & \vdots \\
		y_nx_1 & \cdots & y_nx_m & y_nx'\\
		y'x_1 & \cdots & y'x_n & -
	\end{pmatrix}.
	$$
	\item A term symbol $sT$, say introduced as
	$$
	sT
	\begin{pmatrix}
		x_1y_1 & \cdots & x_1y_n & x_1y'\\
		\vdots & \ddots & \vdots & \vdots\\
		x_my_1 & \cdots & x_my_n & x_my'
	\end{pmatrix},
	$$
	is sent to the expression
	$$
	Ts
	\begin{pmatrix}
		y_1x_1 & \cdots & y_1x_m\\
		\vdots & \ddots & \vdots\\
		y_nx_1 & \cdots & y_nx_m\\
		y'x_1 & \cdots & y'x_m
	\end{pmatrix}.
	$$
	
	\item A term symbol $St$, say introduced as
	$$
	St
	\begin{pmatrix}
		x_1y_1 & \cdots & x_1y_n\\
		\vdots & \ddots & \vdots\\
		x_my_1 & \cdots & x_my_n\\
		x'y_1 & \cdots & x'y_n
	\end{pmatrix},
	$$
	is sent to the expression
	$$
	tS
	\begin{pmatrix}
		y_1x_1 & \cdots & y_1x_m & y_1x'\\
		\vdots & \ddots & \vdots & \vdots\\
		y_nx_1 & \cdots & y_nx_m & y_nx'
	\end{pmatrix}.
	$$
\end{itemize}
Denote this preinterpretation by $I_{\bbA,\bbB}$. As explained in \cite{Car86}, it defines by recursion a map $\hat{I}_{\bbA,\bbB}:\Sigma(\bbA \otimes \bbB) \rightarrow \Sigma(\bbB \otimes \bbA)$ between the respective sets of derivable judgments. Similarly, by swapping the roles of $\bbA$ and $\bbB$ we obtain a preinterpretation $I_{\bbB,\bbA}$ of $\bbB \otimes \bbA$ in $\bbA \otimes \bbB$, and it is immediate that $\hat{I}_{\bbA,\bbB}$ and $\hat{I}_{\bbB,\bbA}$ are inverse to each other.

Hence as soon as we prove that $I_{\bbA,\bbB}$ (and, analogously, $I_{\bbB,\bbA}$) is an interpretation, it will be an isomorphism of \textsc{gat}s.

\begin{proposition}
\label{prop: symmetry interpretation}
$I_{\bbA,\bbB}$ is an interpretation of $\bbA \otimes \bbB$ in $\bbB \otimes \bbA$.
\end{proposition}

We must verify that if $J$, $J'$ are axioms in $\bbA$, $\bbB$, respectively, then $\hat{I}_{\bbA,\bbB}(J \odot J')$ is derivable in $\bbB \otimes \bbA$.

The reason why this is not clear is that, even after taking into account the modifications of symbols and their arities described above, the definition of tensor expressions from \S\ref{subsec: expressions from pairs of judgments} does not become symmetric in the two factors. Indeed, the operations $\otimes_t$ and $\varotimes_t$ do not play a symmetric role: whenever the recursion requires using a previously defined tensor expression, it will use, say, $u^U \otimes_t v^V$ rather than $u^U \varotimes_t v^V$. But that choice was completely arbitrary, and we expect that swapping the roles of $\otimes_t$ and $\varotimes_t$ in the definition of $J \odot J'$ will also yield a derivable judgment with respect to the previously constructed theory.

\vspace{0.5em}

We will now explain more precisely how this modification is made and how it leads to a proof that $I_{\bbA,\bbB}$ is an interpretation.

\begin{construction}
Consider \textsc{gat}s $\bbA$ and $\bbB$. In the notation of \S\ref{sec: tensor product of generalized algebraic theories}, we sketch a definition of operations
\begin{align*}
	\otimes^*_t:&(\text{Der}_s^+(\bbA) \cup \text{Der}_t(\bbA)) \times (\text{Der}_s^+(\bbB) \cup \text{Der}_t(\bbB)) \longrightarrow \text{Exp}_t(\Sigma(\bbA) \otimes \Sigma(\bbB)),\\[0.5em]
	\varotimes^*_t:&(\text{Der}_s^+(\bbA) \cup \text{Der}_t(\bbA)) \times (\text{Der}_s^+(\bbB) \cup \text{Der}_t(\bbB)) \longrightarrow \text{Exp}_t(\Sigma(\bbA) \otimes \Sigma(\bbB)),\\[0.5em]
	\otimes^*_s:&\text{Der}_s^+(\bbA) \otimes \text{Der}_s^+(\bbB) \longrightarrow \text{Exp}_s(\Sigma(\bbA) \otimes \Sigma(\bbB)),
\end{align*}
as well as $\otimes^*$ for precontexts and $\odot^*$ for derivable judgments, by using the same recursive procedure as for $\otimes_t$, $\varotimes_t$, $\otimes_s$, $\otimes$, $\odot$, respectively, except for two cases: if $J = (x_1:X_1, ..., x_m:X_m \vdashcustom u:U)$ and $J' = (y_1:Y_1, ..., y_n:Y_n \vdashcustom v:V)$ where $u$, $v$ are not variables, we let
\begin{align*}
u^U \varotimes^* v^V & = J \varotimes^*_t J' = (x^U \otimes^* v^V)[u^U \otimes^* y_1^{Y_1} \mid xy_1, ..., u^U \otimes^* y_n^{Y_n} \mid xy_n],\\[0.5em]
u^U \otimes^* v^V & = J \otimes_t^* J' = (u^U \otimes^* y^V)[x_1^{X_1} \otimes^* v^V \mid x_1y, ..., x_m^{X_m} \otimes^* v^V \mid x_my].
\end{align*}
Note that these are the only cases where the use of $\otimes_t$ and $\varotimes_t$ is asymmetric.
\end{construction}

Arguing by induction yields the following result on the preinterpretation $I_{\bbA,\bbB}$ defined above:

\begin{proposition}
For derivable judgments $J$, $J'$ in $\bbA$, $\bbB$, respectively, we have
$$
\hat{I}_{\bbA,\bbB}(J \odot J') = J' \odot^* J.
$$
\end{proposition}

Thus Proposition \ref{prop: symmetry interpretation} will follow once we verify that $J' \odot^* J$ is derivable.

\begin{proposition}
Let $J$, $J'$ be derivable judgments in \textsc{gat}s $\bbA$, $\bbB$, respectively. Then $J \odot^* J'$ is derivable in $\bbA \otimes \bbB$.
\end{proposition}

\begin{proof}[Proof sketch]
We will outline a proof by induction that the following holds for all $h \ge 0$: if $\Ht(J)\Ht(J')$, then the judgment $J \odot^* J'$ is derivable and has the same interpretation in $\mathcal C(\bbA \otimes \bbB)$ as $J \odot J'$.

Given $h \ge 0$, suppose that the claim holds for all $h' < h$, and let $J$, $J'$ be such that $\Ht(J)\Ht(J') = h$. Several cases need to be verified, corresponding to Table \ref{table: 1}. Firstly, note that in all of them, the precontext of $J \odot^* J'$ is either
\begin{itemize}
	\item of the form $\textbf{X} \otimes^* \textbf{Y}$ where $\Ht(\textbf{X})\Ht(\textbf{Y}) < h$, or
	
	\item of the form $\partial(\textbf{X}' \otimes^* \textbf{Y}')$ where $\textbf{X} = \partial\textbf{X}'$ and $\textbf{Y} = \partial \textbf{Y}'$ satisfy $\Ht(\textbf{X})\Ht(\textbf{Y}')$, $\Ht(\textbf{X}')\Ht(\textbf{Y}) < h$.
\end{itemize}
In either case, every sort expression occurring in the context is, by construction, $U \otimes^* V$ for some derivable judgments $\textbf{A} \vdashcustom U \tp$ and $\textbf{B} \vdashcustom V \tp$ such that $\Ht(\textbf{A} \vdashcustom U \tp)\Ht(\textbf{B} \vdashcustom V \tp) < h$. But by the induction hypothesis, $\partial(\textbf{A}' \otimes \textbf{B}') \vdashcustom U \otimes^* V \equiv U \otimes V \tp$ is derivable where $\textbf{A}' = (\textbf{A}, a:U)$, $\textbf{B}' = (\textbf{B}, b:V)$. Hence in the first (resp. second) case, $\textbf{X} \otimes^* \textbf{Y}$ (resp. $\partial(\textbf{X}' \otimes^* \textbf{Y}')$) is a context and is provably equal to $\textbf{X} \otimes \textbf{Y}$ (resp. to $\partial(\textbf{X}' \otimes \textbf{Y}')$).

It remains analyze the second part of $J \odot^* J'$:
\begin{itemize}
	\item Suppose that $J = (\textbf{X} \vdashcustom u:U)$ and $J' = (\textbf{Y} \vdashcustom v:V)$ where $u$, $v$ are not variables. Then $J \odot^* J'$ is
	\begin{align*}
		\textbf{X} \otimes^* \textbf{Y} &\vdashcustom u^U \otimes^* v^V \equiv u^U \varotimes^* v^V\\
		&: (U \otimes^* V)[x_1^{X_1} \otimes^* v^V \mid x_1y, ..., x_m^{X_m} \otimes^* v^V \mid x_my, ..., u^U \otimes^* y_1^{Y_1} \mid xy_1, ..., u^U \otimes^* y_n^{Y_n} \mid xy_n].
	\end{align*}
	Letting $\textbf{X}' = (\textbf{X}, x:U)$ and $\textbf{Y}' = (\textbf{Y},y:V)$, the induction hypothesis implies that
	\begin{align*}
		&\partial(\textbf{X}' \otimes^* \textbf{Y}') \equiv \partial(\textbf{X}' \otimes \textbf{Y}') \ctx\\
		&\partial(\textbf{X}' \otimes \textbf{Y}') \vdashcustom U \otimes^* V \equiv U \otimes V \tp\\		
		&\textbf{X} \otimes \textbf{Y} \vdashcustom x_i^{X_i} \otimes^* v^V \equiv x_i^{X_i} \otimes v^V \tm \tag{$1 \le i \le m$}\\
		&\textbf{X} \otimes \textbf{Y} \vdashcustom u^U \otimes^* y_j^{Y_j} \equiv u^U \otimes y_j^{Y_j} \tm \tag{$1 \le j \le n$}
	\end{align*}
	are derivable. Now, by substitution we obtain
	\begin{align*}
		\textbf{X} \otimes \textbf{Y} & \vdashcustom (U \otimes^* V)[x_1^{X_1} \otimes^* v^V \mid x_1y, ..., x_m^{X_m} \otimes^* v^V \mid x_my, ..., u^U \otimes^* y_1^{Y_1} \mid xy_1, ..., u^U \otimes^* y_n^{Y_n} \mid xy_n]\\
		& \equiv (U \otimes V)[x_1^{X_1} \otimes v^V \mid x_1y, ..., x_m^{X_m} \otimes v^V \mid x_my, ..., u^U \otimes y_1^{Y_1} \mid xy_1, ..., u^U \otimes y_n^{Y_n} \mid xy_n] \tp.
	\end{align*}
	Since $J \odot J'$, which is
	\begin{align*}
		\textbf{X} \otimes \textbf{Y} &\vdashcustom u^U \otimes v^V \equiv u^U \varotimes v^V\\
		&: (U \otimes V)[x_1^{X_1} \otimes v^V \mid x_1y, ..., x_m^{X_m} \otimes v^V \mid x_my, ..., u^U \otimes y_1^{Y_1} \mid xy_1, ..., u^U \otimes y_n^{Y_n} \mid xy_n],
	\end{align*}
	is derivable, to obtain $J \odot^* J'$ it suffices to derive
	\begin{align*}
		&\textbf{X} \otimes \textbf{Y} \vdashcustom u^U \otimes^* v^V \equiv u^U \varotimes v^V \tm,\\
		&\textbf{X} \otimes \textbf{Y} \vdashcustom u^U \varotimes^* v^V \equiv u^U \otimes v^V \tm.
	\end{align*}
	But the induction hypothesis yields
	\begin{align*}
		u^U \otimes^* v^V & = (u^U \otimes^* y^V)[x_1^{X_1} \otimes^* v^V \mid x_1y, ..., x_m^{X_m} \otimes^* v^V \mid x_my]\\
		& \equiv (u^U \otimes y^V)[x_1^{X_1} \otimes v^V \mid x_1y, ..., x_m^{X_m} \otimes v^V \mid x_my]\\
		& = u^U \varotimes v^V,\\[1em]
		u^U \varotimes^* v^V & = (x^U \otimes^* v^V)[u^U \otimes^* y_1^{Y_1} \mid xy_1, ..., u^U \otimes^* y_n^{Y_n} \mid xy_n]\\
		& \equiv (x^U \otimes v^V)[u^U \otimes y_1^{Y_1} \mid xy_1, ..., u^U \otimes y_n^{Y_n} \mid xy_n]\\
		& = u^U \otimes v^V,
	\end{align*}
	which concludes the proof in this case.
	
	\item In all other cases, the definition of $J \odot' J'$ does not involve any term expressions $u \otimes^* v$ corresponding to judgments $K = (\textbf{A} \vdashcustom u:U)$, $K' = (\textbf{B} \vdashcustom v:V)$ where $u$, $v$ are not variables and $\Ht(K)\Ht(K') \ge h$. This implies (details have been omitted) that:
	\begin{itemize}
		\item[-] Any sort expression $U \otimes^* V$ occurring in $J \odot' J'$ can be expanded and proved equal to $U \otimes V$ using the induction hypothesis.
		
		\item[-] For a term expression $u \otimes^* v$ occurring in $J \odot' J'$, say with $u$, $v$ corresponding to judgments $K$, $K'$, if $u$ and $v$ are not variables, then the induction hypothesis $u \otimes^* v \equiv u \otimes v$.
		
		\item[-] Otherwise, if $u$ or $v$ is a variable, then $u \otimes^* v$ can be expanded and proved equal to $u \otimes v$ via the induction hypothesis.
	\end{itemize}
\end{itemize}
This concludes our sketch of proof that $J \odot^* J'$ is derivable in $\bbA \otimes \bbB$ for any $J$, $J'$ derivable in $\bbA$, $\bbB$.
\end{proof}

As a consequence we obtain Proposition \ref{prop: symmetry interpretation}, which in turn implies that $\bbA \otimes \bbB$ and $\bbB \otimes \bbA$ are isomorphic via $I_{\bbA,\bbB}$ and $I_{\bbB,\bbA}$.

\vspace{0.5em}

Like the situation with associativity, a categorical study of this symmetry isomorphism is out of reach in the current article. We will return to this problem in \cite{Alm26}.

\appendix

\section{An overview of generalized algebraic theories}

We now give an overview of the concept of a generalized algebraic theory (\textsc{gat} for short), introduced by J. Cartmell in his doctoral thesis \cite{Car78} (see also \cite{Car86}).

Our approach to defining a generalized algebraic theory differs from Cartmell's in the choice of derivation rules. However, as must be the case for our presentation to be valid, our list of derivation rules is ``as expressive" as Cartmell's in the following sense (see Proposition \ref{prop: derivable iff c-derivable}): for any pretheory, the sets of derivable judgments obtained by the two procedures are the same. In other words, the classes of ``theories" obtained in the two cases are identical as long as one only considers the \emph{set} of derivable judgments (rather than the structure of derivations).

The reason for this modification is that, with the new derivation rules, we have a finer control over the recursive process by which a judgment is derived. Ultimately, this will allow us to show that a certain \emph{height} function -- which, in a sense, measures the complexity of a derivable judgment in terms of a ``shortest" derivation -- has several good properties allowing us to prove statements on judgments by induction.

\subsection{Alphabets and pretheories}

The starting point is what we will call a \emph{sort-and-term alphabet}, which provides the symbols required for constructing a generalized algebraic theory:

\begin{definition}
\label{def: alphabet}
A \emph{sort-and-term alphabet} $-$ or, in our setting, simply \emph{alphabet} $-$ is a triple $\Sigma = (\Sigma^{\text{var}}, \Sigma^{\text{sort}}, \Sigma^{\text{term}})$ consisting of a countably infinite set $\Sigma^{\text{var}}$ of \emph{variables}, a set $\Sigma^{\text{sort}}$ of \emph{sort symbols} (or \emph{type symbols}), and a set $\Sigma^{\text{term}}$ of \emph{term symbols} (or \emph{operation symbols}).
\end{definition}

Variables will usually denoted by lowercase letters such as $a$, $b$, $x$, $y$, often with indices as in $x_1$, ..., $x_n$ when we work with lists of variables. Sort symbols will be denoted by uppercase letters such as $S$ and $T$, and term symbols by lowercase letters such as $s$ and $t$.

\begin{construction}
Consider the (single-sorted) algebraic theory $\bbA$ having $\mathbb N \times (\Sigma^{\text{sort}} \sqcup \Sigma^{\text{term}})$ as its set of function symbols, where $(n,w)$ has arity $n$ for each $n$ and $w$, and having no axioms.

We define $\Sigma^*$ as the free model of $\bbA$ on the set $\Sigma^{\text{var}}$. An element of $\Sigma^*$ is called an \emph{expression} on the alphabet $\Sigma$.

To provide an explicit description of $(\Sigma^{\text{var}},\Sigma^{\text{sort}},\Sigma^{\text{term}})^*$, let $X$ be the set of all finite sequences $w_1$, ..., $w_n$ (possibly with $n = 0$, i.e. we include the empty sequence) of elements of a set obtained from $\mathbb N \times (\Sigma^{\text{sort}} \sqcup \Sigma^{\text{term}})$ by adjoining two elements, which we denote by $($ and $)$, to be treated as left and right parentheses. Precisely, we have
$$
X = \bigsqcup_{n \ge 0}\bigg( \mathbb N \times (\Sigma^{\text{sort}} \sqcup \Sigma^{\text{term}})\; \sqcup \; \{\;(,\;)\;\} \bigg)^n.
$$
Then $(\Sigma^{\text{var}},\Sigma^{\text{sort}},\Sigma^{\text{term}})^*$ is the smallest subset $Y \subset X$ that satisfies the following conditions:
\begin{itemize}
	\item If $x$ is a variable, then the length-$1$ sequence $x$ belongs to $Y$.
	
	\item Suppose given $\sigma \in \Sigma^{\text{sort}} \sqcup \Sigma^{\text{term}}$ and a sequence $w_1$, ..., $w_n$ of elements of $Y$ where $n \ge 0$. Then the length-$(n+3)$ sequence
	$$
	\sigma,\; (,\; w_1, ..., w_n,\; )
	$$
	belongs to $Y$.
	
	We will denote it by $\sigma(w_1, ..., w_n)$ when $n \ge 0$, or simply by $\sigma$ when $n = 0$.
\end{itemize}
The structure of model of $\bbA$ on this set is given as follows: for $(\sigma,n) \in \bbN \times (\Sigma^{\text{sort}} \sqcup \Sigma^{\text{term}})$, i.e. an element of $\Sigma^{\text{sort}} \sqcup \Sigma^{\text{term}}$ viewed as an $n$-ary operation, its action on $X$ is given by sending a sequence $w_1$, ..., $w_n$ in $X$ to the sequence $\sigma(w_1, ..., w_n)$ described above.
\end{construction}

\begin{definition}
In the above notation, if $w \in (\Sigma^{\text{var}},\Sigma^{\text{sort}},\Sigma^{\text{term}})^*$ and $\sigma$ is a variable, sort symbol or term symbol, we define an \emph{occurrence} of $\sigma$ as an index $i$ such that the $i$-th entry of $w$ equals $\sigma$. If there exists an occurrence of $\sigma$ in $w$, we say that $\sigma$ \emph{occurs} in $w$.

Given $w$, $u \in (\Sigma^{\text{var}},\Sigma^{\text{sort}},\Sigma^{\text{term}})^*$ and a variable $x$, we write $w[u \mid x]$ for the element of $(\Sigma^{\text{var}},\Sigma^{\text{sort}},\Sigma^{\text{term}})^*$ obtained by replacing every occurrence of $x$ by $u$

More generally, for $k \ge 1$, a sequence of variables $x_1$, ..., $x_k$, and a sequence $u_1$, ..., $u_k$ in $(\Sigma^{\text{var}},\Sigma^{\text{sort}},\Sigma^{\text{term}})^*$, we write $w[u_1 \mid x_1, ..., u_k \mid x_k]$ for the expression obtained by \emph{simultaneously} replacing every occurrence of $x_i$ by $u_i$ for each $i$.

An expression $w \in (\Sigma^{\text{var}},\Sigma^{\text{sort}},\Sigma^{\text{term}})^*$ is said to be a \emph{term expression} if no sort symbol occurs in $w$. We say that $w$ is a \emph{sort expression} if it is of the form $T(u_1, ..., u_n)$ where $T$ is a sort symbol and $u_1$, ..., $u_n$ are term expressions.
\end{definition}

Following our convention for sort symbols, we will usually denote term expressions by using lowercase letters, and sort expressions by using uppercase ones.

\begin{definition}
\label{def: precontext}
Given an alphabet $\Sigma=(\Sigma^{\text{var}},\Sigma^{\text{sort}},\Sigma^{\text{term}})$, a \emph{(length-$n$) precontext} is a sequence of ordered pairs $((x_1,X_1),...,(x_n,X_n))$, with $n \ge 0$, where, for each $i$, $x_i$ is a variable and $X_i$ is a sort expression that only contains variables among $x_1$, ..., $x_{i-1}$. We denote it by
	$$
	x_1:X_1,\ldots,x_n:X_n.
	$$
We will usually denote contexts by boldface letters such as $\textbf{X}$ and $\textbf{Y}$.

Given contexts
$$
\textbf{X} = (x_1:X_1, ..., x_m:X_m), \qquad\qquad \textbf{Y} = (y_1:Y_1, ..., y_n:Y_n),
$$
we write $\textbf{X},\textbf{Y}$ or $\textbf{X},y_1:Y_1, ..., y_n:Y_n$ for the concatenation
$$
x_1:X_1, ..., x_m:X_m,y_1:Y_1, ..., y_n:Y_n
$$
of $\textbf{X}$ and $\textbf{Y}$.

The unique precontext of length $0$ will be said to be \emph{empty}.
\end{definition}

\begin{notation}
For a precontext $\textbf{X} = (x_1:X_1, ..., x_n:X_n)$, we denote its length (i.e. $n$) by $\bbl(\textbf{X})$. Also, for $i \in \{0, ..., n\}$ we will write $\partial_i\textbf{X}$ for the truncation $(x_1:X_1, ..., x_i:X_i)$.
\end{notation}
	
\begin{definition}
\label{def: judgment}
We define a
	\begin{itemize}
		\item \emph{sort judgment} as a pair $(\textbf{X},U)$ consisting of a precontext $\textbf{X}$ and a sort expression $U$. We denote it by
		$$
		\textbf{X} \vdashcustom U \tp.
		$$
		
		\item \emph{term judgment} as a triple $(\textbf{X},U,u)$ consisting of a precontext $\textbf{X}$, a sort expression $U$, and a term expression $u$. We denote it by
		$$
		\textbf{X} \vdashcustom u:U.
		$$
		
		\item \emph{sort equality judgment} as a triple $(\textbf{X},U,V)$ consisting of a precontext $\textbf{X}$ and sort expressions $U$, $V$. We denote it by
		$$
		\textbf{X} \vdashcustom U \equiv V \tp.
		$$
		
		\item \emph{term equality judgment} as a quadruple $(\textbf{X}, U, u, v)$ consisting of a precontext $\textbf{X}$, a sort expression $U$, and term expressions $t$, $t'$. We denote it by
		$$
		\textbf{X} \vdashcustom u \equiv v: U.
		$$
	\end{itemize}
	
	We will use the word ``judgment" to refer collectively to the structures introduced in these four items, but, more generally, also to other kinds of structure, introduced later.
	
	In contrast, a judgment of one of the four kinds above will be called a \emph{standard judgment}.
	
	We will usually denote judgments by the letter $J$ and variants such as $J'$ and $J_1$, ..., $J_n$.
	
	If $J$ is a standard judgment, the precontext $\textbf{X}$ preceding the turnstile symbol will be referred to as the \emph{context of} $J$.
	
	When the context of a standard judgment is empty, we omit it from the notation: $\textbf{X} \vdashcustom U \tp$ becomes $\vdashcustom U \tp$.
	
	\vspace{0.5em}
	
	A precontext $\textbf{X}$ corresponds to a \emph{context judgment}, which we denote by\footnote{Technically, a context judgment is the same as a precontext. However, the shift in terminology and notation corresponds to the fact that, in what follows, the same data will be treated from two different perspectives.}
	$$
	\textbf{X} \ctx.
	$$
	
	We will write $\mathscr J(\Sigma)$, or just $\mathscr J$ when the alphabet is implicit, for the union of the set of standard judgments and that of context judgments of $\Sigma$.
\end{definition}

\begin{definition}[\cite{Car78}]
A \emph{pretheory} is a pair $\bbT = (\Sigma, \mathscr A)$ where $\Sigma$ is a sort-and-term alphabet and $\mathscr A$ is a subset of the set of standard judgments of $\Sigma$ $-$ whose elements we call the \emph{axioms} of $\bbT$ $-$ with the following properties:
	\begin{enumerate}[label=(\roman*)]
		\item If $T$ is a sort symbol, then $\mathscr A$ contains a unique judgment of the form
		$$
		x_1:X_1,\ldots,x_n:X_n \vdashcustom T(x_1,...,x_n) \tp,
		$$
		which we refer to as axiom that \emph{introduces} $T$.
		
		Moreover, $\mathscr A$ does not contain any other sort judgments.
		
		\item If $t$ is a term symbol, then $\mathscr A$ contains a unique judgment of the form
		$$
		x_1:X_1,\ldots,x_n:X_n \vdashcustom t(x_1,...,x_n):U,
		$$
		which we refer to as axiom that \emph{introduces} $t$.
		
		Moreover, $\mathscr A$ does not contain any other term judgments.
	\end{enumerate}
\end{definition}

Generally, an axiom $J \in \mathscr A$ is called a \emph{sort} (resp. \emph{term}, \emph{sort equality}, \emph{term equality}) \emph{axiom} if it is, respectively, a sort (resp. term, sort equality, term equality) judgment.

In the above notation, we refer to $x_1:X_1,\ldots,x_n:X_n \vdashcustom T(x_1,...,x_n) \tp$, resp. $x_1:X_1,\ldots,x_n:X_n \vdashcustom t(x_1,...,x_n):U$, as the \emph{introduction axiom} of $T$, resp. $t$.

\vspace{0.5em}

For a pretheory $\bbT$, we will write $\Sigma^{\text{var}}(\bbT)$, $\Sigma^{\text{sort}}(\bbT)$, $\Sigma^{\text{term}}(\bbT)$, and $\mathscr A(\bbT)$, respectively, for its set of variables, of sort symbols, of term symbols, and of axioms.

\begin{definition}
\label{def: GAT inference rules}
An \emph{inference step} with respect to a given sort-and-term alphabet is a sequence $(J_1,...,J_n,J_{n+1})$ in $\mathscr J$ with $n \ge 0$ (hence each entry is either a standard judgment or a context judgment). We denote it by
$$
\inferrule{J_1 \\ J_2 \\ \ldots\ \\ J_n}{J_{n+1}}.
$$
The judgments $J_1$, ..., $J_n$ are called the \emph{premises}, and $J_{n+1}$ is called the \emph{conclusion}.

By an \emph{inference rule} we will mean a set whose elements are inference steps. If $R$ is an inference rule, we refer to an inference step in $R$ as an \emph{instance} of $R$. If $(J_1, ..., J_n,J_{n+1})$ is an instance of $R$, we write
$$
\inferrule{J_1 \\ J_2 \\ \ldots\ \\ J_n}{J_{n+1}}(R).
$$

For a fixed pretheory $\bbT$, we write $\mathscr R_{gat}$ for the set of all inference steps, with respect to the alphabet of $\bbT$, of one of the sixteen forms below. Note that for each item we obtain an inference rule by taking the set of all inference steps of the given form.

\begin{enumerate}
	\item[\textbf{(ctx)}] $$
	\inferrule{\;}{\varnothing \ctx}
	$$
	where $\varnothing$ denotes the empty precontext. And, for $n \ge 0$,
	$$
	\inferrule{x_1:X_1, ..., x_n:X_n \vdashcustom U \tp}{x_1:X_1, ..., x_n:X_n, x_{n+1}:U \quad \ctx}
	$$
	whenever $x_{n+1}$ is a variable distinct from $x_1$, ..., $x_n$.
	
	\item[\textbf{(s1)}] $$
	\inferrule{\textbf{X} \vdashcustom U \tp}{\textbf{X} \vdashcustom U \equiv U \tp}.
	$$
	\item[\textbf{(s2)}] $$
	\inferrule{\textbf{X} \vdashcustom U\equiv U' \tp}{\textbf{X} \vdashcustom U'\equiv U \tp}.
	$$
	\item[\textbf{(s3)}] $$
	\inferrule{\textbf{X} \vdashcustom U\equiv U' \tp \\ \textbf{X} \vdashcustom U'\equiv U'' \tp}{\textbf{X} \vdashcustom U\equiv U'' \tp}.
	$$
	\item[\textbf{(t1)}] $$
	\inferrule{\textbf{X} \vdashcustom u:U}{\textbf{X} \vdashcustom u \equiv  u:U}.
	$$
	\item[\textbf{(t2)}] $$
	\inferrule{\textbf{X} \vdashcustom u\equiv u':U}{\textbf{X} \vdashcustom u'\equiv u:U}.
	$$
	\item[\textbf{(t3)}] $$
	\inferrule{\textbf{X} \vdashcustom u\equiv u':U \\ \textbf{X} \vdashcustom u'\equiv u'':U}{\textbf{X} \vdashcustom u\equiv u'':U}.
	$$
	\item[\textbf{(seq/t)}] $$
	\inferrule{\textbf{X} \vdashcustom U\equiv U' \tp \\ \textbf{X} \vdashcustom u:U}{\textbf{X} \vdashcustom u:U'}.
	$$
	\item[\textbf{(seq/teq)}] $$
	\inferrule{\textbf{X} \vdashcustom U\equiv U' \tp \\ \textbf{X} \vdashcustom u\equiv u':U \\ \textbf{X} \vdashcustom u:U' \\ \textbf{X} \vdashcustom u':U'}{\textbf{X} \vdashcustom u\equiv u':U'}.
	$$
	
	\item[\textbf{(var)}] $$
	\inferrule{x_1:X_1, ..., x_n:X_n \vdashcustom X_i \tp}{x_1:X_1, ..., x_n:X_n \vdashcustom x_i:X_i}
	$$
	where $n \ge 1$ and $1 \le i \le n$.
	
	\item[\textbf{(s-a)}] For each sort axiom $\textbf{X} \vdashcustom T(x_1,...,x_n) \tp$ we consider
	$$
	\inferrule{\textbf{X} \ctx \\ \textbf{X} \vdashcustom x_1:X_1 \\ \cdots \\ \textbf{X} \vdashcustom x_n:X_n}{\textbf{X} \vdashcustom T(x_1,...,x_n) \tp}
	$$
	where $\textbf{X} = (x_1:X_1, ..., x_n:X_n)$.
	
	\item[\textbf{(t-a)}] For each term axiom $\textbf{X} \vdashcustom t(x_1,...,x_n):U$ we consider
	$$
	\inferrule{\textbf{X} \vdashcustom U \tp \\ \textbf{X} \vdashcustom x_1:X_1 \\ \cdots \\ \textbf{X} \vdashcustom x_n:X_n}{\textbf{X} \vdashcustom t(x_1,...,x_n):U}
	$$
	where $\textbf{X} = (x_1:X_1, ..., x_n:X_n)$.
	
	\item[\textbf{(seq-a)}] For each sort equality axiom $\textbf{X} \vdashcustom U \equiv  U' \tp$ we consider
	$$
	\inferrule{\textbf{X} \vdashcustom U \tp \\ \textbf{X} \vdashcustom U' \tp}{\textbf{X} \vdashcustom U \equiv  U' \tp}.
	$$
	
	\item[\textbf{(teq-a)}] For each term equality axiom $\textbf{X} \vdashcustom u \equiv  u':U$, we consider
	$$
	\inferrule{\textbf{X} \vdashcustom u:U \\ \textbf{X} \vdashcustom u':U}{\textbf{X} \vdashcustom u \equiv  u':U}.
	$$	
	
	\item[\textbf{(s-sub)}] Suppose given a sort axiom $J$, say $\textbf{X} \vdashcustom T(x_1, ..., x_n) \tp$, where $\textbf{X}$ is $x_1:X_1, ..., x_n:X_n$ (possibly with $n = 0$), a precontext $\textbf{Y}$, and term expressions $f_1$, ..., $f_n$. For $i = 1$, ..., $n$, let $J_i$ be the judgment
	$$
	\textbf{Y} \vdashcustom f_i : X_i[f_1 \mid x_1, ..., f_{i-1} \mid x_{i-1}].
	$$
	We consider
	$$
	\inferrule{J \\ \textbf{Y} \ctx \\ J_1 \\ J_2 \\ \cdots \\ J_n}{\textbf{Y} \vdashcustom T(f_1, ..., f_n) \tp}.
	$$
	
	\item[\textbf{(t-sub)}] Suppose given a term axiom $J$, say $\textbf{X} \vdashcustom t(x_1, ..., x_n):U$ where $\textbf{X}$ is $x_1:X_1, ..., x_n:X_n$ (possibly with $n = 0$), a precontext $\textbf{Y}$, and term expressions $f_1$, ..., $f_n$. For $i = 1$, ..., $n$, let $J_i$ be the judgment
	$$
	\textbf{Y} \vdashcustom f_i : X_i[f_1 \mid x_1, ..., f_{i-1} \mid x_{i-1}].
	$$
	Writing $J'$ for $\textbf{Y} \vdashcustom U[f_1 \mid x_1, ..., f_n \mid x_n] \tp$, we consider
	$$
	\inferrule{J \\ J' \\ J_1 \\ J_2 \\ \cdots \\ J_n}{\textbf{Y} \vdashcustom t(f_1, ..., f_n):U[f_1 \mid x_1, ..., f_n \mid x_n]}.
	$$
	
	\item[\textbf{(seq-sub-1)}] Suppose given a sort equality axiom $J$, say $\textbf{X} \vdashcustom U \equiv  U' \tp$ where $\textbf{X} = (x_1:X_1, ..., x_n:X_n)$. Suppose given a precontext $\textbf{Y}$, term expressions $f_1$, ..., $f_n$, and, for $i = 1$, ..., $n$, let $J_i$ be the judgment
	$$
	\textbf{Y} \vdashcustom f_i: X_i[f_1 \mid x_1, ..., f_{i-1} \mid x_{i-1}].
	$$
	Then, writing $J'$ for $\textbf{Y} \vdashcustom U[f_1 \mid x_1, ..., f_n \mid x_n] \tp$ and $J''$ for $\textbf{Y} \vdashcustom U'[f_1 \mid x_1, ..., f_n \mid x_n] \tp$, we consider
	$$
	\inferrule{J \\ J' \\ J'' \\ J_1 \\ \cdots \\ J_n}{\textbf{Y} \vdashcustom U[f_1 \mid x_1, ..., f_n \mid x_n] \equiv  U'[f_1 \mid x_1, ..., f_n \mid x_n] \tp}.
	$$
	
	\item[\textbf{(seq-sub-2)}]
	
	Suppose given a sort axiom $J$, say $\textbf{X} \vdashcustom T(x_1, ..., x_n) \tp$ where $\textbf{X} = (x_1:X_1, ..., x_n:X_n)$. Suppose given a precontext $\textbf{Y}$, term expressions $f_1$, ..., $f_n$, $g_1$, ..., $g_n$, and, for $i = 1$, ..., $n$, let $J_i$ be the judgment
	$$
	\textbf{Y} \vdashcustom f_i \equiv  g_i : X_i[f_1 \mid x_1, ..., f_{i-1} \mid x_{i-1}].
	$$
	Then, writing $J'$ for $\textbf{Y} \vdashcustom T(f_1, ..., f_n) \tp$ and $J''$ for $\textbf{Y} \vdashcustom T(g_1, ..., g_n) \tp$, we consider
	$$
	\inferrule{J \\ J' \\ J'' \\ J_1 \\ \cdots \\ J_n}{\textbf{Y} \vdashcustom T(f_1, ..., f_n) \equiv T(g_1, ..., g_n) \tp}.
	$$
	
	\item[\textbf{(teq-sub-1)}]
	
	Suppose given a term equality axiom $J$, say $\textbf{X} \vdashcustom u \equiv  u' : U$ where $\textbf{X}$ is $x_1:X_1, ..., x_n:X_n$. Suppose given a precontext $\textbf{Y}$, term expressions $f_1$, ..., $f_n$, and, for $i = 1$, ..., $n$, let $J_i$ be the judgment
	$$
	\textbf{Y} \vdashcustom f_i : X_i[f_1 \mid x_1, ..., f_{i-1} \mid x_{i-1}].
	$$
	Then, writing $J'$ for $\textbf{Y} \vdashcustom u[f_1 \mid x_1, ..., f_n \mid x_n]:U[f_1 \mid x_1, ..., f_n \mid x_n]$ and $J''$ for $\textbf{Y} \vdashcustom u'[f_1 \mid x_1, ..., f_n \mid x_n]:U[f_1 \mid x_1, ..., f_n \mid x_n]$, we consider
	$$
	\inferrule{J \\ J' \\ J'' \\ J_1 \\ \cdots \\ J_n}{\textbf{Y} \vdashcustom u[f_1 \mid x_1, ..., f_n \mid x_n] \equiv  u'[f_1 \mid x_1, ..., f_n \mid x_n] : U[f_1 \mid x_1, ..., f_n \mid x_n]}.
	$$
	
	\item[\textbf{(teq-sub-2)}] Suppose given a term axiom $J$, say $\textbf{X} \vdashcustom t(x_1, ..., x_n): U$ where $\textbf{X}$ is $x_1:X_1, ..., x_n:X_n$. Suppose given a precontext $\textbf{Y}$, term expressions $f_1$, ..., $f_n$, $g_1$, ..., $g_n$, and, for $i = 1$, ..., $n$, let $J_i$ be the judgment
	$$
	\textbf{Y} \vdashcustom f_i \equiv  g_i : X_i[f_1 \mid x_1, ..., f_{i-1} \mid x_{i-1}].
	$$
	Then, writing $J'$ for $\textbf{Y} \vdashcustom t(f_1, ..., f_n):U[f_1 \mid x_1, ..., f_n \mid x_n]$ and $J''$ for $\textbf{Y} \vdashcustom t(g_1, ..., g_n):U[f_1 \mid x_1, ..., f_n \mid x_n]$, we consider
	$$
	\inferrule{J \\ J' \\ J'' \\ J_1 \\ \cdots \\ J_n}{\textbf{Y} \vdashcustom t(f_1, ..., f_n) \equiv t(g_1, ..., g_n) : U[f_1 \mid x_1, ..., f_n \mid x_n]}.
	$$
\end{enumerate}

\end{definition}

\subsection{Derivable judgments}

We will now define \emph{derivable judgments} with respect to a pretheory $\bbT$.

\begin{definition}
\label{def: derivable judgment}
Let $\bbT$ be a pretheory, and recall the definition of the set of inference steps $\mathscr R_{gat}$ given above.

We have an increasing sequence $(\mathscr D_n)_{n \ge 0}$ of subsets of $\mathscr R_{gat}$ constructed recursively as follows:
\begin{enumerate}[label=(\roman*)]
	\item $\mathscr D_0$ is the empty set.
	
	\item For each $n \ge 0$, $\mathscr D_{n+1}$ is the union of $\mathscr D_n$ with the set of all inference steps
	$$
	\inferrule{J_1 \\ \cdots \\ J_k}{J}
	$$
	(where $k \ge 0$) in $\mathscr R_{gat}$ such that $J_1$, ..., $J_k$ occur as the conclusion of some inference step in $\mathscr D_n$.
\end{enumerate}

\begin{definition}
\label{def: derivable judgment, height, initial inference}
A judgment $J$ is \emph{derivable} if it is the conclusion of a derivation step in $\mathscr D_n$ for some $n \ge 0$. The \emph{height} of $J$ is defined as the smallest such $n$. We will denote it by $\Ht(J)$.

An \emph{initial inference} of a derivable judgment $J$ is an inference step in $\mathscr D_{\Ht(J)}$ whose conclusion is $J$. (Hence whose premises have height $< \Ht(J)$.)
\end{definition}

A \emph{context} is a precontext $\textbf{X}$ such that $\textbf{X} \ctx$ is derivable. The height of $\textbf{X}$ is defined as follows: expressing $\textbf{X}$ as $x_1:X_1, ..., x_n:X_n$, we let $\Ht(\textbf{X})$ be $0$ if $n = 0$, and $\Ht(x_1:X_1, ..., x_{n-1}:X_{n-1} \vdashcustom X_n \tp)$ otherwise. (Note, in particular, that $\Ht(\textbf{X} \ctx) = \Ht(\textbf{X}) + 1$.)
\end{definition}

\begin{definition}
A \emph{generalized algebraic theory} (which we will often refer to simply as a \emph{theory}, or abbreviate as \textsc{gat}) is a pretheory in which every axiom is derivable.
\end{definition}

This means, for instance, that if $J$ is a sort axiom, say $\textbf{X} \vdashcustom T(x_1, ..., x_n) \tp$, then the premises of the corresponding instance of (s-a), i.e.
$$
\textbf{X} \ctx, \quad \textbf{X} \vdashcustom x_1:X_1, \quad \ldots, \quad \textbf{X} \vdashcustom x_n:X_n,
$$
are derivable. Similarly for the other three kinds of axioms.

\subsection{Other kinds of judgment, the canonical sort of a term, and context morphisms}

\label{subsec: other kinds of judgment}

We will also consider the following structures:

\begin{itemize}
	\item An \emph{partial term judgment} is a pair $(\textbf{X},u)$ consisting of a precontext $\textbf{X}$ and a term expression $u$. We denote it by
	$$
	\textbf{X} \vdashcustom u \tm.
	$$
	
	We say that it is \emph{derivable} if there exists a sort expression $U$ such that $\textbf{X} \vdashcustom u:U$ is derivable.
	
	\item An \emph{partial term equality judgment} is a triple $(\textbf{X},u,u')$ of a precontext $\textbf{X}$ and term expressions $u$, $u'$. We denote it by
	$$
	\textbf{X} \vdashcustom u \equiv u' \tm.
	$$
	
	We say that it is \emph{derivable} if there exists a sort expression $U$ such that $\textbf{X} \vdashcustom u \equiv u':U$ is derivable.
	
	\begin{remark}
		The absence of a sort symbol from these judgments is not intended to suggest that it is meaningful to specify terms that do not have a sort. The judgment $\textbf{X} \vdashcustom u \tm$ can be viewed, by the definition of when it is derivable, as an abbreviation for ``$u$ is a term of \emph{some} sort $U$ in context $\textbf{X}$". In fact, we make sense of $u$ being a ``term in context $\textbf{X}$" at the same time as we assign a sort to it.
		
		But for $\textbf{X} \vdashcustom u \tm$ to be semantically unambiguous, it is crucial that the syntax of \textsc{gat}s be such that if $\textbf{X} \vdashcustom u:U$ and $\textbf{X} \vdashcustom u:V$ are derivable, then $\textbf{X} \vdashcustom U \equiv V \tp$ is derivable. This is indeed the case, as noted by Cartmell in \cite{Car78}, page 1.35, as a corollary of the ``derivation lemma".
		
		Considering such judgments turned out to be quite useful in the present text for the following reason: if we have derivable judgments $\textbf{X} \vdashcustom u:U$, \;\;$\textbf{X} \vdashcustom v:V$, and $\textbf{X} \vdashcustom U \equiv V \tp$, then both $\textbf{X} \vdashcustom u \equiv v:U$ and $\textbf{X} \vdashcustom u \equiv v:V$ are well-formed (in the sense of \cite{Car78}), and one is derivable precisely when the other is. The judgment $\textbf{X} \vdashcustom u \equiv v \tm$ encodes the meaning of either without the need of artificially choosing a sort expression such as $U$, $V$, etc.
	\end{remark}
	
	\item A \emph{context equality judgment} is a pair $(\textbf{X},\textbf{Y})$ consisting of two precontexts of the same length, say $\textbf{X} = (x_1:X_1, ..., x_n:X_n)$ and $\textbf{Y} = (y_1:Y_1, ..., y_n:Y_n)$. We denote it by
	$$
	\textbf{X} \equiv \textbf{Y} \ctx.
	$$
	We say that it is \emph{derivable} if $\textbf{X} \ctx$ and $\textbf{Y} \ctx$ are derivable and, for $1 \le i \le n$, the judgment
	$$
	\partial_{i-1}\textbf{X} \vdashcustom X_i \equiv Y_i[x_1 \mid y_1, ..., x_{i-1} \mid y_{i-1}] \tp
	$$
	is derivable.
	
	\item Consider a derivable judgment $x_1:X_1, ..., x_n:X_n \vdashcustom u \tm$. The \emph{canonical sort} of $u$ (with the context implicit), denoted by $\Type(u)$, is defined as follows: if $u$ is a variable $x_i$, then $\Type(u) = X_i$; if $u = t(f_1, ..., f_k)$ with $t$ introduced by the axiom $\textbf{A} \vdashcustom t(a_1, ..., a_k):V$, then $\Type(u) = V[f_1 \mid a_1, ..., f_k \mid a_k]$.
	
	(Later we will check that $x_1:X_1, ..., x_n:X_n \vdashcustom u:\Type(u)$ is derivable.)
	
	\item A \emph{premorphism} is a triple $(\textbf{X},\textbf{Y},\textbf{f})$ consisting of precontexts $\textbf{X}$, $\textbf{Y} = (y_1:Y_1, ..., y_n:Y_n)$ and a sequence of term expressions $\textbf{f} = (f_1, ..., f_n)$. We denote it by
	$$
	\textbf{f}:\textbf{X} \longrightarrow \textbf{Y}.
	$$
	
	We say that it is \emph{derivable} or that it is a \emph{(context) morphism} if $\textbf{X} \ctx$ and $\textbf{Y} \ctx$ are derivable and, for $1 \le i \le n$, the judgment
	$$
	\textbf{X} \vdashcustom f_i: Y_i[f_1 \mid y_1, ..., f_{i-1} \mid y_{i-1}]
	$$
	is derivable.
	
	\item A \emph{premorphism equality} is a quadruple $(\textbf{X}, \textbf{Y}, \textbf{f}, \textbf{g})$ consisting of precontexts $\textbf{X}$, $\textbf{Y} = (y_1:Y_1, ..., y_n:Y_n)$ and sequences of term expressions $\textbf{f} = (f_1, ..., f_n)$, $\textbf{g} = (g_1, ..., g_n)$. We denote it by
	$$
	\textbf{f} \equiv \textbf{g}: \textbf{X} \longrightarrow \textbf{Y}.
	$$
	We say that it is \emph{derivable} or that it is a \emph{(context) morphism equality} if $\textbf{f}:\textbf{X} \rightarrow \textbf{Y}$ and $\textbf{g}:\textbf{X} \rightarrow \textbf{Y}$ are derivable and, for $1 \le i \le n$, the judgment
	$$
	\textbf{X} \vdashcustom f_i \equiv g_i: Y_i[f_1 \mid y_1, ..., f_{i-1} \mid y_{i-1}]
	$$
	is derivable.
\end{itemize}

We extend the height function $\Ht$, previously defined for standard judgments and context judgments, to the newly introduced ones in the following way:

\begin{itemize}
	\item For derivable partial term judgments,
	$$
	\Ht(\textbf{X} \vdashcustom u \tm) = \min_U \Ht(\textbf{X} \vdashcustom u:U)
	$$
	where the minimum is taken over all sort expressions $U$ such that $\textbf{X} \vdashcustom u:U$ is derivable.
	
	\item For derivable partial term equality judgments,
	$$
	\Ht(\textbf{X} \vdashcustom u \equiv u' \tm) = \min_U \Ht(\textbf{X} \vdashcustom u \equiv u':U)
	$$
	where the minimum is taken over all sort expressions $U$ such that $\textbf{X} \vdashcustom u \equiv u':U$ is derivable.
	
	\item For derivable context equality judgments,
	$$
	\Ht(\textbf{X} \equiv \textbf{Y} \ctx) = \max\{\Ht(\textbf{X}),\;\; \Ht(\textbf{Y}), \;\;\Ht(\partial_0 \textbf{X} \vdashcustom X_1 \equiv Y'_1 \tp),\;\;\; ...,\;\;\; \Ht(\partial_{n-1} \textbf{X} \vdashcustom X_n \equiv Y'_n \tp)\}
	$$
	where $\textbf{X} = (x_1:X_1, ..., x_n:X_n)$, $\textbf{Y} = (y_1:Y_1, ..., y_n:Y_n)$, and $Y'_i = Y_i[x_1 \mid y_1, ..., x_{i-1} \mid y_{i-1}]$.
	
	\item For morphisms,
	$$
	\Ht(\textbf{f}:\textbf{X} \rightarrow \textbf{Y}) = \max\{\Ht(\textbf{X}),\;\; \Ht(\textbf{Y}),\;\; \Ht(\textbf{X} \vdashcustom f_1:Y'_1),\;\; ...,\;\; \Ht(\textbf{X} \vdashcustom f_n:Y'_n)\}
	$$
	where $\textbf{Y} = (y_1:Y_1, ..., y_n:Y_n)$ and $Y'_i = Y_i[f_1 \mid y_1, ..., f_{i-1} \mid y_{i-1}]$.
	
	\item For morphism equalities,
	$$
	\Ht(\textbf{f} \equiv \textbf{g}:\textbf{X} \rightarrow \textbf{Y}) = \max\{\Ht(\textbf{f}:\textbf{X} \rightarrow \textbf{Y}),\;\; \Ht(\textbf{g}:\textbf{X} \rightarrow \textbf{Y}),\;\; \Ht(\textbf{X} \vdashcustom f_1 \equiv g_1: Y'_1),\;\; ..., \;\; \Ht(\textbf{X} \vdashcustom f_n \equiv g_n:Y'_n)\}
	$$
	where $\textbf{Y} = (y_1:Y_1, ..., y_n:Y_n)$ and $Y'_i = Y_i[f_1 \mid y_1, ..., f_{i-1} \mid y_{i-1}]$.
\end{itemize}

\begin{notation}
Given a sequence of term expressions $\textbf{f} = (f_1, ..., f_n)$, a context $\textbf{X} = (x_1:X_1, ..., x_n:X_n)$, and an expression $w$ that only contains variables among $x_1$, ..., $x_n$, we write $w[\textbf{f}]$ for $w[f_1 \mid x_1, ..., f_n \mid x_n]$, with $\textbf{X}$ implicit. We will use this, for example, when $\textbf{f}$ is a morphism (or premorphism) from some context $\textbf{Y}$ to $\textbf{X}$.

For $i = 0$, ..., $n$, we let $\partial_i\textbf{f} = (f_1, ..., f_i)$. Note that for a premorphism (resp. morphism) $\textbf{f}:\textbf{X} \rightarrow \textbf{Y}$, we have a premorphism (resp. morphism) $\partial_i\textbf{f}:\textbf{X} \rightarrow \partial_i\textbf{Y}$.
\end{notation}

\begin{definition}
\label{def: composite of context morphisms}
Suppose given contexts $\textbf{X} = (x_1:X_1, ..., x_m:X_m)$, $\textbf{Y} = (y_1:Y_1, ..., y_n:Y_n)$, and $\textbf{Z} = (z_1:Z_1, ..., z_p:Z_p)$, and context morphisms
$$
\textbf{f} = (f_1, ..., f_n):\textbf{X} \longrightarrow \textbf{Y},
$$
$$
\textbf{g} = (g_1, ..., g_p):\textbf{Y} \longrightarrow \textbf{Z}.
$$
Their \emph{composite} is the premorphism $(\textbf{X}, \textbf{Z},\textbf{g} \circ \textbf{f})$ where the tuple of expressions $\textbf{g} \circ \textbf{f}$ is given by
$$
(\textbf{g} \circ \textbf{f})_i = g_i[\textbf{f}] = g_i[f_1 \mid y_1, ..., f_n \mid y_n]
$$
for $1 \le i \le p$.
\end{definition}

It is true, but not clear at this point, that $\textbf{g} \circ \textbf{f}$ is a morphism from $\textbf{X}$ to $\textbf{Z}$. This will be verified in Corollary \ref{cor: composite is a morphism} as a consequence of Proposition \ref{prop: substitution lemma}, which states that derivable judgments are closed under substitution along context morphisms.

\subsection{Properties of the height function}

We now verify some properties of the height of derivable judgments, contexts, etc., as defined above.

\begin{proposition}
\label{prop: properties height}
Let $\bbT$ be a pretheory. The following hold:
\begin{enumerate}[label=(\alph*)]
	\item Every derivable judgment has height at least $1$. The only context of height $0$ is $\varnothing$, and the only context morphism of height $0$ is the identity $():\varnothing \rightarrow \varnothing$.
	
	\item For a derivable sort judgment $\textbf{X} \vdashcustom U \tp$ we have
	$$
	\Ht(\textbf{X} \vdashcustom U \tp) > \Ht(\textbf{X}).
	$$
	As a consequence, for a length-$n$ context $\textbf{X}$ we have $\Ht(\partial_i\textbf{X}) < \Ht(\textbf{X})$ for $0 \le i < n$.
	
	\item For a derivable sort equality judgment $\textbf{X} \vdashcustom U = U' \tp$ we have
	$$
	\Ht(\textbf{X} \vdashcustom U \equiv U' \tp) > \Ht(\textbf{X} \vdashcustom U \tp),
	$$
	$$
	\Ht(\textbf{X} \vdashcustom U \equiv U' \tp) > \Ht(\textbf{X} \vdashcustom U' \tp).
	$$
	
	\item For a derivable term judgment $\textbf{X} \vdashcustom u:U$ we have
	$$
	\Ht(\textbf{X} \vdashcustom u:U) > \Ht(\textbf{X} \vdashcustom U \tp).
	$$
	
	\item For a derivable term equality judgment $\textbf{X} \vdashcustom u \equiv u':U$ we have
	$$
	\Ht(\textbf{X} \vdashcustom u \equiv u':U) > \Ht(\textbf{X} \vdashcustom u:U),
	$$
	$$
	\Ht(\textbf{X} \vdashcustom u \equiv u':U) > \Ht(\textbf{X} \vdashcustom u':U).
	$$
	
	\item If the judgment $\textbf{X} \vdashcustom u:U$ is derivable, so is $\textbf{X} \vdashcustom u:\Type(u)$. Moreover, we have
	$$
	\Ht(\textbf{X} \vdashcustom u:U) \ge \Ht(\textbf{X} \vdashcustom u:\Type(u)),
	$$
	$$
	\Ht(\textbf{X} \vdashcustom u:U) \ge \Ht(\textbf{X} \vdashcustom \Type(u) \equiv U \tp),
	$$
	and the first inequality is strict if and only if $U$ and $\Type(u)$ are distinct.
	
	\item If $\textbf{X} \vdashcustom u \tm$ is derivable, then $\Ht(\textbf{X} \vdashcustom u \tm) = \Ht(\textbf{X} \vdashcustom u:\Type(u))$.
	
	\item For a derivable judgment $\textbf{X} \vdashcustom T(f_1, ..., f_k) \tp$, we have
	$$
	\Ht(\textbf{X} \vdashcustom T(f_1, ..., f_k) \tp) > \Ht(\textbf{X} \vdashcustom f_i:\Type(f_i))
	$$
	for $1 \le i \le k$. Also, if $\textbf{X} \vdashcustom T(f_1, ..., f_k)$ is not an axiom, then, letting $\textbf{A} \vdashcustom T(a_1, ..., a_k) \tp$ be the axiom that introduces $T$ and $\textbf{f} = (f_1, ..., f_k)$,
	$$
	\Ht(\textbf{X} \vdashcustom T(f_1, ..., f_k) \tp) > \Ht(\textbf{f}:\textbf{X} \rightarrow \textbf{A}).
	$$
	
	\item For a derivable judgment $\textbf{X} \vdashcustom t(f_1, ..., f_k):U$, we have
	$$
	\Ht(\textbf{X} \vdashcustom t(f_1, ..., f_k):U) > \Ht(\textbf{X} \vdashcustom f_i \tm)
	$$
	for $1 \le i \le k$.
	
	\item For a context morphism $\textbf{f}:\textbf{X} \rightarrow \textbf{Y}$, where $\textbf{f} = (f_1, ..., f_n)$,
	$$
	\Ht(\textbf{f}:\textbf{X} \rightarrow \textbf{Y}) \ge \Ht(\textbf{X} \vdashcustom f_i \tm)
	$$
	for $1 \le i \le n$. Also, for a morphism equality $\textbf{f} \equiv \textbf{g}:\textbf{X} \rightarrow \textbf{Y}$ we have
	$$
	\Ht(\textbf{f} \equiv \textbf{g}:\textbf{X} \rightarrow \textbf{Y}) \ge \Ht(\textbf{f}:\textbf{X} \rightarrow \textbf{Y}), \;\; \Ht(\textbf{g}:\textbf{X} \rightarrow \textbf{Y}).
	$$
\end{enumerate}
\end{proposition}

\begin{proof}
\leavevmode
\begin{enumerate}[label=(\alph*)]
	\item These follow directly from the definition.
	
	\item Note that $\textbf{X} \vdashcustom U \tp$ is the conclusion of an inference step (s-a) or (s-sub). Choosing an initial inference, in either case $\textbf{X} \ctx$ is among the premises, so $\Ht(\textbf{X} \vdashcustom U \tp) > \Ht(\textbf{X} \ctx) > \Ht(\textbf{X})$.
	
	\item We proceed inductively: letting $J = (\textbf{X} \vdashcustom U \equiv U' \tp)$ and $\Ht(J) = h$, let us prove that
	$$
	\Ht(\textbf{X} \vdashcustom U \equiv U' \tp) > \Ht(\textbf{X} \vdashcustom U \tp),
	$$
	$$
	\Ht(\textbf{X} \vdashcustom U \equiv U' \tp) > \Ht(\textbf{X} \vdashcustom U' \tp)
	$$
	hold assuming that the analogous inequalities hold for $\textbf{Y} \vdashcustom V \equiv V' \tp$ whenever $\Ht(\textbf{Y} \vdashcustom V \equiv V' \tp) < h$. Note that $J$ has an initial inference of the form (s1), (s2), (s3), (seq-a), (seq-sub-1), or (seq-sub-2). Let $I$ be one such initial inference.
	
	If $I$ is (s1), (seq-a), (seq-sub-1) or (seq-sub-2), then it has $\textbf{X} \vdashcustom U \tp$ as a premise, so the desired inequalities hold trivially.
	
	If $I$ is (s2), then it has $\textbf{X} \vdashcustom U' \equiv U \tp$ as a premise. But by the induction hypothesis,
	$$
	\Ht(\textbf{X} \vdashcustom U' \equiv U \tp) > \Ht(\textbf{X} \vdashcustom U'), \;\; \Ht(\textbf{X} \vdashcustom U).
	$$
	If $I$ is (s3), it has premises $\textbf{X} \vdashcustom U \equiv U'' \tp$ and $\textbf{X} \vdashcustom U'' \equiv U' \tp$. It follows from the induction hypothesis that
	$$
	\Ht(\textbf{X} \vdashcustom U \equiv U' \tp) > \Ht(\textbf{X} \vdashcustom U \equiv U'' \tp) > \Ht(\textbf{X} \vdashcustom U \tp),
	$$
	$$
	\Ht(\textbf{X} \vdashcustom U \equiv U' \tp) > \Ht(\textbf{X} \vdashcustom U'' \equiv U' \tp) > \Ht(\textbf{X} \vdashcustom U' \tp).
	$$
	
	\item Let $I$ be an initial inference of $\textbf{X} \vdashcustom u:U$.
	
	If $I$ is (seq/t), it has a premise of the form $\textbf{X} \vdashcustom U' \equiv U \tp$. Then by item (c) we have
	$$
	\Ht(\textbf{X} \vdashcustom u:U) > \Ht(\textbf{X} \vdashcustom U' \equiv U \tp) > \Ht(\textbf{X} \vdashcustom U \tp).
	$$
	If $I$ is (var), (t-a) or (t-sub), it has $\textbf{X} \vdashcustom U \tp$ as a premise, so the inequality holds trivially.
	
	\item We will prove by induction on $h$ that the desired inequalities hold for all derivable $\textbf{X} \vdashcustom U \equiv U' \tp$ such that $\Ht(\textbf{X} \vdashcustom U \equiv U' \tp) = h$.
	
	Given $J = (\textbf{X} \vdashcustom U \equiv U' \tp)$, assume that the claim holds for $0$, ..., $\Ht(J) - 1$, and let $I$ be an initial inference of $J$.
	
	If $I$ is (t1), (seq/teq), (teq-a), (teq-sub-1) or (teq-sub-2), it has $\textbf{X} \vdashcustom u:U$ and $\textbf{X} \vdashcustom u':U$ among its premises, so the desired inequalities hold trivially.
	
	If $I$ is (t2), it has $\textbf{X} \vdashcustom u' \equiv u:U$ as a premise. But by the induction hypothesis, $\Ht(\textbf{X} \vdashcustom u' \equiv u:U) > \Ht(\textbf{X} \vdashcustom u':U)$, $\Ht(\textbf{X} \vdashcustom u:U)$.
	
	If $I$ is (t3), it has premises $\textbf{X} \vdashcustom u \equiv u'':U$ and $\textbf{X} \vdashcustom u'' \equiv u':U$, so the induction hypothesis yields
	$$
	\Ht(\textbf{X} \vdashcustom u \equiv u':U) > \Ht(\textbf{X} \vdashcustom u \equiv u'':U) > \Ht(\textbf{X} \vdashcustom u:U),
	$$
	$$
	\Ht(\textbf{X} \vdashcustom u \equiv u':U) > \Ht(\textbf{X} \vdashcustom u'' \equiv u':U) > \Ht(\textbf{X} \vdashcustom u':U).
	$$
	
	\item Let us prove by induction on $h$ that the claim holds whenever $\Ht(\textbf{X} \vdashcustom u:U) = h$.
	
	Given $J = (\textbf{X} \vdashcustom u:U)$ such that $\Ht(J) = h$, assume the statement for $0$, ..., $\Ht(J) - 1$, and let $I$ be an initial inference of $J$.
	
	If $I$ is (var), (t-a) or (t-sub), then $U = \Type(u)$. Hence the first inequality holds trivially, and the second one follows, using items (c) and (d), from
	$$
	\Ht(\textbf{X} \vdashcustom u:U) > \Ht(\textbf{X} \vdashcustom U \tp), \qquad \Ht(\textbf{X} \vdashcustom U \equiv U \tp) = \Ht(\textbf{X} \vdashcustom U \tp) + 1.
	$$
	
	If $I$ is (seq/t), it has premises $\textbf{X} \vdashcustom U' \equiv U \tp$ and $\textbf{X} \vdashcustom u:U'$. Now, the induction hypothesis implies that $\textbf{X} \vdashcustom u:\Type(u)$ is derivable and
	$$
	\Ht(\textbf{X} \vdashcustom u:U) > \Ht(\textbf{X} \vdashcustom u:U') \ge \Ht(\textbf{X} \vdashcustom u:\Type(u)).
	$$
	Also, since $\Ht(\textbf{X} \vdashcustom \Type(u) \equiv U' \tp)$, $\Ht(\textbf{X} \vdashcustom U' \equiv U \tp) < h$, an instance of (s3) yields $\Ht(\textbf{X} \vdashcustom \Type(u) \equiv U \tp) \le h$.
	
	\item It follows directly from (f) and the definition of $\Ht(\textbf{X} \vdashcustom u \tm)$.
	
	\item Let $I$ be an initial inference of $\textbf{X} \vdashcustom T(f_1, ..., f_k) \tp$. Then $I$ is of the form (s-a) or (s-sub) and, in either case, it has a premise $\textbf{X} \vdashcustom f_i:U$ for each $i = 1$, ..., $k$. Now, by item (f) we have
	$$
	\Ht(\textbf{X} \vdashcustom T(f_1, ..., f_k) \tp) > \Ht(\textbf{X} \vdashcustom f_i:U) \ge \Ht(\textbf{X} \vdashcustom f_i:\Type(f_i)).
	$$
	If $\textbf{X} \vdashcustom T(f_1, ..., f_k) \tp$ is not an axiom, then $I$ is (s-sub); if $\textbf{A} \vdashcustom T(a_1, ..., a_k) \tp$ is the axiom that introduces $T$, by using the above inequalities and comparing the premises of (s-sub) with the definition of $\Ht(\textbf{f}:\textbf{X} \rightarrow \textbf{A})$ we obtain
	$$
	\Ht(\textbf{X} \vdashcustom T(f_1, ..., t_k) \tp) > \Ht(\textbf{f}:\textbf{X} \rightarrow \textbf{A}).
	$$
	
	\item We will verify by induction on $h$ that the claim holds whenever $\Ht(\textbf{X} \vdashcustom t(f_1, ..., f_k):U) = h$.
	
	Given $J = (\textbf{X} \vdashcustom t(f_1, ..., f_k):U)$ such that $\Ht(J) = h$, assume the statement for $0$, ..., $\Ht(J) - 1$, and let $I$ be an initial inference of $J$.
	
	If $I$ is (seq/t), then it has a premise $\textbf{X} \vdashcustom t(f_1, ..., f_k):U'$, and by the induction hypothesis we obtain
	$$
	\Ht(\textbf{X} \vdashcustom t(f_1, ..., f_k):U) > \Ht(\textbf{X} \vdashcustom t(f_1, ..., f_k):U') \ge \Ht(\textbf{X} \vdashcustom f_i:\Type(f_i))
	$$
	for $i = 1$, ..., $k$. If $I$ is (t-a) or (t-sub), then it has a premise of the form $\textbf{X} \vdashcustom f_i:V$ for each $i$; by (f) and (g) we conclude that
	$$
	\Ht(\textbf{X} \vdashcustom t(f_1, ..., f_k):U) > \Ht(\textbf{X} \vdashcustom f_i:V) \ge \Ht(\textbf{X} \vdashcustom f_i:\Type(f_i)) = \Ht(\textbf{X} \vdashcustom f_i \tm).
	$$
	
	\item Immediate from (f), (g), and the definitions of $\Ht(\textbf{f}:\textbf{X} \rightarrow \textbf{Y})$ and $\Ht(\textbf{f} \equiv \textbf{g}:\textbf{X} \rightarrow \textbf{Y})$.
\end{enumerate}
\end{proof}

\subsection{The substitution property}

\begin{definition}
Let $\bbT$ be a generalized algebraic theory. We say that a derivable standard judgment $\textbf{X} \vdashcustom \blacksquare$ has the \emph{substitution property} if for every context morphism $\textbf{f}:\textbf{Y} \rightarrow \textbf{X}$ the judgment $\textbf{Y} \vdashcustom \blacksquare[\textbf{f}]$ is derivable.\footnote{Here, ``$\blacksquare$" stands for ``$U \tp$", ``$u:U$", ``$U \equiv U' \tp$", or ``$u \equiv u':U$", and ``$\blacksquare[\textbf{f}]$" denotes ``$U[\textbf{f}] \tp$", ``$u[\textbf{f}]:U[\textbf{f}]$", ``$U[\textbf{f}] \equiv U'[\textbf{f}] \tp$", or ``$u[\textbf{f}] \equiv u'[\textbf{f}]: U[\textbf{f}]$", respectively.}
\end{definition}

The following result is essentially the same as the \emph{substitution lemma} from \cite{Car86}:

\begin{proposition}
\label{prop: substitution lemma}
Every derivable standard judgment has the substitution property.
\end{proposition}

\begin{proof}[Proof sketch]
It suffices to check that for any inference step in the list from Definition \ref{def: GAT inference rules} whose premises are derivable and whose conclusion is a standard judgment (i.e. we do not consider (ctx)), the conclusion has the substitution property provided that the premises do.

We have different approaches depending on the kind of inference rule:

\begin{itemize}
	\item For an inference rule among (s1), (s2), (s3), (t1), (t2), (t3), (seq/t), (seq/teq), we apply substitution to each of the premises in order to obtain an inference step of the same kind. For example, suppose that each premise in
	$$
	\inferrule{\textbf{X} \vdashcustom U \equiv U' \tp \\ \textbf{X} \vdashcustom u \equiv u':U \\ \textbf{X} \vdashcustom u:U' \\ \textbf{X} \vdashcustom u':U'}{\textbf{X} \vdashcustom u \equiv u':U'}\text{(seq/teq)}
	$$
	has the substitution property, and consider a context morphism $\textbf{f}:\textbf{Y} \rightarrow \textbf{X}$. Then the judgments
	$$
	\textbf{Y} \vdashcustom U[\textbf{f}] \equiv U'[\textbf{f}] \tp \qquad \textbf{Y} \vdashcustom u[\textbf{f}] \equiv u'[\textbf{f}]:U[\textbf{f}] \qquad \textbf{Y} \vdashcustom u[\textbf{f}]:U'[\textbf{f}] \qquad \textbf{Y} \vdashcustom u'[\textbf{f}]:U'[\textbf{f}]
	$$
	are derivable, so by (seq/teq) we can derive $\textbf{Y} \vdashcustom u[\textbf{f}] \equiv u'[\textbf{f}]: U'[\textbf{f}]$. The other inference rules are treated analogously.
	
	\item For $\textbf{X} = (x_1:X_1, ..., x_n:X_n)$, suppose that the premise of
	$$
	\inferrule{\textbf{X} \vdashcustom X_i \tp}{\textbf{X} \vdashcustom x_i:X_i}\text{(var)}
	$$
	has the substitution property. For a morphism $\textbf{f} = (f_1, ..., f_n):\textbf{Y} \rightarrow \textbf{X}$, we have that $\textbf{Y} \vdashcustom x_i[\textbf{f}]:X_i[\textbf{f}]$ is $\textbf{Y} \vdashcustom f_i:X_i[\textbf{f}]$, which is derivable by definition of a context morphism.
	
	\item For an inference rule among (s-a), (t-a), (seq-a), (teq-a), applying substitution to its conclusion yields a derivable judgment by (s-sub), (t-sub), (seq-sub-1), (teq-sub-1), respectively. For example, consider a context $\textbf{X} = (x_1:X_1, ..., x_n:X_n)$, an axiom $J = (\textbf{X} \vdashcustom u \equiv u':U)$, and suppose that each premise in
	$$
	\inferrule{\textbf{X} \vdashcustom u:U \\ \textbf{X} \vdashcustom u':U}{\textbf{X} \vdashcustom u \equiv  u':U}\text{(teq-a)}
	$$
	has the substitution property. Given a context morphism $\textbf{f}:\textbf{Y} \rightarrow \textbf{X}$, we obtain derivable judgments
	$$
	J'=(\textbf{Y} \vdashcustom u[\textbf{f}]:U[\textbf{f}]), \qquad J''=(\textbf{Y} \vdashcustom u'[\textbf{f}]:U[\textbf{f}]).
	$$
	Also, for $i = 1$, ..., $n$ let $J_i$ be the (derivable) judgment $\textbf{Y} \vdashcustom f_i: X_i[\textbf{f}]$. Then, by (teq-sub-1), we can use $J$, $J'$, $J''$, $J_1$, ..., $J_n$ to derive $\textbf{Y} \vdashcustom u[\textbf{f}] \equiv u'[\textbf{f}]: U[\textbf{f}]$.
	
	The other kinds of axioms are treated similarly.
	
	\item For an inference rule among (s-sub), (t-sub), (seq-sub-1), (seq-sub-2), (teq-sub-1), (teq-sub-2), an application of substitution to its conclusion can be described by substitution of an axiom along a certain composite premorphism.
	
	To illustrate how this is done, let us work out the cases (teq-sub-1) and (teq-sub-2). The argument for (s-sub), (t-sub), (seq-sub-1) and (seq-sub-2) is analogous.

	For (teq-sub-1), suppose given an axiom $J = (\textbf{X} \vdashcustom u \equiv u':U)$, where $\textbf{X} = (x_1:X_1, ..., x_n:X_n)$, and a context morphism $\textbf{f} = (f_1, ..., f_n):\textbf{Y} \rightarrow \textbf{X}$. For $i = 1$, ..., $n$, let $J_i$ be
	$$
	\textbf{Y} \vdashcustom f_i:X_i[\textbf{f}].
	$$
	Let $J' = (\textbf{Y} \vdashcustom u[\textbf{f}]:U[\textbf{f}])$ and $J'' = (\textbf{Y} \vdashcustom u'[\textbf{f}]:U[\textbf{f}])$.
	
	Assume that each premise in
	$$
	\inferrule{J \\ J' \\ J'' \\ J_1 \\ \cdots \\ J_n}{\textbf{Y} \vdashcustom u[\textbf{f}] \equiv u'[\textbf{f}]:U[\textbf{f}]}\text{(teq-sub-1)}
	$$
	has the substitution property, and consider a morphism $\textbf{g}:\textbf{Z} \rightarrow \textbf{Y}$. Then substitution along $\textbf{g}$ turns
	\begin{itemize}
		\item $J'$ into a derivable judgment $\textbf{Z} \vdashcustom u[\textbf{f}][\textbf{g}]: U[\textbf{f}][\textbf{g}]$, which is
		$$
		\textbf{Z} \vdashcustom u[\textbf{f} \circ \textbf{g}]: U[\textbf{f} \circ \textbf{g}].
		$$
		
		\item $J''$ into a derivable judgment $\textbf{Z} \vdashcustom u'[\textbf{f}][\textbf{g}]: U[\textbf{f}][\textbf{g}]$, which is
		$$
		\textbf{Z} \vdashcustom u'[\textbf{f} \circ \textbf{g}]: U[\textbf{f} \circ \textbf{g}].
		$$
		
		\item $J_i$ into a derivable judgment $\textbf{Z} \vdashcustom f_i[\textbf{g}] :X_i[\textbf{f}][\textbf{g}]$, which is
		$$
		\textbf{Z} \vdashcustom (\textbf{f} \circ \textbf{g})_i : X_i[\textbf{f} \circ \textbf{g}].
		$$
	\end{itemize}
	Now, by (teq-sub-1) we can derive $\textbf{Z} \vdashcustom u[\textbf{f} \circ \textbf{g}] \equiv u'[\textbf{f} \circ \textbf{g}]: U[\textbf{f} \circ \textbf{g}]$. This proves that $\textbf{Y} \vdashcustom u[\textbf{f}] \equiv u'[\textbf{f}]: U[\textbf{f}]$ has the substitution property.
	
	\vspace{0.5em}
	
	For (teq-sub-2), suppose given an axiom $J = (\textbf{X} \vdashcustom u :U)$, where $\textbf{X} = (x_1:X_1, ..., x_n:X_n)$, and premorphisms $\textbf{f} = (f_1, ..., f_n)$, $\textbf{g} = (g_1, ..., g_n): \textbf{Y} \rightarrow \textbf{X}$. For $i = 1$, ..., $n$, let $J_i$ be
	$$
	\textbf{Y} \vdashcustom f_i \equiv g_i:X_i[\textbf{f}].
	$$
	Let $J' = (\textbf{Y} \vdashcustom u[\textbf{f}]: U[\textbf{f}])$ and $J'' = (\textbf{Y} \vdashcustom u[\textbf{g}]: U[\textbf{f}])$.
	
	Assume that each premise in
	$$
	\inferrule{J \\ J' \\ J'' \\ J_1 \\ \cdots \\ J_n}{\textbf{Y} \vdashcustom u[\textbf{f}] \equiv u[\textbf{g}] : U[\textbf{f}]}\text{(teq-sub-2)}
	$$
	has the substitution property, and consider a context morphism $\textbf{h}:\textbf{Z} \rightarrow \textbf{Y}$. Then substitution along $\textbf{h}$ turns
	\begin{itemize}
		\item $J'$ into the derivable judgment $\textbf{Z} \vdashcustom u[\textbf{f}][\textbf{h}]:U[\textbf{f}][\textbf{h}]$, which is $\textbf{Z} \vdashcustom u[\textbf{f} \circ \textbf{h}]: U[\textbf{f} \circ \textbf{h}]$.
		
		\item $J''$ into the derivable judgment $\textbf{Z} \vdashcustom u[\textbf{g}][\textbf{h}]: U[\textbf{f}][\textbf{h}]$, which is $\textbf{Z} \vdashcustom u[\textbf{g} \circ \textbf{h}]: U[\textbf{f} \circ \textbf{h}]$.
		
		\item $J_i$ into the derivable judgment $\textbf{Z} \vdashcustom f_i[\textbf{h}] \equiv g_i[\textbf{h}]:X_i[\textbf{f}][\textbf{h}]$, which is $\textbf{Z} \vdashcustom (\textbf{f} \circ \textbf{h})_i \equiv (\textbf{g} \circ \textbf{h})_i: X_i[\textbf{f} \circ \textbf{h}]$.
	\end{itemize}
	By (teq-sub-2) we can derive $\textbf{Z} \vdashcustom u[\textbf{f} \circ \textbf{h}] \equiv u[\textbf{g} \circ \textbf{h}]: U[\textbf{f} \circ \textbf{h}]$. This proves that $\textbf{Y} \vdashcustom u[\textbf{f}] \equiv u[\textbf{g}]: U[\textbf{f}]$ has the substitution property.
\end{itemize}
\end{proof}

\begin{corollary}
\label{cor: composite is a morphism}
Given context morphisms $\textbf{f}:\textbf{X} \rightarrow \textbf{Y}$ and $\textbf{g}:\textbf{Y} \rightarrow \textbf{Z}$, the premorphism $\textbf{g} \circ \textbf{f}$ (see Definition \ref{def: composite of context morphisms}) is a morphism from $\textbf{X}$ to $\textbf{Z}$.
\end{corollary}

\begin{proof}
Writing $\textbf{Z} = (z_1:Z_1, ..., z_n:Z_n)$ and $\textbf{g} = (g_1, ..., g_n)$, by assumption we have that $\textbf{Y} \vdashcustom g_i:Z_i[\textbf{g}]$ is derivable for $1 \le i \le n$. Since every derivable judgment has the substitution property, $\textbf{X} \vdashcustom g_i[\textbf{f}]:Z_i[\textbf{g}][\textbf{f}]$ is derivable. As the latter equals $\textbf{X} \vdashcustom (\textbf{g} \circ \textbf{f})_i: Z_i[\textbf{g} \circ \textbf{f}]$, we conclude that $\textbf{g} \circ \textbf{f}$ is a morphism from $\textbf{X}$ to $\textbf{Z}$.
\end{proof}

\begin{corollary}
\label{cor: sorts in a context are derivable}
For a context $\textbf{X} = (x_1:X_1, ..., x_n:X_n)$, the judgment $\textbf{X} \vdashcustom X_i \tp$ is derivable for $i = 1$, ..., $n$.
\end{corollary}

\begin{proof}
We proceed by induction. Given $i \in \{1, ..., n\}$, suppose that $\textbf{X} \vdashcustom X_j$ is derivable for all $j < i$. By using (var), we derive $\textbf{X} \vdashcustom x_j:X_j$ for each $j < i$. Then we have a context morphism $(x_1, ..., x_{i-1}):\textbf{X} \rightarrow \partial_{i-1}\textbf{X}$ and, as $\partial_{i-1}\textbf{X} \vdashcustom X_i$ has the substitution property, we conclude that $\textbf{X} \vdashcustom X_i$ is derivable.
\end{proof}

\subsection{Some results on equality judgments}

\begin{proposition}
\label{prop: properties morphism equality}
For a pretheory $\bbT$:
\begin{enumerate}[label=(\arabic*)]		
	\item If $\textbf{f} \equiv \textbf{g}:\textbf{X} \rightarrow \textbf{Y}$ is derivable, then, writing $\textbf{Y} = (y_1:Y_1, ..., y_n:Y_n)$, the judgment
	$$
	\textbf{X} \vdashcustom Y_i[\textbf{f}] \equiv Y_i[\textbf{g}] \tp
	$$
	is derivable for $i = 1$, ..., $n$.
	
	\item If $\textbf{f} \equiv \textbf{g}:\textbf{X} \rightarrow \textbf{Y}$ is derivable, then so are $\textbf{f}:\textbf{X} \rightarrow \textbf{Y}$ and $\textbf{g}:\textbf{X} \rightarrow \textbf{Y}$.
	
	\item If $\textbf{f} \equiv \textbf{g}:\textbf{X} \rightarrow \textbf{Y}$ and $\textbf{Y} \vdashcustom U \tp$ are derivable, then so is
	$$
	\textbf{X} \vdashcustom U[\textbf{f}] \equiv U[\textbf{g}] \tp.
	$$
	
	\item If $\textbf{f} \equiv \textbf{g}:\textbf{X} \rightarrow \textbf{Y}$ and $\textbf{Y} \vdashcustom u:U$ are derivable, then so is
	$$
	\textbf{X} \vdashcustom u[\textbf{f}] \equiv u[\textbf{g}]:U[\textbf{f}].
	$$
	
	\item If $\textbf{f} \equiv \textbf{g}:\textbf{X} \rightarrow \textbf{Y}$ and $\textbf{h}:\textbf{Y} \rightarrow \textbf{Z}$ are derivable, then so is $\textbf{h} \circ \textbf{f} \equiv \textbf{h} \circ \textbf{g}:\textbf{X} \rightarrow \textbf{Z}$.
\end{enumerate}
\end{proposition}

\begin{proof}
We will prove by well-founded induction (with respect to the poset $\omega \times \omega$) that, for all $h$, $n \ge 0$, the above statements hold whenever all the structures involved have height at most $h$, and $\textbf{Y}$ has length-$n$.

Below, for a given $h$ we verify each item under the assumption that all items hold with respect to every $h' < h$.

\begin{enumerate}[label=(\arabic*)]	
	\item If $\textbf{Y} = \varnothing$, the claim holds vacuously. Otherwise, let $i \in \{1, ..., n\}$. Since
	$$
	\Ht(\partial_{i-1}\textbf{f} \equiv \partial_{i-1}\textbf{g}:\textbf{X} \rightarrow \partial_{i-1}\textbf{Y}) \le \Ht(\textbf{f} \equiv \textbf{g}:\textbf{X} \rightarrow \textbf{Y})
	$$
	$$
	\Ht(\partial_{i-1}\textbf{Y} \vdashcustom Y_i) \le \Ht(\textbf{Y}),
	$$
	$$
	\bbl(\partial_{i-1}\textbf{Y}) < \bbl(\textbf{Y}),
	$$
	by item (1) we can derive $\textbf{X} \vdashcustom Y_i[\partial_{i-1}\textbf{f}] \equiv Y_i[\partial_{i-1}\textbf{g}] \tp$, which equals $\textbf{X} \vdashcustom Y_i[\textbf{f}] \equiv Y_i[\textbf{g}] \tp$.
	
	\item Write $\textbf{Y} = (y_1:Y_1, ..., y_n:Y_n)$, $\textbf{f} = (f_1, ..., f_n)$, and $\textbf{g} = (g_1, ..., g_n)$. By assumption, $\textbf{X} \vdashcustom f_i \equiv g_i: Y_i[f_1 \mid x_1, ..., f_{i-1} \mid x_{i-1}]$ is derivable for $i = 1$, ..., $n$. By Proposition \ref{prop: properties height}, so are the judgments $\textbf{X} \vdashcustom f_i:Y_i[f_1 \mid x_1, ..., f_{i-1} \mid x_{i-1}]$, hence also $\textbf{f}:\textbf{X} \rightarrow \textbf{Y}$. On the other hand, $\textbf{g}:\textbf{X} \rightarrow \textbf{Y}$ is obtained from $\textbf{f} \equiv \textbf{g}:\textbf{X} \rightarrow \textbf{Y}$ by using item (1) and (seq/teq).
	
	\item Express $U = V[\textbf{h}]$ for an axiom $\textbf{A} \vdashcustom V \tp$ and a morphism $\textbf{h}:\textbf{Y} \rightarrow \textbf{A}$. Since $\Ht(\textbf{h}:\textbf{Y} \rightarrow \textbf{A}) < \Ht(\textbf{Y} \vdashcustom U \tp)$, by item (5) we can derive $\textbf{h} \circ \textbf{f} \equiv \textbf{h} \circ \textbf{g}:\textbf{X} \rightarrow \textbf{A}$. By item (2), we have context morphisms $\textbf{h} \circ \textbf{f}$, $\textbf{g} \circ \textbf{f}:\textbf{X} \rightarrow \textbf{A}$; it now follows, as $\textbf{A} \vdashcustom V \tp$ has the substitution property, that $\textbf{X} \vdashcustom V[\textbf{h} \circ \textbf{f}] \tp$ and $\textbf{X} \vdashcustom V[\textbf{h} \circ \textbf{g}] \tp$ are derivable. This allows us to obtain from (seq-sub-2) the judgment $\textbf{X} \vdashcustom V[\textbf{h} \circ \textbf{f}] \equiv V[\textbf{h} \circ \textbf{g}] \tp$, which equals $\textbf{X} \vdashcustom U[\textbf{f}] \equiv U[\textbf{g}] \tp$.
	
	\item Express $u = v[\textbf{h}]$ for an axiom $\textbf{A} \vdashcustom v:V \tp$ and a morphism $\textbf{h}:\textbf{Y} \rightarrow \textbf{A}$. Then, by Proposition \ref{prop: properties height},
	$$
	\Ht(\textbf{h}:\textbf{Y} \rightarrow \textbf{A}), \;\; \Ht(\textbf{Y} \vdashcustom V[\textbf{h}] \equiv U \tp) < \Ht(\textbf{Y} \vdashcustom u:U).
	$$
	By item (5), we can derive $\textbf{h} \circ \textbf{f} \equiv \textbf{h} \circ \textbf{g}:\textbf{X} \rightarrow \textbf{A}$, from which (2) yields morphisms $\textbf{h} \circ \textbf{f}$, $\textbf{h} \circ \textbf{g}:\textbf{X} \rightarrow \textbf{A}$. As $\textbf{A} \vdashcustom v:V$ has the substitution property, we can derive $\textbf{X} \vdashcustom v[\textbf{h} \circ \textbf{f}]:V[\textbf{h} \circ \textbf{f}]$ and $\textbf{Y} \vdashcustom v[\textbf{h} \circ \textbf{g}]:V[\textbf{h} \circ \textbf{g}]$.
		
		 Also, by applying (3) to $\textbf{f} \equiv \textbf{g}:\textbf{X} \rightarrow \textbf{Y}$ and $\textbf{Y} \vdashcustom V[\textbf{h}] \tp$ we derive $\textbf{X} \vdashcustom V[\textbf{h}][\textbf{f}] \equiv V[\textbf{h}][\textbf{g}] \tp$; and from the latter, by (s2) and (seq/t), we obtain $\textbf{X} \vdashcustom v[\textbf{h} \circ \textbf{g}]:V[\textbf{h} \circ \textbf{f}]$.
		 
		 Now, note that $\textbf{Y} \vdashcustom V[\textbf{h}] \equiv U \tp$ having the substitution property yields $\textbf{Y} \vdashcustom V[\textbf{h} \circ \textbf{f}] \equiv U[\textbf{f}] \tp$; hence the desired judgment is derived via
		 \small
		 \[
		 \inferrule{\textbf{A} \vdashcustom v:V \\ \textbf{X} \vdashcustom v[\textbf{h} \circ \textbf{f}]:V[\textbf{h}][\textbf{f}] \\ \textbf{X} \vdashcustom v[\textbf{h} \circ \textbf{g}]:V[\textbf{h}][\textbf{f}] \\ \textbf{h} \circ \textbf{f} \equiv \textbf{h} \circ \textbf{g}:\textbf{X} \rightarrow \textbf{A}}{\textbf{X} \vdashcustom v[\textbf{h} \circ \textbf{f}] \equiv v[\textbf{h} \circ \textbf{g}]: V[\textbf{h} \circ \textbf{f}]}\text{(teq-sub-2)},
		 \]
		 \[
		 \inferrule{\textbf{X} \vdashcustom V[\textbf{h} \circ \textbf{f}] \equiv U[\textbf{f}] \tp  \\ \textbf{X} \vdashcustom v[\textbf{h} \circ \textbf{f}] \equiv v[\textbf{h} \circ \textbf{g}]: V[\textbf{h} \circ \textbf{f}] \\ \textbf{X} \vdashcustom v[\textbf{h} \circ \textbf{f}]:U[\textbf{f}] \\ \textbf{X} \vdashcustom v[\textbf{h} \circ \textbf{g}]:U[\textbf{f}]}{\textbf{X} \vdashcustom v[\textbf{h} \circ \textbf{f}] \equiv v[\textbf{h} \circ \textbf{g}]:U[\textbf{f}]}\text{(seq/teq)}.
		 \]
		
		\normalsize
		\item Let $\textbf{Z} = (z_1:Z_1, ..., z_p:Z_p)$ and $\textbf{h} = (h_1, ..., h_p)$. Then for $i = 1$, ..., $p$, the judgment $\textbf{Y} \vdashcustom h_i: Z_i[\textbf{h}]$ is derivable and its height is at most that of $\textbf{h}:\textbf{Y} \rightarrow \textbf{Z}$. By item (4) we obtain $\textbf{X} \vdashcustom h_i[\textbf{f}] \equiv h_i[\textbf{g}]:Z_i[\textbf{h}][\textbf{f}]$, which equals $\textbf{X} \vdashcustom (\textbf{h} \circ \textbf{f})_i \equiv (\textbf{h} \circ \textbf{g})_i: Z_i[\textbf{h} \circ \textbf{f}]$. This means, in turn, that $\textbf{h} \circ \textbf{f} \equiv \textbf{h} \circ \textbf{g}:\textbf{X} \rightarrow \textbf{Z}$ is derivable.
\end{enumerate}
\end{proof}

\subsection{Comparing the derivation rules with Cartmell's}

We now provide the deferred proof that, in a pretheory, a judgment is derivable from the inference rules described above precisely when it is derivable from the ones given in \cite{Car86}.

Written in a different but equivalent form, Cartmell considers the following inference rules (where, when a rule is already in the list considered previously, we only write its name):

\begin{enumerate}
	\item[\textbf{(ctx)}] \hspace{1mm} \textbf{(s1)} \hspace{1mm} \textbf{(s2)} \hspace{1mm} \textbf{(s3)} \hspace{1mm} \textbf{(t1)} \hspace{1mm} \textbf{(t2)} \hspace{1mm} \textbf{(t3)} \hspace{1mm} \textbf{(seq/t)} \hspace{1mm} \textbf{(s-a)} \hspace{1mm} \textbf{(t-a)} \hspace{1mm} \textbf{(seq-a)} \hspace{1mm} \textbf{(teq-a)} \hspace{1mm} \textbf{(s-sub)}\\
	
	\item[\textbf{(seq/teq)'}] $$
	\inferrule{\textbf{X} \vdashcustom U\equiv U' \tp \\ \textbf{X} \vdashcustom u\equiv u':U}{\textbf{X} \vdashcustom u\equiv u':U'}.
	$$
	
	\item[\textbf{(var)'}] $$
	\inferrule{x_1:X_1, ..., x_n:X_n \ctx}{x_1:X_1, ..., x_n:X_n \vdashcustom x_i:X_i}
	$$
	where $n \ge 1$ and $1 \le i \le n$.
	
	\item[\textbf{(t-sub)'}] Suppose given a term axiom $J$, say $\textbf{X} \vdashcustom t(x_1, ..., x_n):U$ where $\textbf{X}$ is $x_1:X_1, ..., x_n:X_n$ (possibly with $n = 0$), a precontext $\textbf{Y}$, and term expressions $f_1$, ..., $f_n$. For $i = 1$, ..., $n$, let $J_i$ be the judgment
	$$
	\textbf{Y} \vdashcustom f_i : X_i[f_1 \mid x_1, ..., f_{i-1} \mid x_{i-1}].
	$$
	We consider
	$$
	\inferrule{J \\ J_1 \\ \cdots \\ J_n}{\textbf{Y} \vdashcustom t(f_1, ..., f_n):U[f_1 \mid x_1, ..., f_n \mid x_n]}.
	$$
	
	\item[\textbf{(seq-sub)'}] Suppose given a sort equality judgment $J$, say $\textbf{X} \vdashcustom U \equiv  U' \tp$ where $\textbf{X} = (x_1:X_1, ..., x_n:X_n)$. Suppose given a precontext $\textbf{Y}$, term expressions $f_1$, ..., $f_n$, $g_1$, ..., $g_n$, and, for $i = 1$, ..., $n$, let $J_i$ be the judgment
	$$
	\textbf{Y} \vdashcustom f_i \equiv  g_i : X_i[f_1 \mid x_1, ..., f_{i-1} \mid x_{i-1}].
	$$
	Then we consider
	$$
	\inferrule{J \\ J_1 \\ \cdots \\ J_n}{\textbf{Y} \vdashcustom U[f_1 \mid x_1, ..., f_n \mid x_n] \equiv  U'[g_1 \mid x_1, ..., g_n \mid x_n] \tp}.
	$$
	
	\item[\textbf{(teq-sub)'}] Suppose given a term equality judgment $J$, say $\Gamma \vdashcustom u \equiv  u' : U$ where $\Gamma$ is $x_1:X_1, ..., x_n:X_n$. Suppose given a precontext $\textbf{Y}$, term expressions $f_1$, ..., $f_n$, $g_1$, ..., $g_n$, and, for $i = 1$, ..., $n$, let $J_i$ be the judgment
	$$
	\textbf{Y} \vdashcustom f_i \equiv  g_i : X_i[f_1 \mid x_1, ..., f_{i-1} \mid x_{i-1}].
	$$
	Then we consider
	$$
	\inferrule{J \\ J_1 \\ \cdots \\ J_n}{\textbf{Y} \vdashcustom u[f_1 \mid x_1, ..., f_n \mid x_n] \equiv  u'[g_1 \mid x_1, ..., g_n \mid x_n] : U[f_1 \mid x_1, ..., f_n \mid x_n]}.
	$$
\end{enumerate}

The idea is that by considering (seq/teq), (var), (t-sub) instead of (seq/teq)', (var)', (t-sub)', more premises are required to derive a given judgment, but the extra premises are themselves proved, inductively, to be derivable. Also, the role of (seq-sub)' is played by (seq-sub-1)' and (seq-sub-2), and that of (teq-sub)' is played by (teq-sub-1) and (teq-sub-2).

In what follows, we will say that a judgment is \emph{C-derivable} if it is derivable from Cartmell's inference rules as listed above. When we say that a judgment is \emph{derivable} we will mean so in the sense of Definition \ref{def: GAT inference rules}.

\begin{proposition}
\label{prop: derivable iff c-derivable}
Let $\bbT$ be a pretheory. Then a judgment is C-derivable if and only if it is derivable.
\end{proposition}

\begin{proof}
Let us start by verifying the claim in the simpler direction: that every derivable judgment is C-derivable -- in other words, Cartmell's rules are at least as strong as the ones from Definition \ref{def: GAT inference rules}.

Firstly, each of Cartmell's rules except for (var)', (seq-sub)' and (teq-sub)' is obtained by removing some (or none) of the premises from a rule in Definition \ref{def: GAT inference rules}.

Also, (seq-sub)' is at least as strong as both (seq-sub-1) and (seq-sub-2): (seq-sub-1) corresponds to the case where the lists of expressions $(f_1, ..., f_n)$ and $(g_1, ..., g_n)$ are equal, and (seq-sub-2) to the case where $U$ and $U'$ are equal. In the first case, starting from the premises of (seq-sub-1), we use (t1) to derive the sequence of term equalities $J_1$, ..., $J_n$ required for (seq-sub)'; in the second case, we use (s1) to derive the sort equality $J' = (\textbf{X} \vdashcustom U \equiv U \tp)$ required for (seq-sub)'.

For (var)', the premise $x_1:X_1, ..., x_n:X_n \vdashcustom X_i \tp$ is replaced by $x_1:X_1, ..., x_n:X_n \ctx$; but the former being C-derivable implies that so is the latter.

\vspace{0.5em}

On the other hand, the proof that the rules from Def. \ref{def: GAT inference rules} are as strong as Cartmell's requires more work.

We will verify that for each of Cartmell's derivation steps that are not in Def. \ref{def: GAT inference rules}, if all of its premises are derivable, then so is its conclusion. We have the following cases:

\begin{itemize}
	\item[\textbf{(seq/teq)'}] $$
	\inferrule{\textbf{X} \vdashcustom U\equiv U' \tp \\ \textbf{X} \vdashcustom u\equiv u':U}{\textbf{X} \vdashcustom u\equiv u':U'}.
	$$
	If the premises are derivable, by Proposition \ref{prop: properties height} so is $\textbf{X} \vdashcustom u:U$. Now, by using (seq/t) we derive $\textbf{X} \vdashcustom u:U'$. Similarly, we can derive $\textbf{X} \vdashcustom u':U'$. The desired judgment is then obtained by
	$$
	\inferrule{\textbf{X} \vdashcustom U\equiv U' \tp \\ \textbf{X} \vdashcustom u\equiv u':U \\ \textbf{X} \vdashcustom u:U' \\ \textbf{X} \vdashcustom u':U'}{\textbf{X} \vdashcustom u\equiv u':U'}\text{(seq/teq)}.
	$$
	
	\item[\textbf{(var)'}] $$
	\inferrule{x_1:X_1, ..., x_n:X_n \ctx}{x_1:X_1, ..., x_n:X_n \vdashcustom x_i:X_i}.
	$$
	Here, by Corollary \ref{cor: sorts in a context are derivable} we can derive $x_1:X_1, ..., x_n:X_n \vdashcustom X_i \tp$, so from (var) we obtain $x_1:X_1, ..., x_n:X_n \vdashcustom x_i:X_i$.
	
	\item[\textbf{(t-sub)'}] $$
	\inferrule{J \\ J_1 \\ \cdots \\ J_n}{\textbf{Y} \vdashcustom t(f_1, ..., f_n):U[f_1 \mid x_1, ..., f_n \mid x_n]},
	$$
	where we use the notation of the list above the proposition's statement. If the premises are derivable, then we have a morphism $\textbf{f}:\textbf{Y} \rightarrow \textbf{X}$ and, as $J$ has the substitution property, $\textbf{Y} \vdashcustom t(f_1, ..., f_n): U[\textbf{f}]$ is derivable, as required.
	
	\item[\textbf{(seq-sub)'}] $$
	\inferrule{J \\ J_1 \\ \cdots \\ J_n}{\textbf{Y} \vdashcustom U[f_1 \mid x_1, ..., f_n \mid x_n] \equiv  U'[g_1 \mid x_1, ..., g_n \mid x_n] \tp},
	$$
	where we use the notation of the list above the proposition's statement. Assume that the premises are derivable. Then, letting $\textbf{f} = (f_1, ..., f_n)$ and $\textbf{g} = (g_1, ..., g_n)$, we have that $\textbf{f} \equiv \textbf{g}:\textbf{X} \rightarrow \textbf{Y}$ is derivable. On the one hand, since $\textbf{f}:\textbf{X} \rightarrow \textbf{Y}$ is a morphism (Proposition \ref{prop: properties morphism equality}(2)) and $\textbf{X} \vdashcustom U \tp$ has the substitution property, $\textbf{Y} \vdashcustom U[\textbf{f}] \equiv U'[\textbf{f}]$ is derivable. On the other hand, Proposition \ref{prop: properties morphism equality}(3) yields $\textbf{Y} \vdashcustom U'[\textbf{f}] \equiv U'[\textbf{g}] \tp$. These two judgments allow us to derive $\textbf{Y} \vdashcustom U[\textbf{f}] \equiv U'[\textbf{g}] \tp$ via (s3).
	
	\item[\textbf{(teq-sub)'}]
	$$
	\inferrule{J \\ J_1 \\ \cdots \\ J_n}{\textbf{Y} \vdashcustom u[f_1 \mid x_1, ..., f_n \mid x_n] \equiv  u'[f'_1 \mid x_1, ..., f'_n \mid x_n] : U[f_1 \mid x_1, ..., f_n \mid x_n]},
	$$
	where we use the notation of the list above the proposition's statement. Assume that the premises are derivable. Then, letting $\textbf{f} = (f_1, ..., f_n)$ and $\textbf{g} = (g_1, ..., g_n)$, we have that $\textbf{f} \equiv \textbf{g}:\textbf{X} \rightarrow \textbf{Y}$ is derivable. On the one hand, since $\textbf{f}:\textbf{X} \rightarrow \textbf{Y}$ is a morphism (Proposition \ref{prop: properties morphism equality}(2)) and $\textbf{X} \vdashcustom u \equiv u':U$ has the substitution property, $\textbf{Y} \vdashcustom u[\textbf{f}] \equiv u'[\textbf{f}]:U[\textbf{f}]$ is derivable. On the other hand, Proposition \ref{prop: properties morphism equality}(4) yields $\textbf{Y} \vdashcustom u'[\textbf{f}] \equiv u'[\textbf{f}']:U[\textbf{f}]$. These two judgments allow us to derive $\textbf{Y} \vdashcustom u[\textbf{f}] \equiv u'[\textbf{f}']:U[\textbf{f}]$ via (t3).
\end{itemize}
This concludes the proof that if a judgment is C-derivable, then it is derivable.
\end{proof}

As a consequence, we can use all known concepts and results related to generalized algebraic theories, such as those in \cite{Car78} and \cite{Car86}. In particular, we will use the definition of a morphism of generalized algebraic theories, the corresponding category $\GAT$, the category $\Cont$ of contextual categories and contextual functors, and the equivalence of categories $\mathcal C:\GAT \rightarrow \Cont$ that assigns to each theory $\bbT$ its \emph{syntactic category}.

\nocite{*}

\printbibliography

@book {AdaRos94,
    AUTHOR = {Ad\'amek, Ji\v r\'i{} and Rosick\'y, Ji\v r\'i},
     TITLE = {Locally presentable and accessible categories},
    SERIES = {London Mathematical Society Lecture Note Series},
    VOLUME = {189},
 PUBLISHER = {Cambridge University Press, Cambridge},
      YEAR = {1994},
     PAGES = {xiv+316},
      ISBN = {0-521-42261-2},
   MRCLASS = {18Axx (18-02)},
  MRNUMBER = {1294136},
MRREVIEWER = {J.\ R.\ Isbell},
       DOI = {10.1017/CBO9780511600579},
       URL = {https://doi.org/10.1017/CBO9780511600579},
shorthand = {AdaRos94}
}

@article {Age92,
    AUTHOR = {Ageron, Pierre},
     TITLE = {The logic of structures},
   JOURNAL = {J. Pure Appl. Algebra},
  FJOURNAL = {Journal of Pure and Applied Algebra},
    VOLUME = {79},
      YEAR = {1992},
    NUMBER = {1},
     PAGES = {15--34},
      ISSN = {0022-4049,1873-1376},
   MRCLASS = {03G30 (03B70 18B99 68Q55)},
  MRNUMBER = {1164119},
MRREVIEWER = {Peter\ Johnstone},
       DOI = {10.1016/0022-4049(92)90123-W},
       URL = {https://doi.org/10.1016/0022-4049(92)90123-W},
}

@article {AhrLum19,
    AUTHOR = {Ahrens, Benedikt and Lumsdaine, Peter LeFanu},
     TITLE = {Displayed categories},
   JOURNAL = {Log. Methods Comput. Sci.},
  FJOURNAL = {Logical Methods in Computer Science},
    VOLUME = {15},
      YEAR = {2019},
    NUMBER = {1},
     PAGES = {Paper No. 20, 18},
      ISSN = {1860-5974},
   MRCLASS = {18A99 (03B15 18D30 68Q55 68T15)},
  MRNUMBER = {3932943},
MRREVIEWER = {Fredrik\ Nordvall Forsberg},
       DOI = {10.23638/LMCS-15(1:20)2019},
       URL = {https://doi.org/10.23638/LMCS-15(1:20)2019},
	shorthand={AhrLum19}
}

@misc{Alm26,
      title={A monoidal category of dependently sorted algebraic theories II: categorical aspects}, 
      author={Daniel Almeida},
      year={2026},
      note={arXiv preprint},
      shorthand={part II}
}

@phdthesis{Bar20,
      title={A model 2-category of enriched combinatorial premodel categories}, 
      author={Reid William Barton},
      year={2020},
      eprint={2004.12937},
      archivePrefix={arXiv},
      primaryClass={math.CT},
      url={https://arxiv.org/abs/2004.12937}, 
}

@article{BarHen25,
      title={Homotopy Languages}, 
      author={César Bardomiano Martínez and Simon Henry},
      year={2025},
      eprint={2510.02607},
      archivePrefix={arXiv},
      primaryClass={math.CT},
      url={https://arxiv.org/abs/2510.02607},
      shorthand={BarHen25}
}

@article {Ben97,
    AUTHOR = {Benson, David B.},
     TITLE = {Multilinearity of sketches},
   JOURNAL = {Theory Appl. Categ.},
  FJOURNAL = {Theory and Applications of Categories},
    VOLUME = {3},
      YEAR = {1997},
     PAGES = {No. 11, 269--277},
      ISSN = {1201-561X},
   MRCLASS = {18C10 (03G30 18A30 68Q65)},
  MRNUMBER = {1482759},
MRREVIEWER = {Peter\ Johnstone},
}

@misc{Ben00,
  author       = {B{\'e}nabou, Jean},
  title        = {Distributors at Work},
  year         = {2000},
  howpublished = {\url{https://www2.mathematik.tu-darmstadt.de/~streicher/FIBR/DiWo.pdf}},
  note         = {June 2000},
}

@phdthesis{Bir84,
  author       = {Bird, Gregory J.},
  title        = {Limits in 2-categories of locally-presented categories},
  school       = {University of Sydney},
  year         = {1984}
}

@article {Ehr68,
    AUTHOR = {Ehresmann, Charles},
     TITLE = {Esquisses et types des structures alg\'ebriques},
   JOURNAL = {Bul. Inst. Politehn. Ia\c si (N.S.)},
  FJOURNAL = {Buletinul Institutului Politehnic din Ia\c si. Serie Nou\u a},
    VOLUME = {14/18},
      YEAR = {1968},
     PAGES = {1--14},
      ISSN = {0032-6100},
   MRCLASS = {18.10},
  MRNUMBER = {238918},
MRREVIEWER = {J.\ Sonner},
}

@article {BasEhr72,
    AUTHOR = {Bastiani, Andr\'ee and Ehresmann, Charles},
     TITLE = {Categories of sketched structures},
   JOURNAL = {Cahiers Topologie G\'eom. Diff\'erentielle},
  FJOURNAL = {Cahiers de Topologie et G\'eom\'etrie Diff\'erentielle},
    VOLUME = {13},
      YEAR = {1972},
     PAGES = {105--214},
      ISSN = {0008-0004},
   MRCLASS = {18A30},
  MRNUMBER = {323856},
MRREVIEWER = {J.\ Sonner},
shorthand = {BasEhr72}
}

@phdthesis{Car78,
  author = {Cartmell, John},
  title = {Generalised algebraic theories and contextual categories},
  school = {Oxford University},
  year = {1978}
}

@article {Car86,
    AUTHOR = {Cartmell, John},
     TITLE = {Generalised algebraic theories and contextual categories},
   JOURNAL = {Ann. Pure Appl. Logic},
  FJOURNAL = {Annals of Pure and Applied Logic},
    VOLUME = {32},
      YEAR = {1986},
    NUMBER = {3},
     PAGES = {209--243},
      ISSN = {0168-0072,1873-2461},
   MRCLASS = {03G30 (18A15 18C10)},
  MRNUMBER = {865990},
MRREVIEWER = {S.\ Comer},
       DOI = {10.1016/0168-0072(86)90053-9},
       URL = {https://doi.org/10.1016/0168-0072(86)90053-9}
}

@article {FioVoe20,
    AUTHOR = {Fiore, Marcelo and Voevodsky, Vladimir},
     TITLE = {Lawvere theories and {C}-systems},
   JOURNAL = {Proc. Amer. Math. Soc.},
  FJOURNAL = {Proceedings of the American Mathematical Society},
    VOLUME = {148},
      YEAR = {2020},
    NUMBER = {6},
     PAGES = {2297--2315},
      ISSN = {0002-9939,1088-6826},
   MRCLASS = {18C10 (03F50 08C99 18C50)},
  MRNUMBER = {4080876},
MRREVIEWER = {Andrzej\ Wi\'snicki},
       DOI = {10.1090/proc/14660},
       URL = {https://doi.org/10.1090/proc/14660},
	shorthand={FioVoe20}
}

@article {Fre66,
    AUTHOR = {Freyd, P.},
     TITLE = {Algebra valued functors in general and tensor products in
              particular},
   JOURNAL = {Colloq. Math.},
  FJOURNAL = {Colloquium Mathematicum},
    VOLUME = {14},
      YEAR = {1966},
     PAGES = {89--106},
      ISSN = {0010-1354,1730-6302},
   MRCLASS = {18.10},
  MRNUMBER = {195920},
MRREVIEWER = {F.\ E. J. Linton},
       DOI = {10.4064/cm-14-1-89-106},
       URL = {https://doi.org/10.4064/cm-14-1-89-106},
}

@article {Fre72,
    AUTHOR = {Freyd, Peter},
     TITLE = {Aspect of topoi},
   JOURNAL = {Bull. Austral. Math. Soc.},
  FJOURNAL = {Bulletin of the Australian Mathematical Society},
    VOLUME = {7},
      YEAR = {1972},
     PAGES = {1--76},
      ISSN = {0004-9727},
   MRCLASS = {18A15},
  MRNUMBER = {396714},
MRREVIEWER = {M.\ Tierney},
       DOI = {10.1017/S0004972700044828},
       URL = {https://doi.org/10.1017/S0004972700044828},
}

@article {Fre25,
    AUTHOR = {Frey, Jonas},
     TITLE = {Duality for Clans: an Extension of Gabriel-Ulmer Duality},
   JOURNAL = {The Journal of Symbolic Logic},
      YEAR = {2025},
     PAGES = {1--38},
       DOI = {10.1017/jsl.2024.79},
       URL = {https://doi.org/10.1017/jsl.2024.79}
}

@book {GabUlm71,
    AUTHOR = {Gabriel, Peter and Ulmer, Friedrich},
     TITLE = {Lokal pr\"asentierbare {K}ategorien},
    SERIES = {Lecture Notes in Mathematics},
    VOLUME = {Vol. 221},
 PUBLISHER = {Springer-Verlag, Berlin-New York},
      YEAR = {1971},
     PAGES = {v+200},
   MRCLASS = {18AXX},
  MRNUMBER = {327863},
MRREVIEWER = {J.\ R.\ Isbell},
shorthand={GabUlm71}
}

@article{Hen16,
      title={Algebraic models of homotopy types and the homotopy hypothesis}, 
      author={Simon Henry},
      year={2016},
      eprint={1609.04622},
      archivePrefix={arXiv},
      primaryClass={math.CT},
      url={https://arxiv.org/abs/1609.04622}, 
}

@article{KapKovAlt19,
author = {Kaposi, Ambrus and Kov\'{a}cs, Andr\'{a}s and Altenkirch, Thorsten},
title = {Constructing quotient inductive-inductive types},
year = {2019},
issue_date = {January 2019},
publisher = {Association for Computing Machinery},
address = {New York, NY, USA},
volume = {3},
number = {POPL},
url = {https://doi.org/10.1145/3290315},
doi = {10.1145/3290315},
journal = {Proc. ACM Program. Lang.},
month = jan,
articleno = {2},
numpages = {24},
shorthand={KapKovAlt19}
}

@article {Kel82,
    AUTHOR = {Kelly, G. M.},
     TITLE = {Basic concepts of enriched category theory},
      NOTE = {Reprint of the 1982 original [Cambridge Univ. Press,
              Cambridge; MR0651714]},
   JOURNAL = {Repr. Theory Appl. Categ.},
  FJOURNAL = {Reprints in Theory and Applications of Categories},
    NUMBER = {10},
      YEAR = {2005},
     PAGES = {vi+137},
   MRCLASS = {18-02 (00B60 18D10 18D20)},
  MRNUMBER = {2177301},
shorthand = {Kel82},
}

@article {Law04,
    AUTHOR = {Lawvere, F. William},
     TITLE = {Functorial semantics of algebraic theories and some algebraic
              problems in the context of functorial semantics of algebraic
              theories},
      NOTE = {Reprinted from Proc. Nat. Acad. Sci. U.S.A. {\bf 50} (1963),
              869--872 [MR0158921] and {\it Reports of the Midwest Category
              Seminar. II}, 41--61, Springer, Berlin, 1968 [MR0231882]},
   JOURNAL = {Repr. Theory Appl. Categ.},
  FJOURNAL = {Reprints in Theory and Applications of Categories},
    NUMBER = {5},
      YEAR = {2004},
     PAGES = {1--121},
   MRCLASS = {18C50},
  MRNUMBER = {2118935},
}

@book{MakPar89,
  author    = {Makkai, Michael and Par{\'e}, Robert},
  title     = {Accessible categories: the foundations of categorical model theory},
  series    = {Contemporary Mathematics},
  volume    = {104},
  publisher = {American Mathematical Society},
  address   = {Providence, RI},
  year      = {1989},
  pages     = {viii+176},
  isbn      = {0-8218-5111-X},
  doi       = {10.1090/conm/104},
shorthand={MakPar89}
}

@article{PalVic07,
	title = {Partial Horn logic and cartesian categories},
	journal = {Annals of Pure and Applied Logic},
	volume = {145},
	number = {3},
	pages = {314-353},
	year = {2007},
	issn = {0168-0072},
	doi = {https://doi.org/10.1016/j.apal.2006.10.001},
	url = {https://www.sciencedirect.com/science/article/pii/S0168007206001229},
	author = {E. Palmgren and S.J. Vickers},
	shorthand={PalVic07}
}

@article {RieVer14,
    AUTHOR = {Riehl, Emily and Verity, Dominic},
     TITLE = {The theory and practice of {R}eedy categories},
   JOURNAL = {Theory Appl. Categ.},
  FJOURNAL = {Theory and Applications of Categories},
    VOLUME = {29},
      YEAR = {2014},
     PAGES = {256--301},
      ISSN = {1201-561X},
   MRCLASS = {55U35 (18D10 18G30)},
  MRNUMBER = {3217884},
MRREVIEWER = {David\ A.\ Blanc},
shorthand={RieVer14}
}

@phdthesis{Sub21,
	TITLE = {{From dependent type theory to higher algebraic structures}},
	AUTHOR = {Leena Subramaniam, Chaitanya},
	URL = {https://theses.hal.science/tel-03369103},
	NOTE = {PhD Thesis, 155 pages, in English (with English and French introductions)},
	SCHOOL = {{Universit{\'e} de Paris}},
	YEAR = {2021},
	MONTH = Sep,
	TYPE = {Theses},
	HAL_ID = {tel-03369103},
	HAL_VERSION = {v1},
}

\end{document}